\documentclass[a4paper,12pt]{amsbook}
\usepackage[latin1]{inputenc}
\usepackage{amsmath,amssymb,amsthm,amsfonts,latexsym,dsfont}
\usepackage[all,knot]{xy}
\usepackage{hyperref}

\usepackage{fancyhdr}

\fancypagestyle{plain}{\fancyhf}
\fancypagestyle{plain}{%
\fancyhead{}

}
\fancypagestyle{plain}{
\fancyhead{}

}

\numberwithin{section}{chapter}

\hypersetup{
    pdftoolbar=true,
    pdfmenubar=true,
    pdffitwindow=false,
    pdfstartview={FitH},
    pdftitle={Lie groups and geometric aspects of isometric and Hamiltonian actions},
    pdfauthor={M. M. Alexandrino, R. G. Bettiol and L. Biliotti},
    pdfsubject={17Bxx, 22C05, 22Exx, 22F05, 53C12, 57R30 (2010 AMS Mathematics Subject Classification)},
    pdfkeywords={Lie groups, Lie algebras, Proper action, Adjoint action, Singular Riemannian foliations}, 
    pdfnewwindow=true,
    colorlinks=true, 
    linkcolor=blue,
    citecolor=blue,
    urlcolor=blue
}

\theoremstyle{plain}
\newtheorem{theorem}{\sc Theorem}[chapter]
\newtheorem{lemma}[theorem]{\sc Lemma}
\newtheorem{proposition}[theorem]{\sc Proposition}
\newtheorem{corollary}[theorem]{\sc Corollary}
\newtheorem{claim}[theorem]{\sc Claim}

\newtheorem{l3thm}[theorem]{\sc Lie's Third Theorem}
\newtheorem{cform}[theorem]{\sc Campbell formulas}
\newtheorem{frobthm}[theorem]{\sc Frobenius Theorem}
\newtheorem{hrthm}[theorem]{\sc Hopf--Rinow Theorem}
\newtheorem{bmthm}[theorem]{\sc Bonnet--Myers Theorem}
\newtheorem{lcthm}[theorem]{\sc Levi--Civita Theorem}
\newtheorem{msthm}[theorem]{\sc Myers--Steenrod Theorem}
\newtheorem{cthm}[theorem]{\sc Cartan Theorem}
\newtheorem{tnthm}[theorem]{\sc Tubular Neighborhood Theorem}
\newtheorem{slthm}[theorem]{\sc Slice Theorem}
\newtheorem{mtthm}[theorem]{\sc Maximal Torus Theorem}
\newtheorem{pothm}[theorem]{\sc Principal Orbit Theorem}
\newtheorem{ptconjecture}[theorem]{\sc Conjecture (Palais and Terng)}
\newtheorem{mconj}[theorem]{\sc Conjecture (Molino)}

\theoremstyle{definition}
\newtheorem{definition}[theorem]{\sc Definition}
\newtheorem{example}[theorem]{\sc Example}

\theoremstyle{remark}
\newtheorem{remark}[theorem]{\it Remark}
\newtheorem{exercise}[theorem]{\sc Exercise}

\numberwithin{equation}{section}

\makeatletter
\renewenvironment{proof}[1][\proofname]{\par
  \pushQED{\qed}%
  \normalfont \topsep6\p@\@plus6\p@\relax
  \trivlist
  \item[\hskip\labelsep
        \it
    #1\@addpunct{.}]\ignorespaces
}{%
  \popQED\endtrivlist\@endpefalse
}
\makeatother

\newcommand{\metric}{\ensuremath{g}}
\newcommand{\dd}{\mathrm{d}}
\newcommand{\Ad}{\ensuremath{\mathrm{Ad} }}
\newcommand{\ad}{\ensuremath{\mathrm{ad} }}
\newcommand{\grad}{\ensuremath{\mathrm{grad}\ }}
\newcommand{\tub}{\ensuremath{\mathrm{Tub} }}
\newcommand{\F}{\ensuremath{\mathcal{F}}}
\newcommand{\Iso}{\ensuremath{\mathrm{Iso}}}
\newcommand{\TO}{\ensuremath{\mathcal{O}}}
\newcommand{\singularF}{\ensuremath{\mathcal{X}_{\mathcal{F}}}}

\newcommand{\imaginario}{i}

\newcommand{\h}{\mathfrak h}

\newcommand{\R}{\mathds{R}}
\newcommand{\Hr}{\mathds{H}}
\newcommand{\Z}{\mathds{Z}}
\newcommand{\N}{\mathds{N}}
\newcommand{\C}{\mathds{C}}
\newcommand\SO{{\rm SO}}
\renewcommand\O{{\rm O}}
\newcommand\SU{{\rm SU}}
\def\Sp{{\rm Sp}}
\newcommand\U{{\rm U}}
\newcommand\GL{{\rm GL}}
\newcommand\SL{{\rm SL}}
\newcommand{\id}{\operatorname{id}}
\newcommand{\Aut}{\operatorname{Aut}}
\newcommand{\End}{\operatorname{End}}
\newcommand{\Zentrum}{\operatorname{Z}}
\newcommand{\Ric}{\operatorname{Ric}}
\newcommand{\scal}{\operatorname{S}}

\newcommand{\trace}{\operatorname{tr}}
\newcommand{\hol}{\mathrm{Hol}}

\def\K{\mbox{$\mathds{K}$}}

\def\CP{\mathds{C}P}

\title{Introduction to Lie groups, isometric and adjoint actions and some generalizations}
\author{Marcos M. Alexandrino \\ Renato G. Bettiol (IME USP, Brazil)}

\makeatletter
\def\thickhrulefill{\leavevmode \leaders \hrule height 1pt\hfill \kern \z@}
\renewcommand{\maketitle}{\begin{titlepage}
\let\footnotesize\small
\let\footnoterule\relax
\parindent \z@
\reset@font
\null
\vskip 10\p@
\hbox{\mbox{\hspace{3em}}
\vrule depth 0.6\textheight
\mbox{\hspace{2em}}
\vbox{
\vskip 45\p@
\begin{flushleft}
\huge \bfseries \@title \par
\end{flushleft}
\vskip 30\p@
\begin{flushleft}
\Large Marcos M. Alexandrino\par
Renato G. Bettiol\par
\vskip 10\p@
\end{flushleft}
\vfil
}}
\null
\end{titlepage}
\setcounter{footnote}{0}
}
\makeatother

\makeindex

\usepackage[usenames]{color}
\definecolor{red}{rgb}{1,0,0}

\begin{document}

\frontmatter
\maketitle

\renewcommand{\thepage}{\roman{page}}

\setcounter{page}{2}
\tableofcontents

\mainmatter
\renewcommand{\thepage}{\roman{page}}


\setcounter{page}{5}
\phantomsection
\chapter*{Preface}

The main purpose of the present lecture notes is to provide an introduction to Lie groups, Lie algebras, and isometric and adjoint actions. We mostly aim at an audience of advanced undergraduate and graduate students. These topics are presented straightforward using tools from Riemannian geometry, which enriches the text with interrelations between these theories and makes it  shorter than usual textbooks on this subject. In addition, the connection between such classic results and the research area of the first author is explored in the final chapters. Namely, generalizations to isoparametric submanifolds, polar actions and singular Riemannian foliations are mentioned, and a brief survey on current research is given in the last chapter. In this way, the text is  divided in two parts.

The goal of the first part (Chapters \ref{chap1} and \ref{chap2}) is to introduce the concepts of Lie groups, Lie algebras and adjoint representation, relating such structures. Moreover, we give basic results on closed subgroups, bi--invariant metrics, Killing forms and splitting in simple ideals. This is done concisely due to the use of Riemannian geometry techniques. To this aim we devote a preliminary section quickly reviewing its fundamental concepts, in order to make the text self--contained.

The second part (Chapters \ref{chap3} to \ref{chap5}) is slightly more advanced, and a few research comments are presented. We begin by recalling some results on proper and isometric actions in Chapter \ref{chap3}. In Chapter \ref{chap4}, classic results from the theory of adjoint action are mentioned, especially regarding maximal tori and roots of compact Lie groups, exploring its connection with isoparametric submanifolds and polar actions. Furthermore, in the last chapter, a few recent research results on such areas are given, together with a discussion of singular Riemannian foliations with sections, which is a generalization of the previous topics.

Prerequisites for the first half of the text are a good knowledge of advanced calculus and linear algebra, together with rudiments of calculus on manifolds. Nevertheless, a brief review of the main definitions and essential results is given in the Appendix \ref{appendix}. In addition, a complete guideline for the second half of the text is given by the first chapters.

The present text evolved from preliminary versions
used by the authors to teach several courses and short courses.
In 2007, 2009 and 2010, graduate courses on Lie groups and proper actions
were taught at the University of S\~ao Paulo, Brazil, exploring mostly the first two parts
of the text. Graduate students working in various fields such as Algebra, Geometry, Topology, Applied Mathematics
and even Physics followed these courses, bringing very positive results. Other relevant contributions to this final
version came from short courses given by the authors during the XV Brazilian School of Differential Geometry (Fortaleza, Brazil, July 2008), the Second S\~ao Paulo Geometry Meeting (S\~ao Carlos, Brazil, February 2009), and the Rey Pastor Seminar at the University of Murcia (Murcia, Spain, July 2009).

There are several important research areas related to the content of this text that will not be treated here. We would like to point out two of these, for which we hope to give the necessary basis in the present notes. Firstly, \emph{representation theory} and \emph{harmonic analysis}, for which we recommend Knapp \cite{Knapp-representation}, Varadarajan \cite{Varadarajan-harmonic}, Gangolli and Varadarajan \cite{Gangolli-Varadarajan}, Katznelson \cite{Katznelson} and Deitmar \cite{Deitmar}. Secondly, \emph{differential equations} and \emph{integrable systems}, for which we recommend Bryant \cite{Bryant}, Guest \cite{Guest}, Olver \cite{Olver}, Noumi \cite{Noumi} and Feh\'er and Pusztai \cite{Feher-Pusztai-1, Feher-Pusztai-2}.

We stress that these lecture notes are still in a \emph{preliminary version}. We expect to improve them in the future and would be grateful for any suggestions, that can be emailed to the authors:

\begin{center}
M. M. Alexandrino: \htmladdnormallink{\tt marcosmalex@yahoo.de}{mailto:marcosmalex@yahoo.de};

\smallskip
R. G. Bettiol: \htmladdnormallink{\tt renatobettiol@gmail.com}{mailto:renatobettiol@gmail.com}.
\end{center}

\noindent
We thank Ren\'ee Abib, L\'aszl\'o Feh\'er, F\'abio Simas and Ion Moutinho for emailed comments on the previous version of these lecture notes.

The first author was supported by CNPq and partially supported by Fapesp. He is very grateful to Gudlaugur Thorbergsson for his consistent support as well as for many helpful discussions along the last years. He also thanks his coauthors Dirk T\"{o}ben, Miguel Angel Javaloyes  and Claudio Gorodski with whom it is a joy to work. He also thanks Rafael Briquet, Leandro Lichtenfelz and Martin Weilandt for their multiplous
comments, corrections and suggestions; and Flausino Lucas Spindola for his many contributions during
his MSc. project in 2008.

The second author was supported by Fapesp, grant 2008/07604-0. He would like to express his profound appreciation
by the constant encouragement given by Paolo Piccione (his former advisor) and Daniel Victor Tausk. Both introduced him
to research in differential geometry and continue providing a unique inspiration with elegant results. Many very special
thanks are also due to the first author and Leonardo Biliotti.

\

\begin{flushright}
S\~{a}o Paulo,

August 2010.
\end{flushright}

\newpage
\setcounter{page}{1}
\renewcommand{\thepage}{\arabic{page}}


\part{Lie groups and Riemannian geometry}
\chapter{Basic results on Lie groups}
\label{chap1}

This chapter gives an introduction to the Lie groups theory, presenting main concepts and proving basic results. For this, some knowledge of group theory, linear algebra and advanced calculus is assumed. However, for the reader's convenience a few facts about differentiable manifolds are recalled in the Appendix~\ref{appendix}, which can be used as a preparation reading.

Further readings for the content of this chapter are Duistermaat and Kolk~\cite{duistermaat}, Gorbatsevich et al.~\cite{gov0}, Helgason~\cite{helgason}, Spivak~\cite{spivak1}, Varadarajan~\cite{varadarajan} and Warner~\cite{warner}.

\section{Lie groups and Lie algebras}
\label{sec:11}

\begin{definition}\label{def:liegroup}
A smooth, respectively analytic, manifold $G$ is said to be a {\it smooth}, respectively {\it analytic}, {\it Lie\footnote{Sophus Lie was a nineteenth century Norwegian mathematician, who laid the foundations of continuous transformation groups. The name ``Lie group'' was introduced by \'Elie Cartan in 1930.} group}\index{Lie group} if $G$ is a group and the maps
\begin{eqnarray}
G\times G\owns (x,y) &\longmapsto & xy\in G\label{eq:mult} \\
G\owns x &\longmapsto & x^{-1}\in G\label{eq:invers}
\end{eqnarray}
are smooth, respectively analytic.
\end{definition}

\begin{remark}\label{re:redundantdef}
The requirement that the maps \eqref{eq:mult} and \eqref{eq:invers} be smooth, respectively analytic, often appears as a requirement that the map $$G\times G\owns (x,y)\longmapsto xy^{-1}\in G$$ be smooth, respectively analytic. Indeed, it is easy to prove that these are equivalent conditions.

Moreover, smoothness, respectively analyticity, of \eqref{eq:invers} is a redundant requirement. In fact, it is possible to use the Implicit Function Theorem to verify that it follows from smoothness, respectively analyticity, of \eqref{eq:mult}, see Exercise~\ref{ex:teorfuncimplicita}.
\end{remark}

In this text we will only deal with smooth Lie groups. Nevertheless, in this direction it is important to mention the next result, which is explored in detail in Duistermaat and Kolk~\cite{duistermaat}.

\begin{theorem}
Each $C^2$ Lie group admits a unique {\it analytic structure}, turning $G$ into an analytic Lie group.
\end{theorem}

We stress that every result proved in this text on smooth Lie groups is hence automatically valid to any analytic Lie group. Henceforth, $G$ will always denote a smooth Lie group, and the word \emph{smooth} will be omitted.

\begin{remark}
One of the outstanding problems in the area was that of determining whether a connected locally Euclidean topological group has a smooth structure. This problem was known as the {\it 5$^\mathrm{th}$ Hilbert's problem}, posed by Hilbert at the International Congress of Mathematics in 1900 and solved by von Neumann in 1933 in the compact case. Only in 1952 the general case was solved, in a joint work of Gleason, Montgomery and Zippen~\cite{montgomery}.
\end{remark}

As almost trivial examples of Lie groups we may consider $(\R^n,+)$, $(S^1,\cdot)$, where the group operation is $e^{i\theta}\cdot e^{i\eta}=e^{i(\theta+\eta)}$; and the $n$--torus $T^n=\underbrace{S^1\times\dots\times S^1}_n$ as a product group.

\begin{example}\label{classicalliegroups}
More interesting examples are the so--called {\it classical Lie groups}\index{Lie group!classical}, which form four families of matrix Lie groups closely related to symmetries of Euclidean spaces. The term {\it classical} appeared in 1940 in Weyl's monograph, probably referring to the {\it classical} geometries in the spirit of Klein's Erlangen program.

We begin with $\GL(n,\R)$, the {\it general linear group} of non singular\footnote{i.e., invertible.} $n\times n$ real matrices. Similarly, $\GL(n,\C)$ and $\GL(n,\Hr)$ are the groups of non singular square matrices over the complex numbers and over the quaternions\footnote{The {\it quaternions} are a non commutative extension of complex numbers, first considered by the Irish mathematician Sir William Hamilton, after whom the usual notation is $\Hr$. It can be thought as a four dimensional normed division algebra over $\R$, or simply as the $\R$-vector space $\R^4$, with the quaternionic multiplication. The canonical basis is $(1,i,j,k)$ and the product of basis elements is given by the equations $i^2=j^2=k^2=ijk=-1$. Any elements of $\Hr$ are of the form $v=a+bi+cj+dk\in\Hr$, and their product is determined by the equations above and the distributive law. Conjugation, norm and division can be defined as natural extensions of $\C$.}, respectively. Furthermore, the following complete the list of classical Lie groups, where $I$ denotes the identity matrix.\index{$\GL(n,\R),\GL(n,\C)$}

\begin{itemize}\index{$\SO(n)$}\index{$\O(n)$}\index{$\U(n)$}\index{$\SU(n)$}\index{$\SL(n,\R),\SL(n,\C)$}
\item[(i)] $\SL(n,\R)=\{M\in\GL(n,\R):\det M=1\}$, $\SL(n,\C)$ and $\SL(n,\Hr)$, the {\it special linear groups};
\item[(ii)] $\O(n)=\{M\in\GL(n,\R):M^tM=I\}$, the {\it orthogonal group}, and $\SO(n)=\O(n)\cap\SL(n,\R)$, the {\it special orthogonal group};
\item[(iii)] $\U(n)=\{M\in\GL(n,\C):M^*M=I\}$, the {\it unitary group}, and $\SU(n)=\U(n)\cap\SL(n,\C)$, the {\it special unitary group};
\item[(iv)] $\Sp(n)=\{M\in\GL(n,\Hr):M^*M=I\}$, the {\it symplectic group}.\index{$\Sp(n)$}
\end{itemize}

In order to verify that those are indeed Lie groups, see Exercise~\ref{ex-classiclie}. For now, we only encourage the reader to bare them in mind as important examples of Lie groups.
\end{example}

\begin{remark}
Another class of examples of Lie groups is constructed by quotients of Lie groups by their normal and closed subgroups (see Corollary~\ref{cor-quocienteLieGroup}). In this class of examples, there are Lie groups that \emph{are not} matrix groups. In fact, consider $$G= \left\{\left( \begin{array}{l l l}
                          1 & a & b\\
                          0&1&c\\
                          0&0& 1\end{array}
                          \right), \; a,b,c\in \R \right\},$$ and $$N= \left\{\left( \begin{array}{l l l}
                          1 & 0 & n\\
                          0&1&0\\
                          0&0& 1\end{array}
                          \right), \; n\in\Z \right\}.$$ Then $G/N$ is a Lie group. It is possible to prove that there are no injective homomorphisms $\varphi: G/N\rightarrow \Aut(V)$, where $\Aut(V)$ denotes the group of linear automorphisms of a finite--dimensional vector space $V$ (see Carter, Segal and MacDonald~\cite{carterSegalMacdonald}).
\end{remark}

\begin{definition}
A {\it Lie algebra}\index{Lie algebra} $\mathfrak{g}$ is a real vector space endowed with a bilinear map $[\cdot,\cdot]:\mathfrak{g}\times\mathfrak{g}\rightarrow\mathfrak{g}$, called the {\it Lie bracket}\index{Lie algebra!bracket}, satisfying for all $X,Y,Z\in\mathfrak g$,
\begin{itemize}
\item[(i)] $[X,Y]=-[Y,X]$ (skew--symmetric property);
\item[(ii)] $[[X,Y],Z]+[[Y,Z],X]+[[Z,X],Y]=0$ (\emph{Jacobi identity}).
\end{itemize}
\end{definition}

\begin{example}\index{$\mathfrak{gl}(n,\R),\mathfrak{gl}(n,\C)$}
Basic examples of Lie algebras are the vector spaces of $n\times n$ square matrices over $\R$ and $\C$, respectively denoted $\mathfrak{gl}(n,\R)$ and $\mathfrak{gl}(n,\C)$, endowed with the Lie bracket given by the commutator of matrices, $[A,B]=AB-BA$.
\end{example}

\begin{exercise}\label{ex-sobre-SO3}
Let $\mathfrak{so}(3)=\{A\in\mathfrak{gl}(3,\R):A^t+A=0\}$.

\begin{itemize}
\item[(i)] Verify that $\mathfrak{so}(3)$ a Lie algebra with Lie bracket given by the matrix commutator;
\item[(ii)] Let $A_{\xi}=\left(\begin{array}{r r r}
                          0 & -\xi_{3} & \xi_{2}\\
                          \xi_{3}&0&-\xi_{1}\\
                          -\xi_{2}&\xi_{1}&0\end{array}
                          \right)$.
Prove that $A_{\xi}v=\xi\times v$.
\item[(iii)] Verify that $A_{\xi\times\eta}=[A_{\xi},A_{\eta}]=A_{\xi}A_{\eta}-A_{\eta}A_{\xi}$. Using the fact that $(\R^3,\times)$ is a Lie algebra endowed with the cross product of vectors, conclude that the map $(\R^3,\times)\owns\xi\mapsto A_{\xi}\in\mathfrak{so}(3)$ is a {\it Lie algebra isomorphism}, i.e., a map that preserves the Lie brackets.
\end{itemize}
\end{exercise}

We now proceed to associate to each Lie group a Lie algebra by considering left--invariant vector fields. For each $g\in G$, denote $L_g$ and $R_g$ the diffeomorphisms of $G$ given by {\it left} and {\it right translation}, respectively. More precisely, $$L_g(x)=gx \;\;\mbox{ and }\;\; R_g(x)=xg.$$

A vector field $X$ on $G$, not assumed a priori to be smooth, is said to be {\it left--invariant}\index{Left--invariant!vector field} if for all $g\in G$, $X$ is $L_g$--related to itself, i.e. $\dd L_g\circ X=X\circ L_g$. This means that $X(gh)=\dd (L_g)_hX(h)$, or shortly $X=\dd L_gX$, for all $g\in G$. Similarly, a vector field is called {\it right--invariant}\index{Right--invariant!vector field} if for all $g\in G$ it is $R_g$--related to itself, meaning that $X=\dd R_gX$, for all $g\in G$. A simultaneously left--invariant and right--invariant vector field is said to be {\it bi--invariant}\index{Bi--invariant!vector field}.

\begin{lemma}
Left--invariant vector fields are smooth.
\end{lemma}

\begin{proof}
Let $X$ be a left--invariant vector field on $G$. Suppose $g\in G$ is an element of an open neighborhood of the identity element $e\in G$ and consider the group operation $$\varphi:G\times G\owns (g,h)\longmapsto gh\in G.$$ Differentiating $\varphi$ we obtain $\dd\varphi:T(G\times G)\simeq TG\times TG\rightarrow TG$, a smooth map. Define $s:G\owns g\mapsto (0_g,X(e))\in TG\times TG$, where $g\mapsto 0_g$ is the null section of $TG$. Since $X=\dd\varphi\circ s$, it follows that $X$ is smooth.
\end{proof}

\begin{remark}
The above result is clearly also valid for \emph{right--invariant} vector fields.
\end{remark}

Given two Lie algebras $\mathfrak{g}_{1}$ and $\mathfrak{g}_{2}$, a linear map $\psi:\mathfrak{g}_{1}\rightarrow\mathfrak{g}_{2}$ is called \emph{Lie algebra homomorphism}\index{Lie algebra!homomorphism} if $$[\psi(X),\psi(Y)]=\psi([X,Y]), \;\; \mbox{ for all }X, Y\in \mathfrak{g}_{1}.$$

\begin{theorem}\label{teo-defLieAlgebraLieGroup}
Let $\mathfrak{g}$ be the set of left--invariant vector fields on the Lie group $G$. Then the following hold.

\begin{itemize}
\item[(i)] $\mathfrak{g}$ is a Lie algebra, endowed with the Lie bracket of vector fields;
\item[(ii)] Consider the tangent space $T_eG$ with the bracket defined as follows. If $X^1,X^2\in T_eG$,  set $[X^1,X^2]= [\widetilde{X^1},\widetilde{X^2}]_{e}$ where $\widetilde{X^i_g}= \dd (L_g)_eX^i$. Define $\psi:\mathfrak{g}\ni X\mapsto X_e\in T_{e}G$. Then $\psi$ is a Lie algebra isomorphism, where $\mathfrak{g}$ is endowed with the Lie bracket of vector fields and $T_{e}G$ with the bracket defined above.
\end{itemize}
\end{theorem}

\begin{proof}
First, note that $\mathfrak{g}$ has an $\R$--vector space structure, which indeed follows directly from linearity of $\dd (L_g)_e$. It is not difficult to see that the Lie bracket of vector fields is a Lie bracket, i.e. it is skew--symmetric and satisfies the Jacobi identity. Equation \eqref{bracketrelated} implies that the Lie bracket of left--invariant vector fields is still left--invariant. Hence $\mathfrak{g}$ is a Lie algebra, proving (i).

To prove item (ii), we first claim that the map $\psi$ is injective. Indeed, if $\psi(X)=\psi(Y)$, for each $g\in G$, $X(g)=\dd L_g(X(e))=\dd L_g(Y(e))=Y(g)$. Furthermore, it is also surjective, since for each $v \in T_eG$, $X(g)=\dd L_g(v)$ is such that $\psi(X)=v$, and $X$ is clearly left--invariant. Therefore, $\psi$ is a linear bijection between two $\R$--vector spaces, hence an isomorphism. From the definition of Lie bracket on $T_{e}G$ we have $[\psi(X),\psi(Y)]= [X,Y]_{e}=\psi([X,Y])$. Thus $\psi$ is a Lie algebra isomorphism.
\end{proof}

\begin{definition}
The {\it Lie algebra of the Lie group $G$}\index{Lie algebra!of a Lie group} is the Lie algebra $\mathfrak{g}$ of left--invariant vector fields on $G$. According to the above theorem, $\mathfrak{g}$ could be equivalently defined as the tangent space $T_eG$ with the bracket defined as in (ii).
\end{definition}

In this way, a Lie group $G$ gives rise to a canonically determined Lie algebra $\mathfrak{g}$. A converse is given in the following interesting and important result, proved in Duistermaat and Kolk~\cite{duistermaat}.

\begin{l3thm}\label{liethird}\index{Theorem!Lie's Third}
Let $\mathfrak{g}$ be a Lie algebra. Then there exists a unique connected and simply connected Lie group $G$ with Lie algebra isomorphic to $\mathfrak{g}$.
\end{l3thm}

\begin{remark}
We will not prove the existence of $G$, but uniqueness will follow from Corollary~\ref{cor-gruposisomorfosSeAlgebrasisomorfas}.
\end{remark}

We end this section with two exercises. The first is related to Remark~\ref{re:redundantdef}, and the second completes Exercise~\ref{ex-sobre-SO3}, verifying that $\mathfrak{so}(3)$ is the Lie algebra of $\SO(3)$.

\begin{exercise}\label{ex:teorfuncimplicita}
Consider $G$ a smooth manifold with a smooth group structure, i.e., such that $$\mu:G\times G\owns (x,y) \longmapsto xy\in G$$ is smooth.

\begin{itemize}
\item[(i)] Fix $x_0,y_0\in G$ and verify that $\frac{\partial\mu}{\partial y}(x_0,y_0):T_{y_0}G\to T_{x_0y_0}G$ is an isomorphism, computing it in terms of the derivative of $L_{x_0}:G\to G$;
\item[(ii)] Use the Implicit Function Theorem to solve $xy=e$ in a neighborhood of $x_0\in G$, obtaining a map $y=y(x)$ between open neighborhoods of $x_0$ and $y_0$, such that $xy=e$ if and only if $y=y(x)$;
\item[(iii)] Conclude that $G\owns x\mapsto x^{-1}\in G$ is smooth.
\end{itemize}
\end{exercise}

\begin{exercise}\label{ex2-sobre-SO3}\index{$\SO(n)$}
Assume the following result to be seen in Exercise~\ref{ex-ad-matrizes}. The tangent space at the identity to a Lie subgroup of $\GL(n,\R)$ endowed with the matrix commutator is isomorphic to its Lie algebra.

Consider $\mathcal{S}\subset\GL(3,\R)$ the subspace formed by symmetric matrices. Define the map $\varphi:\GL(3,\R)\rightarrow\mathcal{S}$ given by $\varphi(A)=AA^t$ and denote $I$ the identity.

\begin{itemize}
\item[(i)] Verify that the the kernel of $\dd\varphi(I):\mathfrak{gl}(3,\R)\rightarrow\mathcal{S}$ is the subspace of skew--symmetric matrices in $\mathfrak{gl}(3,\R)$;
\item[(ii)] Prove that, for all $A\in\GL(3,\R)$, $$\ker \dd\varphi(A)=\{H\in\mathfrak{gl}(3,\R) : A^{-1}H\in\ker \dd\varphi(I)\}$$ and conclude that $\dim\ker\dd\varphi(A)=3$, for all $A\in\GL(3,\R)$;
\item[(iii)] Prove that $I$ is a regular value of $\varphi$;
\item[(iv)] Conclude that the Lie group $\O(3)$ may be written as preimage of this element by $\varphi$ and calculate its dimension.
\item[(v)] Recall that $\SO(3)$ is the subgroup of $\O(3)$ given by the connected component of $I$. Prove that $T_I\SO(3)=\ker\dd\varphi(I)=\mathfrak{so}(3)$. Finally, conclude that $\mathfrak{so}(3)$ is the Lie algebra\index{Lie algebra!of $\SO(3)$} of $\SO(3)$.
\end{itemize}
Observe that analogous results hold for the $n$--dimensional case, using the same techniques as above.
\end{exercise}

\section{Lie subgroups and Lie homomorphisms}

The aim of this section is to establish some basic relations between Lie algebras and Lie groups, and their sub objects and morphisms. 

A group homomorphism between Lie groups $\varphi:G_1\rightarrow G_2$ is called a {\it Lie group homomorphism}\index{Lie group!homomorphism} if it is also smooth. In the sequel, we will prove that continuity is in fact a sufficient condition for a group homomorphism between Lie groups to be smooth, see Corollary~\ref{homosmooth}. Recall that given two Lie algebras $\mathfrak{g}_1,\mathfrak{g}_2$, a linear map $\psi:\mathfrak{g}_1\rightarrow\mathfrak{g}_2$ is a {\it Lie algebra homomorphism}\index{Lie algebra!homomorphism} if $\psi([X,Y])=[\psi(X),\psi(Y)]$, for all $X,Y\in\mathfrak{g}_1$.

A {\it Lie subgroup}\index{Lie group!subgroup} $H$ of a Lie group $G$ is an abstract subgroup such that $H$ is an immersed submanifold of $G$ and $$H\times H\owns (x,y)\longmapsto xy^{-1}\in H$$ is smooth. In addition, if $\mathfrak{g}$ is a Lie algebra, a subspace $\mathfrak{h}$ is a {\it Lie subalgebra}\index{Lie algebra!subalgebra} if it is closed with respect to the Lie bracket.

\begin{proposition}\label{Lsubgroup}
Let $G$ be a Lie group and $H\subset G$ an embedded submanifold of $G$ that is also a group with respect to the group operation of $G$. Then $H$ is a closed Lie subgroup of $G$.
\end{proposition}

\begin{proof}
Consider the map $f:H\times H\owns (x,y)\mapsto xy^{-1}\in G$. Then $f$ is smooth and $f(H\times H)\subset H$. Since $H$ is embedded in $G$, from Proposition~\ref{quasiembedded}, $f:H\times H\rightarrow H$ is smooth. Hence $H$ is a Lie subgroup of $G$.

It remains to prove that $H$ is a closed subgroup of $G$. Consider a sequence $\{h_n\}$ in $H$ that converges to $g_0\in G$. Since $H$ is an embedded submanifold, there exists a neighborhood $W$ of $e$ and a chart $\varphi=(x_1,\ldots,x_k):W\to\R^k$ such that
$$W\cap H=\{g\in G\cap W:x_i(g)=0, i=1,\ldots k\}.$$ The fact that the map $G\times G\ni (x,y)\mapsto x^{-1}y\in G$ is continuous implies the existence of neighborhoods $U$ and $V$ of $e$ such that $U^{-1}V\subset W$. For each $n$ and $m$, set $u_n=g_{0}^{-1}h_{n}\in U$ and $v_{m}=g_{0}^{-1}h_{m}\in V$. By construction, $u^{-1}_{n}v_{m}\in W\cap H$ and hence $x_{i}(u^{-1}_{n}v_{m})=0$. Therefore, $x_{i}(u^{-1}_{n})=0$, since $\{v_{m}\}$ converges to $e$, and $u_{n}^{-1}=h_{n}^{-1}g_{0}\in H$. Thus, $g_{0}\in H$, concluding the proof.
\end{proof}

Another important condition under which subgroups are Lie subgroups will be given in Section~\ref{sec-closedsubgroups}. We now investigate the nature of the Lie algebra of a Lie subgroup.

\begin{lemma}\label{uniquerelated}
Let $G_1$ and $G_2$ be Lie groups and $\varphi:G_1\rightarrow G_2$ be a Lie group homomorphism. Then given any left--invariant vector field $X\in \mathfrak{g}_1$, there exists a unique left--invariant vector field $Y\in\mathfrak{g}_2$ that is $\varphi$--related to $X$.
\end{lemma}

\begin{proof}
First, if $Y\in\mathfrak{g}_2$ is $\varphi$--related to $X$, since $\varphi$ is a Lie group homomorphism, $Y_e=\dd\varphi_e X_e$. Hence, uniqueness follows from left--invariance of $Y$.

Define $Y=\dd (L_g)_e(\dd\varphi_e(X_e))$. It remains to prove that $Y$ is $\varphi$--related to $X$. Observing that $\varphi$ is a Lie group homomorphism, $\varphi\circ L_g=L_{\varphi(g)}\circ \varphi$, for all $g\in G_1$. Therefore, for each $g\in G_1$,
\begin{eqnarray*}
\dd(\varphi)_g(X_g) &=& \dd\varphi_g(\dd(L_g)_e X_e)\\
&=& \dd(\varphi \circ L_g)_eX_e\\
&=& \dd(L_{\varphi(g)}\circ\varphi)_eX_e\\
&=& \dd(L_{\varphi(g)})_e(\dd\varphi_eX_e)\\
&=& Y(\varphi(g)),
\end{eqnarray*}

\noindent
which proves that $Y$ is $\varphi$--related to $X$, completing the proof.
\end{proof}

\begin{proposition}\label{dphi}
Let $G_1$ and $G_2$ be Lie groups and $\varphi:G_1\rightarrow G_2$ be a Lie group homomorphism. Then $\dd\varphi_e:\mathfrak{g}_1\rightarrow\mathfrak{g}_2$ is a Lie algebra homomorphism.
\end{proposition}

\begin{proof}
We want to prove that $\dd\varphi_{e}:(T_{e}G_{1},[\cdot,\cdot])\rightarrow (T_{e}G_{2},[\cdot,\cdot])$ is a Lie algebra homomorphism, where the Lie bracket $[\cdot,\cdot]$ on $T_{e}G_{i}$ was defined in Theorem~\ref{teo-defLieAlgebraLieGroup}. For each $X^{1}, X^{2}\in T_{e}G_{1}$, define the vectors $Y^{i}=\dd\varphi_{e}X^{i}\in T_{e}G_{2}$ and extend them to left--invariant vector fields $\widetilde{X}^{i}\in\mathfrak{X}(G_{1})$ and $\widetilde{Y}^{i}\in\mathfrak{X}(G_{2})$, by setting $\widetilde{X}^{i}_{g}=\dd L_{g}X^{i}$ and $\widetilde{Y}^{i}_{g}=\dd L_{g}Y^{i}$.

On the one hand, it follows from the definition of the Lie bracket in $T_{e}G_{2}$ that $$[\dd\varphi_{e}X^{1},\dd\varphi_{e}X^{2}]=[Y^{1},Y^{2}]=[\widetilde{Y}^{1},\widetilde{Y}^{2}]_{e}.$$ On the other hand, it follows from Lemma~\ref{uniquerelated} that $\widetilde{X}^i$ and $\widetilde{Y}^i$ are $\varphi$--related, and from \eqref{bracketrelated}, $[\widetilde{X}^1,\widetilde{X}^2]$ and $[\widetilde{Y}^1,\widetilde{Y}^2]$ are $\varphi$--related. Therefore $$[\widetilde{Y}^{1},\widetilde{Y}^{2}]_{e}=\dd\varphi_{e}[\widetilde{X}^{1},\widetilde{X}^{2}]_{e}=\dd\varphi_{e}[X^{1},X^{2}].$$ The two equations above imply that $[\dd\varphi_{e}X^{1},\dd\varphi_{e}X^{2}]=\dd\varphi_{e}[X^{1},X^{2}]$, concluding the proof.
\end{proof}

\begin{corollary}\label{subalgebrainclusion}
Let $G$ be a Lie group and $H\subset G$ a Lie subgroup. Then the inclusion map $i:H\hookrightarrow G$ induces an isomorphism $\dd i_e$ between the Lie algebra $\mathfrak h$ of $H$ and a Lie subalgebra $\dd i_e(\mathfrak h)$ of $\mathfrak{g}$.
\end{corollary}

We will see below  a converse result, on conditions under which a Lie subalgebra gives rise to a Lie subgroup.
Before that, we give a result on how an open neighborhood of the identity element can \emph{generate} the whole Lie group.


\begin{proposition}\label{vizger}
Let $G$ be a Lie group and $G^0$ be the connected component of $G$ to which the identity element $e\in G$ belongs. Then $G^0$ is a normal Lie subgroup of $G$ and connected components of $G$ are of the form $gG^0$, for some $g\in G$. Moreover, given an open neighborhood $U$ of $e$, then $G^0=\bigcup_{n\in\N} U^n$, where $U^n=\{g_1^{\pm 1}\cdots g_n^{\pm 1}:g_i\in U\}$.
\end{proposition}

\begin{proof}
Since $G^0$ is a connected component of $G$, it is an open and closed subset of $G$. In order to verify that it is also a Lie subgroup, let $g_0\in G^0$ and consider $g_0G^0=L_{g_0}(G^0)$. Note that $g_0G^0$ is a connected component of $G$, once $L_{g_0}$ is a diffeomorphism. Since $g_0\in G^0\cap g_0G^0$, it follows from the maximality of the connected component that $g_0G^0=G^0$. Similarly, since the inversion map is also a diffeomorphism, the subset $G^0_{-1}=\{g_0^{-1}:g_0\in G^0\}$ is connected, with $e\in G^0_{-1}$. Hence, $G^0_{-1}=G^0$, using the same argument. Therefore, $G^0$ is a subgroup of $G$ and an embedded submanifold of $G$. From Proposition~\ref{Lsubgroup}, it follows that $G^0$ is a Lie subgroup of $G$.

In addition, for each $g\in G$, consider the diffeomorphism given by the conjugation $x\mapsto gxg^{-1}$. Using the same argument of maximality of the connected component, one may conclude that $gG^0g^{-1}=G^0$, for all $g\in G$, hence $G^0$ is normal. The proof that the connected component of $G$ containing $g$ is $gG^{0}$ can be similarly done.

Finally, being $G^0$ connected, to show that $G^0=\bigcup_{n\in\N} U^n$ it suffices to check that $\bigcup_{n\in\N} U^n$ is open and closed in $G^0$. It is clearly open, since $U$ (hence $U^n$) is  open. To verify that it is also closed, let $h\in G^0$ be the limit of a sequence $\{h_j\}$  in $\bigcup_{n\in\N} U^n$, i.e., $\lim h_j=h$. Since $U^{-1}=\{u^{-1}:u\in U\}$ is an open neighborhood of $e\in G$, $hU^{-1}$ is an open neighborhood of $h$. From the convergence of the sequence $\{h_j\}$, there exists $j_0\in\N$ such that $h_{j_0}\in hU^{-1}$, that is, there exists $u\in U$ such that $h_{j_0}=hu^{-1}$. Hence $h=h_{j_0}u\in\bigcup_{n\in\N} U^n$. Therefore this set is closed in $G^0$, concluding the proof.
\end{proof}

\begin{theorem}\label{integralsubgroup}
Let $\mathfrak{g}$ be the Lie algebra of $G$ and $\mathfrak{h}$ a Lie subalgebra of $\mathfrak{g}$. There exists a unique connected Lie subgroup $H\subset G$ with Lie algebra $\mathfrak{h}$.
\end{theorem}

\begin{proof}
Define the distribution $D_q=\{X_q:X_q=\dd L_qX \mbox{ for } X\in\mathfrak{h}\}$. Since $\mathfrak{h}$ is a Lie algebra, $D$ is involutive. It follows from the Frobenius Theorem (see Theorem~\ref{frobbis}) that there exists a unique foliation $\F=\{F_q:q\in G\}$ tangent to the distribution, i.e, $D_q=T_qF_q$, for all $q\in G$. Define $H\subset G$ to be the leaf passing through the identity, $H=F_e$.

Note that the map $L_g$ takes leaves to leaves, once $\dd L_gD_a=D_{ga}$. Furthermore, for each $h\in H$, $L_{h^{-1}}(H)$ is the leaf passing through the identity. Therefore $L_{h^{-1}}(H)=H$, which means that $H$ is a group. Finally, consider $\psi:H\times H\owns (x,y)\mapsto x^{-1}y\in H$. It also follows from the Frobenius Theorem that $H$ is quasi--embedded. Therefore, since the inclusion $i:H\hookrightarrow G$ and $i\circ \psi$ are smooth, $\psi$ is also smooth. Hence $H$ is a connected Lie subgroup of $G$ with Lie algebra $\mathfrak{h}$. The uniqueness follows from the Frobenius Theorem and Proposition~\ref{vizger}.
\end{proof}

\begin{definition}
A smooth surjective map $\pi:E\rightarrow B$ is a {\it covering map}\index{Covering map} if for each $p\in B$ there exists an open neighborhood $U$ of $p$ such that $\pi^{-1}(U)$ is a disjoint union of open sets $U_\alpha\subset E$ mapped diffeomorphically onto $U$ by $\pi$. That is, $\pi|_{U_\alpha}:U_\alpha\rightarrow U$ is a diffeomorphism for each $\alpha$.
\end{definition}

\begin{theorem}\label{liecovering}
Given $G$ a connected Lie group, there exist a unique simply connected Lie group $\widetilde{G}$ and a Lie group homomorphism $\pi:\widetilde{G}\rightarrow G$ which is also a covering map.
\end{theorem}

A proof of this theorem can be found in Boothby~\cite{boothby} or Duistermaat and Kolk~\cite{duistermaat}. An example of covering map that is also a Lie group homomorphism is the usual covering $\pi:\R^n\rightarrow T^n$ of the $n$--torus by Euclidean space. Furthermore, in Chapter~\ref{chap3}, others results on covering maps among Lie groups will be given, such as the classic example of $\SU(2)$ covering $\SO(3)$, see Exercise~\ref{su2so3}.

We assume Theorem~\ref{liecovering} in order to continue towards a more precise description of the relation between Lie groups and Lie algebras.

\begin{proposition}\label{coveringisomorphism}
Let $G_1$ and $G_2$ be connected Lie groups and $\pi:G_1\rightarrow G_2$ be a Lie group homomorphism. Then $\pi$ is a covering map if, and only if, $\dd\pi_e$ is an isomorphism.
\end{proposition}

\begin{proof}
Suppose that $\dd\pi_{e_1}:T_{e_1}G_1\rightarrow T_{e_2}G_2$ is an isomorphism, where $e_i\in G_i$ is the identity element of $G_i$. We claim that $\pi$ is surjective.

Indeed, since $\dd\pi_{e_1}$ is an isomorphism, from the Inverse Function Theorem, there exist open neighborhoods $U$ of the identity $e_1\in G_1$ and $V$ of the identity $e_2\in G_2$, such that $\pi(U)=V$ and $\pi|_U$ is a diffeomorphism. Let $h\in G_2$. From Proposition~\ref{vizger}, there exist $h_i\in V$ such that $h=h_1^{\pm 1}\cdots h_n^{\pm 1}$. Since $\pi|_U$ is a diffeomorphism, for each $1\leq i\leq n$, there exists a unique $g_i\in U$ such that $\pi(g_i^{\pm 1})=h_i^{\pm 1}$. Hence \begin{eqnarray*}
\pi(g_1^{\pm 1}\cdots g_n^{\pm 1}) &=& \pi(g_1^{\pm 1})\cdots\pi(g_n^{\pm 1})\\
&=& h_1^{\pm 1}\cdots h_n^{\pm 1}\\
&=& h.
\end{eqnarray*}

Therefore $\pi$ is a surjective Lie group homomorphism. Hence given $q\in G_2$, there exists $p\in G_1$ with $\pi(p)=q$.

Let $\{g_\alpha\}=\pi^{-1}(e_2)$. Using the fact that $\pi$ is a Lie group homomorphism and $\dd\pi_{e_1}$ is an isomorphism, one can prove that $\{g_\alpha\}$ is discrete. Therefore $\pi$ is a covering map, since
\begin{itemize}
\item[(i)] $\pi^{-1}(qV)=\bigcup_\alpha g_\alpha p\cdot U$;
\item[(ii)] $(g_\alpha pU)\cap (pU)=\emptyset$, if $g_{\alpha}\neq e_{1}$;
\item[(iii)] $\pi\big|_{g_\alpha pU}:g_\alpha pU\rightarrow qV$ is a diffeomorphism.
\end{itemize}
Conversely, if $\pi$ is a covering map, it is locally a diffeomorphism, hence $\dd\pi_{e_1}$ is an isomorphism.
\end{proof}

Let $G_1$ and $G_2$ be Lie groups and $\theta:\mathfrak{g}_1\rightarrow\mathfrak{g}_2$ be a Lie algebra homomorphism. We will prove that if $G_1$ is connected and simply connected then $\theta$ induces a Lie group homomorphism. 
We begin by proving uniqueness in the following lemma.

\begin{lemma}\label{connectedunique}
Let $G_1$ and $G_2$ be Lie groups, with identities $e_1$ and $e_2$ respectively, and $\theta:\mathfrak{g}_1\rightarrow\mathfrak{g}_2$ be a fixed Lie algebra homomorphism. If $G_1$ is connected, and $\varphi,\psi:G_1\rightarrow G_2$ are Lie group homomorphisms with $\dd\varphi_{e_1}=\dd\psi_{e_1}=\theta$, then $\varphi=\psi$.
\end{lemma}

\begin{proof}
It is easy to see that the direct sum of Lie algebras $\mathfrak{g}_1\oplus \mathfrak{g}_2$, respectively product of Lie groups $G_1\times G_2$, has a natural Lie algebra, respectively Lie group, structure.

Consider the Lie subalgebra of $\mathfrak{g}_1\oplus\mathfrak{g}_2$ given by $$\mathfrak{h}=\{(X,\theta(X)):X\in\mathfrak{g}_1\},$$ i.e., the graph of $\theta:\mathfrak{g}_1\rightarrow\mathfrak{g}_2$. It follows from Theorem~\ref{integralsubgroup} that there exists a unique connected Lie subgroup $H\subset G_1\times G_2$ with Lie algebra $\mathfrak{h}$.

Suppose now that $\varphi:G_1\rightarrow G_2$ is a Lie group homomorphism with $\dd\varphi_{e_1}=\theta$. Then $$\sigma:G_1\owns g\longmapsto (g,\varphi(g))\in G_1\times G_2$$ is a Lie group homomorphism, and $$\dd\sigma_{e_1}:\mathfrak{g}_1\owns X\longmapsto (X,\theta(X))\in\mathfrak{g}_1\oplus\mathfrak{g}_2$$ is a Lie algebra homomorphism. Note that $\sigma(G_1)$ is the graph of $\varphi$, hence embedded in $G_1\times G_2$. From Proposition~\ref{Lsubgroup}, $\sigma(G_1)$ is a Lie subgroup of $G_1\times G_2$, with Lie algebra $\mathfrak{h}=\dd\sigma_{e_1}(\mathfrak{g}_1)$. Therefore, from Theorem~\ref{integralsubgroup}, $\sigma(G_1)=H$. In other words, $H$ is the graph of $\varphi$. If $\psi:G_1\rightarrow G_2$ is another Lie group homomorphism with $\dd\psi_{e_1}=\theta$, following the same construction above, the graph of $\psi$ and $\varphi$ would be both equal to $H$, hence $\varphi=\psi$.
\end{proof}

\begin{theorem}\label{homo}
Let $G_1$ and $G_2$ be Lie groups and $\theta:\mathfrak{g}_1\rightarrow\mathfrak{g}_2$ be a Lie algebra homomorphism. There exist an open neighborhood $V$ of $e_1$ and a smooth map $\varphi:V\rightarrow G_2$ that is a local homomorphism\footnote{This means that $\varphi(ab)=\varphi(a)\varphi(b)$, for all $a,b\in V$ such that $ab\in V$.}, with $\dd\varphi_{e_1}=\theta$. In addition, if $G_1$ is connected and simply connected, there exists a unique Lie group homomorphism $\varphi:G_1\rightarrow G_2$ with $\dd\varphi_{e_1}=\theta$.
\end{theorem}

\begin{proof}
Using the same notation from the proof of Lemma~\ref{connectedunique}, let $\mathfrak{h}=\{(X,\theta(X)):X\in\mathfrak{g}_1\}$ be a Lie algebra given by the graph of $\theta$. Then, according to Theorem~\ref{integralsubgroup}, there exists a unique connected Lie subgroup $H\subset G_1\times G_2$ with Lie algebra $\mathfrak{h}$, whose identity element will be denoted $\widetilde{e}\in H$. Consider the inclusion map $i:H\hookrightarrow G_1\times G_2$ and the natural projections $\pi_1:G_1\times G_2\rightarrow G_1$, $\pi_2:G_1\times G_2\rightarrow G_2$. The map $\pi_1\circ i:H\rightarrow G_1$ is a Lie group homomorphism, such that $\dd(\pi_1\circ i)_{\widetilde{e}}(X,\theta(X))=X$, for all $X\in T_{e_1}G_1$. It follows from the Inverse Function Theorem that there exist open neighborhoods $U$ of $\widetilde{e}$ in $H$ and $V$ of $e_1$ in $G_1$, such that $(\pi_1\circ i)|_U:U\rightarrow V$ is a diffeomorphism.

Define $\varphi=\pi_2\circ(\pi_1\circ i)^{-1}:V\rightarrow G_2$. Then $\varphi$ is a local homomorphism and $\dd\varphi_{e_1}=\theta$. In fact, for each $X\in T_{e_1}G_1$, \begin{eqnarray*}
\dd\varphi_{e_1}(X) &=&\dd(\pi_2\circ (\pi_1\circ i)^{-1})_{e_1}(X)\\
&=& \dd(\pi_2)_{\widetilde{e}} \ \dd[(\pi_1\circ i)^{-1}]_{e_1} X \\
&=& \dd(\pi_2)_{\widetilde{e}}(X,\theta(X)) \\
&=& \theta(X).
\end{eqnarray*}

Furthermore, $\pi_1\circ i$ is a Lie group homomorphism and $\dd(\pi_1\circ i)_{\widetilde{e}}$ is an isomorphism. From Proposition~\ref{coveringisomorphism}, it follows that $\pi_1\circ i:H\rightarrow G_1$ is a covering map. Supposing $G_1$ simply connected, since a covering map onto a simply connected space is a diffeomorphism, it follows that $\pi_1\circ i$ is a diffeomorphism. Hence it makes sense to invert $\pi_1\circ i$ globally, and we obtain a global homomorphism $\varphi=\pi_2\circ(\pi_1\circ i)^{-1}:G_1\rightarrow G_2$, with $\dd\varphi_{e_1}=\theta$. The uniqueness of $\varphi$ follows from Lemma~\ref{connectedunique}, under the assumption that $G_1$ is simply connected.
\end{proof}

\begin{corollary}\label{cor-gruposisomorfosSeAlgebrasisomorfas}
If $G_1$ and $G_2$ are connected and simply connected and $\theta:\mathfrak{g}_1\rightarrow\mathfrak{g}_2$ is an isomorphism, there exists a unique Lie group isomorphism $\varphi:G_1\rightarrow G_2$, with $\dd\varphi_{e_1}=\theta$. In other words, if $G_1$ and $G_2$ are as above, then $G_1$ and $G_2$ are isomorphic if, and only if, $\mathfrak{g}_1$ and $\mathfrak{g}_2$ are isomorphic.
\end{corollary}

\begin{proof}
From Theorem~\ref{homo}, there exists a unique Lie group homomorphism $\varphi:G_1\rightarrow G_2$ with $\dd\varphi_{e_1}=\theta$. From Proposition~\ref{coveringisomorphism}, since $\dd\varphi_{e_1}=\theta$ is an isomorphism, $\varphi$ is a covering map. Since $G_2$ is simply connected, $\varphi$ is also a diffeomorphism. Hence $\varphi$ is a Lie group homomorphism and a diffeomorphism, therefore an isomorphism.
\end{proof}

\section{Exponential map and adjoint representation}

In this section, we introduce the concept of Lie exponential map and adjoint representation, discussing a few results.

Let $G$ be a Lie group and $\mathfrak{g}$ its Lie algebra. We recall that a Lie group homomorphism $\varphi:\R\rightarrow G$ is called a {\it $1$--parameter subgroup}\index{1--parameter subgroup} of $G$. Let $X\in\mathfrak{g}$ and consider the Lie algebra homomorphism $$\theta:\R\owns t\longmapsto tX\in\R\cdot X.$$ From Theorems~\ref{integralsubgroup} and~\ref{homo}, there exists a unique $1$--parameter subgroup $\lambda_X:\R\rightarrow G$, such that $\lambda_X'(0)=X$.

\begin{remark}\label{remark-expCamposInvariantesEsq}
Note that $\lambda_X$ is an integral curve of the left--invariant vector field $X$ passing through $e$. In fact, \begin{eqnarray*}
\lambda'_X(t) &=& \frac{\dd}{\dd s} \lambda_X (t+s)\Big|_{s=0} \\
&=& \dd L_{\lambda_X(t)}\lambda'_X(0) \\
&=& \dd L_{\lambda_X(t)}X  \\
&=& X(\lambda_X(t)).
\end{eqnarray*}
\end{remark}

\begin{definition}\label{exp}
The {\it (Lie) exponential map}\index{Exponential map}\index{$\exp$} of $G$ is given by $$\exp:\mathfrak{g}\owns X\longmapsto \lambda_X(1)\in G,$$ where $\lambda_X$ is the unique $1$--parameter subgroup of $G$ such that $\lambda_X'(0)=X$.
\end{definition}

\begin{proposition}
The exponential map satisfies the following properties, for all $X\in\mathfrak g$ and $t\in\R$.
\begin{itemize}
\item[(i)] $\exp(tX)=\lambda_X(t)$;
\item[(ii)] $\exp(-tX)=\exp(tX)^{-1}$;
\item[(iii)] $\exp(t_1X+t_2X)=\exp(t_1X)\exp(t_2X)$;
\item[(iv)] $\exp:T_eG\rightarrow G$ is smooth and $\dd(\exp)_0=\id$, hence $\exp$ is a diffeomorphism of an open neighborhood of the origin of $T_{e}G$ onto an open neighborhood of $e\in G$.
\end{itemize}
\end{proposition}

\begin{proof}
We claim that $\lambda_X(t)=\lambda_{tX}(1)$. Consider the $1$--parameter subgroup $\lambda(s)=\lambda_X(st)$. Deriving at $s=0$, it follows that
\begin{eqnarray*}
\lambda'(0) &=& \frac{\dd}{\dd s}\lambda_X(st)\Big|_{s=0} \\
&=& t\lambda'_X(0) \\
&=& tX.
\end{eqnarray*}

Hence, from uniqueness of the $1$--parameter subgroup in Definition~\ref{exp}, $\lambda_X(st)=\lambda_{tX}(s)$. Choosing $s=1$, we obtain the expression in (i). Items (ii) and (iii) are immediate consequences of (i), since $\lambda_X$ is a Lie group homomorphism.

In order to prove item (iv), we will construct a vector field $V$ on the tangent bundle $TG$. This tangent bundle can be identified with $G\times T_eG$, since $G$ is parallelizable. By construction, the projection of the integral curve of $V$ passing through $(e,X)$ will coincide with the curve $t\to\exp(tX)$. From Theorem~\ref{flow}, the flow of $V$ will depend smoothly on the initial conditions, hence its projection (i.e., the exponential map) will also be smooth.

Consider $G\times T_{e}G\simeq TG$. Note that for all $(g,X)\in G\times T_{e}G$, the tangent space $T_{(g,X)}(G\times T_{e}G)$ can be identified with $T_gG\oplus T_{e}G$. Define a vector field $V\in\mathfrak{X}(G\times T_{e}G)$ by $$V(g,X)= \widetilde{X}(g)\oplus 0\in T_gG\oplus T_{e}G,$$ where $\widetilde{X}(g)=\dd L_{g}X$. It is not difficult to see that $V$ is a smooth vector field. Since $t\mapsto\exp(tX)$ is the unique integral curve of $\widetilde{X}$ for which $\lambda_X(0)=e$, being $\widetilde{X}$ left--invariant, $L_g\circ\lambda_X$ is the unique integral curve of $\widetilde{X}$ that takes value $g$ at $t=0$. Hence, the integral curve of $V$ through $(g,X)$ is $t\mapsto (g\exp(tX),X)$.

In other words, the flow of $V$ is given by $\varphi_{t}^{V}(g,X)=(g\exp(tX),X)$ and, in particular, $V$ is complete. Let $\pi_1:G\times T_{e}G \rightarrow G$ be the projection onto $G$. Then $\exp(X)=\pi_1\circ\varphi_{1}^{V}(e,X)$, hence $\exp$ is given by composition of smooth maps, therefore it is smooth.

Finally, item (i) and Remark~\ref{remark-expCamposInvariantesEsq} imply that $\dd(\exp)_0=\id$. 
\end{proof}

An interesting fact about the exponential map is that, in general, it may not be surjective. The classic example of this situation is given by $\SL(2,\R)$, see Duistermaat and Kolk~\cite{duistermaat}.\index{$\mathfrak{sl}(n,\R),\mathfrak{sl}(n,\C)$}\index{$\SL(n,\R),\SL(n,\C)$}

\begin{remark}\index{$\GL(n,\R),\GL(n,\C)$}\index{$\mathfrak{gl}(n,\R),\mathfrak{gl}(n,\C)$}\index{$\mathfrak{gl}(n,\R),\mathfrak{gl}(n,\C)$}
Considering Lie groups of matrices $\GL(n,\K)$, for $\K=\C$ or $\K=\R$, one may inquire whether the Lie exponential map $\exp:\mathfrak{gl}(n,\K)\rightarrow\GL(n,\K)$ coincides with the usual exponentiation of matrices, given for each $A\in\mathfrak{gl}(n,\K)$ by \begin{equation}\label{exponentiation}
e^A=\sum_{k=0}^\infty \frac{A^k}{k!}.
\end{equation}

We now prove that indeed this equality holds.
\end{remark}

To this aim, we recall two well--known properties of the exponentiation of matrices. First, the right--hand side of the expression \eqref{exponentiation} converges uniformly for $A$ in a bounded region of $\mathfrak{gl}(n,\K)$. This can be easily verified using the Weierstrass $M$--test. In addition, given $A,B\in\GL(n,\K)$, it is true that $e^{A+B}=e^Ae^B$ if, and only if, $A$ and $B$ commute. We will see a generalization of this result in Remark~\ref{exphomo}.

Consider the map $\R\owns t\mapsto e^{tA}\in\GL(n,\K)$. Since each entry of $e^{tA}$ is a power series in $t$ with infinite radius of convergence, it follows that this map is smooth. Differentiating the power series term by term, it is easy to see that its tangent vector at the origin of $\mathfrak{gl}(n,\K)$ is $A$, and from the properties above, this map is also a Lie group homomorphism, hence a $1$--parameter subgroup of $\GL(n,\K)$.

Since $\exp(A)$ is the {\it unique} $1$--parameter subgroup of $\GL(n,\K)$ whose tangent vector at the origin is $A$, it follows that $e^A=\exp(A)$, for all $A\in\mathfrak{gl}(n,\K)$.

\begin{remark}
Similarly, it can be proved that the exponential map $\exp:\End(V)\rightarrow\Aut(V)$, where $V$ is a real or complex vector space, is given by the exponentiation of endomorphisms, defined exactly as in \eqref{exponentiation}, with the usual conventions regarded.
\end{remark}

\begin{proposition}\label{expcomm}
Let $G_1$ and $G_2$ be Lie groups and $\varphi:G_1\rightarrow G_2$ a Lie group homomorphism. Then $\varphi\circ\exp^1=\exp^2\circ\dd\varphi_e$, that is, the following diagram commutes. \begin{displaymath}
\xymatrix@+20pt{
\mathfrak{g}_1 \ar[r]^{\dd\varphi_e} \ar[d]_{\exp^1} & \mathfrak{g}_2 \ar[d]^{\exp^2} \\
G_1 \ar[r]_{\varphi} & G_2
}
\end{displaymath}
\end{proposition}

\begin{proof}
Consider the $1$--parameter subgroups of $G_2$ given by $\alpha(t)=\varphi\circ\exp^1(tX)$ and $\beta(t)=\exp^2\circ\dd\varphi_e(tX)$. Then $\alpha'(0)=\beta'(0)=\dd\varphi_e X$, hence, it follows from Theorem~\ref{homo} that $\alpha=\beta$, that is, the diagram above is commutative.
\end{proof}

\begin{remark}\label{expsubgroup}\index{$\exp$}
Using Proposition~\ref{expcomm}, it is possible to prove that if $H$ is a Lie subgroup of $G$, then the exponential map $\exp^H$ of $H$ coincides with the restriction to $H$ of the exponential map $\exp^G$ of $G$. Consider the inclusion $i:H\hookrightarrow G$, which is a Lie group homomorphism, and its differential $\dd i_e:\mathfrak{h}\hookrightarrow\mathfrak{g}$. According to Corollary~\ref{subalgebrainclusion}, this is an isomorphism between the Lie algebra $\mathfrak{h}$ and a Lie subalgebra of $\mathfrak{g}$. Then the following diagram is commutative \begin{displaymath}
\xymatrix@+20pt{
\mathfrak{h} \ar[r]^{\dd i_e} \ar[d]_{\exp^H} & \mathfrak{g} \ar[d]^{\exp^G} \\
H \ar[r]_{i} & G
}
\end{displaymath}

Hence, with the appropriate identifications, $$\exp^H(X)=i(\exp^H(X))=\exp^G(\dd i_e(X))=\exp^G(X),$$ for all $X\in\mathfrak{h}$, which proves the assertion.
\end{remark}

We proceed with three identities known as the {\it Campbell\footnote{Also known as the {\it Campbell-Baker-Hausdorff formulas}.} formulas}. A proof of the following result can be found in Spivak~\cite{spivak1}.

\begin{cform}\label{campbell}\index{Campbell formulas}
Let $G$ be a Lie group and $X,Y\in\mathfrak{g}$. Then there exists $\varepsilon>0$ such that, for all $|t|<\varepsilon$, the following hold.
\begin{itemize}
\item[(i)] $\exp(tX)\exp(tY)=\exp(t(X+Y)+\frac{t^2}{2}[X,Y]+O(t^3))$;
\item[(ii)] $\exp(tX)\exp(tY)\exp(-tX)=\exp(tY+t^2[X,Y]+O(t^3))$;
\item[(iii)] $\exp(-tX)\exp(-tY)\exp(tX)\exp(tY)=\exp(t^2[X,Y]+O(t^3))$;
\end{itemize}
\noindent
where $\frac{O(t^3)}{t^3}$ is bounded.
\end{cform}

We now pass to the second part of this section, discussing some properties of the adjoint representation, or adjoint action.

\begin{definition}
Let $G$ be any group and $V$ a vector space. A {\it linear representation}\index{Representation} of $G$ on $V$ is a group homomorphism $\varphi:G\rightarrow\Aut(V)$. Recall that $\Aut(V)$ denotes the group of all vector space isomorphisms of $V$ to itself.
\end{definition}

Consider the action of $G$ on itself by conjugations, i.e., $$a:G\times G\owns (g,h)\longmapsto ghg^{-1}=a_g(h)\in G.$$

\begin{definition}\label{defadjoint}
Let $G$ be a Lie group, and $\mathfrak{g}$ its Lie algebra. The linear representation \index{$\Ad$}$\Ad:G\rightarrow\Aut(\mathfrak{g})$ defined by $$g\longmapsto\dd (a_g)_e=(\dd L_g)_{g^{-1}} \circ (\dd R_{g^{-1}})_e$$ is called the {\it adjoint representation}\index{Adjoint!representation}\index{Adjoint!action} of $G$. We will see in the next chapters that it defines an action $\Ad:G\times\mathfrak{g}\rightarrow\mathfrak{g}$ of $G$ on its Lie algebra $\mathfrak{g}$, hence $\Ad$ will also be called the \emph{adjoint action}\index{Action!adjoint} of $G$.
\end{definition}

It follows from the definition above that \begin{equation}\label{Ad}
\Ad(g)X=\frac{\dd}{\dd t}\left(g\exp(tX)g^{-1}\right)\Big|_{t=0}.
\end{equation}

Applying Proposition~\ref{expcomm} to the automorphism $a_g$, it follows that \begin{equation}\label{tad}
\exp(t\Ad(g)X)=a_g(\exp(tX))=g\exp(tX)g^{-1},
\end{equation} and in particular, for $t=1$, \begin{equation}\label{gexpg}
g\exp(X)g^{-1}=\exp(\Ad(g)X).
\end{equation}

The differential of the adjoint representation $\Ad$ is denoted $\ad$,\index{$\ad$}
\begin{equation}\label{ad}
\begin{aligned}
\ad:\mathfrak{g}\owns X & \longmapsto \dd\Ad_eX\in\End(\mathfrak{g}) \\
\ad(X)Y &=\frac{\dd}{\dd t}\left(\Ad\left(\exp(tX)\right)Y\right)\Big|_{t=0}
\end{aligned}
\end{equation}

Once more, it follows from Proposition~\ref{expcomm} that $\Ad(\exp(tX))=\exp(t\ad(X))$, that is, the following diagram is commutative. \begin{eqnarray}\label{expad}
\xymatrix@+22pt{
\mathfrak{g} \ar[d]_\exp \ar[r]^{\ad} & \End(\mathfrak{g})\ar[d]^\exp \\
G \ar[r]_{\Ad} & \Aut(\mathfrak{g})
}
\end{eqnarray}

Hence, for $t=1$ we obtain \begin{equation}\label{adexp}
\Ad(\exp(X))=\exp(\ad(X))
\end{equation}

\begin{proposition}\label{ad[]}
If $X,Y\in\mathfrak{g}$, then $\ad(X)Y=[X,Y]$.
\end{proposition}

\begin{proof}
From the Campbell formulas (Theorem~\ref{campbell}), it follows that $$\exp(tX)\exp(tY)\exp(-tX)=\exp(tY+t^2[X,Y]+O(t^3)).$$

Using \eqref{gexpg} with $g=\exp(tX)$, $$\exp\left(\Ad(\exp(tX))tY\right)=\exp(tY+t^2[X,Y]+O(t^3)).$$ Hence, for sufficiently small $t$, the arguments of $\exp$ in the expression above are identical. This means that $$\Ad(\exp(tX))tY=tY+t^2[X,Y]+O(t^3).$$ Therefore, deriving and applying \eqref{ad}, it follows that $\ad(X)Y=[X,Y]$.
\end{proof}

\begin{exercise}\label{ex-ad-matrizes}\index{Adjoint!representation in $\GL(n,\R)$}
Let $G$ be a Lie subgroup of $\GL(n,\R)$. For each $g\in G$ and $X,Y\in\mathfrak{g}$, verify the following properties.

\begin{itemize}
\item[(i)] $\dd L_gX=gX$ and $\dd R_gX=Xg$;
\item[(ii)] $\Ad(g)Y=gYg^{-1}$;
\item[(iii)] Using Proposition~\ref{ad[]} prove that $[X,Y]=XY-YX$ is the matrix commutator.
\end{itemize}
\end{exercise}

Let us now give a result relating commutativity and the Lie bracket. As mentioned before, the vector fields $X,Y\in\mathfrak{X}(M)$ are said to commute if $[X,Y]=0$.

\begin{proposition}\label{abelianiff}
Let $G$ be a connected Lie group with Lie algebra $\mathfrak{g}$. The Lie algebra $\mathfrak{g}$ is abelian if, and only if, $G$ is abelian.
\end{proposition}

\begin{proof}
Fix any $X,Y\in\mathfrak g$. From \eqref{adexp}, supposing that $[X,Y]=0$, it follows from Proposition~\ref{ad[]} that \begin{eqnarray*}
\Ad(\exp(X))Y &=& \exp(\ad(X))Y \\
&=& \sum_{k=0}^{\infty} \frac{\ad(X)^k}{k!} Y \\
&=& Y.
\end{eqnarray*}

Therefore, from \eqref{gexpg}, $\exp(X)\exp(Y)\exp(-X)=\exp(Y)$, hence $\exp(X)\exp(Y)=\exp(Y)\exp(X)$. This means that there exists an open neighborhood $U$ of $e$ such that if $g_1,g_2\in U$, then $g_1g_2=g_2g_1$. It follows from Proposition~\ref{vizger}, that $G=\bigcup_{n\in\N} U^n$, where $U^n=\{g_1^{\pm 1}\cdots g_n^{\pm 1}:g_i\in U\}$. Therefore $G$ is abelian. Note that one can not infer that $G$ is abelian directly from the commutativity of $\exp$, since $\exp$ might not be surjective.

Conversely, suppose $G$ abelian. In particular, for all $s,t\in\R$ and $X,Y\in\mathfrak{g}$, $\exp(sX)\exp(tY)\exp(-sX)=\exp(tY)$. Deriving at $t=0$, $$\frac{\mathrm d}{\mathrm dt}\exp(sX)\exp(tY)\exp(-sX)\Big|_{t=0}=Y.$$ From \eqref{Ad}, it follows that $\Ad(\exp(sX))Y=Y$. Hence, deriving at $s=0$, it follows from \eqref{ad} and Proposition~\ref{ad[]} that $\ad(X)Y=[X,Y]=0$, which means that $\mathfrak{g}$ is abelian.
\end{proof}

\begin{remark}\label{exphomo}
If $X,Y\in\mathfrak{g}$ commute, that is, $[X,Y]=0$, then $$\exp(X+Y)=\exp(X)\exp(Y).$$

Note that this does not hold in general. To verify this identity, consider $\alpha:\R\owns t\mapsto \exp(tX)\exp(tY)$. From Proposition~\ref{abelianiff}, $\alpha$ is a $1$--parameter subgroup, and deriving $\alpha$ at $t=0$, we have $\alpha'(0)=X+Y$. Hence $\alpha(t)=\exp(t(X+Y))$, and we get the desired equation setting $t=1$.
\end{remark}

We end this section with a result on connected abelian Lie groups.

\begin{theorem}\label{cpctabeliantorus}
Let $G$ be a connected $n$--dimensional abelian Lie group. Then $G$ is isomorphic to $T^k\times\R^{n-k}$, where $T^k=S^1\times\dots\times S^1$ is a $k$--torus. In particular, an abelian connected and compact Lie group is isomorphic to a torus.
\end{theorem}

\begin{proof}
Using Proposition~\ref{abelianiff}, since $G$ is abelian, $\mathfrak{g}$ is also abelian. Thus $\mathfrak{g}$ can be identified with $\R^n$. Since $G$ is connected, it follows from Remark~\ref{exphomo} that $\exp:\mathfrak{g}\rightarrow G$ is a Lie group homomorphism. From Proposition~\ref{coveringisomorphism} it is a covering map.

Consider the normal subgroup given by $\Gamma=\ker\exp$. We will now use two results to be proved in the sequel, which will be essential in this proof. The first result (Theorem~\ref{closedsubgroup}) asserts that any closed subgroup of a Lie group is a Lie subgroup. Being $\exp$ continuous, $\Gamma$ is closed, hence a Lie subgroup of $\R^n$.
The second result (Corollary~\ref{cor-quocienteLieGroup}) asserts that the  quotient of a Lie group with a normal Lie subgroup is also a Lie group. Thus $\R^n/\Gamma$ is a Lie group. 

Note that $G$ is isomorphic to $\R^n/\Gamma$, because $\exp:\R^n\rightarrow G$ is a surjective Lie group homomorphism and $\Gamma=\ker\exp$.
\begin{displaymath}
\xymatrix@+20pt{
\R^n \ar[d]_{\pi}\ar[r]^{\exp} & G \\
\R^n/\Gamma \ar[ru]_\simeq
}
\end{displaymath}
Using the fact that $\exp$ is a covering map, it is possible to prove that the isomorphism between $\R^n/\Gamma$ and $G$ defined above is in fact smooth, i.e., a Lie group isomorphism (this also follows from Corollary~\ref{homosmooth}). 

On the other hand, it is a well--known fact that the only non trivial discrete subgroups of $\R^n$ are integral lattices. In other words, there exists a positive integer $k$ less or equal to $n$, and linearly independent vectors $e_1,\dots,e_k\in\R^n$ such that $\Gamma=\left\{\sum_{i=1}^k n_ie_i : n_i\in\Z\right\}$. Therefore $G$ is isomorphic to $\R^n/\Gamma=T^k\times\R^{n-k}$, where $T^k=S^1\times\dots\times S^1$ is a $k$--torus.
\end{proof}


\section{Closed subgroups}\label{sec-closedsubgroups}

The goal of this section is to prove that closed subgroups of a Lie group are Lie subgroups, and briefly explore some corollaries. This fact is a very useful tool to prove that a subgroup is a Lie subgroup. For instance, it applies to all subgroups of $\GL(n,\K)$, for $\K=\C$ or $\K=\R$, defined in Section~\ref{sec:11}.

\begin{theorem}\label{closedsubgroup}
Let $G$ be a Lie group and $H\subset G$ a closed subgroup of $G$. Then $H$ is an embedded Lie subgroup of $G$.
\end{theorem}

\begin{proof}
We will prove this result through a sequence of five claims. The central idea of the proof is to reconstruct the Lie algebra of $H$ as a Lie subalgebra $\mathfrak{h}\subset\mathfrak{g}$. A natural candidate is $$\mathfrak{h}=\{X\in T_eG: \exp(tX)\in H,\mbox{ for all } t\in\R\}.$$

\begin{claim}\label{cl:closed1}
Let $\{X_i\}$ be a sequence in $T_eG$ with $\lim X_i=X$, and $\{t_i\}$ a sequence of real numbers, with $\lim t_i=0$. If $\exp(t_iX_i)\in H$, for all $i\in\N$, then $X\in\mathfrak{h}$.
\end{claim}

Since $\exp(-t_iX_i)=[\exp(t_iX_i)]^{-1}$, without loss of generality one can assume $t_i>0$. Define $R_i(t)$ to be the largest integer less or equal to $\frac{t}{t_i}$. Then $$\tfrac{t}{t_i}-1< R_i(t)\leq\tfrac{t}{t_i},$$ hence $\lim t_iR_i(t)=t$. Therefore $\lim t_iR_i(t)X_i=tX$. On the one hand, it follows from continuity of $\exp$ that $\lim\exp(t_iR_i(t)X_i)=\exp(tX)$. On the other hand, $\exp(t_iR_i(t)X_i)=\left[\exp(t_iX_i)\right]^{R_i(t)}\in H$. Since $H$ is closed, $\exp(tX)\in H$. Therefore $X\in\mathfrak{h}$.

\begin{claim}\label{cl:closed2}
$\mathfrak{h}\subset T_eG$ is a vector subspace of $\mathfrak{g}$.
\end{claim}

Let $X,Y\in\mathfrak{h}$. It is clear that for all $s\in\R$, $sX\in\mathfrak{h}$. Moreover, from the Campbell formulas (Theorem~\ref{campbell}), $$\exp\left[t_i(X+Y)+\frac{t_i^2}{2}[X,Y]+O(t_i^3)\right]=\exp(t_iX)\exp(t_iY)\in H,$$ therefore $\exp\left[t_i(X+Y+\frac{t_i}{2}[X,Y]+O(t_i^2))\right]\in H$ and $$\left(X+Y+\frac{t_i}{2}[X,Y]+O(t_i^2)\right)$$ tends to $X+Y$ when $t_i$ tends to $0$. From Claim~\ref{cl:closed1}, $X+Y\in\mathfrak{h}$.

\begin{claim}\label{cl:closed3}
Let $\mathfrak{k}$ be a vector space such that $T_eG=\mathfrak{h}\oplus\mathfrak{k}$, and $$\psi:\mathfrak{h}\oplus\mathfrak{k}\owns (X,Y)\longmapsto \exp(X)\exp(Y)\in G.$$ Then there exists an open neighborhood $U$ of the origin $(0,0)\in\mathfrak{h}\oplus\mathfrak{k}$, such that $\psi|_U$ is a diffeomorphism.
\end{claim}

Deriving $\psi$ with respect to each component, it follows that
\begin{eqnarray*}
\dd\psi_{(0,0)}(X,0) &=& \dd(\exp)_0 X \\
&=& X,
\end{eqnarray*}
\begin{eqnarray*}
\dd\psi_{(0,0)}(0,Y) &=& \dd(\exp)_0 Y \\
&=& Y.
\end{eqnarray*}
Hence $\dd\psi_0=\id$, and from the Inverse Function Theorem, there exists $U$ an open neighborhood of the origin, such that $\psi|_U$ is a diffeomorphism.

\begin{claim}\label{cl:closed4}
There exists an open neighborhood $V$ of the origin of $\mathfrak{k}$, such that $\exp(Y)\notin H$ for $Y\in V\setminus\{0\}$.
\end{claim}

Suppose that there exists a sequence $\{Y_i\}$ with $Y_i\in\mathfrak{k}$, $\lim Y_i=0$ and $\exp(Y_i)\in H$. Choose an inner product on $\mathfrak{k}$ and define $t_i=\|Y_i\|$ and $X_i=\tfrac{1}{t_i}Y_i$. Since $\{X_i\}$ is a sequence in the unit sphere of $\mathfrak{k}$, which is compact, up to subsequences one can assume it is convergent. Hence $\lim X_i=X$, $\lim t_i=0$ and $\exp(t_iX_i)\in H$. Therefore, from Claim~\ref{cl:closed1}, $X\in\mathfrak{h}$. This contradicts $\mathfrak{h}\cap\mathfrak{k}=\{0\}$.

\begin{claim}\label{cl:closed5}
There exists an open neighborhood $W$ of the origin in $T_eG$ such that $H\cap\exp(W)=\exp(\mathfrak{h}\cap W)$.
\end{claim}

It follows from the construction of $\mathfrak{h}$ that $H\cap\exp(W)\supset\exp(\mathfrak{h}\cap W)$. According to Claims~\ref{cl:closed3} and \ref{cl:closed4}, there exists a sufficiently small open neighborhood $W$ of the origin of $T_eG$ such that $\exp|_W$ and $\psi|_W$ are diffeomorphisms, and $(W\cap\mathfrak{k})\subset V$.

Let $a\in H\cap\exp(W)$. Being $\psi|_W$ a diffeomorphism, there exist a unique $X\in\mathfrak{h}$ and a unique $Y\in\mathfrak{k}$ such that $a=\exp(X)\exp(Y)$. Hence $\exp(Y)=\left[\exp(X)\right]^{-1}a\in H$. From Claim \ref{cl:closed4}, $Y=0$, that is, $a=\exp(X)$, with $X\in\mathfrak{h}$. Therefore $H\cap\exp(W)\subset\exp(\mathfrak{h}\cap W)$.

From Claim~\ref{cl:closed5}, $H$ is an embedded submanifold of $G$ in a neighborhood of the identity $e\in G$. Hence, being $H$ a group, it is an embedded submanifold. Finally, from Proposition~\ref{Lsubgroup}, $H$ is (an embedded) Lie subgroup of $G$.
\end{proof}

In order to explore two corollaries of the above theorem, we start by recalling  the {\it Rank Theorem}\index{Theorem!Rank}. It states that if a smooth map $f:M\rightarrow N$ is such that $\dd f_x$ has constant rank, then for each $x_0\in M$ there exists a neighborhood $U$ of $x_0$ such that the following hold.

\begin{itemize}
\item[(i)] $f(U)$ is an embedded submanifold of $N$;
\item[(ii)] The partition $\{f^{-1}(y)\cap U\}_{y\in f(U)}$ is a foliation of $U$;
\item[(iii)] For each $y\in f(U)$, $\ker\dd f_x=T_xf^{-1}(y)$, for all $x\in f^{-1}(y)$.
\end{itemize}

\begin{lemma}\label{constantrank}
Let $G_1$ and $G_2$ be Lie groups and $\varphi:G_1\rightarrow G_2$ be a Lie group homomorphism. Then the following hold.
\begin{itemize}
\item[(i)] $\dd\varphi_g$ has constant rank;
\item[(ii)] $\ker\varphi$ is a Lie subgroup of $G_1$;
\item[(iii)] $\ker\dd\varphi_e=T_e\ker\varphi$.
\end{itemize}
\end{lemma}

\begin{proof}
Since $\varphi$ is a Lie group homomorphism, $\varphi\circ L_g^1=L_{\varphi(g)}^2\circ\varphi$, where $L_g^i$ denotes the left multiplication by $g$ on $G_i$. Hence, for all $X\in T_gG_1$, \begin{eqnarray*}
\dd\varphi_gX &=& \dd\varphi_g \dd(L_g^1)_e \dd(L_{g^{-1}}^1)_g X \\
&=& \dd L^2_{\varphi(g)} \dd\varphi_e \dd(L_{g^{-1}}^1)_g X.
\end{eqnarray*}
Since $L_{\varphi(g)}^2$ is a diffeomorphism, it follows that $\dd\varphi_gX=0$ if, and only if, $\dd\varphi_e d(L_{g^{-1}}^1)X=0$. Hence $\dim\ker\dd\varphi_g=\dim\ker\dd\varphi_e$, therefore $\dd\varphi_g$ has constant rank. This proves (i).

Item (ii) follows immediately from Theorem~\ref{closedsubgroup}, since $\ker\varphi=\varphi^{-1}(e)$ is closed. Finally, (iii) follows directly from the Rank Theorem.
\end{proof}

\begin{corollary}\label{homosmooth}
Let $G_1$ and $G_2$ be Lie groups and $\varphi:G_1\rightarrow G_2$ a continuous homomorphism. Then $\varphi$ is smooth.
\end{corollary}

\begin{proof}
Let $R=\{(g,\varphi(g)):g\in G_1\}$ be the graph of $\varphi$. Then $R$ is a closed subgroup of $G_1\times G_2$, hence, from Theorem~\ref{closedsubgroup}, $R$ is an embedded Lie subgroup of $G_1\times G_2$. Consider $i:R\hookrightarrow G_1\times G_2$ the inclusion map, and the projections $\pi_1$ and $\pi_2$, onto $G_1$ and $G_2$, respectively. Then $\pi_1\circ i$ is a Lie group homomorphism, and from Lemma~\ref{constantrank}, $\dd(\pi_1\circ i)_g$ has constant rank. On the other hand, $R$ is a graph, hence by the Rank Theorem, $\pi_1\circ i$ is an immersion.

In addition, $\dim R=\dim G_1$, otherwise $(\pi_1\circ i)(R)$ would have measure zero, contradicting $(\pi_1\circ i)(R)=G_1$. From the Inverse Function Theorem, $\pi_1\circ i$ is a local diffeomorphism. Since it is also bijective, it is a global diffeomorphism, therefore $\varphi=\pi_2\circ (\pi_1\circ i)^{-1}$ is smooth.
\end{proof}

\begin{exercise}\label{su2s3}
Identifying\footnote{More precisely, this identification is given by $$S^3\owns z\simeq \underbrace{(z_0+z_1k)}_{w_1}+j\underbrace{(z_2+z_3k)}_{w_2}\simeq (w_1,w_2)\in\C\times\C\simeq\Hr.$$} $S^3\subset\Hr\simeq\C\times\C$, consider the map $\varphi:S^3\rightarrow\SU(2)$ given by $$\varphi(w_1,w_2)=\left(\begin{array}{l r}w_1 & -\overline{w_2} \\ w_2 & \overline{w_1}\end{array} \right)$$

Verify that $\varphi$ is a continuous group homomorphism, hence a Lie group homomorphism. Conclude that $S^3$ and $\SU(2)$ are isomorphic Lie groups.
\end{exercise}

\begin{exercise}\label{ex-classiclie}\index{$\SO(n)$}\index{$\O(n)$}\index{$\U(n)$}\index{$\SU(n)$}\index{$\SL(n,\R),\SL(n,\C)$}\index{$\Sp(n)$}
Prove that $\GL(n,\R)$, $\GL(n,\C)$, $\SL(n,\R)$, $\SL(n,\C)$, $\O(n)$, $\SO(n)$, $\U(n)$, $\SU(n)$ and $\Sp(n)$ are Lie groups (recall definitions in Example~\ref{classicalliegroups}). Verify that their Lie algebras are, respectively,
\begin{itemize}
\item[(i)] $\mathfrak{gl}(n,\R)$, the space of $n\times n$ square matrices over $\R$;\index{$\GL(n,\R),\GL(n,\C)$}\index{$\mathfrak{gl}(n,\R),\mathfrak{gl}(n,\C)$}
\item[(ii)] $\mathfrak{gl}(n,\C)$, the space of $n\times n$ square matrices over $\C$;
\item[(iii)] $\mathfrak{sl}(n,\R)=\{X\in\mathfrak{gl}(n,\R):\trace X=0\}$;\index{$\mathfrak{sl}(n,\R),\mathfrak{sl}(n,\C)$}
\item[(iv)] $\mathfrak{sl}(n,\C)=\{X\in\mathfrak{gl}(n,\C):\trace X=0\}$;
\item[(v)] $\mathfrak{o}(n)=\mathfrak{so}(n)=\{X\in\mathfrak{gl}(n,\R):X^t+X=0\}$;\index{$\mathfrak{so}(n),\mathfrak{o}(n)$}\index{$\SO(n)$}
\item[(vi)] $\mathfrak{u}(n)=\{X\in \mathfrak{gl}(n,\C):X^*+X=0\}$;\index{$\mathfrak{su}(n),\mathfrak{u}(n)$}\index{$\SU(n)$}
\item[(vii)]$\mathfrak{su}(n)=\mathfrak{u}(n)\cap \mathfrak{sl}(n,\C);$
\item[(viii)] $\mathfrak{sp}(n)=\{X\in\mathfrak{gl}(n,\Hr):X^*+X=0\}$.\index{$\mathfrak{sp}(n)$}
\end{itemize}
Compare (v) with Exercise~\ref{ex2-sobre-SO3}.

\medskip
\noindent {\small \emph{Hint:} Prove directly that $\GL(n,\R)$ and $\GL(n,\C)$ are Lie subgroups and verify that the other are closed subgroups of these. One can use Remark~\ref{expsubgroup} to prove that if $X\in\mathfrak g$ and $\mathfrak h\subset \mathfrak g$, then $X\in\mathfrak h$ if and only if $\exp(tX)\in H$, for all $t\in\R$. One can also use  Lemma~\ref{constantrank} to help the computation of the above Lie algebras, and the fact that for any $A\in\mathfrak{gl}(n,\K)$, $\det e^A=e^{\trace A}$.}
\end{exercise}

We recall that the {\it center}\index{Lie group!center}\index{$\Zentrum(G),\Zentrum(\mathfrak g)$} of a Lie group $G$ is the subgroup given by $$\Zentrum(G)=\{g\in G: gh=hg, \mbox{ for all } h\in G \},$$ and the {\it center}\index{Lie algebra!center} of a Lie algebra $\mathfrak{g}$ is defined as $$\Zentrum(\mathfrak{g})=\{X\in\mathfrak{g} : [X,Y]=0, \mbox{ for all } Y\in\mathfrak{g}\}.$$ The following result relates the centers of a Lie group and of its Lie algebra.

\begin{corollary}\label{zentrum}
Let $G$ be a connected Lie group. Then the following hold.

\begin{itemize}
\item[(i)] $\Zentrum(G)=\ker\Ad$;
\item[(ii)] $\Zentrum(G)$ is a Lie subgroup of $G$;
\item[(iii)] $\Zentrum(\mathfrak{g})=\ker\ad$;
\item[(iv)] $\Zentrum(\mathfrak{g})$ is the Lie algebra of $\Zentrum(G)$.
\end{itemize}
\end{corollary}

\begin{proof}
First, we verify that $\Zentrum(G)=\ker\Ad$. If $g\in\Zentrum(G)$, clearly $\Ad(g)=\id$. Conversely, let $g\in\ker\Ad$. It follows from \eqref{gexpg} that $g\exp(tX)g^{-1}=\exp(tX)$, for all $X\in\mathfrak{g}$. Hence $g$ commutes with the elements of a neighborhood of $e\in G$. Applying Proposition~\ref{vizger} one concludes that $g\in\Zentrum(G)$, and this proves (i).

Item (ii) follows from item (i) and from Lemma~\ref{constantrank}. Item (iii) follows directly from Proposition~\ref{ad[]}. Since $\dd(\Ad)_e=\ad$, item (iv) follows directly of item (iii) of Lemma~\ref{constantrank}.
\end{proof}

\begin{remark}
This result  gives an alternative proof of an assertion in Proposition~\ref{abelianiff} that states that if $\mathfrak{g}$ is abelian, then $G$ is abelian. Indeed, if $[X,Y]=0$ for all $X,Y\in\mathfrak{g}$, then $\Zentrum(\mathfrak{g})=\mathfrak{g}$. Hence, from item (iv) of Corollary~\ref{zentrum}, $\Zentrum(G)$ is open in $G$ and it follows from Proposition~\ref{vizger} that $G=\Zentrum(G)$. Hence $G$ is abelian.
\end{remark}

\begin{exercise}\index{$\SU(n)$}\index{$\U(n)$}
Verify the following.
\begin{itemize}
\item[(i)] $\Zentrum(\U(3))=\{z\id : z\in\C, |z|=1\}$;
\item[(ii)] $\SU(3)\cap\Zentrum(\U(3))=\{z\id:z^3=1\}$.
\end{itemize}
\end{exercise}

\chapter{Lie groups with bi--invariant metrics}
\label{chap2}

This chapter deals with Lie groups with a special Riemannian metric, a \emph{bi--invariant} metric. These metrics play an important role in the study of compact Lie groups, since each compact Lie group admits a such metric (see Proposition~\ref{cptbi}). In the sequel, we will use tools from Riemannian geometry to give concise proofs of several classic results on compact Lie groups.

We begin by reviewing some auxiliary facts of Riemannian geometry. Then basic results on bi--invariant metrics and Killing forms are discussed. For example, we prove that a semisimple Lie group is compact if and only if its Killing form is negative--definite. We also prove that a simply connected Lie group admits a bi--invariant metric if and only if it is a product of a compact Lie group with a vector space. Finally, we also prove that if the Lie algebra of a compact Lie group $G$ is simple, then the bi--invariant metric on $G$ is unique up to multiplication by constants.

Further readings on this chapter contents are Milnor~\cite{milnorLeftInvariantMetric}, Ise and Takeuchi~\cite{IseTakeuchi} and Fegan~\cite{Fegan}.

\section{Basic facts of Riemannian geometry}\label{sec:riemgeom}

The main objective of this section is to review basic definitions and introduce some results of Riemannian geometry that will be used in the next sections and chapters. A proof of most results in this section can be found in any textbook in Riemannian geometry, as do Carmo~\cite{manfredo}, Jost~\cite{Jost}, Petersen~\cite{petersen} and Lee~\cite{lee}.

Recall that a {\it Riemannian manifold}\index{Riemannian!manifold} is a smooth manifold $M$ endowed with a {\it (Riemannian) metric}\index{Riemannian!metric}\index{Metric}, that is, a $(0,2)$--tensor field $g$ on $M$ that is
\begin{itemize}
\item[(i)] Symmetric: $g(X,Y)=g(Y,X)$, for all $X,Y\in TM$;
\item[(ii)] Positive--definite: $g(X,X)>0$, if $X\neq 0$.
\end{itemize}

This means that a metric determines an inner product $\langle\cdot,\cdot\rangle_p$ on each tangent space $T_pM$, by $\langle X,Y\rangle_p = g_p(X,Y)$, for all $X,Y\in T_pM$. Using partitions of the unity, it is not difficult to prove that every manifold can be endowed with a metric.

Consider $M$ and $N$ manifolds and $f:M\rightarrow N$ a smooth map. Any $(0,s)$--tensor $\tau$ on $N$ may be {\em pulled back} by $f$, resulting a $(0,s)$--tensor $f^*\tau$ on $M$, as explained in Section~\ref{sec:appdiffint} of the Appendix~\ref{appendix}. A diffeomorphism $f:(M,g^M)\rightarrow (N,g^N)$ satisfying $f^*g^N=g^M$ is called a {\it (Riemannian) isometry}.\index{Riemannian!isometry}\index{Isometry} This means that $$g^M_p(X,Y)=g^N_{f(p)}(\dd f_pX,\dd f_pY),$$ for all $p\in M$ and $X,Y\in T_pM$.

It is possible to associate to each given metric a map called {\it connection}. Such map allows to {\it parallel translate} vectors along curves, {\it connecting} tangent spaces of $M$ at different points. It is actually possible to define connections on any vector bundles over a manifold, however this is beyond the objectives of this text.

\begin{definition}\index{$\nabla$}
Let $(M,g)$ be a Riemannian manifold. A {\it linear connection}\index{Connection} on $M$ is a map $$\nabla:\mathfrak{X}(M)\times\mathfrak{X}(M)\owns (X,Y)\longmapsto\nabla_X Y\in\mathfrak{X}(M),$$ satisfying the following properties:
\begin{itemize}
\item[(i)] $\nabla_X Y$ is $C^\infty(M)$-linear in $X$, i.e., for all $f,g\in C^\infty(M)$, $$\nabla_{fX_1+gX_2}Y=f\nabla_{X_1}Y+g\nabla_{X_2}Y;$$
\item[(ii)] $\nabla_X Y$ is $\R$-linear in $Y$, i.e., for all $a,b\in\R$, $$\nabla_X(aY_1+bY_2)=a\nabla_X Y_1+b\nabla_X Y_2;$$
\item[(iii)] $\nabla$ satisfies the Leibniz rule, i.e., for all $f\in C^\infty(M)$, $$\nabla_X (fY)=f\nabla_X Y+(Xf)Y.$$
\end{itemize}
Moreover, a linear connection is said to be {\it compatible with the metric}\index{Connection!compatible with metric} $g$ of $M$ if $$Xg(Y,Z)=g(\nabla_X Y,Z)+g(Y,\nabla_X Z),$$ for all $X,Y,Z\in\mathfrak{X}(M)$.
\end{definition}

It turns out that requiring a connection to be compatible with the metric does not determine a unique connection on $(M,g)$. To this purpose, define the {\it torsion tensor}\index{Torsion} of the connection to be the $(1,2)$--tensor field given by $$T(X,Y)=\nabla_X Y -\nabla_Y X -[X,Y].$$ A connection is said to be {\it symmetric}\index{Connection!symmetric} if its torsion vanishes identically, that is, if $[X,Y]=\nabla_X Y-\nabla_Y X$ for all $X,Y\in\mathfrak{X}(M)$.

\begin{lcthm}\index{Theorem!Levi--Civita}
Let $(M,g)$ be a Riemannian manifold. There exists a unique linear connection $\nabla$ on $M$ that is compatible with $g$ and symmetric, called the {\it Levi--Civita connection}.\index{Connection!Levi--Civita}
\end{lcthm}

The key fact on the proof of this theorem is the equation known as {\it connection formula}\index{Connection!formula}, or Koszul formula. It exhibits the natural candidate to the {\it Levi--Civita} connection and shows that it is uniquely determined by the metric, \begin{eqnarray}\label{connectionformula}
\langle \nabla_Y X, Z\rangle &=& \tfrac{1}{2}\Big( X\langle Y,Z\rangle -Z\langle X,Y\rangle+Y\langle Z,X\rangle \nonumber \\
&& \hspace{0.5cm}-\langle[X,Y],Z \rangle-\langle [X,Z],Y\rangle-\langle[Y,Z],X \rangle \Big).
\end{eqnarray}

This classic result is due to the Italian mathematician Tullio Levi--Civita in the beginning of the twentieth century. The unique symmetric linear connection compatible with the metric is also called {\it Riemannian connection},\index{Connection!Riemannian} and we will refer to it simply as {\it connection}.

Using the connection $\nabla$ given above one can differentiate vector fields on a Riemannian manifold $(M,g)$ as described in the next result.

\begin{proposition}\index{Covariant derivative}
Let $M$ be a manifold with linear connection $\nabla$. There exists a unique correspondence that to each vector field $X$ along a smooth curve $\gamma:I\rightarrow M$ associates another vector field $\frac{\mathrm D}{\dd t}X$ along $\gamma$, called the {\it covariant derivative} of $X$ along $\gamma$, satisfying the following properties.
\begin{itemize}
\item[(i)] $\R$--linearity, i.e., for all $X,Y\in\mathfrak{X}(M)$, $$\frac{\mathrm D}{\mathrm dt}(X+Y)=\frac{\mathrm D}{\mathrm dt}X+\frac{\mathrm D}{\mathrm dt}Y;$$
\item[(ii)] Leibniz rule, i.e., for all $X\in\mathfrak{X}(M),f\in C^\infty(I)$, $$\frac{\mathrm D}{\mathrm dt}(fX)=\frac{\mathrm df}{\mathrm dt}X+f\frac{\mathrm D}{\mathrm dt}X;$$
\item[(iii)] If $X$ is induced from a vector field $\widetilde{X}\in\mathfrak{X}(M)$, that is $X(t)=\widetilde{X}(\gamma(t))$, then $\frac{\mathrm D}{\mathrm dt}X=\nabla_{\gamma'}\widetilde{X}$.
\end{itemize}
\end{proposition}

Note that to each linear connection on $M$ the proposition above gives a covariant derivative operator for vector fields along $\gamma$. As mentioned before, we will only consider the Levi--Civita connection, hence the covariant derivative is uniquely defined. 

Armed with this notion, it is possible to define the {\it acceleration}\index{Acceleration} of a curve as the covariant derivative of its tangent vector field. Moreover, {\it geodesics}\index{Geodesic} are curves with null acceleration. More precisely, $\gamma:I\rightarrow M$ is a {\it geodesic}   if $\frac{\mathrm D}{\mathrm dt}\gamma'=0$. Writing a local expression for the covariant derivative, it is easy to see that a curve is a geodesic if, and only if, it satisfies a second--order system of ODEs, called the {\it geodesic equation}. Hence, applying the classic ODE theorem that guarantees existence and uniqueness of solutions, one can prove the following result.

\begin{theorem}
For any $p\in M$, $t_0\in\R$ and $v\in T_pM$, there exist an open interval $I\subset\R$ containing $t_0$ and a geodesic $\gamma:I\rightarrow M$ satisfying the initial conditions $\gamma(t_0)=p$ and $\gamma'(t_0)=v$. In addition, any two geodesics with those initial conditions agree on their common domain.
\end{theorem}

Furthermore, from uniqueness of the solution, it is possible to obtain a maximal geodesic with this prescribed initial data.

Another construction that involves covariant differentiation along curves is parallel translation. A vector field $X\in\mathfrak{X}(M)$ is said to be {\it parallel along} $\gamma$ if $\frac{\mathrm D}{\mathrm dt}X=0$. Thus, a geodesic $\gamma$ can be characterized as a curve whose tangent field $\gamma'$ is parallel along $\gamma$. A vector field is called {\it parallel} if it is parallel\index{Vector field!parallel} along every curve.

\begin{proposition}
Let $\gamma:I\rightarrow M$ be a curve, $t_0\in I$ and $v_0\in T_{\gamma(t_0)}M$. There exists a unique parallel vector field $X$ along $\gamma$ such that $X(t_0)=v_0$.
\end{proposition}

This vector field is called the {\it parallel translate}\index{Parallel translation} of $v_0$ along $\gamma$. Once more, the proof depends on basic ODE results. One can also prove that parallel translation is an isometry. 

Having existence and uniqueness of geodesics with prescribed initial data, an important question is how do geodesics change under perturbations of initial data. We now proceed to define a similar concept to the Lie exponential map, which will coincide with it when considering an appropriate metric on the Lie group, turning it into a Riemannian manifold.

For this, we consider the {\it geodesic flow}\index{Geodesic!flow} of a Riemannian manifold $(M,g)$. This is the flow $\varphi^\Gamma:U\subset \R\times TM\to TM$ defined in an open subset $U$ of $\R\times TM$ that contains $\{0\}\times TM$, of the unique vector field on the tangent bundle $\Gamma\in\mathfrak{X}(TM)$, whose integral curves are of the form $t\mapsto(\gamma(t),\gamma'(t))$, where $\gamma$ is a geodesic in $M$. This means that:
\begin{itemize}
\item[(i)] $\gamma(t)=\pi\circ \varphi^\Gamma(t,(p,v))$ is the geodesic with initial conditions $\gamma(0)=p$ and $\gamma'(0)=v$, where $\pi$ is the canonical projection;
\item[(ii)] $\varphi^\Gamma(t,(p,cv))=\varphi^\Gamma(ct,(p,v))$, for all $c\in\R$ such that this equation makes sense.
\end{itemize}

In fact, supposing that such vector field $\Gamma$ exists, one can obtain conditions in local coordinates that this field must satisfy (corresponding to the geodesic equation). Defining the vector field as its solutions, one may use Theorem~\ref{flow} to guarantee existence and smoothness of the geodesic flow $\varphi^\Gamma$.

\begin{definition}
Let $p\in M$. The {\it (Riemannian) exponential map}\index{Riemannian!exponential map} $\exp_p:B_\epsilon(0)\subset T_pM\rightarrow M$ is the map given by $$\exp_p v= \varphi^\Gamma(1,(p,v)).$$
\end{definition}

From Theorem~\ref{flow}, it is immediate that $\exp_p$ is smooth. Using the Inverse Function Theorem one can verify that for any $p\in M$, there exist a neighborhood $V$ of the origin in $T_pM$ and a neighborhood $U$ of $p\in M$, such that $\exp_p|_V:V\rightarrow U$ is a diffeomorphism. Such neighborhood $U$ is called a {\it normal neighborhood} of $p$.

The {\it length}\index{Length}\index{Riemannian!length} of a curve segment $\gamma:[a,b]\rightarrow M$ is defined as $$\ell_g(\gamma)=\int_a^b \sqrt{g(\gamma'(t),\gamma'(t))}\;\mathrm{d}t.$$

Define the {\it (Riemannian) distance}\index{Riemannian!distance} $d(p,q)$ for any pair of points $p,q\in M$ to be the infimum of lengths of all piecewise regular curve segments joining $p$ and $q$. Then $(M,d)$ is a metric space, and the topology induced by this distance coincides with the original topology from the atlas of $M$. Moreover, every geodesic locally minimizes $\ell$.

A Riemannian manifold is called {\it geodesically complete}\index{Complete!geodesically} if every maximal geodesic is defined for all $t\in\R$. It is not difficult to see that a sufficient condition for a manifold to be complete is to be compact. An important result is that the completeness notions above mentioned are indeed equivalent.

\begin{hrthm}\label{hopfrinow}\index{Theorem!Hopf--Rinow}
Let $M$ be a connected Riemannian manifold and $p\in M$. The following statements are equivalent.
\begin{itemize}
\item[(i)] $\exp_p$ is globally defined, that is, $\exp_p:T_pM\rightarrow M$;
\item[(ii)] Every closed bounded set in $M$ is compact;
\item[(iii)] $(M,d)$ is a complete metric space;
\item[(iv)] $M$ is geodesically complete.
\end{itemize}
If $M$ satisfies any (hence all) of the above items, each two points of $M$ can be joined by a minimal segment of geodesic. In particular, for each $x\in M$ the exponential map $\exp_x:T_xM\rightarrow M$ is surjective.
\end{hrthm}

An interesting result due to Myers and Steenrod~\cite{myerssteenrod} is that the group of isometries of a Riemannian manifold is a finite--dimensional Lie group that acts smoothly on $M$.

\begin{msthm}\label{theorem-Myers-Steenrod}\index{Theorem!Myers--Steenrod}\index{Isometry!group}\index{$\Iso(M)$}
Let $M$ be a Riemannian manifold and denote by $\Iso(M)$ its isometry group. Then every closed isometry subgroup of $\Iso(M)$ in the compact--open topology is a Lie group. In particular, $\Iso(M)$ is a Lie group.
\end{msthm}

\begin{remark}\index{Compact--open topology}
A subset $G\subset\Iso(M)$ is closed in the compact--open topology if the following condition holds. 
Let $\{f_{n}\}$ be a sequence of isometries in $G.$ Assume that, for each compact subset $K\subset M,$ the sequence $\{f_{n}\}$  converges uniformly in $K$ to a continuous map $f:M\rightarrow M$ (with respect to the distance $d$). Then $f\in G$.  
\end{remark}

Let us recall a special class of vector fields on $M$ that is closely related to $\Iso(M)$. A \emph{Killing vector field}\index{Killing!vector field}\index{Vector field!Killing} is a vector field whose local flow is a local isometry.

\begin{theorem}
The set $\mathfrak{iso}(M)$ of Killing fields on $M$ is a Lie algebra. In addition, if $M$ is complete, then $\mathfrak{iso}(M)$ is the Lie algebra of  $\mathrm{Iso}(M)$.
\end{theorem}

\begin{proposition}\label{proposition-eq-Killing}
Let $M$ be a Riemannian manifold. A vector field $X\in\mathfrak{X}(M)$ is a Killing field if and only if $$g(\nabla_{Y}X,Z)=-g(\nabla_{Z}X,Y),$$ for all $Y,Z \in \mathfrak{X}(M)$.
\end{proposition}

We now recall a few results on submanifolds that will be used later. Let $(M,g)$ be a Riemannian manifold. A submanifold $i:L\hookrightarrow M$ with the pull--back metric\footnote{Also called \emph{induced} metric.} $g_{L}=i^{*}g$ is called a \emph{Riemannian submanifold}\index{Riemannian!submanifold}\index{Submanifold!Riemannian} of $M$.

Two vector fields $X$ and $Y$ of $L$ can be locally extended to vector fields $\widetilde{X}$ and $\widetilde{Y}$ of $M$. It is then possible to prove that the Riemannian connection associated to $g_{L}$ coincides with $(\nabla_{\widetilde{X}}\widetilde{Y})^\top$, i.e., the component of $\nabla_{\widetilde{X}}\widetilde{Y}$ that is tangent to $L$. It is possible to define a bilinear symmetric form $B_\xi$ that measures the difference between such connections, called the \emph{second fundamental form}\index{Second fundamental form} of $L$. For each normal\footnote{The \emph{normal} space to $L$ will be denoted $\nu_{p}L$} vector $\xi$ to $L$ at $p$, $$B_{\xi}(X,Y)_{p}=g_{p}(\xi,\nabla_{\widetilde{X}}\widetilde{Y}-(\nabla_{\widetilde{X}}\widetilde{Y})^\top).$$

If $B_{\xi}$ of $L$ vanishes for all $\xi$, $L$ is called a \emph{totally geodesic}\index{Totally geodesic submanifold}\index{Submanifold!totally geodesic} submanifold. It is not difficult to prove that this is equivalent to each geodesic of $L$ being a geodesic of $M$. In general, to be totally geodesic is a very strong property. Nevertheless, this class of submanifolds will appear in a natural way along these notes (see Exercise~\ref{exercise-subgrouptotalmentegeoesico}, Chapter~\ref{chap4} and~\ref{chap5}).

Since $B_{\xi}$ is symmetric, there exists a self--adjoint operator $\mathcal{S}_{\xi}$ with respect to $g$, called the \emph{shape operator},\index{Shape operator} satisfying $$g(\mathcal{S}_{\xi}X, Y)=B_{\xi}(X,Y).$$ It is not difficult to prove that $\mathcal{S}_{\xi}(X)=(-\nabla_{X}\widetilde{\xi})^\top$ where $\widetilde{\xi}$ is any normal field that extends $\xi$. Eigenvalues and eigenvectors of $\mathcal{S}_{\eta}(X)$ are respectively called \emph{principal curvatures}\index{Curvature!principal} and \emph{principal directions}\index{Principal direction} (see Remark~\ref{remark-gausscurvature}).

An essential concept in Riemannian geometry is {\it (Riemannian) curvature}\index{Riemannian!curvature}\index{Curvature}. Recall that $\nabla$ denotes the Levi--Civita connection of a given metric $g$ on $M$. Then the curvature tensor on $(M,g)$ is defined as the $(1,3)$--tensor field given by the following expression, for all $X,Y,Z\in\mathfrak{X}(M)$.
\begin{eqnarray*}
R(X,Y)Z &=&\nabla_{[X,Y]}Z -\nabla_X\nabla_Y Z+\nabla_Y\nabla_X Z \\
&=&\nabla_{[X,Y]}Z -[\nabla_X,\nabla_Y]Z.
\end{eqnarray*} Furthermore, it is possible to use the metric to deal with this tensor as a $(0,4)$--tensor, given for each $X,Y,Z,W\in\mathfrak{X}(M)$ by $$R(X,Y,Z,W)= g(R(X,Y)Z,W).$$

There are several possible interpretations of curvature. The first one is that it measures second covariant derivatives' failure to commute.

There are other less trivial interpretations.
The curvature tensor is  part of the so--called {\it Jacobi equation}\index{Jacobi equation} along a geodesic $\gamma$, given by$$\frac{\mathrm D}{\mathrm dt}\frac{\mathrm D}{\mathrm dt}J+R(\gamma'(t),J(t))\gamma'(t)=0.$$ This is an ODE whose solutions $J$ are vector fields along $\gamma$ (called {\it Jacobi fields}). Such vector fields describe how quickly two geodesics with the same starting point move away one from each other.
Riemannian curvature can also be used to describe how parallel transport along a loop differs from the identity.
Finally, it also measures non integrability of a special kind of distribution defined in the frame bundle. These fundamental interpretations of Riemannian curvature are explained for instance in Jost~\cite{Jost} and Bishop and Crittenden~\cite{Bishop}.

There are many important symmetries of this tensor that we recall below. For each $X,Y,Z,W\in\mathfrak{X}(M)$,
\begin{itemize}
\item[(i)] $R$ is skew--symmetric in the first two and last two entries: $$R(X,Y,Z,W)=-R(Y,X,Z,W)=R(Y,X,W,Z);$$
\item[(ii)] $R$ is symmetric between the first two and last two entries: $$R(X,Y,Z,W)=R(Z,W,X,Y);$$
\item[(iii)] $R$ satisfies a cyclic permutation property, the {\it Bianchi first identity}:\index{Bianchi first identity} $$R(X,Y)Z+R(Z,X)Y+R(Y,Z)X=0.$$
\end{itemize}

Using the curvature tensor, we can define the {\it sectional curvature}\index{Curvature!sectional} of the plane spanned by (linearly independent) vectors $X$ and $Y$ as
\begin{equation}\label{seccurvature}
K(X,Y)=\frac{R(X,Y,X,Y)}{g(X,X)g(Y,Y)-g(X,Y)^2},
\end{equation} and it is possible to prove that $K(X,Y)$ depends only on the {\em plane spanned} by $X$ and $Y$, and not directly on the vectors $X$ and $Y$.

\begin{remark}\label{remark-gausscurvature}
Let $M$ be an embedded surface in $\R^3$ with the induced metric. According to the Gauss equation, sectional curvature $K$ of $M$ coincides with Gaussian curvature of $M$. Recall that {\it Gaussian curvature}\index{Curvature!Gaussian} is given by the product of eigenvalues $\lambda_1\lambda_2$ (principal curvatures) of the shape operator $\mathcal{S}_{\xi}(\cdot)=-\nabla_{(\cdot)} \xi$, where $\xi$ is a unitary normal vector to $M$.
An important fact in differential geometry is that each embedded surface with nonzero Gaussian curvature is, up to rigid motions, locally given by the graph of $$f(x_1,x_2)=\tfrac{1}{2}\left(\lambda_1 x_1^2+\lambda_2 x_2^2\right)+O(\|x\|^3).$$ In other words,
if $K>0$, respectively $K<0$, $M$ is locally a small perturbation of elliptic, respectively hyperbolic, paraboloid.
\end{remark}



It would be interesting to summarize information contained in the curvature $(0,4)$--tensor $R$ constructing simpler tensors. The next curvature we recall is {\it Ricci curvature}. It should be thought as an approximation of the Laplacian of the metric, i.e., a measure of the volume distortion on $M$ (see Morgan and Tian~\cite{MorganTian}). It is named after the Italian mathematician Gregorio Ricci--Curbastro that had Levi--Civita as student and collaborator in the beginning of the twentieth century.

Ricci curvature\index{Curvature!Ricci}\index{$\Ric$} is a $(0,2)$--tensor field given by the trace of the curvature endomorphism on its first and last indexes. More precisely, if $(e_1,\dots,e_n)$ is an orthonormal basis of $T_pM$ and $X,Y\in T_pM$, \begin{eqnarray*}
\Ric(X,Y) &=& \trace R(X,\cdot\,)Y \\
&=& \sum_{i=1}^n g(R(X,e_i)Y,e_i).
\end{eqnarray*}

A Riemannian metric is called an {\it Einstein metric}\index{Metric!Einstein}\index{Einstein!manifold} if its Ricci tensor is a scalar multiple of the metric at each point, that is, if for all $p\in M$, $$\Ric_p(X,Y)=\lambda(p) g_p(X,Y), \; \mbox{ for all } X,Y\in T_pM.$$ 

It is easy to see that if $(M,g)$ has constant sectional curvature $k$, then it is Einstein, with constant $\lambda(p)=(\dim M -1)k$. The converse is not true. Indeed there are very important Einstein metrics that do not have constant curvature, e.g., $\C P^n$ with the Fubini--Study metric. As we will see later, compact Lie groups with simple Lie algebras admit Einstein metrics.

\begin{remark}
Einstein metrics appear in general relativity as solutions of the Einstein equation\index{Einstein!equation in vacuum} in vacuum $$\Ric=\left(\tfrac{1}{2}\scal-\Lambda\right)g,$$ where $\scal$ is the {\it scalar curvature} and $\Lambda$ is the cosmological constant. More about Einstein equation and its impact in Physics can be found in Misner, Thorne and Wheeler~\cite{gravitation} or Besse~\cite{besse}.
\end{remark}

We end this section stating an important result that connects Ricci curvature and compactness.

\begin{bmthm}\label{bonnetmyers}\index{Theorem!Bonnet--Myers}
Let $(M,g)$ be a connected complete $n$--dimensional Riemannian manifold, with $n\geq 2$. Suppose that there exists $r>0$ such that $\Ric\geq\frac{n-1}{r^{2}} g$. Then the following hold.
\begin{itemize}
\item[(i)] The diameter of $M$ is less or equal to $\pi r$. In particular, $M$ is compact;
\item[(ii)] The universal cover of $M$ is compact, hence $\pi_1(M)$ is finite.
\end{itemize}
\end{bmthm}

\section{Bi--invariant metrics}

The main goal of this section is to study a special Riemannian structure on Lie groups given by \emph{bi--invariant} metrics. Notation for the Riemannian metric (on Lie groups) will be that of inner product $\langle\cdot,\cdot\rangle_p$ in the tangent space $T_pG$ instead of the tensorial notation $g$, since $g$ will be used as the typical element of the group $G$.

\begin{definition}
A Riemannian metric $\langle\cdot,\cdot\rangle$ on a Lie group $G$ is {\it left--invariant}\index{Metric!left--invariant}\index{Left--invariant!metric} if $L_g$ is an isometry for all $g\in G$, that is, if for all $g,h\in G$ and $X,Y\in T_hG$, $$\langle \dd(L_g)_h X,\dd(L_g)_h Y\rangle_{gh}=\langle X,Y\rangle_h.$$
\end{definition}

Similarly, {\it right--invariant}\index{Metric!right--invariant}\index{Right--invariant!metric} metrics are those that turn the right translations $R_g$ into isometries. Note that given an inner product $\langle\cdot,\cdot\rangle_e$ in $T_{e}G$, it is possible to define a left--invariant metric on $G$ by setting for all $g\in G$ and $X,Y\in T_gG$, $$\langle X,Y\rangle_g=\langle \dd(L_{g^{-1}})_gX,\dd(L_{g^{-1}})_gY\rangle_e,$$ and the right--invariant case is analogous.

\begin{definition}
A {\it bi--invariant metric}\index{Metric!bi--invariant}\index{Bi--invariant!metric} on a Lie group $G$ is a Riemannian metric that is simultaneously left and right--invariant.
\end{definition}

The natural extension of these concepts to $k$--forms is that a $k$--form $\omega\in\Omega^k(G)$ is {\it left--invariant}\index{Left--invariant!$k$--form} if it coincides with its pullback by left translations, i.e., $L_g^*\omega=\omega$ for all $g\in G$. {\it Right--invariant}\index{Right--invariant!$k$--form} and {\it bi--invariant}\index{Bi--invariant!$k$--form} forms are analogously defined. Once more, given any $\omega_e\in\wedge^k T_{e}G$, it is possible to define a left--invariant $k$--form $\omega\in\Omega(G)$ by setting for all $g\in G$ and $X_i\in T_gG$, $$\omega_g(X_1,\dots,X_k)=\omega_e(\dd(L_{g^{-1}})_g X_1,\dots,\dd(L_{g^{-1}})_g X_k),$$ and the right--invariant case is once more analogous.

\begin{proposition}\label{cptbi}
Let $G$ be a compact $n$--dimensional Lie group. Then $G$ admits a bi--invariant metric.
\end{proposition}

\begin{proof}
Let $\omega$ be a right--invariant volume form\footnote{i.e., $\omega\in\Omega^n(G)$ is a non zero $n$--form, see Section~\ref{sec:appdiffint}.} on $G$ and $\langle\cdot,\cdot\rangle$ a right--invariant metric. Define for all $X,Y\in T_xG$, $$\leqslant X,Y\geqslant_x=\int_G \langle \dd L_gX,\dd L_gY\rangle_{gx}\omega.$$

First, we claim that $\leqslant\cdot,\cdot\geqslant$ is left--invariant. Indeed,
\begin{eqnarray}\label{linv1}
\leqslant \dd L_hX,\dd L_hY\geqslant_{hx} &=& \int_G \langle \dd L_g(\dd L_hX),\dd L_g(\dd L_hY)\rangle_{g(hx)}\omega \nonumber \\
&=& \int_G \langle \dd L_{gh}X,\dd L_{gh}Y\rangle_{(gh)x}\omega.
\end{eqnarray}

Fix $X,Y\in T_xG$ and let $f(g)=\langle \dd L_gX,\dd L_gY\rangle_{gx}$. Then
{\allowdisplaybreaks
\begin{eqnarray}\label{linv2}
\int_G \langle \dd L_{gh}X,\dd L_{gh}Y\rangle_{(gh)x}\omega &=& \int_G f(gh)\omega \nonumber \\
&=& \int_G R^*_h(f\omega) \nonumber \\
&=& \int_G f\omega \\
&=& \int_G \langle \dd L_gX,\dd L_gY\rangle_{gx}\omega \nonumber \\
&=& \leqslant X,Y\geqslant_x. \nonumber
\end{eqnarray}}From \eqref{linv1} and \eqref{linv2}, it follows that $\leqslant\cdot,\cdot\geqslant$ is left--invariant. In addition, we claim that $\leqslant\cdot,\cdot\geqslant$ is also right--invariant. Indeed,
{\allowdisplaybreaks
\begin{eqnarray*}
\leqslant \dd R_h X,\dd R_h Y \geqslant_{xh}
&=& \int_G \langle \dd L_g(\dd R_h X), \dd L_g(\dd R_h Y) \rangle_{g(xh)} \ \omega \\
&=& \int_G \langle \dd R_h \dd L_g X, \dd R_h \dd L_g Y \rangle_{(gx) h} \ \omega \\
&=& \int_G \langle \dd L_gX, \dd L_{g} Y \rangle_{gx} \ \omega \\
&=& \leqslant X,Y\geqslant_x.\qedhere
\end{eqnarray*}}
\end{proof}

\begin{exercise}\label{ex-metrica-su}\index{Bi--invariant!metric of $\SU(n)$}\index{Lie algebra!of $\SU(n)$}\index{$\SU(n)$}\index{$\mathfrak{su}(n),\mathfrak{u}(n)$}
Consider $$\mathfrak{su}(n)=\{A\in\mathfrak{gl}(n,\C):A^*+A=0, \trace A=0\},$$ the Lie algebra of $\SU(n)$ (see Exercise~\ref{ex-classiclie}). Verify that the inner product in $T_e\SU(n)$ defined by $$\langle X,Y\rangle=\mathrm{Re}\trace(XY^*)$$ can be extended to a bi--invariant metric.
\end{exercise}

\begin{proposition}\label{curvatureLie}
Let $G$ be a Lie group endowed with a bi--invariant metric $\langle\cdot,\cdot\rangle$, and $X,Y,Z\in\mathfrak{g}$. Then the following hold.
\begin{itemize}
\item[(i)] $\langle [X,Y],Z\rangle=-\langle Y,[X,Z]\rangle$;
\item[(ii)] $\nabla_X Y=\tfrac{1}{2}[X,Y]$;
\item[(iii)] $R(X,Y)Z=\tfrac{1}{4}[[X,Y],Z]$;
\item[(iv)] $R(X,Y,X,Y)=\tfrac{1}{4}\|[X,Y]\|^2$.
\end{itemize}
In particular, the sectional curvature \eqref{seccurvature} is non negative.
\end{proposition}

\begin{proof}
Deriving the formula $\langle \Ad(\exp(tX))Y,\Ad(\exp(tX))Z\rangle=\langle Y,Z\rangle$ it follows from Proposition~\ref{ad[]} and \eqref{ad} that $\langle[X,Y],Z\rangle+\langle Y,[X,Z]\rangle=0$, which proves (i).

Furthermore, (ii) follows from the Koszul formula \eqref{connectionformula} using (i) and the fact that $\nabla$ is symmetric.

To prove (iii), we use (ii) to compute $R(X,Y)Z$ as follows.
{\allowdisplaybreaks
\begin{eqnarray*}
R(X,Y)Z &=& \nabla_{[X,Y]}Z - \nabla_X \nabla_Y Z +\nabla_Y \nabla_X Z \\
&=& \tfrac{1}{2}[[X,Y],Z] -\tfrac{1}{2}\nabla_X [Y,Z] +\tfrac{1}{2}\nabla_Y [X,Z]\\
&=& \tfrac{1}{2}[[X,Y],Z] -\tfrac{1}{4}[X,[Y,Z]] +\tfrac{1}{4}[Y,[X,Z]] \\
&=& \tfrac{1}{4}[[X,Y],Z] +\tfrac{1}{4}\big([[X,Y],Z]+[[Z,X],Y]+[[Y,Z],X] \big)\\
&=& \tfrac{1}{4}[[X,Y],Z].
\end{eqnarray*}}

Finally, to prove (iv), we use (i) to verify that
{\allowdisplaybreaks
\begin{eqnarray*}
\langle R(X,Y)X,Y \rangle &=& \tfrac{1}{4}\langle [[X,Y],X],Y \rangle \\
&=& -\tfrac{1}{4}\langle [X,[X,Y]],Y \rangle \\
&=& \tfrac{1}{4}\langle [X,Y],[X,Y] \rangle \\
&=& \tfrac{1}{4}\|[X,Y]\|^2.\qedhere
\end{eqnarray*}}
\end{proof}

\begin{theorem}\label{expagree}
The Lie exponential map and the Riemannian exponential map at identity agree in Lie groups endowed with bi--invariant metrics. In particular, Lie exponential map of a compact Lie group is surjective.
\end{theorem}

\begin{proof}
Let $G$ be a Lie group endowed with a bi--invariant metric and $X\in\mathfrak{g}$. To prove that the exponential maps coincide, it suffices to prove that $\gamma:\R\rightarrow G$ given by $\gamma(t)=\exp(tX)$ is the geodesic with $\gamma(0)=e$ and $\gamma'(0)=X$. First, recall that $\gamma$ is the integral curve of the left--invariant vector field $X$ passing through $e\in G$ at $t=0$, that is, $\gamma'(t)=X(\gamma(t))$ and $\gamma(0)=e$. Furthermore, from Proposition~\ref{curvatureLie},
{\allowdisplaybreaks
\begin{eqnarray*}
\frac{\mathrm D}{\mathrm dt}\gamma' &=& \frac{\mathrm D}{\mathrm dt} X(\gamma(t)) \\
&=& \nabla_{\gamma'} X \\
&=& \nabla_X X \\
&=& \tfrac{1}{2}[X,X] \\
&=& 0.
\end{eqnarray*}}Hence $\gamma$ is a geodesic. Therefore the Lie exponential map coincides with the Riemannian exponential map.

To prove the second assertion, if $G$ is compact, from Proposition~\ref{cptbi}, it admits a bi--invariant metric. Using that exponential maps coincide and that the Lie exponential map is defined for all $X\in\mathfrak{g}$, it follows from Hopf--Rinow Theorem~\ref{hopfrinow} that $G$ is a complete Riemannian manifold. Therefore $\exp=\exp_e:T_{e}G\rightarrow G$ is surjective.
\end{proof}

\begin{exercise}
Prove that $\SL(2,\R)$ does not admit a metric such that the Lie exponential map and the Riemannian exponential map coincide in $e$. For this, use the fact that $\exp$ is not surjective in $\SL(2,\R)$.
\end{exercise}

In the next result, we prove that each Lie group $G$ with bi--invariant metric is a \emph{symmetric space},\index{Symmetric space} i.e., for each $a\in G$ there exists an isometry $I^{a}$ that reverses geodesics through $a$ (see Remark~\ref{polarsymmetric}).

\begin{theorem}\label{liegeodesic}
Let $G$ be a connected Lie group endowed with a bi--invariant metric. For each $a\in G$ define $$I^a:G\owns g\longmapsto ag^{-1}a\in G.$$ Then $I^a$ is an isometry that reverses geodesics through $a$. In other words, $I^a\in\Iso(G)$ and if $\gamma$ is a geodesic with $\gamma(0)=a$, then $I^a(\gamma(t))=\gamma(-t)$.
\end{theorem}

\begin{proof}
Since $I^e(g)=g^{-1}$, the map $\dd(I^e)_e:T_{e}G\rightarrow T_{e}G$ is the multiplication by $-1$, i.e., $\dd(I^e)_e=-\id$. Hence it is an isometry of $T_{e}G$. Since $\dd(I^e)_{a}=\dd(R_{a^{-1}})_{e}\circ \dd(I^e)_{e}\circ \dd(L_{a^{-1}})_{a}$, for any $a\in G$, the map $\dd(I^e)_a:T_aG\rightarrow T_{a^{-1}}G$ is also an isometry. Hence $I^e$ is an isometry. It clearly reverses geodesics through $e$, and since $I^a=R_aI^eR_a^{-1}$, it follows that $I^a$ is also an isometry that reverses geodesics through $a$.
\end{proof}

\begin{exercise}\label{exercise-subgrouptotalmentegeoesico}
Let $G$ be a compact Lie group endowed with a bi--invariant metric. Prove that each closed subgroup $H$ is a totally geodesic submanifold.
\end{exercise}

\section{Killing form and semisimple Lie algebras}

To continue our study of bi--invariant metrics we introduce the Killing form, named after the German mathematician Wilhelm Killing. Using it, we establish classic algebraic conditions under which a Lie group is compact. Equivalent definitions of semisimple Lie algebras are also discussed.

\begin{definition}
Let $G$ be a Lie group and $X,Y\in\mathfrak{g}$. The {\it Killing form}\index{Killing!form} of $\mathfrak{g}$ (also said Killing form of $G$) is defined as the symmetric bilinear form given by $$B(X,Y)=\trace \left[\ad(X)\ad(Y)\right].$$ If $B$ is non degenerate, $\mathfrak{g}$ is said to be {\it semisimple}.\index{Lie algebra!semisimple}\index{Semisimple!Lie algebra}\index{Semisimple!Lie group}
\end{definition}

We will see in Theorem~\ref{teo-semisimples-equivalencias}, other equivalent definitions of semisimplicity for a Lie algebra. A Lie group $G$ is said to be {\it semisimple}\index{Lie group!semisimple} if its Lie algebra $\mathfrak{g}$ is semisimple.

\begin{proposition}
The Killing form is $\Ad$--invariant, i.e., $B(X,Y)=B(\Ad(g)X,\Ad(g)Y)$.
\end{proposition}

\begin{proof}
Let $\varphi:\mathfrak{g}\rightarrow\mathfrak{g}$ be a Lie algebra automorphism. Then $$\ad(\varphi(X))\varphi(Y)=\varphi\circ\ad(X)Y.$$ Hence $\ad(\varphi(X))=\varphi\circ\ad(X)\circ\varphi^{-1}$. Therefore 
{\allowdisplaybreaks
\begin{eqnarray*}
B(\varphi(X),\varphi(Y)) &=& \trace[\ad(\varphi(X))\ad(\varphi(Y))] \\
&=& \trace[\varphi\cdot\ad(X)\ad(Y)\cdot\varphi^{-1}] \\
&=& \trace[\ad(X)\ad(Y)] \\
&=& B(X,Y).
\end{eqnarray*}}Since $\Ad(g)$ is a Lie algebra automorphism, the proof is complete.
\end{proof}

\begin{corollary}\label{killingmetric}
Let $G$ be a semisimple Lie group with negative--definite Killing form $B$. Then $-B$ is a bi--invariant metric.
\end{corollary}

\begin{remark}\label{riccikilling}
Let $G$ be a Lie group endowed with a bi--invariant metric. From Proposition~\ref{curvatureLie}, it follows that
{\allowdisplaybreaks
\begin{eqnarray*}
\Ric(X,Y) &=& \trace R(X,\cdot)Y \\
&=& \trace \tfrac{1}{4}[[X,\cdot],Y] \\
&=& -\tfrac{1}{4}\trace [Y,[X,\cdot]] \\
&=& -\tfrac{1}{4} B(Y,X) \\
&=& -\tfrac{1}{4} B(X,Y).
\end{eqnarray*}}Hence $\Ric(X,Y)=-\tfrac{1}{4}B(X,Y)$. Therefore, the Ricci curvature of $G$ is independent of the bi--invariant metric.
\end{remark}

\begin{theorem}\label{compactkilling}
Let $G$ be a $n$--dimensional semisimple connected Lie group. Then $G$ is compact if, and only if, its Killing form $B$ is negative--definite.
\end{theorem}

\begin{proof}
First, suppose that $B$ is negative--definite. From Corollary~\ref{killingmetric}, $-B$ is a bi--invariant metric on $G$. Hence, Theorem~\ref{liegeodesic} and Hopf--Rinow Theorem~\ref{hopfrinow} imply that $(G,-B)$ is a complete Riemannian manifold whose Ricci curvature satisfies the equation in Remark~\ref{riccikilling}. It follows from Bonnet--Myers Theorem~\ref{bonnetmyers} that $G$ is compact.

Conversely, suppose $G$ is compact. From Proposition~\ref{cptbi}, it admits a bi--invariant metric. Hence, using item (i) of Proposition~\ref{curvatureLie} and Proposition~\ref{ad[]}, it follows that if $(e_1,\dots,e_n)$ is an orthonormal basis of $\mathfrak{g}$, then
{\allowdisplaybreaks
\begin{eqnarray*}
B(X,X) &=& \trace(\ad(X)\cdot\ad(X)) \\
&=& \sum_{i=1}^n \langle \ad(X)\ad(X)e_i,e_i\rangle \\
&=& -\sum_{i=1}^n \langle \ad(X)e_i,\ad(X)e_i\rangle \\
&=& -\sum_{i=1}^n \|\ad(X)e_i\|^2\leq 0.
\end{eqnarray*}}

Note that if there exists a $X\neq 0$ such that $\|\ad(X)e_i\|^2= 0$ for all $i$, then by definition of the Killing form, $B(Y,X)=0$ for each $Y$. This would imply that $B$ is degenerate, contradicting the fact that $\mathfrak{g}$ is semisimple. Therefore, for each $X\neq 0$ we have $B(X,X)<0$. Hence $B$ is negative--definite.
\end{proof}

The next result is an immediate consequence of Corollary~\ref{killingmetric}, Remark~\ref{riccikilling} and Theorem~\ref{compactkilling}.

\begin{corollary}
Let $G$ be a semisimple compact connected Lie group with Killing form $B$. Then $(G,-B)$ is an Einstein manifold.
\end{corollary}

We conclude this section with a discussion on equivalent definitions of semisimple Lie algebras.

Recall that a Lie subalgebra $\mathfrak{h}$ is an {\it ideal}\index{Lie algebra!ideal}\index{Ideal} of a Lie algebra $\mathfrak {g}$ if $[X,Y]\in\mathfrak{h}$, for all $X \in\mathfrak{h},Y\in\mathfrak{g}$. If an ideal $\mathfrak{h}$ has no ideals other than the trivial, $\{0\}$ and $\mathfrak{h}$, it is called {\it simple}.\index{Lie algebra!simple}\index{Simple!Lie algebra} Following the usual convention, by {\it simple Lie algebras} we mean Lie sub algebras that are non commutative simple ideals. We stress that simple ideals, which may be commutative, will not be called Lie algebras, but simply referred to as {\it simple ideals}.\index{Ideal!simple}\index{Simple!ideal}

Given a Lie algebra $\mathfrak{g}$, consider the decreasing sequence of ideals $$\mathfrak{g}^{(1)}=[\mathfrak{g},\mathfrak{g}], \mathfrak{g}^{(2)}=[\mathfrak{g}^{(1)},\mathfrak{g}^{(1)}], \ldots, \mathfrak{g}^{(k)}=[\mathfrak{g}^{(k-1)},\mathfrak{g}^{(k-1)}],\ldots$$ If there exists a positive integer $m$ such that $\mathfrak{g}^{(m)}=\{0\}$, then $\mathfrak{g}$ is said to be {\it solvable}.\index{Lie algebra!solvable}

\begin{example}
Consider the ideal of all $n\times n$ matrices $(a_{ij})$ over $K=\R$ or $K=\C$, with $a_{ij}=0$ if $i>j$. One can easily verify that this is a solvable Lie algebra. Other trivial examples are nilpotent Lie algebras.
\end{example}

It is also possible to prove that every Lie algebra admits a maximal solvable ideal $\tau$, called its {\it radical}.\index{Lie algebra!radical} We recall some results whose prove can be found in Ise and Takeuchi~\cite{IseTakeuchi}.

\begin{proposition}\label{prop-decomposicao-formaKilling-ideais}
The following hold.
\begin{enumerate}
\item[(i)] If $\mathfrak{h}$ is an ideal of $\mathfrak{g}$, then the Killing form $B_{\mathfrak{h}}$ of $\mathfrak{h}$ satisfies $B(X,Y)=B_{\mathfrak{h}}(X,Y)$, for all $X,Y\in\mathfrak{h}$;
\item[(ii)] If $\mathfrak{g}=\mathfrak{g}_{1}\oplus\mathfrak{g}_{2}$ is direct sum of ideals, then $\mathfrak{g}_{1}$ is orthogonal to $\mathfrak{g}_{2}$ with respect to $B$. Thus $B$ is the sum of the Killing forms $B_{1}$ and $B_{2}$ of $\mathfrak{g}_{1}$ and $\mathfrak{g}_{2}$, respectively.
\end{enumerate}
\end{proposition}

\begin{cthm}\label{teo-condicaoSolubilidadeCartan}\index{Theorem!Cartan}
A Lie algebra $\mathfrak{g}$ is solvable if and only if $B(\mathfrak{g},\mathfrak{g}^{(1)})=\{0\}$. In particular, if $B$ vanishes identically, then $\mathfrak{g}$ is solvable.
\end{cthm}

We are now ready to present a theorem that gives equivalent definitions of semisimple Lie algebras.

\begin{theorem}\label{teo-semisimples-equivalencias}
Let $\mathfrak{g}$ be a Lie algebra with Killing form $B$. Then the following are equivalent.

\begin{enumerate}
\item[(i)] $\mathfrak{g}=\mathfrak{g}_1\oplus\dots\oplus\mathfrak{g}_n$ is the direct sum of simple Lie algebras $\mathfrak{g}_i$ (i.e., non commutative simple ideals);
\item[(ii)] $\mathfrak{g}$ has trivial radical $\tau=\{0\}$;
\item[(iii)] $\mathfrak{g}$ has no commutative ideal other than $\{0\}$;
\item[(iv)] $B$ is non degenerate, i.e., $\mathfrak{g}$ is semisimple.
\end{enumerate}
\end{theorem}

Before proving this theorem, we present two important properties of semisimple Lie algebras.

\begin{remark}\label{rem-teo-semisimples-equivalencias}
If the Lie algebra $\mathfrak{g}$ is the direct sum of non commutative simple ideals $\mathfrak{g}=\mathfrak{g}_1\oplus\dots\oplus\mathfrak{g}_n$, then $[\mathfrak{g},\mathfrak{g}]=\mathfrak{g}$.

Indeed, on the one hand, if $i\neq j$, $[\mathfrak{g}_{i},\mathfrak{g}_{j}]=\{0\}$, since $\mathfrak{g}_{i}$ and $\mathfrak{g}_{j}$ are ideals. On the other hand, $[\mathfrak{g}_{i},\mathfrak{g}_{i}]=\mathfrak{g}_{i}$, since $\mathfrak{g}_i^{(1)}=[\mathfrak{g}_{i},\mathfrak{g}_{i}]$ is an ideal and $\mathfrak{g}_{i}$ is a non commutative simple ideal.
\end{remark}

\begin{remark}
If $\mathfrak{g}=\mathfrak{g}_1\oplus\dots\oplus\mathfrak{g}_n$ is the direct sum of simple Lie algebras, then this decomposition is unique up to permutations.

In fact, consider another decomposition $\mathfrak{g}=\widetilde{\mathfrak{g}}_1\oplus\dots\oplus\widetilde{\mathfrak{g}}_m$ and let $\widetilde{X}\in\widetilde{\mathfrak{g}}_{k}$. Then $\widetilde{X}=\sum_{i} X_{i}$, where $X_{i}\in\mathfrak{g}_{i}$. Since $\mathfrak{g}_{i}$ is a non commutative simple ideal, for each $i$ such that $X_{i}\neq 0$, there exists $V_{i}\in\mathfrak{g}_{i}$ different from $X_{i}$, such that $[V_{i},X_{i}]\neq 0$. Hence $[\widetilde{X},V_{i}]\neq 0$ is a vector that belongs to both $\mathfrak{g}_{i}$ and $\widetilde{\mathfrak{g}}_{k}$. Therefore, the ideal $\mathfrak{g}_{i}\cap \widetilde{\mathfrak{g}}_{k}$ is different from $\{0\}$. Since $\mathfrak{g}_{i}$ and $\widetilde{\mathfrak{g}}_{k}$ are simple ideals, it follows that $\mathfrak{g}_{i}=\mathfrak{g}_{i}\cap \widetilde{\mathfrak{g}}_{k}=\widetilde{\mathfrak{g}}_{k}$.
\end{remark}

We now give a proof of Theorem~\ref{teo-semisimples-equivalencias}, based on Ise and Takeuchi~\cite{IseTakeuchi}.

\begin{proof}
We proceed proving the equivalences (i) $\Leftrightarrow$ (ii), (i) $\Leftrightarrow$ (iii) and (i) $\Leftrightarrow$ (iv) one by one.

(i) $\Leftrightarrow$ (ii). Assume that $\mathfrak{g}=\mathfrak{g}_1\oplus\dots\oplus\mathfrak{g}_n$ is the direct sum of non commutative simple ideals. Let $\pi_{i}:\mathfrak{g}\rightarrow\mathfrak{g}_{i}$ denote the projection onto each factor and note that $\pi_{i}$ is a Lie algebra homomorphism. Thus, the projection $\tau_{i}=\pi_{i}(\tau)$ is a solvable ideal of $\mathfrak{g}_{i}$. Since $\mathfrak{g}_{i}$ is simple, then either $\tau_{i}$ is equal to $\{0\}$ or to $\mathfrak{g}_{i}$. However, the solvable ideal $\tau_{i}$ cannot be equal to $\mathfrak{g}_{i}$, since $\mathfrak{g}_{i}=[\mathfrak{g}_{i},\mathfrak{g}_{i}]$ (see Remark~\ref{rem-teo-semisimples-equivalencias}). Therefore, $\mathfrak{\tau}_{i}=\{0\}$. Hence $\mathfrak{\tau}=0$.

Conversely, assume that $\tau=\{0\}$. For any ideal $\mathfrak{h}$ of $\mathfrak{g}$, set $\mathfrak{h}^{\perp}=\{X\in\mathfrak{g} : B(X,\mathfrak{h})=0\}$. Note that $\mathfrak{h}^{\perp}$ and $\mathfrak{h}^{\perp}\cap\mathfrak{h}$ are ideals, and $B$ restricted to $\mathfrak{h}^{\perp}\cap\mathfrak{h}$ vanishes identically. It then follows from Proposition~\ref{prop-decomposicao-formaKilling-ideais} and the Cartan Theorem~\ref{teo-condicaoSolubilidadeCartan} that $\mathfrak{h}^{\perp}\cap\mathfrak{h}$ is solvable. Since the radical is trivial, we conclude that $\mathfrak{h}^{\perp}\cap\mathfrak{h}=\{0\}$. Therefore $\mathfrak{g}=\mathfrak{h}\oplus\mathfrak{h}^{\perp}$. Being $\mathfrak{g}$ finite--dimensional, by induction, $\mathfrak{g}$ is the direct sum of simple ideals. Moreover, the fact that $\tau=\{0\}$ implies that each simple ideal is non commutative.

(i) $\Leftrightarrow$ (iii). Suppose that there exists a nontrivial commutative ideal $\mathfrak{a}$. Then, the radical must contain $\mathfrak{a}$, hence it is nontrivial. It then follows from (i) $\Leftrightarrow$ (ii) that $\mathfrak{g}$ is not a direct sum of non commutative simple ideals.

Conversely, assume that $\mathfrak{g}$ is not a direct sum of non commutative simple ideals. Then, from (i) $\Leftrightarrow$ (ii), the radical $\tau$ is nontrivial, and there exists a positive integer $m$ such that $\tau^{(m-1)}\neq\{0\}$. Set $\mathfrak{a}=\tau^{(m-1)}$ and note that $\mathfrak{a}$ is a nontrivial commutative ideal.

(i) $\Leftrightarrow$ (iv). Assume that $\mathfrak{g}=\mathfrak{g}_1\oplus\dots\oplus\mathfrak{g}_n$ is the direct sum of non commutative simple ideals. From Proposition ~\ref{prop-decomposicao-formaKilling-ideais}, it suffices to prove that $B|_{\mathfrak{g}_{i}}$ is nondegenerate for each $i$. Consider a fixed $i$ and let $$\mathfrak{h}=\{X\in \mathfrak{g}_{i}: B(X,\mathfrak{g}_{i})=0\}.$$ Note that $\mathfrak{h}$ is an ideal of $\mathfrak{g}_{i}$. Since $\mathfrak{g}_{i}$ is a simple ideal, either $\mathfrak{h}=\mathfrak{g}_{i}$ or $\mathfrak{h}=\{0\}$. If $\mathfrak{h}=\mathfrak{g}_{i}$, then the Cartan Theorem~\ref{teo-condicaoSolubilidadeCartan} implies that $\mathfrak{g}_{i}$ is solvable, contradicting the fact that $[\mathfrak{g}_{i},\mathfrak{g}_{i}]=\mathfrak{g}_{i}$. Therefore $\mathfrak{h}=\{0\}$, hence $B|_{\mathfrak{g}_{i}}$ is nondegenerate.

Conversely, assume that $B$ is non degenerate. From the last equivalence, it suffices to prove that $\mathfrak{g}$ has no commutative ideals other than $\{0\}$. Let $\mathfrak{a}$ be a commutative ideal of $\mathfrak{g}$. For each $X\in\mathfrak{a}$ and $Y\in\mathfrak{g}$ we have $\ad(X)\ad(Y)(\mathfrak{g})\subset\mathfrak{a}$. Therefore $B(X,Y)=\trace[\ad(X)\ad(Y)]|_{\mathfrak{a}}$. On the other hand, since $\mathfrak{a}$ is commutative, $\ad(X)\ad(Y)|_{\mathfrak{a}}=0$. Therefore $B(X,Y)=0$ for each $X\in\mathfrak{a}$ and $Y\in\mathfrak{g}$. Since $B$ is non degenerate, it follows that $\mathfrak{a}=\{0\}$.
\end{proof}

\section{Splitting Lie groups with bi--invariant metrics}

In Proposition~\ref{cptbi}, we proved that each compact Lie group admits a bi--invariant metric. In this section we will prove that the only simply connected Lie groups that admit bi--invariant metrics are products of compact Lie groups with vector spaces. We will also prove that if the Lie algebra of a compact Lie group $G$ is simple, then the bi--invariant metric on $G$ is unique up to multiplication by constants.


\begin{theorem}
Let $\mathfrak{g}$ be a Lie algebra endowed with a bi--invariant metric. 
Then $$\mathfrak{g}=\mathfrak{g}_1\oplus\ldots\oplus\mathfrak{g}_n$$ is the direct orthogonal sum of simple ideals $\mathfrak{g}_i$.

In addition, let $\widetilde{G}$ be the connected and simply connected Lie group with Lie algebra isomorphic to $\mathfrak{g}$. Then $\widetilde{G}$ is isomorphic to the product of normal Lie subgroups $$G_1\times\ldots\times G_n,$$ such that $G_i=\R$ if $\mathfrak{g}_i$ 
is commutative and $G_i$ is compact if $\mathfrak g_i$ is non commutative\footnote{For each $i$, $\mathfrak g_i$ denotes the Lie algebra of the normal subgroup $G_i$.}.
\end{theorem}

\begin{proof}
In order to verify that $\mathfrak{g}$ is direct orthogonal sum of simple ideals, it suffices to prove that if $\mathfrak{h}$ is an ideal, then $\mathfrak{h}^\perp$ is also an ideal. Let $X\in\mathfrak{h}^\perp$, $Y\in\mathfrak{g}$ and $Z\in\mathfrak{h}$. Then, using Proposition~\ref{curvatureLie}, it follows that
{\allowdisplaybreaks
\begin{eqnarray*}
\langle [X,Y],Z \rangle &=& -\langle [Y,X],Z \rangle \\
&=& \langle X,[Y,Z] \rangle \\
&=& 0.
\end{eqnarray*}}Hence $[X,Y]\in\mathfrak{h}^\perp$, and this proves the first assertion.

In addition, from Lie's Third Theorem~\ref{liethird}, given $\mathfrak{g}_i$, there exists a unique connected and simply connected Lie group $G_i$ with Lie algebra isomorphic to $\mathfrak{g}_i$. Hence $G_1\times\ldots\times G_n$ is a connected and simply connected Lie group, with Lie algebra $\mathfrak{g}_1\oplus\ldots\oplus\mathfrak{g}_n=\mathfrak{g}$. From uniqueness in Lie's Third Theorem, it follows that $G_1\times\ldots\times G_n=\widetilde{G}$.

Moreover, if $\mathfrak{g}_i$ is commutative and simple, then $\mathfrak{g}_i=\R$. Hence, since $G_i$ is connected and simply connected, $G_i\simeq\R$. Else, if $\mathfrak{g}_i$ is non commutative, we observe that there does not exist $X\in\mathfrak{g}_i, X\neq 0$ such that $[X,Y]=0$, for all $Y\in\mathfrak{g}_i$. Indeed, if there existed such $X$, then $\{\R X\}\subset\mathfrak{g}_i$ would be a non trivial ideal. From the proof of Theorem~\ref{compactkilling}, the Killing form of $\mathfrak{g}_i$ is negative--definite, and from the same theorem, $G_i$ is compact.

Finally, to verify that $G_i$ is a normal subgroup, let $X\in\mathfrak{g}_i$, and $Y\in\mathfrak{g}$. Then $[X,Y]\in\mathfrak{g}_i$, and from \eqref{adexp}, \begin{eqnarray*}
\Ad(\exp(Y))X &=& \exp(\ad(Y))X \\
&=& \sum_{k=0}^{\infty} \frac{\ad(Y)^k}{k!} X \in\mathfrak{g}_i.
\end{eqnarray*}
On the other hand, from \eqref{gexpg}, $$\exp(Y)\exp(X)\exp(Y)^{-1}=\exp(\Ad(\exp(Y))X),$$ and hence $\exp(Y)\exp(X)\exp(Y)^{-1}\in G_i$. From Proposition~\ref{vizger}, $G_i$ is normal.
\end{proof}

The above theorem and Remark~\ref{rem-teo-semisimples-equivalencias} imply the next corollary.

\begin{corollary}
Let $\mathfrak{g}$ be a Lie algebra with a bi--invariant metric. Then $\mathfrak{g}=\Zentrum(\mathfrak{g})\oplus \widetilde{\mathfrak{g}}$ is the direct sum of ideals, where $\widetilde{\mathfrak{g}}$ is a semisimple Lie algebra. In particular, $[\widetilde{\mathfrak{g}},\widetilde{\mathfrak{g}}]=\widetilde{\mathfrak{g}}$.
\end{corollary}

\begin{proposition}\label{prop-Liegroup-uniqueMetric-Einstein}
Let $G$ be a compact simple Lie group with Killing form $B$ and bi--invariant metric $\langle\cdot,\cdot\rangle$. Then the bi--invariant metric is unique up to multiplication by constants. In addition, $(G,\langle\cdot,\cdot\rangle)$ is an Einstein manifold. More precisely, there exists $\lambda\in\R$ such that $$\Ric(X,Y)=\lambda\langle X,Y\rangle.$$\index{$\Ric$}
\end{proposition}

\begin{proof}
Let $(\cdot,\cdot)$ be another bi--invariant metric on $G$. Then there exists a symmetric positive--definite matrix $A$, such that $(X,Y)=\langle AX,Y\rangle$. We claim that $A\ad(X)=\ad(X)A$. Indeed,
{\allowdisplaybreaks
\begin{eqnarray*}
\langle A\ad(X)Y,Z \rangle &=& (\ad(X)Y,Z) \\
&=& -(Y,\ad(X)Z) \\
&=& -\langle A Y,\ad (X)Z \rangle \\
&=& \langle\ad(X)A Y,Z \rangle.
\end{eqnarray*}}

Furthermore, eigenspaces of $A$ are $\ad(X)$-invariant, that is, they are ideals. In fact, let $Y\in\mathfrak{g}$ be an eigenvector of $A$ associated to an eigenvalue $\mu$. Then \begin{eqnarray*}
A\ad(X)Y &=& \ad(X)AY \\
&=& \mu \ad(X)Y.
\end{eqnarray*} Since $\mathfrak{g}$ is simple, it follows that $A=\mu\id$, hence $(X,Y)=\mu\langle X,Y\rangle$, for all $X,Y\in\mathfrak{g}$.

Since $G$ is compact, it follows from Theorem~\ref{compactkilling} and Corollary~\ref{killingmetric} that $-B$ is a bi--invariant metric. Hence, there exists $\lambda$ such that $-B(X,Y)=4\lambda\langle X,Y\rangle$. Therefore, from Remark~\ref{riccikilling}, $G$ is Einstein, \begin{eqnarray*}
\Ric(X,Y) &=& -\tfrac{1}{4}B(X,Y) \\
&=& \tfrac{1}{4} 4\lambda \langle X,Y \rangle \\
&=& \lambda \langle X,Y \rangle.\qedhere
\end{eqnarray*}
\end{proof}

\begin{exercise}
Let $G$ be a compact semisimple Lie group with Lie algebra $\mathfrak{g}=\mathfrak{g}_1\oplus\dots\oplus\mathfrak{g}_n$ given by the direct sum of non commutative simple ideals. Consider $\langle\cdot,\cdot\rangle$ a bi--invariant metric on $G$. Prove that there exist positive numbers $\lambda_j$ such that $$\langle\cdot,\cdot\rangle =\sum_j -\lambda_j B_j(\cdot,\cdot),$$ where $B_j=B|_{\mathfrak{g}_j}$ is the restriction of the Killing form to $\mathfrak{g}_j$.
\end{exercise}

\begin{exercise}\index{Killing!form of $\SU(n)$}\index{$\SU(n)$}\index{$\mathfrak{su}(n),\mathfrak{u}(n)$}
In order to compute the Killing form $B$ of $\SU(n)$, recall that its Lie algebra is $$\mathfrak{su}(n)=\{A\in\mathfrak{gl}(n,\C):A^*+A=0, \trace A=0\}$$ (see Exercise~\ref{ex-classiclie}). Consider special diagonalizable matrices in $\mathfrak{su}(n)$, $$X=\left[ \begin{array}{l l l}
                          i\theta_1 & & \\
                           & \ddots & \\
			  & & i\theta_n \end{array}\right] \mbox{ and } Y=\left[ \begin{array}{l l l}
                          i\zeta_1 & & \\
                           & \ddots & \\
			  & & i\zeta_n \end{array}\right].$$

Use the fact that $\trace X=\trace Y=0$, hence $\sum_{i=1}^n \theta_i=\sum_{i=1}^n \zeta_i=0$, and Exercise~\ref{ex-ad-matrizes} to verify that $B$ calculated in $X$ and $Y$ gives $$B(X,Y)=\trace\left(\ad(X)\ad(Y)\right)=-2n\sum_{i=1}^n \theta_i\zeta_i.$$

From Exercise~\ref{ex-metrica-su}, the inner product $\langle X,Y \rangle=\mathrm{Re} \trace(XY^*)$ in $T_e\SU(n)$ can be extended to a bi--invariant metric $\langle\cdot,\cdot\rangle$. Since $\SU(n)$ is simple, from Proposition~\ref{prop-Liegroup-uniqueMetric-Einstein}, there exists a constant $c\in\R$ such that $B(\cdot,\cdot)=c\langle\cdot,\cdot\rangle$. Conclude that $c=-2n$ and that the Killing form of $\SU(n)$ is $$B(Z,W)=-2n\mathrm{Re}\trace(ZW^*), \; \mbox{ for all } Z,W\in\mathfrak{su}(n).$$
\end{exercise}

\part{Isometric and adjoint actions and some generalizations}
\chapter{Proper and isometric actions}
\label{chap3}

In this chapter, we present a concise introduction to the theory of proper and isometric actions. We begin with some results on fiber bundles, among which the most important are the Slice Theorem~\ref{slicethm} and the Tubular Neighborhood Theorem~\ref{theorem-tubularneighborhood}. These will be used to establish strong relations between proper and isometric actions in Section~\ref{sec32}. As a consequence of Proposition~\ref{proposition-acaoisometrica-eh-propria} and Theorem~\ref{theorem-acaopropria-isometrica}, these are essentially {\em the same} actions, in the sense explained in Remark~\ref{re:isopropsame}. Furthermore, some properties of the so--called principal orbits are explored. Finally, orbit types and the correspondent stratification are studied in the last section.

Further references on the content of this chapter are Palais and Terng~\cite{PalaisTerng}, Duistermaat and Kolk~\cite{duistermaat}, Kawakubo~\cite{livrokawakubo}, Pedrosa~\cite[Part I]{escola} and Spindola~\cite{lucas}.

\section{Proper actions and fiber bundles}\label{Section-properaction-fiberbundle}

In this section, proper actions are introduced together with a preliminary study of principal and associated fiber bundles. The main results are the Slice Theorem~\ref{slicethm} and the Tubular Neighborhood Theorem~\ref{theorem-tubularneighborhood}, that will play an essential role in the theory.

\begin{definition}
Let $G$ be a Lie group and $M$ a smooth manifold. A smooth map $\mu:G\times M\rightarrow M$ is called a {\it left action}\index{Action!left} of $G$ on $M$ if
\begin{itemize}
\item[(i)] $\mu(e,x)=x$, for all $x\in M$;
\item[(ii)] $\mu(g_1,\mu(g_2,x))=\mu(g_1g_2,x)$, for all $g_1,g_2\in G, x\in M$.
\end{itemize}
{\it Right actions}\index{Action!right} of $G$ on $M$ are analogously defined.
\end{definition}

The simplest examples of actions are seen in linear algebra courses, for instance $\GL(n,\R)$ acting on $\R^n$ by multiplication. An important example is the \emph{adjoint action}\index{Adjoint!action}\index{Action!adjoint}\index{$\Ad$} of a Lie group $G$ on its Lie algebra $\mathfrak{g}$, given by the adjoint representation $$\Ad:G\times\mathfrak{g}\ni (g,X)\longmapsto \Ad(g)X\in\mathfrak{g}.$$ This particular action will be studied in detail in Chapter~\ref{chap4}. Other typical examples are actions of a Lie subgroup $H\subset G$ on $G$ by left multiplication or conjugation.

\begin{definition}
Let $\mu:G\times M\rightarrow M$ be a left action and $x\in M$. The subgroup $G_x=\{g\in G:\mu(g,x)=x\}$ is called {\it isotropy group}\index{Isotropy group}\index{$G_x$} or {\it stabilizer}\index{Stabilizer} of $x\in M$ and $G(x)=\{\mu(g,x):g\in G\}$ is called the {\it orbit}\index{Orbit}\index{$G(x)$} of $x\in M$.

In addition, if $\bigcap_{x\in M} G_x=\{e\}$, the action is said to be {\it effective}\index{Action!effective} and if $G_x=\{e\}$, for all $x\in M$, it is said to be {\it free}.\index{Action!free} Finally, if given $x,y\in M$ there exists $g\in G$ with $\mu(g,x)=y$, $\mu$ is said to be {\it transitive}.\index{Action!transitive}
\end{definition}

\begin{remark}\label{remark-effective-action}
Every action can be reduced to an effective action. This is done by taking the quotient of $G$ by the normal subgroup $\bigcap_{x\in M} G_x$, that corresponds to the {\em kernel} of the action.
\end{remark}

If $G(x)$ and $G(y)$ have nontrivial intersection, then they coincide. Hence, orbits of an action of $G$ on $M$ constitute a partition of $M$ and one can consider the quotient space $M/G$, called the {\it orbit space}.\index{Orbit!space}

\begin{exercise}\label{ex-conjisotropy}
Let $\mu:G\times M\rightarrow M$ be any left action. Prove that $G_{\mu(g,x)}=gG_xg^{-1}$.
\end{exercise}

The next result asserts that given a smooth action, we can associate to each element of $\mathfrak{g}$ a vector field on $M$. 

\begin{proposition}
\label{proposition-campoXxi}
Consider a smooth action $\mu:G\times M\rightarrow M$.
\begin{itemize}
\item[(i)] Each $\xi\in\mathfrak{g}$ induces a smooth vector field $X^\xi$ on $M$, given by $$X^\xi(p)=\frac{\mathrm d}{\mathrm dt}\mu(\exp(t\xi),p)\Big|_{t=0};$$
\item[(ii)] The flow of $X^\xi$ is given by $$\varphi^{X^\xi}_{t}(\cdot)=\mu(\exp(t\xi),\cdot).$$
\end{itemize}
\end{proposition}

\begin{proof}
Define \begin{eqnarray*} \sigma^{\xi}:M &\longrightarrow & TG\times TM \\ p &\longmapsto & (\xi_{e},0_{p}). \end{eqnarray*} Note that $\dd\mu\circ\sigma^\xi$ is smooth. Item (i) follows from \begin{eqnarray*}
\dd\mu\circ\sigma^\xi (p) &=& \dd\mu\left(\frac{\mathrm d}{\mathrm dt}(\exp(t\xi),p)\Big|_{t=0}\right)\\
                        &=&\frac{\mathrm d}{\mathrm dt}\left(\mu(\exp(t\xi),p)\Big|_{t=0}\right)\\
                        &=& X^\xi(p).
\end{eqnarray*}

In order to prove (ii), for each $p\in M$ set $\alpha_{p}(t)= \mu(\exp(t\xi),p)$. Then $\alpha_{p}(0)=p$ and \begin{eqnarray*}
\frac{\mathrm d}{\mathrm dt}\alpha_{p}(t)\Big|_{t_{0}}&=&\frac{\mathrm d}{\mathrm ds}\alpha(s+t_{0})\Big|_{s=0}\\
                              &=& \frac{\mathrm d}{\mathrm ds}\mu(\exp(s\xi)\cdot\exp(t_{0}\xi) ,p))\Big|_{s=0}\\
                              &=& X^\xi(\mu(\exp(t_{0}X),p))\\
                              &=& X^\xi(\alpha_{p}(t_{0})).\qedhere
\end{eqnarray*}
\end{proof}

\begin{remark}
The association $\mathfrak{g}\ni\xi\mapsto X^\xi\in\mathfrak{X}(M)$ defined in item (i) is a Lie anti--homomorphism (see Spindola~\cite{lucas}).
\end{remark}

\begin{remark}
\label{remark-proposition-campoXxi}
As we will see in Proposition~\ref{orbitsubmanifold}, each orbit $G(x)$ is an immersed submanifold. Moreover, for each $v\in T_{x}G(x)$ there is a vector field $X^\xi$ such that $X^\xi(x)=v$. Assuming this fact, it follows that the partition $\mathcal{F}=\{G(x)\}_{x\in M}$ by orbits of an action is a foliation, provided that all orbits have the same dimension (see Definition~\ref{definition-foliation}). More generally, without this dimensional assumption, $\F$ is a \emph{singular} foliation (see Remark~\ref{re:dimsemicont} and Definition~\ref{definition-s.r.f}).
\end{remark}

\begin{remark}
Vector fields induced by an action play an important role in the study of \emph{equations of Lie type}. As a reference for this subject we indicate Bryant~\cite{Bryant}.
\end{remark}

\begin{exercise}\index{$\SO(n)$}
\label{ex-rotacaoS0}
Consider the action $\mu:\SO(3)\times\R^{3}\rightarrow \R^{3}$ by multiplication.
\begin{itemize}
\item[(i)] Determine the isotropy groups and orbits of $\mu$;
\item[(ii)] Verify that if $\xi\in\mathfrak{so}(3)$, then $X^\xi(p)=\xi\times p$;
\item[(iii)] Consider $A_{\xi}$ defined in Exercise~\ref{ex-sobre-SO3}. Verify that $e^{tA_{\xi}}$ is a rotation on the axis $\xi$ with angular speed $\|\xi\|$.
\end{itemize}
\end{exercise}

In the sequel, we will need the following technical result, whose proof can be found in Spindola~\cite{lucas}.

\begin{proposition}\label{proposition-kern-mu}
The following hold.

\begin{itemize}
\item[(i)] Let $\mu:G\times M\rightarrow M$ be a left action and $\mu_{x}(\cdot)=\mu(\,\cdot,x)$. Then $\ker\dd(\mu_{x})_{g_{0}}=T_{g_{0}}g_{0}G_{x};$
\item[(ii)] Let $\mu:M\times G\rightarrow M$ be a right action and $\mu_{x}(\cdot)=\mu(x,\cdot)$. Then $\ker\dd(\mu_{x})_{g_{0}}=T_{g_{0}}G_{x}g_{0}.$
\end{itemize}
\end{proposition}

In the sequel, the notation $\mu_x$ introduced above will be constantly used. More precisely, given a left action $\mu:G\times M\to M$, define two auxiliary maps\index{$\mu^g$}\index{$\mu_x$}
\begin{equation*}
\begin{aligned}
\mu^g:M&\longrightarrow M\quad& \mu_x:G&\longrightarrow M \\
x&\longmapsto\mu(g,x)\quad& g&\longmapsto\mu(g,x).
\end{aligned}
\end{equation*}

\begin{definition}
An action $\mu:G\times M\rightarrow M$ is {\it proper}\index{Action!proper} if the map $G\times M\owns (g,x)\mapsto (\mu(g,x),x)\in M\times M$ is proper, i.e., the preimage of any compact set is compact.
\end{definition}

It follows directly from the above definition that each isotropy group of a proper action is compact.

\begin{proposition}\label{proposition-def-equivalente-propria}
An action $\mu:G\times M\rightarrow M$ is proper if, and only if, the following property is satisfied. Let $\{g_n\}$ be any sequence in $G$ and $\{x_n\}$ be a convergent sequence in $M$, such that $\{\mu(g_n,x_n)\}$ also converges. Then $\{g_n\}$ admits a convergent subsequence.
\end{proposition}

\begin{corollary}
Actions of compact groups are always proper.
\end{corollary}

\begin{exercise}\label{exercise-rightmulitiplication-is-proper}
Let $H$ be a closed subgroup of the Lie group $G$. Prove that $G\times H\owns (g,h)\mapsto gh\in G$ is a free proper right action.
\end{exercise}

\begin{exercise}
An action $G\times M\rightarrow M$ is called \emph{properly discontinuous}\index{Action!properly discontinuous} if for all $x\in M$ there exists a neighborhood $U\owns x$ such that for all $g\in G, g\neq e$, then $gU\cap U=\emptyset$. Let $G$ be a discrete group that acts on a manifold $M$. Prove that this action is properly discontinuous if and only if it is free and proper.
\end{exercise}

Proper actions are closely related to principal fiber bundles. Before describing this relation, we give some basic definitions.

\begin{definition}
Let $E$, $B$ and $F$ be manifolds and $G$ a Lie group. Assume that $G\times F\rightarrow F$ is an effective left action and $\pi: E\rightarrow B$ a smooth submersion. Suppose that $B$ admits an open covering $\{U_{\alpha}\}$ and that there exist diffeomorphisms $\psi_{\alpha}:U_{\alpha}\times F\rightarrow \pi^{-1}(U_{\alpha})$ satisfying
\begin{itemize}
\item[(i)] $\pi\circ\psi_{\alpha}=\pi_{1}$, where $\pi_{1}(b,f)=b$;
\item[(ii)] if $U_{\alpha}\cap U_{\beta}\neq\emptyset$, then $\psi_{\beta}^{-1}\circ\psi_{\alpha}(b,f)=(b,\theta_{\alpha, \beta}(b)f)$, where $\theta_{\alpha,\beta}(b)\in G$ and $\theta_{\alpha,\beta}:U_{\alpha}\cap U_{\beta}\rightarrow G$ is smooth.
\end{itemize}
Then, $(E,\pi,B,F,G,U_{\alpha},\psi_{\alpha})$ is called \emph{coordinate bundle}. Moreover, $(E,\pi,B,F,G,\{U_{\alpha}\},\{\psi_{\alpha}\})$ and $(E,\pi,B,F,G,\{V_{\beta}\},\{\varphi_{\beta}\})$ are said to be \emph{equivalent} if $\varphi_{\beta}^{-1}\circ\psi_{\alpha}(b,f)=(b,\widetilde{\theta}_{\alpha, \beta}(b)f)$, where $\widetilde{\theta}_{\alpha,\beta}:U_{\alpha}\cap V_{\beta}\rightarrow G$ is smooth.

An equivalence class of coordinate bundles, denoted $(E,\pi,B,F,G)$, is called a \emph{fiber bundle}.\index{Fiber bundle} $E$ is called the \emph{total space},\index{Fiber bundle!total space} $\pi$ the \emph{projection},\index{Fiber bundle!projection} $B$ the \emph{base space},\index{Fiber bundle!base space} $F$ the \emph{fiber}\index{Fiber bundle!fiber} and $G$ the \emph{structure group}.\index{Fiber bundle!structure group} For each $b\in B$, $\pi^{-1}(b)$ is called the \emph{fiber over $b$} and is often denoted $E_{b}$. Furthermore, $\psi_{\alpha}$ are called \emph{coordinate functions}\index{Fiber bundle!coordinate function} and $\theta_{\alpha,\beta}$ \emph{transition functions}.\index{Fiber bundle!transition function} Fiber bundles are usually denoted only by its total space $E$ if the underlying structure is evident from the context.
\end{definition}

\begin{remark}
The total space $E$ of a fiber bundle can be reconstructed, in rough terms, \emph{gluing} trivial products $U_{\alpha}\times F$ according to transition functions $\theta_{\alpha,\beta}$. More precisely, consider the disjoint union $\bigsqcup_{\alpha} U_{\alpha}\times F$. Whenever $x\in U_{\alpha}\cap U_{\beta}$, identify $(x,f)\in U_{\alpha}\times F$ with $(x,\theta_{\alpha,\beta}(x)(f))\in U_{\beta}\times F$. Then the quotient space $\bigsqcup_{\alpha} U_{\alpha}\times F/\sim$ corresponds to $E$ and the projection on the first factor to the bundle's projection.
\end{remark}

In the previous chapters, we have already seen examples of fiber bundles, such as the tangent bundle $TM$ of a manifold $M$. This is a particular example of a special class of fiber bundles called \emph{vector bundles}.\index{Fiber bundle!vector bundle}\index{Vector bundle} Vector bundles are fiber bundles with fiber $\R^{n}$ and structure group $\GL(n,\R)$.

Another example of fiber bundle that was previously presented is $(\widetilde{M},\rho,M,F,F)$, where $\widetilde{M}$ denotes the universal covering of a manifold $M$, $\rho$ the associated covering map and $F$ a discrete group isomorphic to the fundamental group $\pi_{1}(M)$. This is a particular example of an important class of bundles called principal bundles. A fiber bundle $(P,\rho,B,F,G)$ is a \emph{principal fiber bundle}\index{Fiber bundle!principal} if $F=G$ and the action of $G$ on itself is by left translations. Principal fiber bundles are also called \emph{principal $G$--bundles} or simply \emph{principal bundles}.\index{Principal bundle}

\begin{example}\label{exbtm}
Important examples of principal fiber bundles are the so--called frame bundles. The \emph{frame bundle}\index{Frame bundle} of a manifold $M$ is given by $$B(TM)=\bigcup_{x\in M} B(T_{x}M),$$ where $B(T_{x}M)$ denotes the set of all frames (ordered bases) of the vector space $T_{x}M$. It is easy to see that $B(T_{x}M)$ is diffeomorphic to $\GL(n,\R)$. Then $(B(TM),\rho,M,\GL(n,\R))$ is a principal bundle, where $\rho:B(TM)\rightarrow M$ is the \emph{footpoint projection}.\footnote{This means that if $\xi_{x}$ is a frame of $T_{x}M$, then $\rho(\xi_{x})=x$.}
\end{example}

\begin{proposition}
Principal fiber bundles $(P,\rho,B,G)$ have an underlying free proper right action $\mu:P\times G\rightarrow P$, whose orbits are the fibers.
\end{proposition}

\begin{proof}
Consider coordinate functions $\psi_{\alpha}:U_{\alpha}\times G\rightarrow \rho^{-1}(U_{\alpha})$ and define $$\mu(x,g)=\psi_{\alpha}(\psi_{\alpha}^{-1}(x)\cdot g),$$ where $(b,f)\cdot g=(b,f\cdot g),$ for all $b\in B, f,g\in G$. Let us first check that this action is well--defined, i.e., the definition does not depend on the choice of $\psi_{\alpha}.$ This follows from the fact that the structure group acts on the \emph{left} and the action defined above is a \emph{right} action. In fact, \begin{eqnarray*}
\psi_{\alpha}(\psi_{\alpha}^{-1}(x)\cdot g)&=&\psi_{\alpha}(b,fg)\\
                                           &=&\psi_{\beta}(b,\theta_{\alpha, \beta}(b)fg)\\
                                           &=&\psi_{\beta}((b,\theta_{\alpha, \beta}(b)f)\cdot g)\\
                                           &=&\psi_{\beta}(\psi_{\beta}^{-1}(x)\cdot g).
\end{eqnarray*}

The definition of $\mu$ and Exercise~\ref{exercise-rightmulitiplication-is-proper} imply that $\mu$ is a free proper right action. The observation that its orbits coincide with fibers is immediate from the definition of $\mu$.
\end{proof}

The next theorem provides a converse to the above result, and a method to build principal bundles.

\begin{theorem}\label{theorem-acaoproprialivre-eh-fibrado}
Let $\mu:M\times G\rightarrow M$ be a proper free right action. Then $M/G$ admits a smooth structure such that $(M,\rho,M/G,G)$ is a principal fiber bundle, where $\rho:M\rightarrow M/G$ is the canonical projection.
\end{theorem}

\begin{remark}
The smooth structure on $M/G$ has the following pro\-per\-ties that guarantee its uniqueness:
\begin{itemize}
\item[(i)] $\rho: M\rightarrow M/G$ is smooth;
\item[(ii)] For any manifold $N$ and any map $h:M/G\rightarrow N$, $h$ is smooth if and only if $h\circ\rho$ is smooth. 
\end{itemize}
\end{remark}

\begin{proof}
We will only sketch the main parts of this proof. The first is to find \emph{trivializations}\footnote{That is, diffeomorphisms between open neighborhoods on $M$ and products of open subsets of the base and fibers.}\index{Trivialization} of the desired fiber bundle. Consider a submanifold $S$ containing $x$ that is transverse to the orbit $G(x)$, i.e., such that $T_{x}M=T_{x}S\oplus \dd(\mu_x)_{e}\mathfrak{g}.$

\begin{claim}\label{cl:thmacaoproprialivrefibr1}
Up to reducing $S$, there exists a $G$--invariant neighborhood $U$ of $G(x)$ such that $\varphi:S\times G\rightarrow U$ is a diffeomorphism and $\varphi(s,ga)=\mu(\varphi(s,g),a)$.
\end{claim}

In order to prove this claim, one must first prove that $\varphi$ is a local diffeomorphism in a neighborhood of $S\times\{e\}$, up to reducing $S$. This can be done using the fact that the action is free, and Proposition~\ref{proposition-kern-mu}. It then follows from $$\dd\varphi_{(s,g)}(X,\dd R_{g}Y)=\dd(\mu^{g})_{s}\circ\dd\varphi_{(s,e)}(X,Y)$$ that $\varphi$ is a local diffeomorphism around each point of $S\times G$. Since the action is proper, $\varphi$ is injective and this concludes the proof of this first assertion. Note that $\varphi$ is not a coordinate function, since $S$ is contained in $M$ and not in $M/G$.

The next part of the proof is to identify $S$ with an open subset of $M/G$. To this aim, we recall basic facts from topology. First, the projection $\rho:M\rightarrow M/G$ is continuous if $M/G$ is endowed with the quotient topology.\footnote{In this topology, a subset $U\subset M/G$ is open if $\rho^{-1}(U)$ is open.} It is possible to prove that $M/G$ is Hausdorff if the action is proper, and that $\rho:M\rightarrow M/G$ is an open map. These considerations and Claim~\ref{cl:thmacaoproprialivrefibr1} prove the following.

\begin{claim}
$\rho(S)$ is an open subset of $M/G$ and $\rho|_{S}:S\rightarrow\rho(S)$ is a homeomorphism.
\end{claim}

We are now ready to exhibit an atlas of $M/G$ and coordinate functions of $(M,\rho,M/G,G)$.

\begin{claim}\label{cl:claimcomtheta}
Consider two submanifolds $S_1$ and $S_2$ transverse to two different orbits, with $W=\rho(S_{1})\cap\rho(S_{2})\neq\emptyset$. Set $\rho_{i}=\rho|_{S_{i}}:S_{i}\rightarrow \rho(S_{i})$ and $V_{i}=\rho_{i}^{-1}(W).$ Then
\begin{itemize}
\item[(i)]$\rho_{1}^{-1}\circ\rho_{2}:V_{2}\rightarrow V_{1}$ is a diffeomorphism.
\item[(ii)] $\psi_{2}^{-1}\circ\psi_{1}(b,g)=(b,\theta(b)g),$ where $\psi_{i}:\rho(S_{i})\times G\rightarrow \rho^{-1}(\rho(S_{i}))$ is given by $\psi_{i}(b,g)=\varphi_{i}(\rho_{i}^{-1}(b),g)$ and $b\in W$.
\end{itemize}
\end{claim}

In order to prove this last claim, note that $\varphi_{2}^{-1}(V_{1})$ is a graph, i.e., $\varphi_{2}^{-1}(V_{1})=\{(s,\widetilde{\theta}(s)):s\in V_{2}\}$. This implies that $\widetilde{\theta}$ is smooth. Therefore $\rho_{1}^{-1}\circ\rho_{2}(s)=\mu(s,\widetilde{\theta}(s))$ is smooth and this proves (i).

As for (ii), the map $\theta$ is defined by $\theta(b)=\widetilde{\theta}(\rho_{2}^{-1}(b))$.
\end{proof}

\begin{corollary}\label{cor-quocienteLieGroup}
Let $G$ be a Lie group and $H\subset G$ a closed subgroup acting by right multiplication on $G$. Then $G/H$ is a manifold and $(G,\rho, G/H, H)$ is a principal fiber bundle, where $\rho:G\rightarrow G/H$ denotes the canonical projection. In addition, if $H$ is  a normal subgroup then $G/H$ is a Lie group and $\rho$ is a Lie group homomorphism. 
\end{corollary}

\begin{proof}
The first part of the above result follows directly from Theorem~\ref{theorem-acaoproprialivre-eh-fibrado} and Exercise~\ref{exercise-rightmulitiplication-is-proper}.

To prove that $G/H$ is a Lie group when $H$ is a normal subgroup, define \begin{eqnarray*}\alpha:G\times G\ni (a,b) &\longmapsto & ab^{-1}\in G \\ \widetilde{\alpha}: G/H\times G/H\ni(aH,bH) &\longmapsto &ab^{-1}H\in G/H.\end{eqnarray*} Note that $\widetilde{\alpha}$ is well--defined since $H$ is normal. The fact that $\rho:G\rightarrow G/H$ is a projection of a bundle and  $\rho\circ\alpha=\widetilde{\alpha}\circ(\rho\times\rho)$ imply that $\widetilde{\alpha}$ is smooth. Therefore $G/H$ is a Lie group.
\end{proof}

\begin{exercise}
Let $G$ be a Lie group and consider $\rho:\widetilde{G}\rightarrow G$ its universal covering. Prove that
\begin{itemize}
\item[(i)] $H=\rho^{-1}(e)$ is a normal discrete closed subgroup, and $gh=hg$, for all $h\in H, g\in G$;
\item[(ii)] $G$ is isomorphic to $\widetilde{G}/H$;
\item[(iii)] $\pi_1(G)$ is abelian.
\end{itemize}
\end{exercise}

\begin{exercise}\label{su2so3}
Prove that $\SU(2)$ is the universal covering of $\SO(3)$, using the following items.\index{$\SU(n)$}\index{$\SO(n)$} Recall that $S^3\subset\Hr$ is isomorphic to $\SU(2)$ (see Exercise~\ref{su2s3}).
\begin{itemize}
\item[(i)] Let $g\in S^3$, $\theta\in\R$ and $u\in S^2$ be such that $g=\cos\theta+\sin\theta u$. Define $T_g(v)= gvg^{-1}$ for all $v\in\R^3$. Prove that $T_g$ is a linear (orthogonal) transformation of $\R^3$ and $T_g=e^{A_{2\theta u}}$, where $A_{\xi}$ is as in Exercise~\ref{ex-sobre-SO3};
\item[(ii)] Prove that $\varphi: S^3\owns g\mapsto T_g\in\SO(3)$ is a covering map. Conclude that $\pi_1(\SO(3))\simeq\Z_2$.
\end{itemize}
\end{exercise}

It follows from Corollary~\ref{cor-quocienteLieGroup} that the quotient $G/G_x$ is a smooth manifold. As we prove in the sequel, the orbit $G(x)$ is the image of an immersion of $G/G_x$ into $M$.

\begin{proposition}\label{orbitsubmanifold}
Let $\mu:G\times M\rightarrow M$ be a left action. Consider $\widetilde{\mu}_x:G/G_x\rightarrow M$ defined by $\widetilde{\mu}_x\circ\rho=\mu_x,$ where $\rho:G\rightarrow G/G_x$ is the canonical projection. Then $\widetilde{\mu}_{x}$ is an injective immersion, whose image is $G(x)$. In particular, $G(x)$ is an immersed submanifold of $M$. In addition, if the action is proper, then $\widetilde{\mu}_x$ is an embedding and $G(x)$ is an embedded submanifold of $M$.
\end{proposition}

\begin{proof}
According to Corollary~\ref{cor-quocienteLieGroup}, $(G,\rho,G/G_{x},G_{x})$ is a principal fiber bundle. This implies that $\widetilde{\mu}_{x}$ is smooth. It follows from Proposition~\ref{proposition-kern-mu} that the derivative of $\widetilde{\mu}_x$ at every point is injective, and hence $\widetilde{\mu}_{x}$ is an injective immersion. The fact that $\widetilde{\mu}_x$ is an embedding when the action is proper can be proved using Proposition~\ref{proposition-def-equivalente-propria}.
\end{proof}

\begin{remark}\label{re:dimsemicont}
From Proposition~\ref{orbitsubmanifold}, orbits $G(x)$ of a left action $\mu:G\times M\to M$ are immersed submanifolds of $M$. The dimension of these submanifolds $G(x)$ clearly depends on $x\in M$. This dependence satisfies an important property, namely, {\em lower semi--continuity}, i.e., for each $x_0\in M$, the dimension of orbits $G(x)$ for $x$ near $x_0$ is greater or equal to $\dim G(x_0)$. This can be proved by observing that $\dim G(x)$ is given by the rank of the linear map $\dd(\mu_x)_e:\mathfrak g\to T_xM$, which is in fact invariant for points in $G(x)$. From continuity of the map $x\mapsto\dd(\mu_x)_e$ and lower semi--continuity of the rank function for a continuous family of linear maps, it follows that $\dim G(x)$ is lower semi--continuous.
\end{remark}

\begin{exercise}\label{ex-diffeos}
Verify that there exists diffeomorphisms that give the following identifications:
\begin{itemize}
\item[(i)] $S^{n}=\SO(n+1)/\SO(n)$;
\item[(ii)] $\R P^n=\SO(n+1)/S(\O(n)\times\O(1))$;\index{$\SO(n)$}\index{$\R P^n$}
\item[(iii)] $\C P^n=\SU(n+1)/S(\U(n)\times\U(1))$.\index{$\SU(n)$}\index{$\C P^n$}
\end{itemize}
Here $S(\O(n)\times\O(1))$ denotes the subgroup of $\SO(n+1)$ consisting of matrices $$A=\left( \begin{array}{l r} B & 0\\ 0 & \pm1 \end{array}\right),$$ where $B\in\O(n)$ and $\det A=1$. Analogously for $S(\U(1)\times\U(n))$.

\medskip
\noindent {\small \emph{Hint:} Use Proposition~\ref{orbitsubmanifold}. For instance, to prove (ii) note that the action of $\SO(n+1)$ in $S^n$ induces an action in $\R P^n$.}
\medskip
\end{exercise}

\begin{remark}\index{$k$--Grassmannian}\index{Grassmannian}\index{$\mathrm{Gr}_k(V)$}
The same technique can be used to prove statements more general than (ii) and (iii) of Exercise~\ref{ex-diffeos}, on $k$--Grassmannians of $\R^n$ and $\C^n$. Recall that if $V$ is a finite--dimensional vector space, the \emph{$k$--Grassmannian}\index{$k$--Grassmannian} of $V$ is given by $$\mathrm{Gr}_k(V)=\{W\mbox{ subspace of } V:\dim W=k\}.$$ Hence $\mathrm{Gr}_1(\R^{n+1})=\R P^n$ and $\mathrm{Gr}_1(\C^{n+1})=\C P^n$. Observe that $\SO(n)$, respectively $\SU(n)$, acts on $\mathrm{Gr}_k(\R^n)$, respectively $\mathrm{Gr}_k(\C^n)$, by multiplication. These more general statements are

\begin{itemize}
\item[(ii')] $\mathrm{Gr}_k(\R^n)=\SO(n)/S(\O(n-k)\times\O(k))$;
\item[(iii')] $\mathrm{Gr}_k(\C^n)=\SU(n)/S(\U(n-k)\times\U(k))$.
\end{itemize}
\end{remark}

\begin{remark}
\label{remark-flagmanifold}
These techniques can also be used to prove that $\SU(n)/T$ is a \emph{complex flag manifold}\index{Flag}\index{Complex flag manifold}, where $T$ is the subgroup of $\SU(n)$ of diagonal matrices. Indeed, let $F_{1,\ldots,n-1}(\C^{n})$ be the set of \emph{complex flags}, i.e.,
$$F_{1,\ldots,n-1}(\C^{n})=\{\{0\}\subset E_1\subset \ldots\subset E_{n-1}:E_i\in\mathrm{Gr}_i(\C^{n})\}.$$
Set $\widetilde{F}=\{(l_1,\ldots,l_n):l_1\dots l_n \mbox{ orthogonal lines of } \C^{n}\}.$ On the one hand, there is a natural bijection between $\widetilde{F}$ and $F_{1,\ldots,n-1}(\C^{n})$ given by $$(l_1,\ldots,l_n)\longmapsto \{0\}\subset E_1\subset\ldots\subset E_{n-1},$$ where $E_i=l_1\oplus\cdots\oplus l_i$. On the other hand, $\SU(n)$ acts naturally on $\widetilde{F}$ and this action is transitive. Note that, for an orthogonal basis $e_1,\ldots,e_n$ of $\C^{n}$, the isotropy group of the point $(\C e_1,\ldots,\C e_n)\in \widetilde{F}$ is $T$. Therefore, $\SU(n)/T=\widetilde{F}=F_{1,\ldots,n-1}(\C^{n})$.
\end{remark}

A fundamental concept in the theory of proper actions is that of {\it slice}, which is defined as follows.

\begin{definition}\label{definition-slice}
Let $\mu:G\times M\rightarrow M$ be an action. A {\it slice}\index{Slice}\index{$S_x$}\index{Action!slice} at $x_0\in M$ is an embedded submanifold $S_{x_0}$ containing $x_0$ and satisfying the following properties:
\begin{itemize}
\item[(i)] $T_{x_0}M=\dd\mu_{x_0}\mathfrak{g}\oplus T_{x_0}S_{x_0}$ and $T_xM=\dd\mu_x\mathfrak{g}+T_xS_{x_0}$, for all $x\in S_{x_{0}}$;
\item[(ii)] $S_{x_0}$ is invariant under $G_{x_0}$, i.e., if $x\in S_{x_0}$ and $g\in G_{x_0}$, then $\mu(g,x)\in S_{x_0}$;
\item[(iii)] Consider $x\in S_{x_0}$ and $g\in G$ such that $\mu(g,x)\in S_{x_0}$. Then $g\in G_{x_0}$.
\end{itemize}
\end{definition}

\begin{example}\label{ex-action-slice-isotropy}
Consider the action of $S^{1}\times \R$ on $\C\times\R=\R^{3}$ defined as $\mu((s,l),(z,t))=(s\cdot z,t+l)$. One can easily check that this is a proper action and determine the isotropy groups, orbits and slices at different points. In fact, for $x=(z_{0},t_{0})$ with $z_{0}\neq 0$, the isotropy group $G_{x}$ is trivial, the orbit $G(x)$ is a cylinder with axis $\{(0,t)\in\C\times\R: t\in\R\}$ and a slice $S_{x}$ at $x$ is a small segment of the straight line that joins $x$ to $(0,t_{0})$, that does not intersect the axis. If $x=(0,t_{0})$ then $G_{x}=S^{1}$, the orbit $G(x)$ is the axis $\{(0,t)\in\C\times\R:t\in\R\}$ and a slice $S_{x}$ at $x$ is a disc $\{(z,t_{0}),\in \C\times\R:|z|<\varepsilon\}$.
\end{example}

\begin{slthm}\label{slicethm}
Let $\mu:G\times M\rightarrow M$ be a proper action and $x_{0}\in M$. Then there exists a slice $S_{x_{0}}$ at $x_{0}$.
\end{slthm}
\begin{proof}

We start with the construction of a metric on $M$ such that $G_{x_{0}}$ is a compact subgroup of isometries of $M$. This metric is defined as follows $$\langle X, Y \rangle_{p}= \int_{G_{x_{0}}} b(\dd\mu^{g}X, \dd\mu^{g}Y)_{\mu(g,p)}\ \omega,$$ where $\omega$ is a right--invariant volume form of the compact Lie group $G_{x_{0}}$ and $b$ is an arbitrary metric. Using the same arguments of Proposition~\ref{cptbi}, one can check that $\langle \cdot, \cdot \rangle$ is in fact $G_{x_{0}}$--invariant.

It is not difficult to see that $\dd\mu^{g}$ leaves the tangent space $T_{x_{0}}G(x_{0})$ invariant, where $g\in G_{x_0}$. Therefore, since $G_{x_{0}}$ is a subgroup of isometries of $M$, $\dd\mu^{g}$ also leaves the {\em normal space}\footnote{The normal space at $x\in N$ of a submanifold $N\subset M$ of a Riemannian manifold $M$ is an orthogonal complement to $T_xN$ in $T_xM$, and will be henceforth denoted $\nu_xN$.} $\nu_{x_{0}}G(x_{0})$ invariant.

We now define the candidate to be a slice at $x_{0}$ by setting $$S_{x_{0}}=\exp_{x_{0}}(B_{\varepsilon}(x_{0})),$$ where $B_{\varepsilon}(x_{0})$ is an open ball in the normal space $\nu_{x_{0}}G(x_{0})$.

Since $\dd\mu^{g}$ leaves $\nu_{x_{0}}G(x_{0})$ invariant and isometries map geodesics to geodesics, $S_{x_{0}}$ is invariant by the action of $G_{x_{0}}.$ Therefore, item (ii) of Definition~\ref{definition-slice} is satisfied. Item (i) is satisfied by the construction of $S_{x_{0}}$ and continuity of $\dd\mu$.

It only remains to verify item (iii), which will be proved by contradiction. Suppose that item (iii) is not satisfied. Thus there exists a sequence $\{x_{n}\}$ in $S_{x_{0}}$ and a sequence $\{g_{n}\}$ in $G$ such that $\lim x_n=x_0$, $\lim \mu(g_n,x_n)=x_0$, $\mu(g_n,x_n)\in S_{x_{0}}$ and $g_{n}\notin G_{x_{0}}$. Since the action is proper, there exists a subsequence which will be also denoted $\{g_n\}$, that converges to $g\in G_{x_{0}}$. Set $\widetilde{g}_n=g^{-1}g_{n}$. Then $\lim \widetilde{g}_n=e$, $\lim\mu(\widetilde{g}_n,x_n)=x_0$, $\mu(\widetilde{g}_n,x_n)\in S_{x_{0}}$ and $\widetilde{g}_n\notin G_{x_{0}}$. Using Proposition~\ref{proposition-kern-mu}, the Inverse Function Theorem and reducing $S_{x_{0}}$ if necessary, one can prove the next claim.

\begin{claim}\label{cl:slicethmcl}
There exists a submanifold $C\subset G$ that contains $e$, such that $\mathfrak{g}=\mathfrak{g}_{x_{0}}\oplus T_{e}C$, where $\mathfrak g_{x_0}$ is the Lie algebra of $G_{x_0}$.
Furthermore, the map $\varphi:C\times S_{x_{0}}\ni (c,s)\mapsto\mu(c,s)\in M$ is a diffeomorphism.
\end{claim}

It follows from the Inverse Function Theorem that for each $\widetilde{g}_n$ there exists a unique $c_n\in C$ and $h_n\in G_{x_{0}}$ such that $\widetilde{g}_n=c_{n}h_{n}$. Since $\widetilde{g}_n\notin G_{x_{0}}$, we conclude that $c_{n}\neq e$. On the other hand, $\mu(h_n,x_n)\in S_{x_{0}}$ because $h_n\in G_{x_{0}}$. The fact that $c_{n}\neq e$ and Claim~\ref{cl:slicethmcl} imply that $\mu(c_{n},\mu(h_{n},x_n))\notin S_{x_{0}}$. This contradicts the fact that $\mu(\widetilde{g}_n,x_n)\in S_{x_{0}}$. Therefore such sequence $\{g_n\}$ does not exist, hence item (iii) is satisfied.
\end{proof}

Before presenting the next result, we recall the concept of fiber bundle with fiber $F$ associated to a principle bundle $(P,\rho,B,H)$. Consider the free right proper action $\mu_{1}:P\times H\rightarrow P$ and a left action $\mu_{2}: H\times F\rightarrow F$. 
Then the left action \begin{eqnarray*} \mu: H\times (P\times F) &\longrightarrow & (P\times F) \\ (h,(p,f)) &\longmapsto & (\mu_{1}(p,h^{-1}),\mu_{2}(h,f))\end{eqnarray*} is a left proper action. Define $P\times_{H}F$ as the orbit space of $\mu$ and set $B=P/H$. 


\begin{theorem}\label{theorem-fibradoassociado}
$P\times_{H}F$ is a manifold, called \emph{twisted space},\index{Twisted space} and the 5--uple $(P\times_{H}F,\pi, B,F,H)$ is a fiber bundle where $\pi:P\times_{H}F\rightarrow B$ is defined as $\pi([x,f])=\rho(x)$ and $\rho: P\rightarrow B=P/H$ is the canonical projection.
\end{theorem}

\begin{proof}
We will only sketch the main parts of this proof.

Consider $S$ a transverse submanifold to an $H$--orbit in $P$ and an $H$--invariant neighborhood $U$ in $P$, defined as in Claim~\ref{cl:thmacaoproprialivrefibr1}.

\begin{claim}\label{cl:thmfibrassoc1}
The map $\varphi:S\times F\rightarrow U\times_{H} F$, given by $\varphi(s,f)=[s,f]$, is a diffeomorphism.
\end{claim}

In order to prove Claim~\ref{cl:thmfibrassoc1}, one has to check that \begin{eqnarray*}\widetilde{\varphi}:H\times(S\times F) &\longrightarrow & U\times F \\ (h,(s,f)) & \longmapsto & (\mu_{1}(s,h^{-1}),\mu_{2}(h,f))\end{eqnarray*} is an \emph{$H$--equivariant diffeomorphism}, i.e., $h\cdot\widetilde{\varphi}(p)=\widetilde{\varphi}(h\cdot p)$. Here $H$ acts on $H\times(S\times F)$ by left multiplication on the first factor and $H$ acts on $U\times F$ by the twisted action defined above. The fact that $\widetilde{\varphi}$ is an $H$--equivariant diffeomorphism implies that $\varphi$ is a diffeomorphism between $S\times F= (H\times(S\times F))/H$ and $U\times_{H} F=(U\times F)/H$. This proves Claim~\ref{cl:thmfibrassoc1}.

We are now ready to present the coordinate functions of the fiber bundle. Consider two transverse submanifolds $S_1$ and $S_2$ to two different orbits and set $\rho_{i}=\rho|_{S_{i}}:S_{i}\rightarrow \rho(S_{i})$. Assume also that $W=\rho(S_{1})\cap\rho(S_{2})$ is nonempty. 

\begin{claim}\label{cl:thmfibrassoc2}
The following hold.
\begin{itemize}
\item[(i)] The following map is a diffeomorphism \begin{eqnarray*}\psi_{i}:\rho(S_{i})\times F &\longrightarrow & U\times_{H}F \\ (b,f) &\longmapsto &[\rho_{i}^{-1}(b),f];\end{eqnarray*}
\item[(ii)] $\psi_{2}^{-1}\circ\psi_{1}(b,f)=(b,\theta(b)f)$, where $\theta$ was defined in Claim~\ref{cl:claimcomtheta}.
\end{itemize}
\end{claim}

Claim~\ref{cl:thmfibrassoc1} implies item (i) of Claim~\ref{cl:thmfibrassoc2}. As for (ii), it suffices to prove that if $[s_{1},f_{1}]=[s_{2},f_{2}]$, then $f_{2}=\mu_{2}(\widetilde{\theta}(s_{2}),f_{1})$, where $\widetilde{\theta}$ was defined in Claim~\ref{cl:claimcomtheta}. The fact that $[s_{1},f_{1}]=[s_{2},f_{2}]$ implies that there exists $h$ such that
\begin{eqnarray}
f_{2}&=&\mu_{2}(h,f_{1}),\label{eq-theorem-fibradoassociado-1} \\
s_{1}&=&\mu_{1}(s_{2},h).\label{eq-theorem-fibradoassociado-2}
\end{eqnarray}
On the other hand, from the proof of Theorem~\ref{theorem-acaoproprialivre-eh-fibrado},
\begin{equation}
\label{eq-theorem-fibradoassociado-3}
s_{1}=\mu_{1}(s_{2},\widetilde{\theta}(s_{2})).
\end{equation}
Equations \eqref{eq-theorem-fibradoassociado-2} and \eqref{eq-theorem-fibradoassociado-3}, and the fact that the action on $P$ is free imply that
\begin{equation}
\label{eq-theorem-fibradoassociado-4}
\widetilde{\theta}(s_{2})=h.
\end{equation}
Finally, substituting \eqref{eq-theorem-fibradoassociado-4} in (\ref{eq-theorem-fibradoassociado-1}), we get the desired result.
\end{proof}

The bundle $(P\times_{H}F,\pi,B,F,H)$ is called \emph{fiber bundle associated to the principal bundle}\index{Fiber bundle!associated to a principal bundle} $(P,\rho,B,H)$ with fiber $F$. A typical example of a fiber bundle associated to a principal bundle is the tangent bundle $TM$ of a manifold $M$, that is associated to the principal bundle $B(TM)$, defined in Example~\ref{exbtm}.

Existence of slices allows to consider a \emph{tubular neighborhood}\index{Tubular neighborhood}\index{$\tub$} of each orbit $G(x_0)$, defined by $\tub(G(x_0))=\mu(G,S_{x_0})$. The next theorem asserts that this is the total space of the fiber bundle associated to the principal bundle $(G,\rho,G/G_{x_0},G_{x_0})$ with fiber $S_{x_0}$.

\begin{tnthm}\label{theorem-tubularneighborhood}\index{Theorem!Tubular Neighborhood}
Let $\mu:G\times M\rightarrow M$ be a proper action and $x_{0}\in M$. Then there exists a $G$--equivariant\footnote{Recall that given actions $\mu_1:G\times M\rightarrow M$ and $\mu_2:G\times N\rightarrow N$, a map $f:M\rightarrow N$ is called \emph{$G$--equivariant}\index{Equivariant!map}\index{$G$--equivariant map} if for all $x\in M, g\in G$, $$\mu_2(g,f(x))=f(\mu_1(g,x)).$$} diffeomorphism between $\tub(G(x_{0}))$ and the total space of the fiber bundle associated to the principal fiber bundle $(G,\rho,G/G_{x_{0}},G_{x_{0}})$ with fiber $S_{x_{0}}$. In other words, $\tub(G(x_{0}))=G\times_{G_{x_{0}}} S_{x_{0}}$.
\end{tnthm}

\begin{remark}
The considered action $G\times (G\times_{G_{x_{0}}} S_{x_{0}})\rightarrow G\times_{G_{x_{0}}} S_{x_{0}}$ is given by $h\cdot [g,s]=[hg,s]$.
\end{remark}

\begin{proof}
Define the map \begin{eqnarray*}\varphi: G\times S_{x_{0}} &\longrightarrow & G(S_{x_{0}}) \\ (g,x) &\longmapsto &\mu(g,x), \end{eqnarray*}
and note that $\dd\varphi_{(g,x)}$ is surjective. This can be proved using the fact that $\dd\varphi_{(e,x)}$ is surjective and
$\dd\varphi_{(g,x)}(\dd L_{g} X,\dd L_{g} Y)=\dd(\mu^{g})_{x}\circ \dd\varphi_{(e,x)}(X,Y)$. Since $\dd\varphi_{(g,x)}$ is surjective, $\tub(G(x_{0}))=G(S_{x_{0}})$ is an open neighborhood of $G(x_0)$, that is obviously $G$--invariant.

\begin{claim}\label{cl:tn1}
$\varphi(g,x)=\varphi(h,y)$ if and only if $h=gk^{-1}$ and $y=\mu(k,x)$, where $k\in G_{x_{0}}.$
\end{claim}

One part of Claim~\ref{cl:tn1} is clear. As for the other, assume that $\varphi(g,x)=\varphi(h,y)$. Hence $\mu(g,x)=\mu(h,y)$ and 
$y=\mu(k,x)$, where $k=h^{-1}g$. Since $x,y\in S_{x_{0}}$ it follows from Definition~\ref{definition-slice} that $k\in G_{x_{0}}$.

We are ready to define the candidate to $G$--equivariant diffeomorphism,
\begin{eqnarray*}
\psi:G \times_{G_{x_{0}}} S_{x_{0}} & \longrightarrow & \tub(G(x_{0})) \\
\left[g,s\right] & \longmapsto & \mu(g,s).
\end{eqnarray*}
Claim~\ref{cl:tn1} and the fact that $\pi:G\times S_{x_{0}}\rightarrow G\times_{G_{x_{0}}}S_{x_{0}}$
is a projection of a principal fiber bundle imply that $\psi$ is well--defined, smooth and bijective.

\begin{claim}\label{cl:tn2}
$\psi$ is a $G$--equivariant diffeomorphism.
\end{claim}

To prove Claim~\ref{cl:tn2}, note that $\dd\pi$ and $\dd\varphi$ are surjective, and $\varphi=\psi\circ\pi$. Therefore $\dd\psi$ is surjective.
On the other hand,
\begin{eqnarray*}
\dim G\times_{G_{x_{0}}} S_{x_{0}}&=& \dim(G/G_{x_{0}})+\dim S_{x_{0}}\\
&=&\dim M\\
&=&\dim \tub(G(x_{0})).
\end{eqnarray*}
The above equation and the fact that $\dd\psi$ is surjective imply that $\dd\psi$ is an isomorphism.
Hence, by the Inverse Function Theorem, $\psi$ is a local diffeomorphism. Since $\psi$ is bijective, we conclude that $\psi$ is a diffeomorphism. By construction, $\psi$ is also $G$--equivariant, concluding the proof.
\end{proof}

\begin{remark}\label{remark-SliceVaiEmSlice}
Using the above result, one can prove that $S_{\mu(g,x_{0})}=\mu^{g}(S_{x_{0}})$. It also follows from the above theorem that there is a unique $G$--equivariant retraction $r:\tub(G(x_{0}))\rightarrow G(x_{0})$ such that $r\circ\psi=\widetilde{\mu}_{x_{0}}\circ\pi$, where $\pi:G\times_{G_{x_{0}}} S_{x_{0}}\rightarrow G/G_{x_{0}}$ is the projection of the fiber bundle and $\widetilde{\mu}_{x_{0}}$ defined as in Proposition~\ref{orbitsubmanifold}.
\end{remark}

\begin{remark}\label{remarkslicerep}\index{Holonomy!group}
Let $\mu:G\times M\rightarrow M$ be an action whose orbits have constant dimension. Then the Tubular Neighborhood Theorem~\ref{theorem-tubularneighborhood} implies that the holonomy group\index{$\hol$} $\hol(G(x),x)$ of the leaf $G(x)$ of the foliation $\{G(y)\}_{y\in M}$ coincides with the image of the \emph{slice representation} of $G_x$ in $S_x$ (see Definition~\ref{slicerepresentation} and Remark~\ref{remarkslicerep2}). In the general case of foliations with compact leaves and finite holonomy, there is an analogous result to the Tubular Neighborhood Theorem~\ref{theorem-tubularneighborhood} known as {\it Reeb Local Stability Theorem} (see Moerdijk and Mr\v{c}un~\cite{Moerdijk}).
\end{remark}

\section{Isometric actions and principal orbits}
\label{sec32}

This section is devoted to the important relation between proper and isometric actions, and to a special type of orbits of such actions. More precisely, we prove in Proposition~\ref{proposition-acaoisometrica-eh-propria} that isometric actions are proper, and in Theorem~\ref{theorem-acaopropria-isometrica} that for each proper action there exists an invariant metric that turns it into an isometric action. Furthermore, the concept of principal orbit is studied together with a few of its geometric properties, see Proposition~\ref{prop-equivariant-field}. In addition, we prove the Principal Orbit Theorem~\ref{teo-orbita-principal} that, for instance, ensures that the set of points in principal orbits is open and dense.

\begin{proposition}\label{proposition-acaoisometrica-eh-propria}\index{$\Iso(M)$}
Let $M$ be a Riemannian manifold and $G$ a closed subgroup of the isometry group $\Iso(M)$. Then the action $\mu:G\times M\owns (g,x)\mapsto g(x) \in M$ is a proper action.
\end{proposition}

\begin{proof}
We will only prove the case in which $M$ is complete.

Consider a sequence $\{g_{n}\}$ in $G$ and a sequence $\{x_{n}\}$ in $M$ such that $\lim \mu(g_{n},x_{n})=y$ and $\lim x_{n}=x$. We have to prove that there exists a convergent subsequence of $\{g_{n}\}$ in $G$. Choose $N_{0}$ such that $d(y,\mu(g_n,x_n))<\tfrac{\varepsilon}{2}$ and $d(x,x_n)<\tfrac{\varepsilon}{2}$ for $n>N_{0}$.

\begin{claim}\label{cl:isomprop1}
For a fixed $R_1$, $\mu(g_n,B_{R_1}(x))\subset B_{\varepsilon+R_{1}}(y)$ for $n>N_{0}$.
\end{claim}

In fact, for $z\in B_{R_1}(x)$, \begin{eqnarray*}
d(\mu(g_n,z),y)&\leq & d(y,\mu(g_n,x_n))+d(\mu(g_n,x_n),\mu(g_n,z))\\
               &< & \tfrac{\varepsilon}{2}+ d(x_n,z)\\
               &\leq& \tfrac{\varepsilon}{2}+d(z,x)+d(x,x_n)\\
               &<& \varepsilon+R_1. 
\end{eqnarray*}

The fact that each $g_n$ is an isometry implies the following assertion.

\begin{claim}\label{cl:isomprop2}
$\{g_n\}$ is an equicontinuous family of functions.
\end{claim}

The Arzel\`a--Ascoli Theorem\footnote{This theorem gives a criterion for convergence of continuous maps in the compact--open topology. More precisely, a sequence of continuous functions $\{g_n:K\rightarrow B\}$ between compact metric spaces $K$ and $B$ admits a uniformly convergent subsequence if $\{g_n\}$ is \emph{equicontinuous}, i.e., for every $\varepsilon$ there exists a $\delta>0$ such that $d(g_n(x),g_n(y))<\varepsilon$ whenever $d(x,y)<\delta$, $x,y\in K$ and $n\in\N$.}, Claims~\ref{cl:isomprop1} and~\ref{cl:isomprop2} and the fact that closed balls in $M$ are compact (since $M$ was assumed complete) imply that there exists a subsequence\footnote{Here we denote by $\{g_{\; n}^{(1)}\}$ a generic subsequence of $\{g_n\}$. Inductively, $\{g_{\; n}^{(j+1)}\}$ will denote a subsequence of a certain sequence $\{g_{\; n}^{(j)}\}$. Indexes $j$ indicate the choice of a subsequence, and should not be thought as exponents.} $\{g_{\; n}^{(1)}\}$ of $\{g_n\}$ that converges uniformly on $\overline{B_{R_1}(x)}$ to a continuous map from $\overline{B_{R_1}(x)}$ to $M$. Using the same argument, we inductively define a subsequence $\{g_{\; n}^{(i)}\}$ of $\{g_{\; n}^{(i-1)}\}$ that converges uniformly on $\overline{B_{R_i}(x)}$ to a continuous map from $\overline{B_{R_i}(x)}$ to $M$, where $R_i>R_{i-1}$.

Finally, consider the \emph{diagonal} subsequence $\{g_{\; i}^{(i)}\}$. Note that it converges uniformly on each $\overline{B_{R_i}(x)}$ to a continuous map $g:M\rightarrow M$. It then follow from Myers--Steenrod Theorem~\ref{theorem-Myers-Steenrod} that $g$ is an isometry that belongs to $G$.
\end{proof}

Using the Slice Theorem~\ref{slicethm} we prove the converse result, see Remark~\ref{re:isopropsame}.

\begin{theorem}\label{theorem-acaopropria-isometrica}
Let $\mu:G\times M\rightarrow M$ be a proper action. Then there exists a $G$--invariant metric of $M$ such that $\mu^G=\{\mu^g:g\in G\}$ is a closed subgroup of $\Iso(M)$.
\end{theorem}

\begin{proof}
We start by recalling the following result due to Palais~\cite{Palais-partition}.

\begin{claim}\label{cl:palais}
If $U_\alpha$ is a locally finite open cover of $G$--invariant open sets, then there exists a smooth partition of unity $\{f_{\alpha}\}$ subordinate to $U_\alpha$ such that each $f_\alpha$ is $G$--invariant.
\end{claim}

We want to find a $G$--invariant locally finite open cover $\{U_\alpha\}$ of $M$ such that $U_{\alpha}=G(S_{x_{\alpha}})$, where $S_{x_{\alpha}}$ is a slice at a point $x_\alpha$. To this aim, note that $M/G$ is paracompact, i.e., an arbitrary open covering of $M/G$ has a locally finite refinement. This follows from a classic topology result that asserts that a Hausdorff locally compact space is paracompact if it is a union of countably many compact spaces.

Since $M/G$ is paracompact, there exists a locally finite open covering $\{\pi(S_{x_{\alpha}})\}$ of $M/G$. We then set $U_{\alpha}=\pi^{-1}(\pi(S_{x_{\alpha}}))$. Let $\{f_{\alpha}\}$ be a $G$--invariant partition of unit subordinate to $U_\alpha$ (see Claim~\ref{cl:palais}).

Define an orthogonal structure in $TM|_{S_{x_{\alpha}}}$ by $$\langle X, Y\rangle^{\alpha}_{p}=\int_{G_{x_{\alpha}}} b(\dd\mu^{g}X,\dd\mu^{g}Y)_{\mu(g,p)}\, \omega,$$ where $\omega$ is a right invariant volume form of $G_{x_{\alpha}}$ and $b$ is an arbitrary metric. It is not difficult to see that this orthogonal structure is $G_{x_{\alpha}}$--invariant.

We now define a $G$--invariant metric on $U_\alpha$ by $$\langle \dd \mu^{g}X, \dd\mu^{g}Y \rangle^{\alpha}_{\mu(g,p)}=\langle X, Y\rangle^{\alpha}_{p}.$$ This metric is well--defined since $\langle\cdot,\cdot\rangle^{\alpha}|_{S_{x_{\alpha}}}$ is $G_{x_{\alpha}}$--invariant. Finally, we define the desired metric by $$\langle\cdot,\cdot\rangle=\sum_{\alpha}f_{\alpha}\langle \cdot, \cdot\rangle^{\alpha}.$$

To conclude the proof, we must verify that $\mu^{G}$ is a closed subgroup of $\Iso(M)$. Assume that $\mu^{g_{n}}$ converges uniformly in each compact of $M$ to an isometry $f:M\rightarrow M$. Thus, for each point $x\in M$ we have that $\lim \mu(g_{n},x)=f(x)$. Hence, properness of the action implies that a subsequence $\{g_{n_{i}}\}$ converges to $g\in G$. It is not difficult to check that if $\lim g_{n_{i}}=g$, then $\mu^{g_{n_{i}}}$ converges uniformly in each compact of $M$ to $\mu^{g}$. Therefore $\mu^{g}=f$, concluding the proof.
\end{proof}

\begin{remark}\label{re:isopropsame}
If the action $\mu:G\times M\rightarrow M$ is effective, the above theorem implies that we can identify $G$ with a closed subgroup of $\Iso(M)$ for a particular metric. In this sense, the theorem above is a converse to Proposition~\ref{proposition-acaoisometrica-eh-propria}. Thus, proper effective actions and actions of closed subgroups of isometries are essentially the same topic.
\end{remark}

We now present some results about a particular type of orbits, the so--called principal orbits.

\begin{definition}
Let $\mu:G\times M\rightarrow M$ be a proper action. Then $G(x)$ is a {\it principal orbit}\index{Orbit!principal} if there exists a neighborhood $V$ of $x$ in $M$ such that for each $y\in V$, $G_x\subset G_{\mu(g,y)}$ for some $g\in G$.
\end{definition}

In rough terms, principal orbits are the ones that have the smallest isotropy group (or the largest \emph{type}, see Definition~\ref{def-type}) among the nearby orbits.

\begin{proposition}\label{proposition-equivalencia-orbitasprincipais}
Let $\mu:G\times M\rightarrow M$ be a proper left action. Then the following are equivalent.
\begin{itemize}
\item[(i)] $G(x)$ is a principal orbit;
\item[(ii)] Consider $S_x$ a slice at $x$. For each $y\in S_{x}$, $G_{y}=G_{x}$.
\end{itemize}
\end{proposition}

\begin{proof}
We will prove both implications separately. First, let us prove that (i) implies (ii). From the definition of slice, for each $y\in S_{x}$ we have $G_{y}\subset G_{x}$. On the other hand, since $G(x)$ is a principal orbit, there exists $g$ such that $gG_{x}g^{-1}\subset G_{y}$. Therefore $gG_{x}g^{-1}\subset G_{y}\subset G_{x}.$ Since $G_x$ is compact, we conclude that inclusions above are equalities. In particular, $G_{y}=G_{x}$.

To prove the converse, we have to prove that for $z$ in a tubular neighborhood of $G(x)$, there exists $g$ such that $G_{x}\subset g G_{z}g^{-1}$. From the Tubular Neighborhood Theorem~\ref{theorem-tubularneighborhood}, the orbit $G(z)$ intersects the slice $S_x$ at least at one point, say $y$. Since $y$ and $z$ are in the same orbit, there exists $g$ such that $G_{y}=gG_{z}g^{-1}$. Therefore $G_{x}=G_{y}=gG_{z}g^{-1}$.
\end{proof}

\begin{remark}
As we will see in Chapter~\ref{chap5}, the above proposition implies that each principal orbit is a leaf with trivial holonomy of the foliation $\F=\{G(x)\}_{x\in M}$.
\end{remark}

\begin{exercise}\label{exercise-induce-effectiveaction}
Let $\mu:G\times M\rightarrow M$ be a proper action. Set $H=\bigcap_{x\in M} G_{x}$ and consider the induced effective action $\widetilde{\mu}:G/H\times M\rightarrow M$. Prove that $\widetilde{\mu}$ is smooth and proper. Check that if an orbit $G(x)$ is principal with respect to $\widetilde{\mu}$, then it is also principal with respect to $\mu$.
\end{exercise}

Let $\mu:G\times M\rightarrow M$ be a proper isometric action and $\nu_{x}G(x)$ the normal space of $G(x)$ at $x$. It is not difficult to prove that the image by the exponential map at $x$ restricted to an open subset of $\nu_{x}G(x)$ is a slice $S_{x}$ at $x$. In other words, there exists $\varepsilon>0$ such that $$S_{x}=\{\exp_{x}(\xi):\xi\in\nu_{x}G(x),\|\xi\|<\varepsilon\},$$ compare with Example~\ref{ex-action-slice-isotropy}. The slice constructed in this way will be called \emph{normal slice at}\index{Slice!normal}\index{Normal slice} $x$.

\begin{definition}\label{slicerepresentation}
Let $\mu:G\times M\rightarrow M$ be an isometric proper action and $S_x$ a normal slice at $x$. The \emph{slice representation} of $G_x$ in $S_x$ is defined by $$G_{x}\ni g\longmapsto \dd\mu^{g}|_{S_{x}}\in\O(\nu_{x}G(x)).$$
\end{definition}

\begin{exercise}\label{exercise-representacaoisotropica}
Let $\mu:G\times M\rightarrow M$ be an isometric proper action and $S_x$ a normal slice at $x$. Prove that the statements below are equivalent.
\begin{itemize}
\item[(i)] $G(x)$ is a principal orbit;
\item[(ii)] The slice representation is trivial.
\end{itemize}
\end{exercise}

Let us explore some geometric properties of orbits of isometric actions.

\begin{proposition}\label{prop-equivariant-field}
Let $\mu:G\times M\rightarrow M$ be an isometric proper action and $G(x)$ a principal orbit. Then the following hold.
\begin{itemize}
\item[(i)] A geodesic $\gamma$ orthogonal to an orbit $G(\gamma(0))$ remains orthogonal to all  orbits it intersects.
\item[(ii)] Given $\xi\in\nu_xG(x)$, $\widehat{\xi}_{\mu(g,x)}=\dd(\mu^{g})_{x}\xi$ is a well--defined normal vector field along $G(x),$ called  \emph{equivariant normal field};\index{Equivariant!normal field}
\end{itemize}
For the following, let $\widehat{\xi}$ be an equivariant normal field.
\begin{itemize}
\item[(iii)] $\mathcal S_{\widehat{\xi}_{\mu(g,x)}}=\dd\mu^{g} \mathcal S_{\widehat{\xi}_{x}} \dd\mu^{g^{-1}}$, where $\mathcal S_{\widehat{\xi}}$ is the shape operator of $G(x)$;
\item[(iv)] Principal curvatures of $G(x)$ along an equivariant normal field $\widehat{\xi}$ are constant;
\item[(v)] $\{\exp(\widehat{\xi}_y):y\in G(x)\}$ is an orbit of $\mu$.
\end{itemize}
\end{proposition}

\begin{proof}
In order to prove item (i), by Proposition~\ref{proposition-campoXxi} and Remark~\ref{remark-proposition-campoXxi}, it suffices to prove that if a Killing vector field $X$ is orthogonal to $\gamma'(0)$, then $X$ is orthogonal to $\gamma'(t)$ for all $t$. Since $\langle\nabla_{\gamma'(t)}X,\gamma'(t)\rangle=0$ (see Proposition~\ref{proposition-eq-Killing}) and $\gamma$ is a geodesic, we have that $\frac{\mathrm d}{\mathrm dt}\langle X, \gamma'(t)\rangle=0$. Therefore $X$ is always orthogonal to $\gamma'(t)$.
Item (ii) follows from Exercise~\ref{exercise-representacaoisotropica}.

Item (iii) follows from
\begin{eqnarray*}
\langle \dd\mu^{g^{-1}}\mathcal S_{\widehat{\xi}_{\mu(g,x)}}\dd\mu^{g}W ,Z\rangle_{x} &=& \langle\mathcal S_{\widehat{\xi}_{\mu(g,x)}}\dd\mu^{g}W ,\dd\mu^{g} Z\rangle_{\mu(g,x)}\\
&=&\langle -\nabla_{\dd\mu^{g}W}\dd(\mu^{g})_{x}\xi, \dd\mu^{g}Z\rangle_{\mu(g,x)}\\
 &=&\langle -\nabla_{W}\widehat{\xi},Z\rangle_{x}\\
 &=&\langle \mathcal S_{\widehat{\xi}_{x}} W, Z\rangle_{x}.
\end{eqnarray*}

As for (iv), note that if $\mathcal S_{\widehat{\xi}}X=\lambda X$, then $\dd\mu^{g^{-1}}\mathcal S_{\widehat{\xi}_{\mu(g,x)}}\dd\mu^{g}X=\lambda X$. Hence $\mathcal S_{\widehat{\xi}_{\mu(g,x)}}\dd\mu^{g}X=\lambda \dd\mu^{g} X$. Finally, item (v) follows from
\begin{eqnarray*}
\exp_{\mu(g,x)}(\widehat{\xi}_{\mu(g,x)})&=&\exp_{\mu(g,x)}(\dd\mu^{g}\xi_{x})\\
&=&\mu^{g}\exp_{x}(\xi).\qedhere
\end{eqnarray*}
\end{proof}

\begin{remark}\label{remarkequivariantfield}
The above proposition illustrates a few concepts and results of Chapters~\ref{chap4} and~\ref{chap5}. Item (i) implies that the foliation defined by the partition by orbits of a proper isometric action is a \emph{singular Riemannian foliation} (see Definition~\ref{defsrfs}). Item (v) implies that one can reconstruct the partition by orbits of an action taking all parallel submanifolds to a principal orbit. This is a consequence of \emph{equifocality}, which is valid for every singular Riemannian foliation. Item (iv) and the fact that equivariant normal fields are parallel normal fields when the action is \emph{polar} imply that principal orbits of a polar action on Euclidean space are \emph{isoparametric} (see Definitions~\ref{definitionPolarAction} and~\ref{definition-isoparametric}).
\end{remark}

We conclude this section proving the so--called Principal Orbit Theorem, that among others asserts that the subset of points on principal orbits is open and dense.

\begin{pothm}\label{teo-orbita-principal}
Let $\mu:G\times M\rightarrow M$ be a proper action, where $M$ is connected, and denote by $M_{\mathrm{princ}}$\index{$M_\mathrm{princ}$} the set of points of $M$ contained in principal orbits. Then the following hold.
\begin{itemize}
\item[(i)] $M_{\mathrm{princ}}$ is open and dense in $M$;
\item[(ii)] $M_{\mathrm{princ}}/G$ is a connected embedded submanifold of $M/G$;
\item[(iii)] Let $G(x)$ and $G(y)$ be two principal orbits. Then there exists $g\in G$ such that $G_x=gG_yg^{-1}$.
\end{itemize}
\end{pothm}

\begin{proof}
We first prove existence of a principal orbit. Since $G$ has finite dimension and isotropy groups are compact, we can choose $x\in M$ such that $G_{x}$ has the smallest dimension among isotropy groups and, for that dimension, the smallest number of components. Let $S_{x}$ be a slice at $x$. The definition of slice implies that, for each $y\in S_{x},$ we have $G_{y}\subset G_{x}$. By construction, we conclude that $G_{y}=G_{x}$. Hence, from Proposition~\ref{proposition-equivalencia-orbitasprincipais}, $G(x)$ is a principal orbit.

In order to prove that $M_{\mathrm{princ}}$ is open, consider a point $x\in M_{\mathrm{princ}}$. Proposition~\ref{proposition-equivalencia-orbitasprincipais} implies that, for each $y\in S_{x}$ we have $G_{y}=G_{x}$. We claim that each $y\in S_{x}$ belongs to a principal orbit. Indeed, if a point $z$ is close to $y$, then $z$ is in a tubular neighborhood of $G(x)$. From the Tubular Neighborhood Theorem~\ref{theorem-tubularneighborhood}, the orbit of $z$ intersect $S_{x}$ at least in one point. Hence, there exists $g$ such that $\mu(g,z)=w\in S_{x}$. Therefore $G_{y}=G_{x}=G_{w}=gG_{z}g^{-1}.$ This implies that each point $y\in S_x$ belongs to a principal orbit. Thus, once more from the Tubular Neighborhood Theorem~\ref{theorem-tubularneighborhood}, each point in the tubular neighborhood of $G(x)$ belongs to a principal orbit.

To prove that $M_{\mathrm{princ}}$ is dense, consider $p\notin M_{\mathrm{princ}}$ and $V$ a neighborhood of $p$. Again, choose a point $x\in V\cap G(S_{p})$ such that $G_{x}$ has the smallest dimension among isotropy groups and, for that dimension, the smallest number of components. Then from the argument above we conclude that $x\in M_{\mathrm{princ}}$.

Proposition~\ref{proposition-equivalencia-orbitasprincipais} and some arguments from Theorem~\ref{theorem-acaoproprialivre-eh-fibrado} can be used to prove that $M_{\mathrm{princ}}/G$ is a manifold. Item (iii) is a consequence of item (ii). We will only prove that $M_{\mathrm{princ}}/G$ is connected. To this aim, assume that the action is proper and isometric. We will say that a set $K\subset M/G$ \emph{does not locally disconnect} $M/G$ if for each $p\in M/G$ we can find a neighborhood $U$ such that 
$U\setminus K$ is path--connected. Using (i), it is not difficult to verify the following.

\begin{claim}
If $(M\setminus M_{\mathrm{princ}})/G$ does not locally disconnect $M/G$, then $M_{\mathrm{princ}}/G$ is path--connected.
\end{claim}

Therefore, to prove that $M_{\mathrm{princ}}/G$ is connected, it suffices to prove that $(M\setminus M_{\mathrm{princ}})/G$ does not locally disconnect $M/G$. This can be done using the fact that $S_{x}/G_{x}=G(S_{x})/G,$ the slice representation, Exercise~\ref{exercise-induce-effectiveaction} and the next claim.

\begin{claim}\label{cl:pothm2}
Let $K$ be a closed subgroup of $\O(n)$ acting in $\R^{n}$ by multiplication. Then $\R^{n}_{\mathrm{princ}}/K$ is path--connected.
\end{claim}

In order to prove Claim~\ref{cl:pothm2}, we proceed by induction. If $n=1$, then $K=\Z_{2}$ or $K=1$. In both cases, $\R_{\mathrm{princ}}/K$ is path--connected. For each sphere $S^{n-1}$ centered at the origin, we can apply the induction hypothesis and the slice representation to conclude that $(S^{n-1}\setminus S^{n-1}_{\mathrm{princ}})/K$ does not locally disconnect $S^{n-1}/K.$ Therefore $S^{n-1}_{\mathrm{princ}}/K$ is path--connected. Now consider $x,y\in \R^{n}_{\mathrm{princ}}$ and set $\widetilde{x}$ the projection of $x$ in the sphere $S^{n-1}$ that contains $y$. Note that the points  in the line joining $x$ to $\widetilde{x}$ belong to principal orbits. Therefore the orbits $K(x)$ and $K(\widetilde{x})$ are connected by a path in $\R^{n}_{\mathrm{princ}}/K$. As we have already proved, $K(y)$ and $K(\widetilde{x})$ are path--connected in $S^{n-1}_{\mathrm{princ}}/K$ and hence in $\R^{n}_{\mathrm{princ}}/K$. Therefore $K(x)$ and $K(y)$ are connected by a path in $\R^{n}_{\mathrm{princ}}/K$, concluding the proof.
\end{proof}

\section{Orbit types}

In this section, we discuss some properties of orbit types of proper actions. For instance, we prove that on compact manifolds there exists only a finite number of these, see Theorem~\ref{theorem-n-finito-types-orbits}. In addition, it is also proved that each connected component of a set of orbits of same type is a connected component of the total space of a certain fiber bundle, see Theorem~\ref{theorem-mesmotipo-esptotalFibrado}. Finally, it is also proved that the connected components of a set of orbits of the same type gives a stratification of $M$, see Theorem~\ref{teo-stratification}.

\begin{definition}\label{def-type}
Let $\mu:G\times M\rightarrow M$ be a proper action.
\begin{itemize}
\item[(i)] Two orbits\footnote{It is also said that each pair of points $x'\in G(x)$ and $y'\in G(y)$ \emph{are of the same type}. See also Exercise~\ref{ex-conjisotropy}.} $G(x)$ and $G(y)$ are of the {\it same type}\index{Orbit!same type} if there exists $g\in G$ such that $G_x=G_{\mu(g,y)}$;
\item[(ii)] The orbit $G(x)$ has a type {\it larger}\index{Orbit!larger type} than $G(y)$ if there exists $g\in G$ such that $G_{x}\subset G_{\mu(g,y)}$;
\item[(iii)] An orbit $G(x)$ is said to be {\it regular}\index{Orbit!regular} if the dimension of $G(x)$ coincides with the dimension of principal orbits;
\item[(iv)] A non principal regular orbit is called {\it exceptional}\index{Orbit!exceptional};
\item[(v)] A non regular orbit is called {\it singular}.\index{Orbit!singular}
\end{itemize}
\end{definition}

Note that Theorem~\ref{teo-orbita-principal} asserts not only that $M_{\mathrm{princ}}$ is open and dense in $M$, but also that there exists a unique type of principal orbit.

\begin{exercise}\label{exercise-effectiveaction-type}
Let $\mu:G\times M\rightarrow M$ be a proper action. Set $H=\bigcap_{x\in M} G_{x}$ and consider the induced effective action $\widetilde{\mu}:G/H\times M\rightarrow M$. Verify that if $(G/H)(x)$ and $(G/H)(y)$ are of the same type with respect to $\widetilde{\mu}$, then $G(x)$ and $G(y)$ are of the same type with respect to $\mu$.
\end{exercise}

\begin{exercise}
Let $\mu:G\times M\to M$ be a proper action and $G(p)$ a principal orbit. Prove that $G(x)$ is an exceptional orbit if and only if $\dim G(x)=\dim G(p)$ and the number of connected components of $G_x$ is greater than the number of connected components of $G_p$. 
\end{exercise}

\begin{exercise}\label{ex-RP2-orbita-excepcional}
Consider the action $\widetilde{\mu}:S^{1}\times S^{2}\rightarrow S^{2}$ of the circle $S^1$ on the sphere $S^2$ by rotations around the $z$--axis. Verify that this action induces an action $\mu:S^{1}\times \R P^{2}\rightarrow \R P^{2}$ that has an exceptional orbit.
\end{exercise}


\begin{exercise}\label{ex-conjugacaoSU(3)}\index{$\SU(n)$}
Consider the action $\mu:\SU(3)\times\SU(3)\rightarrow \SU(3)$ by conjugation. Prove that orbits are diffeomorphic to one of the following manifolds.
\begin{itemize}
\item[(i)] $\{\lambda I:\lambda^{3}=1\}$, where $I\in\SU(3)$ is the identity;
\item[(ii)] $\SU(3)/T$, where $T$ is the group of diagonal complex matrices $(t_{ij})$, such that $t_{ii}=\pm1$ and $t_{11}\cdot t_{22}\cdot t_{33}=1$;
\item[(iii)] $\SU(3)/S(\U(2)\times \U(1))$.
\end{itemize}
Conclude, using Exercise~\ref{ex-diffeos} and Remark~\ref{remark-flagmanifold}, that orbits are diffeomorphic to a point, or a complex flag manifold or $\C P^2$.
 
\medskip
\noindent {\small \emph{Hint:} Use the fact that each matrix of $\SU(3)$ is conjugate to a matrix of $T$.}
\end{exercise}

\begin{remark}\label{remarkpolar}
As we will see in the next chapter, Exercises~\ref{ex-RP2-orbita-excepcional} and~\ref{ex-conjugacaoSU(3)} give examples of polar actions. An isometric action is a \emph{polar action}\index{Action!polar} if for each regular point $x$, the set $\exp_{x}(\nu_{x}G(x))$ is a totally geodesic manifold that intersects every orbit orthogonally (for details, see Definition~\ref{definitionPolarAction}). In Exercise~\ref{ex-RP2-orbita-excepcional}, the polar action admits an exceptional orbit, and in Exercise~\ref{ex-conjugacaoSU(3)}, the polar action admits only principal and singular orbits. Note that $\R P^2$ is not simply connected and $\SU(3)$ is simply connected. As we will see in Chapter~\ref{chap5}, polar actions do not admit exceptional orbits if the ambient space is simply connected.
\end{remark}

\begin{remark}
We stress that there are several examples of isometric actions that are not polar. For instance, $\mu:S^{1}\times (\C\times\C)\rightarrow (\C\times\C)$ defined by $\mu(s,(z_{1},z_{2}))=(s\cdot z_{1}, s \cdot z_{2})$.
\end{remark}

\begin{theorem}\label{theorem-n-finito-types-orbits}
Let $\mu:G\times M\rightarrow M$ be a proper action. For each $x\in M$, there exists a slice $S_x$ such that the tubular neighborhood $\mu(G,S_{x})$ contains only finitely many different types of orbits. In particular, if $M$ is compact, there is only a finite number of different types of orbits in $M$.
\end{theorem}

\begin{proof}
From Theorem~\ref{theorem-acaopropria-isometrica}, consider a metric on $M$ such that $\mu^G\subset\Iso(M)$. It is not difficult to check the following.

\begin{claim}\label{cl:nfinito1}
If $y,z\in S_{x}$ have the same $G_{x}$--orbit type, then $y$ and $z$ have the same $G$--orbit type.
\end{claim}

Therefore, it suffices to prove the result for the action of $G_{x}$ on $S_{x}$. Due to the slice representation (see Definition~\ref{slicerepresentation}) and Exercise~\ref{exercise-effectiveaction-type}, we can further reduce the problem to proving the result for an action of $K\subset \O(n)$ on $\R^{n}$ by multiplication.

To this aim, we proceed by induction on $n$. If $n=1$, then $K=\Z_{2}$ or $K=1$. In both cases there exists only a finite number of different types of orbits. For a sphere $S^{n-1}$, we can apply the induction hypothesis, Claim~\ref{cl:nfinito1} and the slice representation, to conclude that there exists a finite number of types of $K$--orbits. Since for each $p\in S^{n-1}$ the points in the segment joining the origin to $p$ (apart from the origin) have the same isotropy type, we conclude that there exists a finite number of types of $K$--orbits in $\R^{n}$.
\end{proof}

In what follows, we will prove that each connected component of a set of orbits of same type is a connected component of the total space of a certain fiber bundle. To this aim we present the following propositions (whose proofs can be found in Duistermaat and Kolk~\cite{duistermaat}) and fix some notations.

\begin{proposition}\label{proposition-conjuntofixo}
Let $\mu:G \times M \rightarrow M$ be a proper (isometric) action and $H\subset G$ a compact subgroup. Then the connected components of $$M^{H}=\{x\in M:\mu(h,x)=x, \mbox{ for all }h\in H\}$$ are (totally geodesic) submanifolds. In addition, $$T_{p}M^{H}=\{X\in T_{p}M, \dd\mu^{h}(p)X=X, \mbox{ for all } h\in H\}.$$
\end{proposition}

\begin{proof}
Let us sketch the main idea of the proof. Consider a closed subgroup $K\subset\O(n)$. Note that $(\R^{n})^{K}$ is a vector subspace of $\R^n$. In fact, if $k(x)=x$, then $k$ fixes the line segment that joins $x$ to the origin. It is easy to see that if $k$ fixes two vectors, then $k$ fixes the entire subspace generated by these vectors. At this point, one can prove the general result using the above observation, the isotropy representation $\dd\mu_{x}:G_{x}\times T_{x}M\to T_{x}M$, Theorem~\ref{theorem-acaopropria-isometrica} and a connectedness argument.
\end{proof}

We now study the concept of {\em local types}\footnote{See Remark~\ref{re:duistermaatfeztd}.} of orbits, which contains more geometric information on the group action than the {\em types} of orbits, given in Definition~\ref{def-type}.

\begin{definition}
Two orbits $G(x)$ and $G(y)$ of a proper action are of the \emph{same local type}\index{Orbit!same local type} if there exists a $G$--equivariant diffeomorphism $\varphi:G(S_{x})\rightarrow G(S_{y})$, for slices $S_x$ and $S_y$.
\end{definition}

\begin{proposition}\label{proposition-localtype}
Let $\mu:G \times M\rightarrow M$ be a proper action and consider two orbits $G(y)$ and $G(z)$. Then $G(y)$ and $G(z)$ are of the same local type if and only if
\begin{itemize}
\item[(i)] There exists $g_0\in G$ such that $G_{z}=g_0G_{y}g_0^{-1}$ (i.e., $y$ and $z$ are of the same type);
\item[(ii)] There exists a linear isomorphism $A:T_{y}S_{y}\rightarrow T_{z}S_{z}$ such that $A(\dd(\mu^{h})_{y}V)=\dd(\mu^{g_{0}hg_{0}^{-1}})_{z}A(V)$ for all $h\in G_{y}$.
\end{itemize}
\end{proposition}

In this way, we have two equivalence relations, namely \emph{to be of the same orbit type} (denoted $\sim$); and \emph{to be of the same local orbit type} (denoted $\approx$). In the same fashion, denote $M_{x}^{\sim}$ the set of $y\in M$ with the same type of $x$. This equivalence class will be called the \emph{orbit type}\index{Orbit!type} of $x$. Denote also $M_{x}^{\approx}$\index{$M_{x}^{\approx},M_{x}^{\sim}$} the set of $y\in M$ with the same local type of $x$, called the \emph{local orbit type}\index{Orbit!local type} of $x$. 

Recall that if $H\subset G$ is a subgroup, the \emph{normalizer}\index{Normalizer} of $H$ in $G$ is given by $N(H)=\{g\in G:gHg^{-1}=H\}$. We are now ready to prove that $M_{x}^{\approx}$ is the total space of a fiber bundle with fiber $G/G_{x}$.

\begin{theorem}\label{theorem-mesmotipo-esptotalFibrado}
Let $\mu:G\times M\rightarrow M$ be a proper action. Then
\begin{itemize}
\item[(i)] Each local orbit type is an open and closed subset of the corresponding orbit type;
\item[(ii)] The set $M_{x}^{\approx}\cap M^{G_{x}}$ is a $N(G_{x})$--invariant manifold open in $M^{G_{x}}$;
\item[(iii)] The action of $N(G_{x})/G_{x}$ on $M_{x}^{\approx}\cap M^{G_{x}}$ is proper and free;
\item[(iv)] There exists a $G$--equivariant diffeomorphism from the total space $G/G_{x}\times_{N(G_{x})/G_{x}} (M_{x}^{\approx}\cap M^{G_{x}})$ onto the manifold $M_{x}^{\approx}$.
\item[(v)]$M_{x}^{\approx}$ is  the total space of a fiber bundle  with fiber $G/G_{x}$ and basis  $(M_{x}^{\approx}\cap M^{G_{x}})/ (N(G_{x})/G_{x}).$ 
\end{itemize}
\end{theorem}

\begin{remark}\label{re:duistermaatfeztd}
To our knowledge, the concept of {\em local types} was introduced by Duistermaat and Kolk~\cite{duistermaat}, on which our proof is based on.
One advantage of local types is that all connected components of $M_{x}^{\approx}$ have the same dimension. This does not necessarily occur with $M_{x}^{\sim}$. Furthermore, it seems that a generalization of this concept would be useful in the theory of \emph{singular Riemannian foliations} (see Definition~\ref{defsrfs}).
\end{remark}

\begin{remark}
A simplification of the proof of Theorem~\ref{theorem-mesmotipo-esptotalFibrado} gives existence of a $G$--equivariant diffeomorphism between the orbit type of $x$, $M_{x}^{\sim}$, and $G/G_{x}\times_{N(G_{x})/G_{x}} (M_{x}^{\sim}\cap M^{G_{x}})$. Nevertheless, as remarked above, $M_{x}^{\sim}$ can be a union of manifolds with different dimensions. As we will see in Claim~\ref{cl:duist2} below, $y\in M_{x}^{\sim}\cap M^{G_{x}}$ if and only if $G_{y}=G_{x}.$
\end{remark}


\begin{proof}
We will only sketch the main parts of this proof. We also stress that the most important result of this theorem is item (iv), and its proof is similar to that of the Tubular Neighborhood Theorem~\ref{theorem-tubularneighborhood}.

Using the properties that define a slice and the Tubular Neighborhood Theorem~\ref{theorem-tubularneighborhood}, one can prove the following.

\begin{claim}\label{cl:duist1}
The following hold.
\begin{itemize}
\item[(i)] For each $x\in M$, denote $M_{x}^{\leq}$ the set of $y\in M$ such that $G_{x}\subset G_{\mu(g,y)}$ for some $g$. Then $M_{x}^{\leq}=\mu^G(M^{G_{x}})$. In addition, $M_{x}^{\leq}$ is closed in $M$;
\item[(ii)] $M_{x}^{\sim}\cap \mu^G(S_{x})=\mu^G(S_{x}^{G_{x}})=M^{\leq}\cap \mu^G(S_{x})$ is a submanifold of $M$, \emph{$G$--equivalent}\footnote{Two manifolds are said to be \emph{$G$--equivalent}\index{$G$--equivalent} if there exists a $G$--equivariant diffeomorphism between them.} to $G/G_{x}\times S_{x}^{G_{x}}$.
\end{itemize}
\end{claim}

Proposition~\ref{proposition-localtype} and the slice representation imply that the points of $S_{x}^{G_{x}}$ are of the same local type. From this fact and Claim~\ref{cl:duist1}, it follows that each local orbit type is an open subset of the corresponding orbit type. It is also closed, since its complement is the union of all other local orbit types in the same orbit type. This proves (i).

\begin{claim}\label{cl:duist2}
$y\in M_{x}^{\sim}\cap M^{G_{x}}$ if and only if $G_{y}=G_{x}.$
\end{claim}

In fact, if $y\in M_{x}^{\sim}\cap M^{G_{x}}$ then $G_{y}=gG_{x}g^{-1}\supset G_{x}$. Since $G_{x}$ is compact, $G_y=G_x$. The other implication is also simple. Using the fact that $M^{G_{x}}\subset M_{x}^{\leq}$, Claim~\ref{cl:duist1} and (i) one can prove that $M_{x}^{\approx}\cap M^{G_{x}}$ is open in $M^{G_{x}}$. Then the next claim implies that that $M_{x}^{\approx}\cap M^{G_{x}}$ is a $N(G_{x})$--invariant manifold, concluding the proof of (ii).

\begin{claim}\label{cl:duist3}
The following hold.
\begin{itemize}
\item[(a)] Let $y\in M_{x}^{\sim}\cap M^{G_{x}}$ and assume that $\mu(g,y)\in M_{x}^{\sim}\cap M^{G_{x}}$. Then $g\in N(G_{x})$;
\item[(b)] If $g\in N(G_{x})$ and $y\in M_{x}^{\sim}\cap M^{G_{x}}$. Then $\mu(g,y)\in M_{x}^{\sim}\cap M^{G_{x}}$;
\item[(c)] $N(G_{x})$ leaves $M_{x}^{\approx}\cap M^{G_{x}}$ invariant;
\item[(d)] For each $y\in M_{x}^{\approx}\cap M^{G_{x}}$, there exists a neighborhood diffeomorphic to $(N(G_{x})/G_{x})\times S_{x}^{G_{x}}$;
\end{itemize}
\end{claim}

As for (iii), note that the action $N(G_{x})\times M_{x}^{\approx}\cap M^{G_{x}}\rightarrow M_{x}^{\approx}\cap M^{G_{x}}$ is proper, since $N(G_{x})\subset G$ is closed. In addition, Claim~\ref{cl:duist2} implies that the action $N(G_{x})/G_{x} \times M_{x}^{\approx}\cap M^{G_{x}}\rightarrow M_{x}^{\approx}\cap M^{G_{x}}$ is free.

Let us now give an idea of how to prove (iv). We must first note that $M_{x}^{\approx}$ is a manifold, and in particular, the connected components have the same dimension. This can be proved using Claim~\ref{cl:duist1}, (i) and the definition of local type. Set \begin{eqnarray*}\varphi:G/G_{x}\times (M_{x}^{\approx}\cap M^{G_{x}}) &\longrightarrow & M_{x}^{\approx} \\ (gG_{x},s)&\longmapsto &\mu(g,s). \end{eqnarray*} Claim~\ref{cl:duist2} implies that $\varphi$ is well--defined.

\begin{claim}\label{cl:duist4}
$\varphi$ and $\dd\varphi$ are surjective.
\end{claim}

In order to prove that $\varphi$ is surjective, consider $y\in M_{x}^{\approx}$. There exists $g$ such that $G_{\mu(g^{-1},y)}=G_{x}$. It then follows from Claim~\ref{cl:duist2} that $\mu(g^{-1},y)\in M_{x}^{\sim}\cap M^{G_{x}}$. Since the action of $G$ leaves $M_{x}^{\approx}$ invariant, $\mu(g^{-1},y)\in M_{x}^{\approx}\cap M^{G_{x}}$ and hence $y=\varphi(g,\mu(g^{-1},y))$.

To prove that $\dd\varphi$ is surjective, note that the restriction $\varphi: (G/G_{x})\times S_{x}^{G_{x}}\rightarrow M_{x}^{\sim}\cap\mu^G(S_{x})$ is a diffeomorphism. This fact and the inclusion of the tangent spaces $$T_{(gG_{x},y)}(G/G_{x}\times S_{x}^{G_{x}})\subset T_{(gG_{x},y)}(G/G_{x}\times (M_{x}^{\approx}\cap M^{G_{x}}))$$ imply that $\dd\varphi$ is surjective.

\begin{claim}\label{cl:duist5}
$\varphi(gG_{x},y)=\varphi(hG_{x},z)$ if and only if $h=gk^{-1}$ and $z=\mu(k,y)$ for some $k\in N(G_{x})$.
\end{claim}

One part of Claim~\ref{cl:duist5} is clear. In order to prove the other, assume that $\varphi(gG_{x},y)=\varphi(hG_{x},z)$. Hence $\mu(g,y)=\mu(h,z)$ and $z=\mu(k,y)$, where $k=h^{-1}g$. Since $z,y\in (M_{x}^{\approx}\cap M^{G_{x}})$, it follows from Claim~\ref{cl:duist3} that $k\in N(G_{x})$ and this concludes the proof of Claim~\ref{cl:duist5}.

We now define the candidate to the desired $G$--equivariant diffeomorphism as \begin{eqnarray*} \psi: G/G_{x}\times_{N(G_{x})/G_{x}} (M_{x}^{\approx}\cap M^{G_{x}}) &\longrightarrow & M_{x}^{\approx} \\ \left[gG_{x},s\right] &\longmapsto & \mu(g,s). \end{eqnarray*} Claim~\ref{cl:duist5} and the fact that $$\pi:G/G_{x}\times (M_{x}^{\approx}\cap M^{G_{x}}) \longrightarrow G/G_{x}\times_{N(G_{x})/G_{x}} (M_{x}^{\approx}\cap M^{G_{x}})$$ is the projection of a principal fiber bundle imply that $\psi$ is well--defined, smooth and bijective.

\begin{claim}\label{cl:duist6}
$\psi$ is a $G$--equivariant diffeomorphism.
\end{claim}

To prove Claim~\ref{cl:duist6}, note that $\dd\pi$ and $\dd\varphi$ are surjective (see Claim~\ref{cl:duist4}) and $\varphi=\psi\circ\pi$. Therefore $\dd\psi$ is surjective. On the other hand, Claim~\ref{cl:duist1} implies that
\begin{equation}\label{eq1-theorem-mesmotipo-esptotalFibrado}
M_{x}^{\approx}\cap \mu^G(S_{x})=(G/G_{x})\times S_{x}^{G_{x}}
\end{equation}
and Claim~\ref{cl:duist3} implies that
\begin{equation}\label{eq2-theorem-mesmotipo-esptotalFibrado}
\mu^G(S_{x})\cap(M_{x}^{\approx}\cap M^{G_{x}})=(N(G_{x})/G_{x})\times S_{x}^{G_{x}}.
\end{equation}
Both \eqref{eq1-theorem-mesmotipo-esptotalFibrado} and \eqref{eq2-theorem-mesmotipo-esptotalFibrado} imply that
\begin{equation}\label{eq3-theorem-mesmotipo-esptotalFibrado}
\dim M_{x}^{\approx}=\dim(G/G_{x}\times_{N(G_{x})/G_{x}}(M_{x}^{\approx}\cap M^{G_{x}})).
\end{equation}

It follows from \eqref{eq3-theorem-mesmotipo-esptotalFibrado} and the fact that $\dd\psi$ is surjective, that $\dd\psi$ is an isomorphism. Hence, from the Inverse Function Theorem, $\psi$ is a local diffeomorphism. Since it is also bijective, we conclude that $\psi$ is a diffeomorphism. By construction, $\psi$ is $G$--equivariant. This concludes the proofs of Claim~\ref{cl:duist6} and (iv).

Finally, item (v) follows from the fact that the action of $N(G_{x})/G_{x}$ is proper and free on $M_{x}^{\approx}\cap M^{G_{x}}.$
\end{proof}

To conclude this section, we present a result on the stratification given by orbit types of a proper action.

\begin{definition}
A {\it stratification}\index{Stratification} of a manifold $M$ is a locally finite partition of $M$ by embedded submanifolds $\{M_i\}_{i\in I}$ of $M$, called {\it strata},\index{Strata} such that the following hold.
\begin{itemize}
\item[(i)] For each $i\in I$, the closure of $M_i$ is $M_i\cup\bigcup_{j\in I_i} M_j$, where $I_i\subset I\setminus\{i\}$;
\item[(ii)] $\dim M_j<\dim M_i$, for each $j\in I_i$.
\end{itemize}
\end{definition}

\begin{example}\label{ex-stratification}
Consider the action of $G=S^{1}$ on $M=\C\times\R=\R^{3}$ defined by $\mu(s,(z,t))=(s z,t)$. Then $$M_{1}=\{(0,t)\in\C\times\R:t \in\R\}$$ is the stratum of orbits that are single points. Set $M_{2}=\R^{3}\setminus M_{1}$. Note that each principal orbit is a circle (of codimension 2) whose center is in $M_{1}$, and that $M_{2}=M_{\mathrm{princ}}$. In particular, $M_{\mathrm{princ}}$ is open and dense in $M$, as asserted in Theorem~\ref{teo-orbita-principal}. One can also check that $$M/G=\{(x_{2},x_{3})\in\R^{2}: x_{2}\geq 0\}$$ and $M_{\mathrm{princ}}=\{(x_{2},x_{3})\in\R^{2}:x_{2}> 0\}$. Hence, $M_{\mathrm{princ}}/G$ is connected in $M/G$, again as inferred in Theorem~\ref{teo-orbita-principal}. In this example, $\{M_i\}_{i\in I}$ is clearly a stratification of $M$, where $I=\{1,2\}$, $I_{2}=\{1\}$ and $I_{1}=\emptyset$.
\end{example}

\begin{theorem}\label{teo-stratification}
Let $\mu:G\times M\rightarrow M$ be a proper action. The connected components of the types of orbits of $M$ give a stratification of $M$.
\end{theorem}

\begin{proof}
In this proof we will use notations and arguments in the proof of Theorem~\ref{theorem-mesmotipo-esptotalFibrado}.

\begin{claim}
Let $(M_{y}^{\sim})^{0}$ be the connected component of ${M_{y}^{\sim}}$ that contains $y$. 
Consider $x\in\overline{(M_{y}^{\sim})^{0}}$ such that $x\notin  (M_{y}^{\sim})^{0}$. Then
\begin{itemize}
\item[(i)] If $(M_{x}^{\sim})^{0}$ is the connected component of $M_{x}^{\sim}$ that contains $x$, then $(M_{x}^{\sim})^{0}$ is contained in $\overline{(M_{y}^{\sim})^{0}}$;
\item[(ii)] $\dim (M_{x}^{\sim})^{0}< \dim (M_{y}^{\sim})^{0}$. 
\end{itemize}
\end{claim}

In order to prove (i), we may assume that $y\in S_{x}$. From Claim~\ref{cl:duist1}, it suffices to prove that if $\widetilde{x}\in S_{x}^{G_{x}}$ then $\widetilde{x}\in \overline{(M_{y}^{\sim})^{0}}$. Due to the slice representation, we may also assume that $x=0$, $K=\mu^{G_{x}}$ acts on an Euclidean space and $S_{x}^{G_{x}}$ is a vector space fixed by the action. Since the action of $K$ fixes each point of $S_{x}^{G_{x}}$, it leaves invariant the normal space $\nu_{p}(S_{x}^{G_{x}})$ that contains $y$. Let $\widetilde{y}$ be the translation of $y$ to the normal space $\nu_{\widetilde{x}}(S_{x}^{G_{x}})$, i.e., $\widetilde{y}=y-p+\widetilde{x}$. Since $K_{\widetilde{x}}=K=K_{p}$, we conclude that
\begin{equation}\label{eq1-teo-stratification}
K_{y}=K_{\widetilde{y}}.
\end{equation}

Since the action of $K$ on the Euclidean space is linear,
\begin{equation}\label{eq2-teo-stratification}
K_{z}=K_{\widetilde{y}}
\end{equation}
for each point $z$ (different from $\widetilde{x}$) in the segment joining $\widetilde{x}$ and $\widetilde{y}$. Equations \eqref{eq1-teo-stratification} and \eqref{eq2-teo-stratification} imply that $\widetilde{x}\in \overline{(M_{y}^{\sim})^{0}}$. This concludes the proof of (i).

As for (ii), let $(S_{x})_{y}^{\sim}$ be the set of points in $S_{x}$ with the same $G_{x}$--orbit type of $y\in S_{x}$. From the above discussion, we infer that $$\dim ((S_{x})_{y}^{\sim})^{0}\geq \dim(S_{x}^{G_{x}})+\dim \mu^{G_{x}}(y)+1.$$ In particular,
\begin{equation}\label{eq3-teo-stratification}
\dim ((S_{x})_{y}^{\sim})^{0}> \dim(S_{x}^{G_{x}}).
\end{equation}
From Claim~\ref{cl:duist1}, we have
\begin{equation}\label{eq4-teo-stratification}
\dim (M_{x}^{\sim})^{0}=\dim(S_{x}^{G_{x}}) +\dim G(x).
\end{equation}
Finally, the fact that each point in $((S_{x})_{y}^{\sim})^{0}$ is also in $(M_{y}^{\sim})^{0}$ implies
\begin{equation}\label{eq5-teo-stratification}
\dim (M_{y}^{\sim})^{0}\geq \dim ((S_{x})_{y}^{\sim})^{0}+ \dim G(x).
\end{equation}
Equations \eqref{eq3-teo-stratification}, \eqref{eq4-teo-stratification} and \eqref{eq5-teo-stratification} imply $\dim (M_{x}^{\sim})^{0}< \dim (M_{y}^{\sim})^{0}$. This concludes the proof of (ii).

Using the slice representation and the same arguments in the proof of Theorem~\ref{theorem-n-finito-types-orbits}, one can prove that the partition is locally finite and this concludes the proof.
\end{proof}

\begin{remark}
With the slice representation, it is possible to prove that the connected components of the orbit types in $M$ form a \emph{Whitney stratification} of $M$ (see Duistermaat and Kolk~\cite{duistermaat} for definitions and proof).
\end{remark}

We conclude this chapter with a few final comments on the slice representation. We have seen in different proofs along this chapter that, using the slice representation, the local study of proper (isometric) actions can be reduced to the local study of an isometric action in Euclidean space.

As remarked by Molino~\cite{Molino}, the same idea is still valid in the local study of \emph{singular Riemannian foliations} (see Definition~\ref{defsrfs}). More precisely, after a suitable change of metrics, the local study of an arbitrary singular Riemannian foliation is reduced to the study of a singular Riemannian foliation on the Euclidean space (with the standard metric). This idea was used, for example, by Alexandrino and T\"{o}ben~\cite{AlexToeben2} and Alexandrino~\cite{Alex6}.

\chapter{Adjoint and conjugation actions}
\label{chap4}

Two important actions play a central role in the theory of compact Lie groups. First, the action of a compact Lie group $G$ on itself by \emph{conjugation}. Second, the \emph{adjoint action}\index{Action!adjoint}\index{Adjoint!action} $$\Ad:G\times \mathfrak g\longrightarrow \mathfrak g$$ of $G$ on its Lie algebra $\mathfrak g$, given by the linearization of the conjugation action (see Definition~\ref{defadjoint}). The adjoint action is also a typical example of a \emph{polar action} (see Remark~\ref{remarkpolar} and Definition~\ref{definitionPolarAction}) and each of its regular orbits is an isoparametric submanifold (see Definition~\ref{definition-isoparametric}). A surprising fact is that several results of adjoint actions can be generalized, not only to the theory of isoparametric submanifolds and polar actions, but also to a general context of \emph{singular Riemannian foliations with sections}, or s.r.f.s. for short (see Definition~\ref{defsrfs}).

The aim of this chapter is to recall classic results of adjoint actions from a differential geometric viewpoint. The relation between adjoint action and isoparametric submanifolds is also briefly studied. Several results presented in this chapter will illustrate results of the theory of s.r.f.s., discussed in the next chapter.

Further references for this chapter are Fegan~\cite{Fegan}, Duistermaat and Kolk~\cite{duistermaat}, Helgason~\cite{helgason}, Hall~\cite{Hall}, San Martin~\cite{SanMartin}, Serre~\cite{Serre}, Thorbergsson~\cite{ThSurvey1,ThSurvey2} and Palais and Terng~\cite{PalaisTerng}.

\section[Maximal tori and polar actions]{Maximal tori, isoparametric submanifolds and polar actions}

The aim of this section is twofold. First, we prove the Maximal Torus Theorem, that generalizes classic results from linear algebra. Second, we present definitions of polar actions and isoparametric submanifolds, briefly discussing some results from these theories.

Recall that a Lie group $T$ is a \emph{torus}\index{Torus} if it is isomorphic to the product $S^{1}\times\cdots\times S^{1}$ and that, in this case, it is abelian and its Lie algebra $\mathfrak t$ is isomorphic to the Euclidean space $\R^n$. An element $p\in T$ is a \emph{generator}\index{Torus!generator}\index{Generator} if the set $\{p^{n}:n\in\Z\}$ is dense in $T$. Analogously, a vector $X\in\mathfrak{t}$ is an \emph{infinitesimal generator}\index{Generator!infinitesimal} if the set $\{\exp(tX):t\in\R\}$ is dense in $T$. It is not difficult to see that each torus contains a generator and an infinitesimal generator. Recall also that it was proved in Theorem~\ref{cpctabeliantorus} that a connected compact abelian Lie group is a torus.

Let $T$ be the subgroup of diagonal matrices of $\SU(n)$. It follows from Theorem~\ref{cpctabeliantorus} that $T$ is a torus, and we know from linear algebra that every matrix $g\in\SU(n)$ is conjugate to an element of $T$. This is a particular case of the following result.

\begin{mtthm}\label{maxtorusthm}\index{Theorem!Maximal Torus}\index{Torus!maximal}
Let $G$ be a connected, compact Lie group. Then
\begin{itemize}
\item[(i)] There exists a maximal torus $T$ (in the sense that if $N$ is a torus and $T\subset N$, then $T=N$);
\item[(ii)] Let $T_1$ and $T_2$ be two maximal tori. There exists $g\in G$ such that $g T_{1} g^{-1}=T_{2}$. In particular, the maximal tori have the same dimension, called the \index{Lie group!rank}\emph{rank} of $G$;
\item[(iii)] Let $T$ be a maximal torus and $g\in G$. There exists $h\in G$ such that $hgh^{-1}\in T$. In particular, each element of $G$ is contained in a maximal torus;
\item[(iv)] For each bi--invariant metric on $G$, the orbits of the conjugation action intersect each maximal torus orthogonally.
\end{itemize}
\end{mtthm}

\begin{proof}
(i) Let $\mathfrak{t}\subset \mathfrak g$ be the maximal abelian sub algebra of the Lie algebra of $G$. From Theorem~\ref{integralsubgroup}, there exists a unique connected subgroup $T\subset G$ with Lie algebra $\mathfrak{t}$. Note that the closure $\overline{T}$ is an abelian connected, compact Lie group. Thus, from Theorem~\ref{cpctabeliantorus}, $\overline{T}$ is a torus. Since $\mathfrak{t}$ is maximal, $\overline{T}=T$. Furthermore, consider a subgroup $H\subset G$ such that $H$ is a torus and $\overline{T}\subset H$. Then, for each $X\in\h$, we have $[X,Z]=0$ for all $Z\in\mathfrak{t}$. By maximality of $\mathfrak{t}$, we conclude that $X\in\mathfrak{t}$, hence $\mathfrak{h}=\mathfrak{t}$. Therefore, from Theorem~\ref{integralsubgroup}, $H=\overline{T}=T$, and this proves (i).

Before proving the next item, we present the following auxiliary result.

\begin{lemma}\label{LemmaMaximalTorusTheorem}
Consider a Lie group $G$ with Lie algebra $\mathfrak g$ in the hypothesis above. Let $T$ be a maximal torus of $G$ with Lie algebra $\mathfrak{t}$, and $X\in\mathfrak{t}$ an infinitesimal generator of $T$. Then $$\mathfrak{t}=\{Y\in\mathfrak{g}:[X,Y]=0\}.$$
\end{lemma}

In order to prove this lemma, we first observe that the inclusion $\mathfrak t\subset\{Y\in\mathfrak{g}:[X,Y]=0\}$ is immediate from Proposition~\ref{abelianiff}. Conversely, consider $Y\in\mathfrak{g}$ such that $[X,Y]=0$. From commutativity of $X$ and $Y$, it follows that $$\exp(sX)\exp(tY)=\exp(tY)\exp(sX), \mbox{ for all } s,t\in\R.$$ Hence $\exp(tY)$ commutes with all elements of $T=\overline{\{\exp(sX):s\in\R\}}$. Therefore, $Y$ commutes with all vectors of $\mathfrak{t}$. From Theorem~\ref{integralsubgroup}, there exists a unique connected Lie subgroup $T_{2}$ of $G$ with Lie algebra $\R Y\oplus \mathfrak{t}$. Note that $\overline{T_{2}}$ is a connected compact abelian Lie group. Thus, from Theorem~\ref{cpctabeliantorus}, $\overline{T_{2}}$ is a torus. Since $T$ is a maximal torus contained in $\overline{T_{2}}$, it follows $\overline{T_{2}}=T$. This implies $Y\in\mathfrak{t}$ and concludes the proof of the lemma.

As for (ii), let $X_{1}$ and $X_{2}$ be infinitesimal generators of $T_{1}$ and $T_{2}$ respectively. Define \begin{eqnarray*} f:G&\longrightarrow &\R \\ g &\longmapsto &\langle \Ad(g)X_{1},X_{2}\rangle, \end{eqnarray*} where $\langle\cdot,\cdot\rangle$ is a bi--invariant metric on $G$ (see Proposition~\ref{cptbi}). Since $G$ is compact, $f$ has a minimum at some point $g_{0}\in G$. Therefore, for all $Y\in \mathfrak{g}$, \begin{eqnarray*}
0&=&\frac{\mathrm d}{\mathrm dt} f\circ (\exp(tY) g_{0})\Big|_{t=0}\\
&=&\frac{\mathrm d}{\mathrm dt} \langle \Ad (\exp(tY)) \Ad (g_{0})X_{1}, X_{2}\rangle\Big|_{t=0}\\
&=&\langle [Y,\Ad(g_{0})X_{1}],X_{2}\rangle\\
&=&\langle Y, [\Ad(g_{0})X_{1},X_{2}]\rangle.
\end{eqnarray*}

The equation above implies $[\Ad(g_{0})X_{1},X_{2}]=0$. It then follows from Lemma~\ref{LemmaMaximalTorusTheorem} that $\Ad(g_{0})X_{1}\in\mathfrak{t_{2}}$, hence $g_{0}\exp(tX_{1}) g_{0}^{-1}\in T_{2}.$ The last equation implies that $g_{0} T_{1} g_{0}^{-1}\subset T_{2}$ and by maximality of $T_{1}$, we conclude $T_{1}=g_{0}^{-1}T_{2}g_{0}$. This proves (ii).

From Theorem~\ref{expagree}, $\exp:\mathfrak{g}\rightarrow G$ is surjective. Thus, given $g\in G$ there exists $Y\in\mathfrak{g}$ such that $\exp(Y)=g$. Let $T_{2}$ be a maximal torus that contains $\overline{\{\exp(tY):t\in\R\}}$. It follows from (ii) that there exists $h\in G$ such that $h T_{2} h^{-1}=T$. In particular, $h g h^{-1}\in T$, and this proves (iii).

Finally, consider $p\in T$ and $G(p)=\{gpg^{-1}:g\in G\}$ its orbit by the conjugation action. It is then easy to see that
\begin{equation}\label{Eq1MaximalTorusTheorem} 
T_{p}G(p)= \{\dd R_{p}Y- \dd L_{p} Y:Y\in\mathfrak{g}\},
\end{equation}
\begin{equation}\label{Eq2MaximalTorusTheorem}
T_{p}T=\{\dd R_{p}Z:Z\in\mathfrak{t}\}.
\end{equation}
Let $\langle \cdot,\cdot\rangle$ be a bi--invariant metric on $G$. Since $\Ad(p)Z=Z$ for all $ Z\in\mathfrak{t}$, $$0=\langle \dd R_{p}Y-\dd L_{p}Y,\dd R_{p}Z \rangle .$$ The above equation, \eqref{Eq1MaximalTorusTheorem} and \eqref{Eq2MaximalTorusTheorem} imply (iv).
\end{proof}

\begin{definition}\label{definitionPolarAction}
An isometric action of a compact Lie group on a Riemannian manifold $M$ is called a \emph{polar action}\index{Action!polar} if it admits \emph{sections}. This means that each regular point $p$ is contained in an immersed submanifold $\Sigma$, called \emph{section},\index{Action!sections} which is orthogonal to each orbit it intersects and whose dimension is equal to the codimension of the regular orbit $G(p)$. In addition, if sections are flat, the action is called \emph{hyperpolar}.\index{Action!hyperpolar}
\end{definition}

The Maximal Torus Theorem~\ref{maxtorusthm} implies that the conjugation action of a compact Lie group $G$ endowed with a bi--invariant metric is polar. In this case, sections are maximal tori. The codimension of a regular orbit is the rank of $G$, and to this aim one can use its \emph{root system} (see Remark~\ref{remarkWeylchamber2}). Moreover, the conjugation action is not only polar, but also hyperpolar. This follows from Proposition~\ref{curvatureLie}, that implies that the sectional curvature on the plane spanned by the (linearly independent) vectors $X$ and $Y$ is $$K(X,Y)=\frac{1}{4}\frac{\|[X,Y]\|^2}{\|X\|^{2}\|Y\|^{2}-\langle X,Y\rangle^{2}}.$$

\begin{remark}\label{polarsymmetric}
Examples of polar actions can be found in symmetric spaces. Recall that a Riemannian manifold $M$ is a \emph{symmetric space}\index{Symmetric space} if, for each $p\in M$, there is an isometry $\sigma_p$ of $M$ that fixes $p$ and reverses geodesics through $p$. In this case, it is not difficult to prove that $M=G/K$, where $G$ is the connected component of $\Iso(M)$ that contains the identity, and $K$ is its isotropy group at some fixed point $p_0\in M$. Such a pair of groups $(G,K)$ is called a \emph{symmetric pair}.\index{Symmetric space!pair} Let $\Sigma$ be a maximal flat and totally geodesic submanifold through $p_0\in M$. Then the action of $K$ on $M$ is hyperpolar, and $\Sigma$ is a section. This action is called \emph{isotropic action}\index{Action!isotropic}\index{Isotropic!action}. Moreover, the action of a compact Lie group on itself by conjugation is a particular case of isotropy action.

Another example of hyperpolar action is the \emph{isotropic representation}\index{Representation!isotropic}\index{Isotropic!representation} of symmetric spaces. This representation is the induced action of $K$ on the tangent space $T_{p_{0}}M$. In this case, $T_{p_{0}}\Sigma$ is a section. Isotropy actions can be generalized in the following way. Assume that $(G,K_{1})$ and $(G,K_{2})$ are symmetric pairs. Then it can be proved that the action of $K_{1}$ on $M=G/K_{2}$ is hyperpolar. This examples are known as \emph{Hermann actions}.\index{Action!Hermann}

Dadok~\cite{Dadok} classified all polar linear representations. They are \index{Action!orbit--equivalent}\emph{orbit--equivalent}\footnote{Isometric actions of Lie groups $G_1$, respectively $G_2$, on Riemannian manifolds $M_1$, respectively $M_2$, are said to be \emph{orbit--equivalent} if there exists an isometry between $M_1$ and $M_2$ mapping orbits of the $G_1$ action in orbits of the $G_2$ action.} to the isotropic representation of symmetric spaces. Podest\`a and Thorbergsson~\cite{PodestaThorbergsson} classified polar actions on compact symmetric spaces of \emph{rank one}\footnote{The \emph{rank}\index{Symmetric space!rank} of a symmetric space $M$ is the maximum dimension of flat, totally geodesic submanifolds of $M$.}. In such spaces there are examples of polar actions which are not hyperpolar. Finally, Kollross~\cite{Kollross} classified all hyperpolar actions on irreducible, simply--connected symmetric spaces of \emph{compact type}\footnote{A symmetric space $M$ is of \emph{compact type}\index{Symmetric space!compact type} if it has non negative sectional curvature (and is not flat).}. If they have cohomogeneity greater then one, they are orbit--equivalent to Hermann actions. Kollross also classified all cohomogeneity one actions on compact irreducible symmetric spaces.

It has been conjectured that all polar actions on irreducible symmetric spaces of compact type with rank greater than one are hyperpolar. A partial answer for this problem was given by Biliotti~\cite{Biliotti}, that proved that all polar actions on irreducible Hermitian symmetric spaces of compact type and rank greater than one are hyperpolar. For a survey on polar actions see Thorbergsson~\cite{ThSurvey1,ThSurvey2}, and for further reading on symmetric spaces, see Helgason~\cite{helgason}.
\end{remark}

We now turn our attention back to the classic theory of compact Lie groups, considering the adjoint action of $G$ on its Lie algebra $\mathfrak{g}$.

\begin{corollary}\label{corollary-MaximalTorus}
Let $G$ be a connected, compact Lie group endowed with a bi--invariant metric, and $\mathfrak{t}$ the Lie algebra of a maximal torus of $G$. Then each orbit of the adjoint action intersects $\mathfrak{t}$ orthogonally.
\end{corollary}

\begin{proof}
It follows from the Maximal Torus Theorem~\ref{maxtorusthm} that each orbit of the adjoint action intersects $\mathfrak{t}$. We only have to prove that such intersections are orthogonal. For each $Z\in\mathfrak{t}$, the tangent space of the adjoint orbit of $Z$ is given by \begin{eqnarray*}
T_{Z}(\Ad(G)Z) &=& \left\{\frac{\mathrm d}{\mathrm dt}\Ad(\exp(tX))Z\Big|_{t=0}: X \in\mathfrak{g}\right\} \\
             &=& \{[X,Z]: X \in\mathfrak{g}\}.
\end{eqnarray*}
On the other hand, for each $V\in\mathfrak{t}$, $$\langle[X,Z],V\rangle =\langle X,[Z,V]\rangle =0.$$

The result follows from the two equations above.
\end{proof}

Corollary~\ref{corollary-MaximalTorus} implies that the adjoint action is a polar action on $\mathfrak{g}$. This is a particular case of the next general result (see Palais and Terng~\cite{PalaisTerng}).

\begin{proposition}
The isotropic representation of a polar action is polar.\index{Representation!isotropic}
\end{proposition}

We conclude this section with some comments on isoparametric submanifolds.

\begin{definition}\label{definition-flatnormal}
Let $L$ be an immersed submanifold of a Riemannian manifold $M$. A section $\xi$ of the normal bundle $\nu(L)$ is said to be a \emph{parallel normal field}\index{Vector field!normal parallel} along $L$ if $\nabla^{\nu}\xi$ vanishes identically, where $\nabla^{\nu}$ is the normal connection.\footnote{$\nabla^{\nu}\xi$ is the component of $\nabla\xi$  that is normal to $L$.} $L$ is said to have \emph{flat normal bundle}\index{Normal bundle!flat}, if any normal vector can be extended to a locally defined parallel normal field. In addition, $L$ is said to have \emph{globally flat normal bundle},\index{Normal bundle!globally flat} if the holonomy of the normal bundle $\nu(L)$ is trivial. This means that any normal vector can be extended to a globally defined parallel normal field.
\end{definition}

\begin{definition}\label{definition-isoparametric}
A submanifold $F$ of a \emph{space form}\footnote{Recall that a \emph{space form}\index{Space form} $M(k)$ is a simply connected complete Riemannian manifold with constant sectional curvature $k$ (see Section~\ref{sec:riemgeom}).} $M(k)$ is called \emph{isoparametric}\index{Isoparametric!submanifold}\index{Submanifold!isoparametric} if its normal bundle is flat and principal curvatures along any parallel normal vector field are constant.
\end{definition}

Principal orbits of a polar action on a Riemannian manifold have globally flat normal bundle, since they are described by an integrable Riemannian submersion (see Remark~\ref{remarkequivariantfield} and Palais and Terng~\cite{PalaisTerng}). In particular, equivariant normal vectors turn out to be parallel normal fields. It is not difficult to prove that principal curvatures along any equivariant normal vector field of a  principal orbit are constant (see Proposition~\ref{prop-equivariant-field}). These two facts imply the next result.

\begin{proposition}\label{proposition-adjointaction-is-isoparametric}\index{Isoparametric submanifold}
Principal orbits of a polar action on Euclidean space are isoparametric submanifolds. In particular, principal orbits of the adjoint action of $G$ are isoparametric submanifolds of $\mathfrak g$.
\end{proposition}

In the particular case of adjoint actions, we will prove that each regular orbit is a principal orbit (see Theorem~\ref{theorem-OrbitaPrincipal=OrbitaRegular}), and calculate principal curvatures and principal directions (see Remark~\ref{remark-Ad-isoparametrica-curvaturas}).

\begin{remark}\label{remark-definition-isoparametric}
It is possible to prove the normal bundle of an isoparametric submanifold is globally flat and hence that each normal vector can be extended to a parallel normal vector field. An important property of isoparametric submanifolds is that parallel sets of an isoparametric submanifold are submanifolds. This means that given an isoparametric submanifold $N$, a parallel normal vector field $\xi$ and $\eta_{\xi}(x)=\exp_{x}(\xi)$ the endpoint map, then $\eta_{\xi}(N)$ is a submanifold (with dimension possibly lower than $\dim N$). An \emph{isoparametric foliation}\index{Foliation!isoparametric} $\F$ on $M(k)$ is a partition of $M(k)$ by submanifolds parallel to a given isoparametric submanifold $N$. It is possible to prove that a leaf $L$ of $\F$ is also isoparametric if $\dim L=\dim N$. In this case, the partition by parallel submanifolds to $L$ turns out to be $\F$.
\end{remark}

\begin{remark}\index{Isoparametric!submanifold}
Isoparametric hypersurfaces in space forms have been studied since Cartan~\cite{Cartan1,Cartan2,Cartan3,Cartan4}. In Euclidean and hyperbolic spaces, they are cylinders or umbilic hypersurfaces. In spheres, there are other examples of isoparametric hypersurfaces, of which several are inhomogeneous (see Ferus, Karcher and M\"{u}nzner~\cite{FerusKarcherMunzner}). Regarding isoparametric hypersurfaces in spheres, we would like to emphasize the work of M\"{u}nzner~\cite{Munzner1,Munzner2}, that proved that the number of principal curvatures of isoparametric hypersurfaces in spheres can only be 1,2,3,4 or 6. Furthermore, all of these numbers are known to occur. Nevertheless, a full classification of isoparametric hypersurfaces in spheres is still an open problem.

The concept of isoparametric submanifolds was independently introduced in the eighties by Harle~\cite{Harle}, Carter and West~\cite{CarterWest1,CarterWest2} and Terng~\cite{Terng}, that remarks the similarity between isoparametric submanifolds and polar orbits in Euclidean spaces. In particular, Coxeter groups are associated to isoparametric submanifolds. Another interesting result relating orbits of polar actions to isoparametric submanifolds is due to Thorbergsson~\cite{Th2}. It is proved that a compact, irreducible isoparametric submanifold $L^{n}$ in $\R^{n+k}$ is homogeneous if $k\geq 3$ and if $L^{n}$ does not lie in any affine hyperplane of $\R^{n+k}$.

A comprehensive description of Terng's work about isoparametric submanifolds on space forms and Hilbert spaces can be found in Palais and Terng~\cite{PalaisTerng}. Another  reference to the theory of isoparametric submanifolds is Berndt, Console and Olmos~\cite{BerndtConsoleOlmos}. More information on the history of isoparametric hypersurfaces and submanifolds and possible generalizations can be found in the surveys by Thorbergsson~\cite{ThSurvey1,ThSurvey2}.
\end{remark}

\section{Roots of a compact Lie group}

Roots of a compact Lie group play a fundamental role in the theory of compact Lie groups. They allow, for instance, to classify compact simple Lie groups (see Section~\ref{Sec-Dynkin Diagrams}) and are also related to principal curvatures of a principal orbit of the adjoint action (see Remark~\ref{remark-Ad-isoparametrica-curvaturas}). \emph{Roots} of a compact Lie group $G$ will be defined in Theorem~\ref{theorem-roots-existencia}.

For the sake of motivation, we start by presenting a result which will follow directly from Theorem~\ref{theorem-roots-existencia} and Remark~\ref{remarkWeylchamber1}.

\begin{proposition}\label{proposition-interpretacao-raiz}
Let $G$ be a connected compact non abelian Lie group and $T\subset G$ a fixed maximal torus with Lie algebras $\mathfrak{g}$ and $\mathfrak{t}$, respectively. Then there exist bidimensional subspaces $V_{i}\subset\mathfrak{g}$ and linear functionals $\alpha_{i}:\mathfrak{t}\rightarrow \R$ such that
\begin{itemize}
\item[(i)] $\mathfrak{g}=\mathfrak{t}\oplus V_{1}\oplus\cdots\oplus V_{k}$;
\item[(ii)] For a bi--invariant metric on $G$ and an orthonormal basis of $V_{i}$,
\[\Ad(\exp(X))\big|_{V_{i}} =\left(\begin{array}{r r}
                    \cos(\alpha_{i}(X))&-\sin(\alpha_{i}(X))\\
                     \sin(\alpha_{i}(X))&\cos(\alpha_{i}(X))
                    \end{array}\right), \]
for all $X\in\mathfrak{t}$;
\item[(iii)] Such decomposition is unique, and $\alpha_{m}\neq\alpha_{n}$ if $m\neq n$.
\end{itemize}
\end{proposition}

Each linear functional $\alpha:\mathfrak{t}\rightarrow \R$ above turns out to be the root associated to the space $V_{\alpha}$ (see the definition in Theorem~\ref{theorem-roots-existencia}). In order to prove the main theorem of this section, Theorem~\ref{theorem-roots-existencia}, we will need a few preliminary results.

\begin{theorem}\label{theorem-posto-1-GL}
Let $G$ be a connected compact Lie group of rank $1$ (see the Maximal Torus Theorem~\ref{maxtorusthm}). Then $G$ is isomorphic to $S^{1}$, $\SO(3)$ or $\SU(2)$.
\end{theorem}

A proof of the above result can be found, for example, in Duistermaat and Kolk~\cite{duistermaat}. We also need some elementary facts of linear algebra that are recalled below.

\begin{lemma}\label{lemma-espaco-invariante-conjugado}
Let $W\subset\C^{n}$ be a complex subspace. Then $W=\overline{W}$, i.e., $W$ is invariant by complex conjugation, if and only if $W=V\oplus \imaginario V$ for a real subspace $V\subset\R^{n}$.
\end{lemma}

\begin{remark}\label{remark-dimensao-real-complexa-esp-invariante}
If $\{v_{i}\}$ is a real basis for a real subspace $V\subset \R^{n}$, then $\{v_{i}\}$ is a complex basis for the complex subspace $W=V\oplus \imaginario V$. In particular,
$\dim_{\C}W=\dim_{\R}V$.
\end{remark}

Let $G$ be a connected compact Lie group with Lie algebra $\mathfrak{g}$. Consider the complexification of $\mathfrak{g}$, $$\mathfrak{g}_{\C}=\mathfrak{g}\otimes_{\R}\C,$$ i.e., $\mathfrak{g}_{\C}=\mathfrak{g}\oplus \imaginario \mathfrak{g}$, as a real vector space. It is not difficult to see that $\mathfrak{g}_{\C}$ is a complex Lie algebra where $[\cdot,\cdot]:\mathfrak{g}_{\C}\times\mathfrak{g}_{\C}\rightarrow\mathfrak{g}_{\C}$ is the canonical extension of the Lie bracket of $\mathfrak{g}$.

\begin{theorem}\label{theorem-roots-existencia}
Let $G$ be a connected compact non abelian Lie group and $T\subset G$ a fixed maximal torus, with Lie algebras $\mathfrak{g}$ and $\mathfrak{t}$, respectively. Consider $\mathfrak{g}_{\C}$ the complexification of $\mathfrak{g}$.

\begin{itemize}
\item[(i)] For each $X\in\mathfrak{g}$, $\ad(X):\mathfrak{g}_{\C}\rightarrow\mathfrak{g}_{\C}$ is linear and diagonalizable with only purely imaginary eigenvalues;
\item[(ii)] There exists a unique (up to permutations) decomposition of $\mathfrak{g}_{\C}$ in complex subspaces $$\mathfrak{g}_{\alpha}=\{Y\in\mathfrak{g}_{\C}:[X,Y]=\imaginario \alpha(X)Y, \mbox{ for all } X \in \mathfrak{t}\},$$ where $\alpha:\mathfrak{t}\rightarrow\R$ is a linear functional called \emph{root}\index{Lie group!root}\index{Root}. In other words, $$\mathfrak{g}_{\C}=\mathfrak{g}_{0}\oplus \sum_{\alpha\in R} \mathfrak{g}_{\alpha},$$ where $R$ denotes the set of roots, also called \emph{root system}\index{Root!system};
\item[(iii)] $\mathfrak{g}_{0}=\mathfrak{t}\oplus \imaginario\mathfrak{t}$ and $\overline{\mathfrak{g}_{\alpha}}=\mathfrak{g}_{-\alpha}$. In particular, $\alpha\in R$ implies $-\alpha\in R$;
\item[(iv)] $\dim_{\C}\mathfrak{g}_{\alpha}=1$ and $\dim V_{\alpha}=2$, where $V_{\alpha}=(\mathfrak{g}_{\alpha}\oplus\mathfrak{g}_{-\alpha})\cap\mathfrak{g}$;
\item[(v)] $\mathfrak{g}_{k\alpha}=0$ if $k\neq -1,0,1$;
\item[(vi)] Let $e_{2}+\imaginario e_{1}$ be the vector that generates $\mathfrak{g}_{\alpha}$, with $e_{1}, e_{2}\in\mathfrak{g}$. Then
\begin{itemize}
\item[(vi-a)] $\{e_{1},e_{2}\}$ is a basis of $V_{\alpha}$;
\item[(vi-b)] $[X,e_{1}]=\alpha(X) e_{2}$ and $[X,e_{2}]=-\alpha(X) e_{1}$, for all $X\in\mathfrak{t}$;
\item[(vi-c)] $\langle e_{1}, e_{2} \rangle=0$ and $\|e_{1}\|=\|e_{2}\|$, with respect to each bi--invariant metric $\langle\cdot,\cdot\rangle$;
\item[(vi-d)] $$\Ad(\exp(X))\big|_{V_{\alpha}} =\left(\begin{array}{r r}
                    \cos(\alpha(X))&-\sin(\alpha(X))\\
                     \sin(\alpha(X))&\cos(\alpha(X))
                    \end{array}\right),$$ with respect to the basis $\left\{\frac{e_{1}}{\|e_{1}\|},\frac{e_{2}}{\|e_{2}\|}\right\}$;
\end{itemize}
\item[(vii)] Let $\alpha^{\vee}$ be the \emph{coroot}\footnote{The \emph{coroot}\index{Lie group!coroot}\index{Root!coroot} $\alpha^\vee$ associated to a root $\alpha$ is a vector in $\mathfrak{t}$ such that $\alpha(\alpha^{\vee})=2$ and $\alpha^{\vee}$ is orthogonal to $\ker\alpha$ with respect to a bi--invariant metric.} of $\alpha$;
\begin{itemize}
\item[(vii-a)] $\alpha^{\vee}$ is orthogonal to $\ker\alpha$ in any bi--invariant metric;
\item[(vii-b)] Set $\mathfrak{g}^{(\alpha)}=\R\alpha^{\vee}\oplus V_{\alpha}$. Then $\mathfrak{g}^{(\alpha)}$ is a Lie algebra isomorphic to $\mathfrak{so}(3)$, and the correspondent Lie group $G^{(\alpha)}=\exp(\mathfrak{g}^{(\alpha)})$ is a compact Lie group;
\item[(vii-c)] For each $g\in G^{(\alpha)}$, $\Ad(g)\big|_{\ker\alpha}=\id$;
\item[(vii-d)] There exists $w\in G^{(\alpha)}$, such that $\Ad(w)\alpha^{\vee}=-\alpha^{\vee}$.
\end{itemize}
\end{itemize}
\end{theorem}

\begin{proof}
Using the Jordan normal form, there exists a decomposition of $\mathfrak{g}_{\C}$ in subspaces $\mathfrak{g}_{j}$ such that $\ad(X)|_{\mathfrak{g}_{j}}=c_{j}I+ N_{j}$ where $N_{j}$ is a \emph{nilpotent matrix}\footnote{This means that $N_{j}^{m}=0$ if $m=m_{j}$, and $N_{j}^{m}\neq 0$ if $m<m_{j}$}. Therefore
\begin{equation}\label{theorem-roots-existencia-Eq1}
\exp{(t\ad(X))}\big|_{\mathfrak{g}|_{j}}=\exp{(tc_{j})}\sum_{i=0}^{m_{j}-1} \frac{t^{i}}{i!} N_{j}^{i}.
\end{equation}

On the other hand, using \eqref{adexp}, compactness of $\Ad(G)$ implies that $\exp(t\ad(X))$ is bounded. This fact and \eqref{theorem-roots-existencia-Eq1} imply that $c_{j}$ is a purely imaginary number, and $N_{j}=0$. This proves (i).

Item (ii) follows from the next two claims.

\begin{claim}\label{cl:mons1}
There exists a decomposition $\mathfrak{g}_{\C}=\mathfrak{g}_{0}\oplus\mathfrak{g}_{1}\oplus\cdots\oplus\mathfrak{g}_{k}$ and linear functionals $\alpha_{i}:\mathfrak{t}\rightarrow \R$, such that for each $Y\in\mathfrak{g}_{i}$, $[X,Y]=\imaginario \alpha_{i}(X)Y$ for all $X\in\mathfrak{t}$. In addition,
$\alpha_{m}\neq\alpha_{n}$ if $m\neq n$.
\end{claim}

Indeed, consider a basis $\{X_{1},\dots,X_{l}\}$ of $\mathfrak{t}$. On the one hand, from (i) each $\ad(X_{k})$ is diagonalizable. On the other hand, $\ad(X_{j})$ commutes with $\ad(X_{k})$ for all $j,k$, once $[X_{j},X_{k}]=0$. This allows to find a decomposition $\mathfrak{g}_{\C}=\sum_{i}\mathfrak{g}_{i}$ in common eigenspaces of the operators $\ad(X_{k})$. Finally, define $$\alpha_{i}\left(\sum_{k=1}^l x_{k}X_{k}\right)=\sum_{k=1}^l x_{k}\alpha_{ik},$$ where $[X_{k},Y]=\imaginario \alpha_{i\,k}Y$, for $Y\in\mathfrak{g}_{i}$. Since, by construction, $\mathfrak{g}_{m}\oplus\mathfrak{g}_{n}$ is not an eigenspace for all operators $\ad(X_{k})$, it follows that $\alpha_{m}\neq\alpha_{n}$ if $m\neq n$.

\begin{claim}\label{cl:mons2}
If $[X,Y]=\imaginario\beta(X)Y$ for all $X\in\mathfrak{t}$, then there exists $i$ such that $Y\in \mathfrak{g}_{i}$ and $\beta=\alpha_{i}$, where $\mathfrak{g}_{i}$ and $\alpha_{i}$ were defined in Claim~\ref{cl:mons1}.
\end{claim}

Indeed, let $Y=Y_{0}+Y_{1}+\dots+ Y_{k}$, where $Y_{i}\in\mathfrak{g}_{i}$. Then \begin{eqnarray*}
\sum\imaginario\beta(X)Y_{i} &=& [X,Y]\\
&=& \sum\imaginario\alpha_{i}(X)Y_{i}.
\end{eqnarray*}

Since $\mathfrak{g}_{\C}=\mathfrak{g}_{0}\oplus\mathfrak{g}_{1}\oplus\cdots\oplus\mathfrak{g}_{k}$, we conclude that $\beta(X)=\alpha_{m}(X)=\alpha_{n}(X)$. Claim~\ref{cl:mons2} follows from the fact that $\alpha_{m}\neq\alpha_{n}$ if $m\neq n$.

Item (iii) follows from Lemma~\ref{lemma-espaco-invariante-conjugado} and the fact that $\mathfrak{t}$ is a maximal abelian Lie algebra of $\mathfrak{g}$.

In the sequel, we prove (iv) and (v). Since the set of roots $R$ is finite, we can choose a vector $X\in\ker\alpha$ such that $\beta(X)\neq 0$, for all roots $\beta$ that are not multiples of $\alpha$. Note that the Lie algebra of the identity's connected component of the isotropy group $G_{X}^{0}$ (with respect to the adjoint action) is \begin{equation}
\label{theorem-roots-existencia-Eq2} \mathfrak{g}_{X}=\mathfrak{t}\oplus \sum _{k\neq 0} (\mathfrak{g}_{k\alpha}\oplus\mathfrak{g}_{-k\alpha})\cap\mathfrak{g}.
\end{equation}
This implies that
\begin{equation}\label{theorem-roots-existencia-Eq3}
\Zentrum(\mathfrak{g}_{X})=\ker \alpha,
\end{equation} where $\Zentrum(\mathfrak{g}_{X})$ is the Lie algebra of $\Zentrum(G_{X}^{0})$ (see Corollary~\ref{zentrum}). Set $G'=G_{X}^{0}/\Zentrum(G_{X}^{0})$ and note that the Lie algebra of the compact connected Lie group $G'$ is $\mathfrak{g}_{X}/\Zentrum(\mathfrak{g}_{X})$. It follows from \eqref{theorem-roots-existencia-Eq2} and \eqref{theorem-roots-existencia-Eq3} that $\mathfrak{t}/\ker\alpha$ is a maximal abelian Lie subalgebra of $\mathfrak{g}_{X}/\Zentrum(\mathfrak{g}_{X})$. Therefore $G'$ has rank 1 and hence by Theorem~\ref{theorem-posto-1-GL}, \begin{equation}\label{theorem-roots-existencia-Eq4}
\mathfrak{g}_{X}/\Zentrum(\mathfrak{g}_{X})=\mathfrak{so}(3).
\end{equation}

The fact that $\dim\mathfrak{g}_{X}/\Zentrum(\mathfrak{g}_{X})=3$ implies that
\begin{equation}\label{theorem-roots-existencia-Eq5}
\dim\sum _{k\neq 0} (\mathfrak{g}_{k\alpha}\oplus\mathfrak{g}_{-k\alpha})\cap\mathfrak{g}=2.
\end{equation}

Items (iv) and (v) follow from \eqref{theorem-roots-existencia-Eq5} and Remark~\ref{remark-dimensao-real-complexa-esp-invariante}.

As for (vi), since $\mathfrak{g}_{\alpha}\cap\mathfrak{g}_{-\alpha}=\{0\}$, $e_{1}\neq \lambda e_{2}$. This and the fact that $\dim V_{\alpha}=2$ imply (vi-a). Item (vi-b) follows directly from $$[X, e_{2}+\imaginario e_{1}]= \imaginario\alpha (X)e_{2}-\alpha(X) e_{1}.$$ Item (vi-c) follows from the fact that $\ad$ is skew--symmetric with respect to every bi--invariant metric (see Proposition~\ref{curvatureLie}). Finally, (vi-d) follows from (vi-b) and from $$\frac{\mathrm d}{\mathrm dt}\Ad(\exp(tX))e_{i}\big|_{t=0}=[X,e_{i}].$$

In order to prove (vii), we first prove that $\mathfrak{g}^{(\alpha)}$ is a Lie algebra and that the definition of $\alpha^{\vee}$ does not depend on the bi--invariant metric. Let $Y$ be an infinitesimal generator of $\mathfrak{t}$. Then Jacobi equation and (vi-b) imply that $[[e_{1},e_{2}],Y]=0$. Hence, from Lemma~\ref{LemmaMaximalTorusTheorem}, $[e_{1},e_{2}]\in\mathfrak{t}$.  
Note that $[e_1,e_2]\neq 0$. Otherwise $\ker\oplus V_{\alpha}$ would be an abelian Lie algebra with dimension greater than the dimension of maximal abelian Lie algebra $\mathfrak{t}$. 
Item (vi-b) and the fact that $\ad$ is skew--symmetric with respect to every bi--invariant metric imply that $[e_{1},e_{2}]$ is orthogonal to $\ker\alpha$. In particular, the definition of $\alpha^{\vee}$ does not depend on the bi--invariant metric.
To verify that $\mathfrak{g}^{(\alpha)}$ is a Lie algebra, it suffices to note that
\begin{equation*}
\begin{aligned}
\left[[e_{1},e_{2}],e_{1}\right] &= \alpha([e_{1},e_{2}])e_{2}, \\
\left[[e_{1},e_{2}],e_{2}\right] &=-\alpha([e_{1},e_{2}])e_{1}.
\end{aligned}
\end{equation*}

Note also that $\mathfrak{g}^{(\alpha)}\subset \mathfrak{g}_{X}$, $\mathfrak{g}^{(\alpha)}\cap \Zentrum(\mathfrak{g}_{X})=0$ and $$\dim \mathfrak{g}^{(\alpha)}=3=\dim \mathfrak{g}_{X}/\Zentrum(\mathfrak{g}_{X}),$$ where $X$ was defined in the proof of (iv). We conclude that the morphism below is an isomorphism of Lie algebras.
\begin{equation}\label{theorem-roots-existencia-Eq6}
\mathfrak{g}^{(\alpha)}\longrightarrow \mathfrak{g}_{X}/\Zentrum(\mathfrak{g}_{X})=\mathfrak{so}(3).
\end{equation}

Since $\mathfrak{g}^{(\alpha)}=\mathfrak{so}(3)$, $G^{(\alpha)}$ is a compact Lie group isomorphic to $\SU(2)$ or $\SO(3)$. This fact implies that
\begin{equation}\label{theorem-roots-existencia-Eq7}
\Ad\left(G^{(\alpha)}\right)\big|_{\mathfrak{g}^{(\alpha)}}=\SO(3).
\end{equation}
From the fact that the sphere is homogeneous in $\mathfrak{g}^{(\alpha)}$ and \eqref{theorem-roots-existencia-Eq7}, there exists $w\in G^{(\alpha)}$ such that $\Ad(w)\alpha^{\vee}=-\alpha^{\vee}$. Since $$\Ad(\exp(Y))Z=\exp(\ad(Y))Z=Z$$ for $Y\in\mathfrak{g}^{(\alpha)}$ and $Z\in\ker\alpha$, we conclude that for each $g\in G^{(\alpha)}$, $\Ad(g)\big|_{\ker\alpha}=\id$.
\end{proof}

In the sequel, notation established in Theorem~\ref{theorem-roots-existencia} will still be used.

\begin{definition}\label{definitionWeylchamber}
Let $\mathfrak{t}^{r}=\mathfrak{t}\setminus\bigcup_{\alpha\in R}\ker\alpha$. A connected component of $\mathfrak{t}^{r}$ is called a \emph{Weyl chamber}\index{Weyl chamber}. A choice of a fixed Weyl chamber $C$, allows to define the set of \emph{positive roots} $$P=\{\alpha\in R:\alpha(X)>0, \mbox{ for all } X\in C\}.$$ Note that $$R=P\cup (-P) \mbox{ and } P\cap(-P)=\emptyset,$$ where $-P=\{-\alpha:\alpha\in P\}$.
\end{definition}

\begin{remark}
In the definition above,  the choice of a fixed Weyl chamber was used to define positive roots. It is also possible to do the converse, starting with a set of roots with special properties and then defining the associated Weyl chamber. More precisely, consider $P$ a set of roots with the following properties:
\begin{itemize}
\item[(i)] $R=P\cup (-P)$;
\item[(ii)] $P\cap(-P)=\emptyset$;
\item[(iii)] If $\sum_{\alpha\in P}c_{\alpha}\alpha=0$, with $c_{\alpha}\geq 0$ for all $\alpha\in P$, then $c_{\alpha}=0$ for all $\alpha\in P$.
\end{itemize}

Such $P$ will be called a \emph{choice of positive roots}\index{Root!positive}. The associated Weyl chamber is then defined as $$C=\{X\in \mathfrak{t}:\alpha(X)>0, \mbox{ for all } \alpha\in P\}.$$ Therefore there exists a bijective correspondence between the set of choices of positive roots and the set of Weyl chambers in $\mathfrak{k}$.
\end{remark}

\begin{remark}\label{remarkWeylchamber1}
From Theorem~\ref{theorem-roots-existencia}, we infer that $$\mathfrak{g}=\mathfrak{t}\oplus \sum_{\alpha\in P}V_{\alpha}.$$
\end{remark}

\begin{remark}\label{remarkWeylchamber2}
If $X\in C$, then $\Ad(G)X$ is a regular orbit and its codimension is equal to the dimension of $T$. Indeed, since $X\notin \ker \alpha$ for any $\alpha\in R$, we conclude that $\mathfrak{g}_{X}=\mathfrak{t}.$ We will see in Theorem~\ref{theorem-OrbitaPrincipal=OrbitaRegular} that all regular orbits of the adjoint action are principal orbits.  
\end{remark}

It was proved in Proposition~\ref{proposition-adjointaction-is-isoparametric} that principal orbits of the adjoint action are isoparametric (recall Definition~\ref{definition-isoparametric}). We now calculate principal curvatures and principal directions of such a principal orbit.

\begin{remark}\label{remark-Ad-isoparametrica-curvaturas}
Let $G$ be a connected compact Lie group with a bi--invariant metric $\langle \cdot,\cdot\rangle$ and $T\subset G$ a fixed maximal torus, with Lie algebras $\mathfrak{g}$ and $\mathfrak{t}$, respectively.

Consider $Z\in\mathfrak{t}$ such that $\Ad(G)Z$ is a principal orbit. From Corollary~\ref{corollary-MaximalTorus}, $$T_{Z}(\Ad(G)Z)=\{[\xi,Z]:\mbox{ for all }\xi \in\mathfrak{g}\}.$$ Let $N\in\mathfrak{t}$. From the Maximal Torus Theorem~\ref{maxtorusthm}, $\widehat{N}(\Ad(g)Z)=\Ad(g)N$ is a well--defined normal field along $\Ad(G)Z$. Deriving the last equation, it follows that \begin{equation}\label{remark-Ad-isoparametrica-curvaturas-Eq1}
\dd\widehat{N}_{Z}([\xi,Z])=[\xi,N].
\end{equation}
We claim that $\widehat{N}$ is normal and parallel. This follows from \eqref{remark-Ad-isoparametrica-curvaturas-Eq1} and $$\langle [\xi,N], X\rangle = \langle \xi, [N,X]\rangle=0, \mbox{ for all } X\in\mathfrak{t}.$$

Since $\widehat{N}$ is normal and parallel, from \eqref{remark-Ad-isoparametrica-curvaturas-Eq1}, \begin{equation}\label{remark-Ad-isoparametrica-curvaturas-Eq2}
\mathcal S_{\widehat{N}}([\xi,Z])=-[\xi,N],
\end{equation}
where $\mathcal S_{\widehat{N}}$ is the shape operator. Let $\{e_{1},e_{2}\}$ be the basis defined in Theorem~\ref{theorem-roots-existencia}. Replacing $\xi$ by $e_{i}$ in \eqref{remark-Ad-isoparametrica-curvaturas-Eq2}, it follows that \begin{equation}\label{remark-Ad-isoparametrica-curvaturas-Eq3}
\mathcal{S}_{\widehat{N}}(e_{1})=-\frac{\alpha(N)}{\alpha(Z)}e_{1} , \quad \mathcal{S}_{\widehat{N}}(e_{2})=-\frac{\alpha(N)}{\alpha(Z)}e_{2}.
\end{equation}
Hence $V_{\alpha}$ is a curvature distribution of the orbit $\Ad(G)Z$. In particular, the spaces $V_{\alpha}$ are orthogonal to each other.
\end{remark}

To conclude this section, we calculate the roots of $\SU(n)$. More roots computations on other classical Lie groups such as $\SO(n)$ and $\Sp(n)$ can be found in Fegan~\cite{Fegan}.

\begin{example}\label{example-roots-SU(n)}\index{Root!of $\SU(n)$}\index{$\SU(n)$}\index{$\mathfrak{su}(n),\mathfrak{u}(n)$}
Let $T$ be the subgroup of diagonal matrices in $\SU(n)$. It follows from Theorem~\ref{cpctabeliantorus} that $T$ is a torus, and it can be proved that $T$ is a maximal torus. As usual, let $\mathfrak{t}$ denote the Lie algebra of $T$. Consider $X\in\mathfrak{t}$. Then$$X  =\left(\begin{array}{c c c}
                    \imaginario \theta_{1}& &  \\
                        & \ddots &  \\
                        & &\imaginario \theta_{n}
                      \end{array}\right),$$ with $\theta_{1} +\dots +\theta_{n}=0$.

Let $E_{ij}$ denote the matrix with $1$ in the $(i,j)^\mathrm{th}$ entry and $-1$ in the $(j,i)^\mathrm{th}$ entry and $F_{ij}$ the matrix with $\imaginario$ in both $(i,j)^\mathrm{th}$ and $(j,i)^\mathrm{th}$ entries. A direct calculation gives
$$\begin{array}{l c c c r}
\ad(X)E_{ij} &=& X E_{ij}-E_{ij} X &=& (\theta_{i}-\theta_{j}) F_{ij}, \\
\ad(X)F_{ij} &=& X F_{ij}-F_{ij} X &=& -(\theta_{i}-\theta_{j}) E_{ij}.
\end{array}$$
From the above equations and uniqueness of the decomposition on $V_{\alpha}$ (see Theorem~\ref{theorem-roots-existencia} and Remark~\ref{remarkWeylchamber1}), the roots of $\SU(n)$ are $$\pm (\theta_{i}^{*}-\theta_{j}^{*}), \mbox{ with } i< j,$$ where $\theta_{i}^{*}(X)=\theta_{i}$. One can also check that $$(\theta_{i}^{*}-\theta_{j}^{*}), \mbox{ with } i<j$$
is a choice of positive roots.

\end{example}


\section{Weyl group}

 Let $G$ be a connected compact Lie group endowed with a bi--invariant metric $\langle \cdot,\cdot\rangle$ and $T\subset G$ a fixed maximal torus, with Lie algebras $\mathfrak{g}$ and $\mathfrak{t}$, respectively. The aim of this section is twofold. First, we prove that \emph{every} regular orbit of the adjoint action is a principal orbit. We then study intersections of orbits of the adjoint action with $\mathfrak{t}$ using the so--called \emph{Weyl group}.

\begin{lemma}\label{lemma-q-in-isotropyofq}
Let $G_{q}^{0}$ be the connected component of the isotropy group $G_{q}=\{g\in G:gqg^{-1}=q\}$. Then $q\in G_{q}^{0}$.
\end{lemma}

\begin{proof}
Since $G$ is compact, $\exp$ is surjective (see Theorem~\ref{expagree}). Thus, there exists $X\in\mathfrak{g}$ such that $q=\exp(X)$. We conclude that
\begin{eqnarray*}
q&=&\exp(X)\\
 &=&\exp(tX)\exp(X)\exp(-tX)\\
 &=&\exp(tX)q \exp(-tX).
\end{eqnarray*}
Therefore $\exp(tX)\in G_{q}$ and hence $q=\exp(X) \in G_{q}^{0}$.
\end{proof}

\begin{lemma}\label{lemma-OrbitaPrincipal=OrbitaRegular}
Let $G_{X}=\{g\in G:\Ad(g)X=X\}$ be the isotropy group of $X\in\mathfrak{t}$. Then $G_{X}$ is a connected Lie group.
\end{lemma}

\begin{proof}
Let $z\in G_{X}$. We want to prove that $z$ is in the identity's connected component $G_{X}^{0}$. Consider $G_{z}$ the isotropy group with respect to the conjugation action of $z$. Then $X$ is in the Lie algebra of $G_{z}$. In fact, since $z\in G_{X}$ it follows that $t\Ad(z)X=tX$, hence $z\exp(tX)z^{-1}=\exp(tX)$. This is equivalent to $\exp(-tX)z\exp(tX)=z$.

Let $\widetilde{T}$ be a maximal torus of $G_{z}^{0}$, tangent to $X$. Since $[Z,X]=0$ for every $Z$ in the Lie algebra of $\widetilde{T}$, $Z$ is in the Lie algebra of $G_{X}$. Thus
\begin{equation}\label{lemma-OrbitaPrincipal=OrbitaRegularEq1}
\widetilde{T}\subset G_{X}^{0}.
\end{equation}
Since $z\in G_{z}^{0}$ (see Lemma~\ref{lemma-q-in-isotropyofq}), from the Maximal Torus Theorem~\ref{maxtorusthm} there exists $g\in G_{z}^{0}$ such that
\begin{equation}\label{lemma-OrbitaPrincipal=OrbitaRegularEq2}
g z g^{-1}\in \widetilde{T}.
\end{equation}
Finally, since $g\in G_{z}^{0}$,
\begin{equation}\label{lemma-OrbitaPrincipal=OrbitaRegularEq3}
z=g z g^{-1}.
\end{equation}
Equations \eqref{lemma-OrbitaPrincipal=OrbitaRegularEq1}, \eqref{lemma-OrbitaPrincipal=OrbitaRegularEq2} and \eqref{lemma-OrbitaPrincipal=OrbitaRegularEq3} imply that $z\in G_{X}^{0}$.
\end{proof}

\begin{theorem}\label{theorem-OrbitaPrincipal=OrbitaRegular}
Each regular orbit of the adjoint action is a principal orbit.
\end{theorem}
\begin{proof}
Let $X_{0}$ be a regular point. This implies that $\dim G_{X_{0}}$ $\leq \dim G_{X}$ for all $X\in\mathfrak{g}$. Consider a slice $S_{X_{0}}$ at $X_{0}$. From the definition of slice, $G_{Y}\subset G_{X_{0}}$ for each $Y\in S_{X_{0}}$. Since $\dim G_{X_{0}}\leq \dim G_{Y}$ and both are connected, it follows that $G_{Y}=G_{X_{0}}$. This implies that $G(X_{0})$ is a principal orbit.
\end{proof}

\begin{definition}\label{definitionWeylgroupAdjointAction}\index{Weyl group}
Consider \begin{eqnarray*} N_{\mathfrak t} &=& \{g\in G:\Ad(g) \mathfrak{t}\subset \mathfrak{t}\}, \\ Z_{\mathfrak t} &=& \{g\in G:\Ad(g)Y=Y, \mbox{ for all } Y \in\mathfrak{t}\}. \end{eqnarray*} Clearly $Z_{\mathfrak t}$ is a normal subgroup of $N_{\mathfrak t}$. The \emph{Weyl group} is defined by $$W=N_{\mathfrak t}/Z_{\mathfrak t}.$$
\end{definition}

Note that the action $W\times\mathfrak{t}\ni (wZ_{\mathfrak t},X)\mapsto\Ad(w)X\in\mathfrak{t}$ is an effective isometric action, and orbits of the action of $W$ on $\mathfrak{t}$ coincide with the intersections of $\mathfrak{t}$ with the orbits of the adjoint action.

\begin{definition}
Let $\alpha(\cdot)=\langle \widetilde{\alpha}, \cdot\rangle $ be a root. Then $\varphi_{\alpha}$ denotes the orthogonal reflection across $\ker(\alpha)$, i.e., $$\varphi_{\alpha}(Z)=Z-2\left\langle Z,\frac{\widetilde{\alpha}}{\|\widetilde{\alpha}\|}\right\rangle \frac{\widetilde{\alpha}}{\|\widetilde{\alpha}\|}.$$
\end{definition}

\begin{remark}
If $\alpha^{\vee}$ is the coroot of $\alpha$ then $\varphi_{\alpha}(Z)=Z- \alpha(Z)\alpha^{\vee}$.
\end{remark}

Item (vii) of Theorem~\ref{theorem-roots-existencia} implies that reflections $\varphi_{\alpha}$ are elements of the Weyl group $W$. Moreover, in the next theorem we prove that $W$ is generated by these reflections.

\begin{theorem}\label{proposition-W-generatedby-reflections}
Let $W$ be the Weyl group. Then
\begin{itemize}
\item[(i)] Let $C_1$ and $C_2$ be two Weyl chambers such that $X\in \partial C_{1}\cap \partial C_{2}$. Then there exist reflections $\varphi_{\alpha_1},\ldots,\varphi_{\alpha_{n}}$ such that $\varphi_{\alpha_{i}}(X)=X$ and 
$\varphi_{\alpha_{n}}\circ \cdots \circ\varphi_{\alpha_{1}}C_{2}=C_{1}$;
\item[(ii)] The closure of each Weyl chamber $C$ is a fundamental domain for the action of $W$ on $\mathfrak{t}$, i.e., each orbit of the adjoint action intersects the closure of $C$ exactly once;
\item[(iii)] $W$ is generated by $\{\varphi_{\alpha}\}_{\alpha\in R},$ where $R$ is the set of roots.
\end{itemize}
\end{theorem}

\begin{proof}
First, note that each isometry $w\in W$ maps regular points to regular points and hence $w C_{1}$ is a Weyl chamber. The facts that $\Ad(G)$ is a compact orbit, $W$ is a subgroup of isometries and $W(p)=W/W_p$ imply that the cardinality of $W$ is finite. Thus, the subgroup $K_X$ generated by the reflections $\varphi_{\alpha}$ such that $\varphi_{\alpha}(X)=X$ is finite.

Let $B_{2\varepsilon}(X)$ be an open ball in the intersection of a slice at $X$ with $\mathfrak{t}$ and consider two regular points $p_1$ and $p_2$ such that $p_{i}\in C_{i}$ and $p_{i}\in B_{\varepsilon}(X)$. Define $f:K_{X}\rightarrow\R$ as $f(\varphi_{\alpha})=\|\varphi_{\alpha}(p_{2})-p_{1}\|$, and let $w_{0}\in K_{X}$ be a minimum of $f$. Suppose that $w_{0}p_{2}\notin C_{1}$. Then the line joining $w_{0}p_{2}$ to $p_{1}$ intersects a wall $\ker \beta$. Since this intersection is contained in $B_{\varepsilon}(X)$, $X\in\ker\beta$ and hence $\varphi_{\beta}\in K_{X}$. Therefore $f(\varphi_{\beta}w_{0})<f(w_{0})$, contradicting the fact that $w_{0}$ is a minimum of $f$. Thus $w_{0}p_{2}\in C_{1}$ and this concludes the proof of (i).

As for (ii), the Maximal Torus Theorem~\ref{maxtorusthm} implies that each orbit of the adjoint action intersects $\mathfrak{t}$. Composing with reflections $\varphi_{\alpha}$, we conclude that each regular (respectively singular) orbit intersects $C$ (respectively $\partial C$) at least once. We claim that each principal orbit intersects $C$ exactly once. Set $$H=\{ h\in G:\Ad(h)X\in C \ ,\mbox{ for all } X\in C\},$$ and assume that there exists $X\in C$ such that the orbit $\Ad(G)X$ meets $C$ more than once. That is, there exists $g\in G$ such that $X\neq \Ad(g)X\in C$. Clearly $g\notin T$. It is not difficult to prove that $g\in H$. Since $\Ad(G)$ is compact, $g$ has finite order\footnote{This means that there exists $n$ such that $g^n=e$ and $g^{n-1}\neq e$.} $n$. Set $Y=X+\Ad(g)X+\dots +\Ad(g^{n-1})X$. Since $g\in H$, $Y\in C$ and $\Ad(g)Y=Y$ from construction. Hence
\begin{equation}\label{proposition-W-generatedby-reflectionsEq2}
g\in G_{Y}.
\end{equation}
On the other hand, since $Y\in C$, it follows from Remark~\ref{remarkWeylchamber2} and Lemma~\ref{lemma-OrbitaPrincipal=OrbitaRegular} that
\begin{equation}\label{proposition-W-generatedby-reflectionsEq3}
G_{Y}=G_{Y}^{0}=T.
\end{equation}
Equations \eqref{proposition-W-generatedby-reflectionsEq3} and \eqref{proposition-W-generatedby-reflectionsEq2} contradict the fact that $g\notin T$.

Let us prove that each singular orbit of the adjoint action intersects $\partial C$ exactly once. Assume that there exists $0\neq X\in\partial C$ and $g\in G$ such that $X\neq \Ad(g)X\in\partial C$. Let $\gamma:[0,1]\rightarrow C$ be a curve with $\gamma(0)=X$ and $\gamma(t)\in C$ for $0<t\leq 1$. Composing with reflections $\varphi_{\alpha}$ if necessary, we conclude that $$\Ad(g)\gamma|_{(0,1]} \subset C.$$ Therefore, for small $t$, $$\gamma(t)\neq \Ad(g)\gamma(t)\in C,$$
contradicting the fact proved above that each principal orbit of the adjoint action intersects $C$ exactly once.

It remains to prove (iii). For any $g\in W$ there exist reflections $\varphi_{\alpha_{1}},\ldots, \varphi_{\alpha_{n}}$ such that $\varphi_{\alpha_{n}}\circ \cdots \circ\varphi_{\alpha_{1}}\circ g (C)\subset C$. From (ii), $$\varphi_{\alpha_{n}}\circ \cdots \circ\varphi_{\alpha_{1}}\circ g (X)\in \{\Ad(G)X\cap C\}=\{X\}$$ for all $X\in C$. Therefore, the isometry $\varphi_{\alpha_{n}}\circ \cdots \circ\varphi_{\alpha_{1}}\circ g|_{\mathfrak{t}}$ is the identity, concluding the proof.
\end{proof}

\begin{remark}\label{remark-proposition-W-generatedby-reflections} 
Let $\F=\{\Ad(G)X\}_{X\in\mathfrak{g}}$ be the singular foliation by orbits of the adjoint action. The above result implies that $\F\cap\mathfrak{t}$ is invariant by reflections across the walls of Weyl chambers.
\end{remark}

\begin{remark}
The above theorem implies that the Weyl group permutes the Weyl chambers and that the cardinality of $W$ is equal to the number of Weyl chambers.
\end{remark}

\section{Normal slice of conjugation action}

In this section, we describe the normal slice of the conjugation action of a compact Lie group $G$ endowed with a bi--invariant metric. We start by recalling that the normal slice $S_{q}$ of an isometric action $G\times M\rightarrow M$ at a point $q$ is $S_{q}=\exp_{q}(B_{\varepsilon}(0))$, where $B_{\varepsilon}(0)$ lies in the normal space $\nu_{q}(G(q))$ of the orbit $G(q)$. We also recall that, for $x\in S_{q}$, $gx\in S_{q}$ if and only if $g$ is in the isotropy group $G_{q}$.

\begin{proposition}\label{proposition-isotropia-uniao-toros}
Let $\mathcal{T}$ be the collection of maximal tori of $G$, and $\Lambda(q)$ the subcollection of maximal tori that contain $q$. Then $\Zentrum(G)=\bigcap_{T\in\mathcal{T}} T$ and $G_{q}^{0}=\bigcup_{T\in\Lambda(q)}T$, where $G_{q}^{0}$ is the connected component of the isotropy group $G_{q}$ with respect to the adjoint action.
\end{proposition}

\begin{proof}
Item (iii) of the Maximal Torus Theorem~\ref{maxtorusthm} implies that $G=\bigcup_{T\in\mathcal{T}}T$ and hence $\Zentrum(G)\supset\bigcap_{T\in\mathcal{T}} T$. The Maximal Torus Theorem~\ref{maxtorusthm} also implies that each $g\in\Zentrum(G)$ is contained in every maximal torus $T$. In fact, for a fixed maximal torus $T$ there exists $h\in G$ such that $hgh^{-1}\in T$. Since $g\in\Zentrum(G)$ we conclude that $g=hgh^{-1}$ and hence $g\in T$.

Furthermore, it is clear that $G_{q}^{0}\supset\bigcup_{T\in\Lambda(q)}T$. Let $g\in G_{q}^{0}$ and $\widetilde{T}$ be a maximal torus of $G_{q}^{0}$ that contains $g$. Since $q\in\Zentrum\left(G_{q}^{0}\right)$ (see Lemma~\ref{lemma-q-in-isotropyofq}) we conclude that $q\in\widetilde{T}$, using the above argument. Let $T$ be a maximal torus of $G$ such that $T\supset\widetilde{T}$. Then $g,q\in T$ and hence $g\in T\in\Lambda(q)$.
\end{proof}

\begin{theorem}\label{theorem-slice-acaoconjugacao}
Let $G$ be a compact connected Lie group with a bi--invariant metric acting on itself by conjugation, and consider $S_{q}$ a normal slice at $q$ and $\varepsilon$ its radius. Then
\begin{itemize}
\item[(i)] Let $\Lambda_{\varepsilon}(q)=\{B_{\varepsilon}(q)\cap T:T\in \Lambda(q)\}$. Then $S_{q}=\bigcup_{\sigma\in\Lambda_{\varepsilon}(q)}\sigma$;
\item[(ii)] If $y\in S_{q}$, then $S_{y}\subset S_{q}$;
\item[(iii)] Let $\F$ denote the partition by the orbits of the conjugation action of $G$, i.e. $\F=\{G(x):x\in G\}$. Then there exists an isoparametric foliation $\widehat{\F}$ on a neighborhood of the origin of $T_{q}S_q$ such that $\exp_{q}(\widehat{\F})=\F\cap S_{q}$.
\end{itemize}
\end{theorem}

\begin{proof}
It follows from Proposition~\ref{proposition-isotropia-uniao-toros} that
\begin{equation}\label{theorem-slice-acaoconjugacaoEq-1}
G_{q}^{0}\cap B_{\varepsilon}(q)=\bigcup_{\sigma\in\Lambda_{\varepsilon}(q)} \sigma.
\end{equation}
Let $\sigma\in\Lambda_{\varepsilon}(q)$. Recall that each torus is totally geodesic in $G$. Therefore, the shortest segment of geodesic $\gamma$ joining $x\in\sigma$ to $q$ is contained in $\sigma$. Since $G(q)$ is orthogonal to $\sigma$ (see the Maximal Torus Theorem~\ref{maxtorusthm}) we conclude that $\gamma$ is orthogonal to $G(q)$. Hence $x\in S_{q}$, and then
\begin{equation}\label{theorem-slice-acaoconjugacaoEq-2}
\bigcup_{\sigma\in\Lambda_{\varepsilon}(q)} \sigma \subset S_{q}.
\end{equation}

Moreover, note that
\begin{equation}\label{theorem-slice-acaoconjugacaoEq-3}
\dim S_{q}=\dim(G)- \dim(G/G_{q})=\dim G_{q}.
\end{equation}
From \eqref{theorem-slice-acaoconjugacaoEq-1}, \eqref{theorem-slice-acaoconjugacaoEq-2} and \eqref{theorem-slice-acaoconjugacaoEq-3}, it follows that
\begin{equation}\label{theorem-slice-acaoconjugacaoEq-4}
G_{q}^{0}\cap B_{\varepsilon}(q) =\bigcup_{\sigma\in\Lambda_{\varepsilon}(q)} \sigma=S_{q}.
\end{equation}

Item (ii) follows from $G_{y}\subset G_{q}$ and \eqref{theorem-slice-acaoconjugacaoEq-4}.

Finally, \eqref{theorem-slice-acaoconjugacaoEq-4} implies that there exists a neighborhood $U$ of $e\in G_{q}^{0}$ such that $L_{q}(U)=S_{q}$. Set $\widetilde{\F}_{q}=\{G_{q}(y):y\in S_{q}\}$ and $\widetilde{\F}_{e}=\{G_{q}(x):x\in U\}$. Since $q\in\Zentrum(G_{q})$,
\begin{equation}\label{theorem-slice-acaoconjugacaoEq-5}
\widetilde{\F}_{q}=L_{q}(\widetilde{\F}_{e}).
\end{equation}
We also know from the properties of a slice of an isometric action that
\begin{equation}\label{theorem-slice-acaoconjugacaoEq-6}
\widetilde{\F}_{q}=\F\cap S_{q}.
\end{equation}
Now set $\widehat{\F}_{e}=\{\Ad(G_{q})X:X\in (B_{\varepsilon}(0)\cap\mathfrak{g}_{q})\}$, where $\mathfrak{g}_{q}$ denotes the Lie algebra of $G_{q}$ and $\widehat{\F}=\dd L_{q}(\widehat{\F}_{e})$. Note that
\begin{equation}\label{theorem-slice-acaoconjugacaoEq-7}
\widetilde{\F}_{e}=\exp(\widehat{\F}_{e}).
\end{equation}

Equations \eqref{theorem-slice-acaoconjugacaoEq-5}, \eqref{theorem-slice-acaoconjugacaoEq-6} and \eqref{theorem-slice-acaoconjugacaoEq-7} imply
{\allowdisplaybreaks
\begin{eqnarray*}
\F\cap S_{q}&=&\widetilde{\F}_{q}\\
            &=&L_{q}(\widetilde{\F}_{e})\\
            &=&L_{q}(\exp(\widehat{\F}_{e}))\\
            &=&\exp_{q}(\dd L_{q}(\widehat{\F})_{e})\\
            &=&\exp_{q}(\widehat{\F}).
\end{eqnarray*}}
Finally, note that $\widehat{\F}$ is isoparametric, once $\widehat{\F}_{e}$ is isoparametric and the left translation $L_{q}$ is an isometry.
\end{proof}

\section{Dynkin diagrams}\label{Sec-Dynkin Diagrams}

In this section, we briefly study Dynkin diagrams and the classification of compact simple Lie groups, without giving proofs. Most results are proved in Hall~\cite{Hall}, Helgason~\cite{helgason}, San Martin~\cite{SanMartin} and Serre~\cite{Serre}. In addition, $G$ will denote a connected compact semisimple Lie group endowed with a bi--invariant metric $\langle \cdot,\cdot\rangle$, and $T\subset G$ a fixed maximal torus. As usual, $\mathfrak{g}$ and $\mathfrak{t}$ are the Lie algebra of $G$ and $T$ respectively, and $R$ the set of roots of $G$.

\begin{definition}
Let $\alpha(\cdot)=\langle \widetilde{\alpha},\cdot\rangle$ and $\beta(\cdot)=\langle \widetilde{\beta}, \rangle$ be elements of $\mathfrak{t}^{*}$, the dual space of $\mathfrak{t}$.
The \emph{inner product} of $\alpha$ and $\beta$ is defined by $\langle \alpha, \beta \rangle=\langle \widetilde{\alpha},\widetilde{\beta}\rangle$.
\end{definition}

\begin{proposition}
Let $\alpha,\beta\in R$ and assume that $\alpha\neq\pm\beta$ and $\langle \alpha, \beta\rangle \leq 0$. Then either $\alpha$ and $\beta$ are orthogonal, or the angle between $\alpha$ and $\beta$ is
\begin{itemize}
\item[(i)] $120^{\mathrm o}$ and $\|\alpha\|=\|\beta\|$;
\item[(ii)] $135^{\mathrm o}$ and $\|\alpha\|=\sqrt{2}\|\beta\|$ or $\|\beta\|=\sqrt{2}\|\alpha\|$;
\item[(iii)] $150^{\mathrm o}$ and $\|\alpha\|=\sqrt{3}\|\beta\|$ or $\|\beta\|=\sqrt{3}\|\alpha\|$;
\end{itemize}
\end{proposition}

\begin{definition}
A positive\footnote{See Definition~\ref{definitionWeylchamber}} root $\alpha$ is \emph{simple}\index{Root!simple} if it is not the sum of others positive roots.
\end{definition}

\begin{theorem}\label{thmparadynkin}
Let $P$ be a choice of positive roots of $R$.
\begin{itemize}
\item[(i)] The subset $\triangle\subset P$ of simple roots is a basis of $\mathfrak{t}^{*}$;
\item[(ii)] Each $\alpha\in R$ is a linear combination of elements of $\triangle$ with integer coefficients, either all non negative or all non positive;
\item[(iii)] If $\alpha,\beta\in\triangle$, then $\langle \alpha,\beta\rangle \leq 0$.
\end{itemize}
\end{theorem}

The set $\triangle$ is called a \emph{base}\index{Root!base} of the root system $R$.

\begin{definition}\label{dynkin}
Let $\triangle=\{\alpha_{1},\dots,\alpha_{n}\}$ be a base of $R$. The \emph{Dynkin diagram}\index{Dynkin diagram}\index{Root!diagram} of $R$ with respect to $\triangle$ is a graph in which vertexes represent simple roots $\alpha_{1},\dots,\alpha_{n}$. Consider distinct indices $i$ and $j$. If $\alpha_i$ and $\alpha_j$ are orthogonal, there is no edge between them.

\smallskip
\centerline{\xymatrix{{\bullet}_{\alpha_{i}} & {\bullet}_{\alpha_{j}}}}
\smallskip

\noindent
Otherwise, we place either one, two or three edges pointing to the \emph{smaller} root respectively if the angle between $\alpha_i$ and $\alpha_j$ is
\begin{itemize}
\item[(i)] $120^{\mathrm o}$

\smallskip
\centerline{\xymatrix{{\bullet}_{\alpha_{i}} \ar@{-}[r] & {\bullet}_{\alpha_{j}}}}
\smallskip

\item[(ii)] $135^{\mathrm o}$ and $\|\alpha_{i}\|=\sqrt{2}\|\alpha_{j}\|$

\smallskip
\centerline{\xymatrix{ {\bullet}_{\alpha_{i}} \ar@2{-}[r]|{\rangle} & {\bullet_{\alpha_{j}}}}}
\smallskip

\item[(iii)] $150^{\mathrm o}$ and $\|\alpha_{i}\|=\sqrt{3}\|\alpha_{j}\|$

\smallskip
\centerline{\xymatrix{ {\bullet_{\alpha_{i}}}\ar@3{-}[r]|{\rangle}& {\bullet_{\alpha_{j}}}}}
\end{itemize}
\end{definition}

It is possible to prove that the Dynkin diagram does not depend on the bi--invariant metric or the choice of positive roots.

\begin{example}\index{Dynkin diagram!of $\SO(7)$}\index{$\SO(n)$}\index{$\mathfrak{so}(n),\mathfrak{o}(n)$}\label{dynkinso7}
Let us determine the Dynkin diagram of $G=\SO(7)$. In this case, the Lie algebra $\mathfrak t$ of a maximal torus is formed by matrices of the form
$$X=\left(\begin{array}{c c c c c c c}
0 & \theta_{1} & & & & &\\
-\theta_{1} & 0 & & & & &\\
& & 0 & \theta_{2} & & & \\
& & -\theta_{2} & 0 & & &\\
& & & & 0 & \theta_{3} & \\
& & & & -\theta_{3} & 0 &\\
& & & & & & 0 \\
\end{array}\right).$$
The same kind of argument used in Example~\ref{example-roots-SU(n)} implies that the roots of $G$ are $$\pm\theta_{i}^{*} \mbox{ and } \pm (\theta_{i}^{*}\pm \theta_{j}^{*}) \mbox{ for } i<j,$$ where $\theta_{i}^{*}(X)=\theta_{i}$. A possible choice of positive roots is $$\theta_{i}^{*} \mbox{ and } (\theta_{i}^{*} \pm \theta_{j}^{*}) \mbox{ for } i<j.$$ For this choice, the basis of simple roots is $\triangle=\{\alpha_{1}, \alpha_{2}, \alpha_{3}\}$, where
\begin{eqnarray*}
\alpha_{1} &=& \theta_{1}^{*}-\theta_{2}^{*} \\
\alpha_{2} &=& \theta_{2}^{*}-\theta_{3}^{*} \\
\alpha_{3} &=& \theta_{3}^{*}.
\end{eqnarray*}

Consider the bi--invariant metric $\langle X,Y\rangle= \trace XY^{t},$ for $X,Y\in\mathfrak{t}$. It is easy to check that $\langle \theta_{i}^{*},\theta_{j}^{*}\rangle=\delta_{ij}$. Hence, a direct calculation gives
\begin{itemize}
\item[(i)] $\|\alpha_{1}\|=\|\alpha_{2}\|=\sqrt{2}$ and $\frac{\langle \alpha_{1}, \alpha_{2} \rangle}{ \|\alpha_{1}\| \,\|\alpha_{2}\|}= -\frac{1}{2}$, i.e, the angle between them is $120^{\mathrm o}$,

\smallskip
\centerline{\xymatrix{{\bullet}_{\alpha_{1}} \ar@{-}[r] & {\bullet}_{\alpha_{2}}}}
\smallskip

\item[(ii)] $\|\alpha_{2}\|=\sqrt{2}$, $\|\alpha_{3}\|=1$ and $\frac{\langle \alpha_{2}, \alpha_{3} \rangle}{ \|\alpha_{2}\| \,\|\alpha_{3}\|}= -\frac{\sqrt{2}}{2}$, i.e., the angle between them is $135^{o}$.

\smallskip
\centerline{\xymatrix{ {\bullet}_{\alpha_{2}} \ar@2{-}[r]|{\rangle} & {\bullet_{\alpha_{3}}}}}
\smallskip

\item[(iii)] $\langle\alpha_{1},\alpha_{3}\rangle =0$

\smallskip
\centerline{\xymatrix{{\bullet}_{\alpha_{1}}  & {\bullet}_{\alpha_{3}}}}
\smallskip
\end{itemize}
\noindent
Therefore, it follows from the above items that the Dynkin diagram of $\SO(7)$ is

\smallskip
\centerline{\xymatrix{{\bullet}\ar@{-}[r] &{\bullet} \ar@2{-}[r]|{\rangle} & {\bullet}.}}
\smallskip
\end{example}

\begin{exercise}\index{$\mathfrak{su}(n),\mathfrak{u}(n)$}\index{$\SU(n)$}
Use the calculus of  roots of $\SU(n)$ in Example~\ref{example-roots-SU(n)} to prove that the Dynkin diagram of $\SU(3)$ is

\smallskip
\centerline{\xymatrix{{\bullet} \ar@{-}[r] &{\bullet}}}

\medskip
\noindent {\small \emph{Hint:} Use Example~\ref{dynkinso7} as a guideline, Theorem~\ref{thmparadynkin} and Definition~\ref{dynkin}.}
\end{exercise}

We conclude this section stating a classification of connected, compact simple Lie groups, through the use of Dynkin diagrams.

\begin{theorem}
If $G_{1}$ and $G_{2}$ are two compact, semisimple, connected and simply connected Lie groups with the same Dynkin diagram, then $G_{1}$ and $G_{2}$ are isomorphic.
\end{theorem}

\begin{theorem}
The possible connected Dynkin diagrams of a connected, compact simple Lie group are

\smallskip
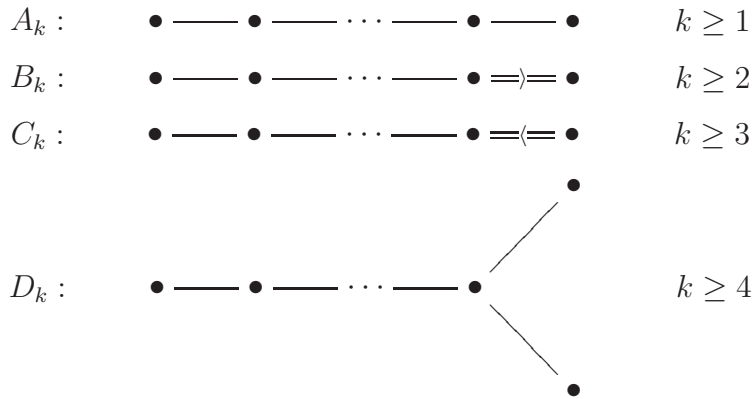
\centerline{\xymatrix{A_{k}: & {\bullet} \ar@{-}[r] & {\bullet}\ar@{-}[r] &\cdots \ar@{-}[r] &{\bullet} \ar@{-}[r] & {\bullet} & \txt{ $k\geq 1$}}}

\smallskip
\centerline{\xymatrix{B_{k}: & {\bullet} \ar@{-}[r] & {\bullet}\ar@{-}[r] &\cdots \ar@{-}[r] &{\bullet} \ar@2{-}[r]|{\rangle} & {\bullet}&\txt{ $k\geq 2$}}}

\smallskip
\centerline{\xymatrix{C_{k}: & {\bullet} \ar@{-}[r] & {\bullet}\ar@{-}[r] &\cdots \ar@{-}[r] &{\bullet} \ar@2{-}[r]|{\langle} & {\bullet} &\txt{ $k\geq 3$}}}

\smallskip
\centerline{\xymatrix{ & & & & & {\bullet}& \\ D_{k}: &  {\bullet} \ar@{-}[r] & {\bullet}\ar@{-}[r] &\cdots \ar@{-}[r] &{\bullet}\ar@{-}[ur]  \ar@{-}[dr] &  & \txt{ $k\geq 4$} \\ & & & & & {\bullet} &}}

\noindent
where $k$ denotes the number of vertexes, and also

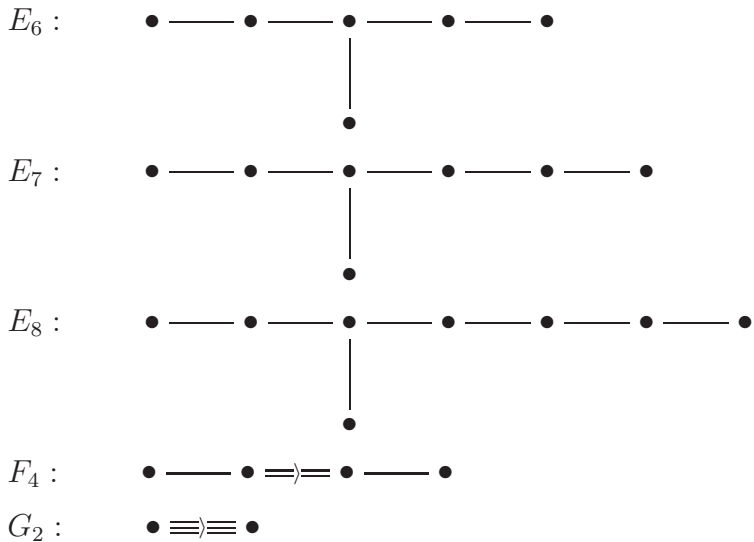
\begin{center}
\begin{minipage}{6cm}
\smallskip
\xymatrix{E_{6}: & {\bullet} \ar@{-}[r]& {\bullet} \ar@{-}[r]& {\bullet} \ar@{-}[r] \ar@{-}[d]& {\bullet} \ar@{-}[r] & {\bullet} \\ & & & {\bullet} &}

\smallskip
\xymatrix{E_{7}: & {\bullet} \ar@{-}[r]& {\bullet} \ar@{-}[r]& {\bullet} \ar@{-}[r] \ar@{-}[d]& {\bullet} \ar@{-}[r] & {\bullet} \ar@{-}[r]& {\bullet} \\  & & & {\bullet} & & &}

\smallskip
\xymatrix{E_{8}: & {\bullet} \ar@{-}[r]& {\bullet} \ar@{-}[r]& {\bullet} \ar@{-}[r] \ar@{-}[d]& {\bullet} \ar@{-}[r] & {\bullet} \ar@{-}[r]& {\bullet}\ar@{-}[r]& {\bullet} \\ & & & {\bullet} & & & &}

\smallskip
\xymatrix{F_{4}: &  {\bullet} \ar@{-}[r]& {\bullet} \ar@2{-}[r]|{\rangle} & {\bullet} \ar@{-}[r] & {\bullet}}

\smallskip
\xymatrix{G_{2}: &  {\bullet}\ar@3{-}[r]|{\rangle}& {\bullet}}
\smallskip
\end{minipage}
\end{center}

\noindent
Each diagram above is the Dynkin diagram of a connected simple compact Lie group. In particular, $A_k$, $B_k$, $C_k$ and $D_k$ are, respectively, the Dynkin diagrams of
\begin{itemize}
\item[(i)] $\SU(k+1)$, for $k\geq 1$;
\item[(ii)] $\SO(2k+1)$, for $k\geq 2$;
\item[(iii)] $\Sp(k)$, for $k\geq 3$;
\item[(iv)] $\SO(2k)$, for $k\geq 4$.
\end{itemize}
\end{theorem}

\chapter{Singular Riemannian foliations with sections}
\label{chap5}

In this chapter, we give a survey of some results on the theory of singular Riemannian foliations with sections (s.r.f.s.). This theory is a generalization of the classic theory of adjoint actions presented in Chapter~\ref{chap4}, and several results of this chapter are extensions of previous results.

\section{Definitions and first examples}

We start by recalling the definition of a singular Riemannian foliation (see Molino~\cite{Molino}).

\begin{definition}\label{definition-s.r.f}
A partition $\F$ of a complete Riemannian manifold $M$ by connected immersed submanifolds, called \emph{leaves}, is a \emph{singular foliation}\index{Foliation!singular} of $M$ if it verifies (i) and a {\it singular Riemannian foliation (s.r.f.)}\index{Foliation!singular Riemannian}\index{s.r.f.} if it also verifies (ii).

\begin{enumerate}
\item[(i)] $\F$ is \emph{singular}, i.e., the module $\singularF$\index{$\singularF$} of smooth vector fields on $M$ that are tangent at each point to the corresponding leaf acts transitively on each leaf. In other words, for each leaf $L$ and each $v\in TL$ with footpoint $p,$ there is $X\in \singularF$ with $X(p)=v$.
\item[(ii)] The partition is \emph{transnormal}, i.e., every geodesic that is perpendicular at one point to a leaf remains perpendicular to every leaf it intersects.
\end{enumerate}
\end{definition}

Let $\F$ be a singular Riemannian foliation on a complete Riemannian manifold $M$. A leaf $L$ of $\F$ (and each point in $L$) is called \emph{regular}\index{Leaf!regular} if the dimension of $L$ is maximal, and it is otherwise called {\it singular}.\index{Leaf!singular} If all the leaves of $\F$ are regular, $\F$ is a \emph{Riemannian foliation}.\index{Foliation!Riemannian}

\begin{definition}\label{defsrfs}
Let $\F$ be a singular Riemannian foliation on a complete Riemannian manifold $M$. Then $\F$ is said to be a \emph{singular Riemannian foliation with sections}\footnote{Recently, s.r.f.s are also being called \emph{polar foliations}.}\index{Singular Riemannian foliation with sections}\index{Polar foliation}\index{Foliation!polar}\index{Foliation!singular Riemannian, with sections}\index{s.r.f.s.} (s.r.f.s.) if for each regular point $p$, the set $\Sigma=\exp_{p}(\nu_p L_{p})$ is a complete immersed submanifold that intersects each leaf orthogonally. In this case, $\Sigma$ is called a \emph{section}.\index{Foliation!section}
\end{definition}

\begin{example}
Using Killing vector fields, one can prove that the partition by orbits of an isometric action on a Riemannian manifold $M$ is a singular Riemannian foliation (recall Proposition~\ref{prop-equivariant-field}). In particular, the partition by orbits of a polar action (see Definition~\ref{definitionPolarAction}) is a s.r.f.s. In addition, regarding isoparametric submanifolds (see Definition~\ref{definition-isoparametric} and Remark~\ref{remark-definition-isoparametric}), an isoparametric foliation on a space forms is a s.r.f.s.
\end{example}

Terng and Thorbergsson~\cite{TTh1} introduced the concept of \index{Submanifold!equifocal} equifocal submanifolds with flat sections in symmetric spaces in order to generalize the definition of isoparametric submanifolds in Euclidean space. We now give a slightly more general definition of equifocal submanifolds in Riemannian manifolds.

\begin{definition}\label{definition-equifocal}\index{Equifocal submanifold}
A connected immersed submanifold $L$ of a complete Riemannian manifold $M$ is \emph{equifocal} if it satisfies
\begin{enumerate}
\item[(i)] The normal bundle $\nu(L)$ is flat;
\item[(ii)] $L$ admits sections, i.e., for each $p\in L$, the set $\Sigma =\exp_{p}(\nu_p L_{p})$ is a complete immersed totally geodesic submanifold;
\item[(iii)] For each $p\in L$, there exists a neighborhood $U\subset L$, such that for each parallel normal field $\xi$ along $U$ the derivative of the \emph{end point} map $\eta_{\xi}:U\ni x\mapsto\exp_{x}(\xi)\in M$ has constant rank.
\end{enumerate}
\end{definition}

Terng and Thorbergsson~\cite{TTh1} proved that the partition by parallel submanifolds of an equifocal submanifold with flat sections in a simply connected compact symmetric space is a s.r.f.s.

\section{Holonomy and orbifolds}

In this section, we recall some facts about Riemannian foliations necessary to understand the next sections. Most of them can be found in Molino~\cite{Molino}.

\begin{proposition}
Let $\F$ be a (singular) foliation on a complete Riemannian manifold $(M,\metric)$. The following statements are equivalent:
\begin{enumerate}
\item[(i)] $\F$ is a (singular) Riemannian foliation;
\item[(ii)] The leaves are locally equidistant, i.e., for each $q\in M$ there exists a tubular neighborhood $$\tub_{\varepsilon}(P_{q})=\{x\in M:d(x,P_{q})<\varepsilon\}$$ of a plaque $P_{q}$ of $L_{q}$ with the following property. 
For $x\in\tub(P_{q})$, each plaque  $P_{x}$ contained in $\tub(P_{q})$ is contained in the cylinder of radius $d(x,P_{q})$ and axis $P_{q}$.
\end{enumerate}
\end{proposition}

It is also convenient to recall equivalent definitions of Riemannian foliations.

\begin{proposition}
\label{prop-equivalencia-df-FR}
Let $\F$ be a foliation on a complete Riemannian manifold $(M,\metric)$. The following statements are equivalent:
\begin{enumerate}
\item[(i)] $\F$ is a Riemannian foliation;
\item[(ii)] For each $q\in M$, there exists a neighborhood $U$ of $q$ in $M$, a Riemannian manifold $(\sigma,b)$ and a Riemannian submersion $f: (U,\metric)\rightarrow (\sigma,b)$ such that the connected components of $\F\cap U$ (plaques) are preimages of $f$;
\item[(iii)] Let $\metric_{T}$ be the transverse metric, i.e., $\metric_{T}=A^{*}\metric$ where $A_{p}:T_{p}M\rightarrow \nu_{p}L$ is the orthogonal projection. Then the Lie derivative $L_{X}\metric_{T}$ is zero for each $X\in \singularF$ (see Definition~\ref{definition-s.r.f}).
\end{enumerate}
\end{proposition}

An important property of Riemannian foliations is \emph{equifocality}. In order to understand this concept, we need some preliminary definitions. A \emph{Bott}\index{Connection!Bott} or \emph{basic connection}\index{Connection!basic} $\nabla$ of a foliation $\F$ is a connection on the normal bundle to the leaves, with $\nabla_XY=[X,Y]^{\nu\F}$ whenever $X\in \singularF$ and $Y$ is a vector field of the normal bundle $\nu\F$ to the foliation. Here, $^{\nu\F}$ denotes projection onto $\nu\F$.

A \emph{normal foliated vector field}\index{Vector field!normal foliated} is a normal field parallel with respect to the Bott connection. If we consider a local submersion $f$ that describes the plaques of $\F$ in a neighborhood of a point of $L$, then a normal foliated vector field is a normal projectable/basic vector field with respect to $f$.

\emph{Equifocality}\index{Equifocality} means that if $\xi$ is a normal parallel vector field (with respect to the Bott connection) along a curve $\beta:[0,1]\rightarrow L$, then the curve $t\mapsto \exp_{\beta(t)}(\xi)$ is contained in the leaf $L_{\exp_{\beta(0)}(\xi)}$.

\begin{remark}
This property still holds even for singular Riemannian foliations, and implies that one can reconstruct the (singular) foliation by taking all parallel submanifolds of a (regular) leaf with trivial holonomy (see Alexandrino and T\"oben~\cite{AlexToeben2}).
\end{remark}

Equifocality allows to introduce the concept of parallel translation (with respect to the Bott connection) of horizontal segments of geodesic.

\begin{definition}\label{paralleltranslation}\index{Parallel translation}
Let $\beta:[a,b]\rightarrow L$ be a piecewise smooth curve and $\gamma:[0,1]\rightarrow M$ a segment of horizontal geodesic such that $\gamma(0)=\beta(a)$. Consider $\xi_0$ a vector of the normal space $\nu_{\beta(a)}L$ such that $\exp_{\gamma(0)}(\xi_0)=\gamma(1)$ and $\xi:[a,b]\rightarrow \nu L$ the parallel translation of $\xi_0$ with respect to the Bott connection along $\beta$. Define $\|_\beta(\gamma)=\widehat{\gamma}$, where $\widehat{\gamma}:[0,1]\rightarrow M$ is the segment of geodesic given by $s\mapsto\widehat{\gamma}(s)=\exp_{\beta(b)}(s\,\xi(b))$.
\end{definition}

Due to equifocality of $\F$, there is an alternative definition of holonomy map of a Riemannian foliation.

\begin{definition}\label{dfn-holonomy-map}\index{Holonomy!map}
Let $\beta:[0,1]\rightarrow L$ be a piecewise smooth curve and $S_{\beta(i)}=\{\exp_{\beta(i)}(\xi):\xi\in \nu_{\beta(i)} L, \|\xi\|<\varepsilon\}$ a \emph{slice} at $\beta(i)$, for $i=0,1$. Then a \emph{holonomy map} is defined by
\begin{eqnarray*}
\varphi_{[\beta]}:S_{\beta(0)} &\longrightarrow &S_{\beta(1)} \\
x & \longmapsto & ||_{\beta}\gamma(r),
\end{eqnarray*}
where $\gamma:[0,r]\rightarrow S_{\beta(0)}$ is the minimal segment of geodesic joining $\beta(0)$ to $x$. Since the Bott connection is locally flat, the parallel translation depends only on the homotopy class $[\beta]$.
\end{definition}

\begin{remark}\label{remark-defClassicaHolonomyMap}
The holonomy map of a foliation $\F$ is usually defined as follows. Consider a partition $0=t_{0}<t_{1}<\ldots < t_{k}=1$ such that $\beta((t_{i-1},t_{i}))$ is contained in a simple open set $U_{i}$ defined by trivialization of $\F$. Then $\{U_{1},\ldots U_{k}\}$ is called a chain of simple open sets. Sliding along the plaques in $U_{k}$ one can define a diffeomorphism $\varphi_{k}$ from an open neighborhood of the transverse manifold $S_{\beta(t_{k-1})}$ onto an open set of $S_{\beta(t_{k})}$. In the same way, sliding along the plaques $U_{k-1}$ one defines a diffeomorphism $\varphi_{k-1}$ from an open neighborhood of $S_{\beta(t_{k-2})}$ onto an open neighborhood of $S_{\beta(t_{k-1})}$. After some steps, we obtain a diffeomorphism $\varphi_{1}$ from an open neighborhood $S_{\beta(t_{0})}$ onto a neighborhood of $S_{\beta(t_{1})}$. Finally, the germ of the diffeomorphism $\varphi_{k}\circ\varphi_{k-1}\circ\cdots\circ\varphi_{1}$ is the holonomy map $\varphi_{[\beta]}$. It is possible to prove that it does not depend on the chain of simple open sets used in its construction, but only on the homotopy class of $\beta$.
\end{remark}

\begin{remark}
Note that, since the holonomy map depends only on the homotopy class, there is a group homomorphism $$\varphi:\pi_{1}(L_{0},x_{0})\longrightarrow \mathrm{Diff}_{x_{0}}(S_{x_{0}}),$$ from the fundamental group of the leaf $L_{0}$ to the group of germs at $x_{0}$ of local diffeomorphisms of $S_{x_{0}}$. Its image is the \emph{holonomy group} of $L_{0}$ at $x_{0}$.\index{Holonomy!group}\index{$\hol$}
\end{remark}

\begin{remark}\label{remarkslicerep2}
Let $G$ be a connected Lie group, $\mu:G\times M\rightarrow M$  a proper action and $\F=\{G(x)\}_{x\in M}$ the partition by orbits, and assume that all orbits are regular. Using Remarks~\ref{remark-SliceVaiEmSlice} and~\ref{remark-defClassicaHolonomyMap} it is possible to check that for each holonomy map $\varphi_{[\beta]}$ there exists $g$ such that $\mu(g,\cdot)|_{S_{x_{0}}}=\varphi_{[\beta]}$. In particular, the holonomy group coincides with the image of the slice representation of $G_{x_{0}}$ in $S_{x_{0}}$ (see Remark~\ref{remarkslicerep}).
\end{remark}

We now recall some facts about pseudogroups and orbifolds related to Riemannian foliations. More details can be found in Salem~\cite[Appendix D]{Molino} or in Moerdijk and  Mr\v{c}un~\cite{Moerdijk}.

\begin{definition}
Let $\Sigma$ be a Riemannian manifold, not necessarily connected. A \emph{pseudogroup}\index{Isometry!pseudogroup} $W$ of isometries of $\Sigma$ is a collection of local isometries $w:U\rightarrow W$, where $U$ and $V$ are open subsets of $\Sigma$ such that
\begin{enumerate}
\item[(i)] If $w\in W$, then $w^{-1}\in W$;
\item[(ii)] If $w:U\rightarrow V$ and $\widetilde{w}:\widetilde{U}\rightarrow\widetilde{V}$ are in $W$, then $\widetilde{w}\circ w:w^{-1}(\widetilde{U})\rightarrow\widetilde{V}$ is also in $W$;
\item[(iii)] If $w:U\rightarrow V$ is in $W$, then its restriction to each open subset $\widetilde{U}\subset U$ is also in $W$;
\item[(iv)] If $w:U\rightarrow V$ is an isometry between open subsets of $\Sigma$ that coincides in a neighborhood of each point of $U$ with an element of $W$, then $w\in W$.
\end{enumerate}
\end{definition}

\begin{definition}
Let $A$ be a family of local isometries of $\Sigma$ containing the identity. The pseudogroup obtained by taking inverses of the elements of $A$, restrictions to open subsets, compositions and unions is called the {\it pseudogroup generated} by $A$.
\end{definition}

\begin{example}\label{ex-holohomypseudogroup}
An important example of a Riemannian pseudogroup is the holonomy pseudogroup of a Riemannian foliation, whose definition we now recall. Let $\F$ be a Riemannian foliation of codimension $k$ on a Riemannian manifold $(M,\metric).$ Then $\F$ can be described by an open cover $\{U_{i}\}$ of $M$ with Riemannian submersions $f_{i}:(U_{i},\metric)\rightarrow(\sigma_{i},b)$, where $\sigma_{i}$ is a submanifold of dimension $k$, such that there are isometries $w_{ij}:f_{i}(U_{i} \cap U_{j})\rightarrow f_{j}(U_{j}\cap U_{i})$ with $f_{j}=w_{ij}\circ f_{i}$. The elements $w_{ij}$ acting on $\Sigma=\sqcup\sigma_{i}$ generate a pseudogroup of isometries of $\Sigma$ called the \emph{holonomy pseudogroup of $\F$}.\index{Holonomy!pseudogroup}
\end{example}

\begin{definition}\label{dfn-Riemannian-orbifold}
A $k$--dimensional \emph{Riemannian orbifold}\index{Riemannian!orbifold}\index{Orbifold} is an equivalence class of pseudogroups $W$ of isometries on a $k$--dimensional Riemannian manifold $\Sigma$ verifying
\begin{enumerate}
\item[(i)] The space of orbits $\Sigma/ W$ is Hausdorff;
\item[(ii)] For each $x\in\Sigma$, there exists an open neighborhood $U$ of $x$ in $\Sigma$ such that the restriction of $W$ to $U$ is generated by a finite group of diffeomorphisms of $U$.
\end{enumerate}

In addition, if $W$ is a group of isometries, then $\Sigma/W$ is called a \emph{good} Riemannian orbifold.\index{Riemannian!good orbifold}\index{Orbifold!good}
\end{definition}

An important example of a Riemannian orbifold is the space of leaves $M/\F$, where $M$ is a Riemannian manifold and $\F$ a Riemannian foliation on $M$ with compact leaves.

\begin{proposition}
\label{prop-M/F-orbifold}
Let $\F$ be a Riemannian foliation with compact leaves on a complete Riemannian manifold $(M,\metric).$ Then $M/\F$ is a Riemannian orbifold isomorphic to $\Sigma/W$, where  $\Sigma$ and $W$ were defined in Example~\ref{ex-holohomypseudogroup}.
\end{proposition}

\begin{remark}
A converse result holds. More precisely, each Riemannian orbifold $\Sigma/W$ is the space of leaves of a Riemannian foliation with compacts leaves (see Moerdijk and Mr\v{c}un~\cite[Proposition 2.23]{Moerdijk}).
\end{remark}

\section{Surgery and suspension of homomorphisms}

The examples of s.r.f.s. presented in the first section have a symmetric space as ambient space. A simple way to construct new s.r.f.s. on nonsymmetric spaces is to consider a s.r.f.s. $\F$ with compact leaves on a manifold $M$ and change either the leaf metric locally, or the transverse metric on a tubular neighborhood of a regular leaf $L$ with trivial holonomy, since the set of this kind of leaves is an open and dense subset in $M$. For a suitable change of metric, $M$ is not symmetric. In this section, we explain some techniques that allow to construct other examples of s.r.f.s., namely surgery and suspension of homomorphisms.

Let us first discuss how \emph{surgery}\index{s.r.f.s.!surgery} can be used to construct s.r.f.s. (see Alexandrino and T\"oben~\cite{AlexToeben}). Let $\F_{i}$ be a s.r.f.s. with codimension $k$ and compact leaves on a complete Riemannian manifold $M_{i}$ for $i=1,2$. Suppose that there exist regular leaves with trivial holonomy $L_{1}\in\F_{1}$ and $L_{2}\in\F_{2}$, such that $L_{1}$ is diffeomorphic to $L_{2}$. Since $L_{i}$ has trivial holonomy, there exists trivializations $\psi_{i}:\tub_{3 \varepsilon}(L_{i})\rightarrow L_{1}\times B_{3 \varepsilon},$ where $\tub_{3 \varepsilon}(L_{i})$ is the tubular neighborhood of $L_{i}$ with radius $3\varepsilon$, and $B_{3 \varepsilon}$ is a ball in the Euclidean $k$--dimensional space.

Now define $\tau: L_{1}\times B_{\varepsilon}\setminus\overline{L_{1}\times B_{\varepsilon/2}}\rightarrow L_{1}\times B_{2\varepsilon}\setminus\overline{L_{1}\times B_{\varepsilon}}$ as the inversion in the cylinder of radius $\varepsilon$ and axis $L_{1}.$ Define also $\phi:\tub_{\varepsilon}(L_{1})\setminus\overline{\tub_{\varepsilon/2}(L_{1})}\rightarrow \tub_{2\varepsilon}(L_{2})\setminus\overline{\tub_{\varepsilon}(L_{2})}$ by $\phi=\psi_{2}^{-1}\circ\tau\circ\psi_{1}$. We now change the transverse metric of $M_{i}$ in the tubular neighborhoods of $L_i$, such that the restriction of $\phi$ to each section in $\tub_{\varepsilon}(L_{1})\setminus\overline{\tub_{\varepsilon/2}(L_{1})}$ is an isometry. This can be done, for instance, by taking $\frac{g_{0}}{\|x\|^{2}}$ as the transverse metric on $\tub_{2\varepsilon}(L_{i})$ and keeping the original transverse metric $g$ outside $\tub_{3\varepsilon}(L_{i})$, using partitions of unity with two appropriate functions. 

Finally, define $\widetilde{M}$ as  $$\widetilde{M}=\left(M\setminus\overline{\tub_{\varepsilon/2}(L_{1})}\right)\sqcup_{\phi} \left(M\setminus\overline{\tub_{\varepsilon}(L_{2})}\right).$$ This new manifold has a singular foliation $\widetilde{F}$, and leaves are locally equidistant with respect to the transverse metric that already exists. To conclude the construction, we only have to define a tangential metric to the leaves with partitions of unity. Note that if $\Sigma_{1}$ (respectively $\Sigma_{2}$) is a section of $\F_{1}$ (respectively $\F_{2}$), then the connected sum $\Sigma_1\#\Sigma_2$ is a section of $\widetilde{\F}$.

Let us now explore an example of s.r.f.s. constructed with \index{s.r.f.s.!suspension of homomorphisms}\emph{suspension of homomorphisms}, which is extracted from Alexandrino~\cite{Alex4}. Other examples can be found in Alexandrino~\cite{Alex2}. We start by recalling the method of suspension. For further details see Molino~\cite[pages 28, 29, 96, 97]{Molino}.

Let $Q$ and $V$ be Riemannian manifolds with dimension $p$ and $n$ respectively, and consider $\rho:\pi_{1}(Q, q_{0})\rightarrow  \Iso(V)$ a homomorphism. Let $\widehat{P}:\widehat{Q}\rightarrow Q$ be the projection of the universal cover of $Q$. Then there is an action of $\pi_{1}(Q,q_{0})$ on $\widetilde{M}=\widehat{Q}\times V$ given by
$$(\widehat{q},v)\cdot [\alpha]=(\widehat{q}\cdot[\alpha],\rho(\alpha^{-1})\cdot v),$$ where $\widehat{q}\cdot[\alpha]$ denotes the deck transformation associated to $[\alpha]$ applied to a point $\widehat{q}\in\widehat{Q}$. We denote by $M$ the set of orbits of this action, and by $\Pi:\widetilde{M}\rightarrow M$ the canonical projection. Moreover, it is possible to prove that $M$ is a manifold. Indeed, given a simple open neighborhood $U_{j}\subset Q$, consider the following bijection \begin{eqnarray*}
\Psi_{j}: \Pi(\widehat{P}^{-1}(U_{j})\times V)& \longrightarrow & U_{j}\times V \\ \Pi(\widehat{q},v) & \longmapsto & (\widehat{P}(\widehat{q})\times v).\end{eqnarray*} If $U_{i}\cap U_{j} \neq\emptyset$ and is connected, we can see that $$\Psi_{i}\cap\Psi_{j}^{-1}(q,v)=(q,\rho([\alpha]^{-1})v)$$ for a fixed $[\alpha]$. Hence there exists a unique manifold structure on $M$ for which $\Psi_{j}$ are local diffeomorphisms.

Define a map $P$ by \begin{eqnarray*} P: M &\longrightarrow& Q\\ \Pi(\widehat{q},v)& \longmapsto & \widehat{P}(\widehat{q}). \end{eqnarray*} It follows that $M$ is the total space of a fiber bundle, and $P$ is its projection over the base space $Q$. In addition, the fiber is $V$ and the structural group is given by the image of $\rho.$

Finally, define $\F=\{\Pi(\widehat{Q},v)\}$, i.e., the projection of the trivial foliation defined as the product of $\widehat{Q}$ with each $v$. It is possible to prove that this is a foliation transverse to the fibers. Furthermore, this foliation is a Riemannian foliation whose transverse metric coincides with the metric of $V$.

\begin{example}\label{example-MolinoConjectureSRFS}\index{Action!polar}
In what follows, we construct a s.r.f.s. such that the intersection of a local section with the closure of a regular leaf is an orbit of an action of a subgroup of isometries of the local section. This isometric action is not a polar action. Therefore, \emph{there exists a s.r.f.s. $\F$ such that the partition formed by the closure of leaves is a singular Riemannian foliation without sections.}

Let $V$ denote the product $\R^{2}\times\C\times\C$ and $\widehat{\F}_{0}$ the singular foliation of codimension $5$ on $V$, whose leaves are the product of points in $\C\times\C$ with circles in $\R^{2}$ centered in the origin. It is easy to see that the foliation $\widehat{\F}_{0}$ is a s.r.f.s. Let $Q$ be the circle $S^{1}$ and $k$ an irrational number. Define the homomorphism \begin{eqnarray*}
\rho:\pi_{1}(Q,q_{0}) &\longrightarrow &\Iso(V) \\
n &\longmapsto& \big((x,z_{1},z_{2})\mapsto(x,e^{ink}\cdot z_{1},e^{ink}\cdot z_{2})\big).\\
\end{eqnarray*}

Finally, set $\F=\Pi(\widehat{Q}\times\widehat{\F_{0}})$. It turns out that $\F$ is a s.r.f.s. and the intersection of the section $\Pi(0\times\R\times\C\times\C)$
with the closure of a regular leaf is an orbit of an isometric action on the section. This isometric action is not a polar action, since the following is not polar
\begin{eqnarray*}
S^{1}\times\C\times\C&\longrightarrow & \C\times\C\\
(s,z_{1},z_{2}) &\longmapsto &(s\cdot z_{1},s\cdot z_{2}).\\
\end{eqnarray*}
\end{example}

\section{Results on s.r.f.s.}

In this section, we give a few results of the theory of s.r.f.s., without proofs. First, Palais and Terng~\cite[Remark 5.6.8]{PalaisTerng} proposed a conjecture that can be formulated as follows.

\begin{ptconjecture}
Let $G$ be an isometric action of a compact Lie group on $M$, such that the distribution of the normal space to the regular orbit is integrable. Then there exists a complete totally geodesic immersed section of the action of $G$ that intersects all orbits perpendicularly.
\end{ptconjecture}

Heintze, Olmos and Liu~\cite{HOL} proved that Palais and Terng were right in their conjecture for isometric actions. In particular, they proved that the set of regular points is dense in each section. A result in Alexandrino~\cite{Alex4} gives a positive answer to this conjecture in the case of singular Riemannian foliations.

\begin{theorem}\label{theorem-PT-conjecture}
Let $\F$ be a singular Riemannian foliation on a complete Riemannian manifold $M$. Suppose that the distribution of normal spaces to the regular leaves is integrable. Then $\F$ is a s.r.f.s. In addition, the set of regular points is open and dense in each section.
\end{theorem}

\begin{remark}
Particular cases of the above result where proved by Molino and Pierrot~\cite{MolinoPierrot}, Boualem~\cite{Boualem} and Lytchak and Thorbergsson~\cite{LytchakThorbergsson}. Szenthe~\cite{Szenthe} also worked on conjecture for isometric actions.
\end{remark}

Henceforth, $\F$ denotes a s.r.f.s. on a complete Riemannian manifold $M$. The next result of Alexandrino~\cite{Alex2} relates s.r.f.s. to equifocal submanifolds. Recall Definition~\ref{definition-equifocal} for definitions of equifocal submanifolds and endpoint map $\eta_{\xi}$.

\begin{theorem}\label{frss-eh-equifocal}
Let $L$ be a regular leaf of a s.r.f.s. $\F$ of a complete Riemannian manifold $M$.
\begin{enumerate}
\item[(i)] $L$ is equifocal. In particular, the union of the regular leaves that have trivial normal holonomy is an open and dense set in $M$, provided that all the leaves are compact;
\item[(ii)] Let $\beta$ be a smooth curve of $L$ and $\xi$ a parallel normal field to $L$ along $\beta$. Then the curve $\eta_{\xi}\circ \beta$ is in a leaf of $\F$;
\item[(iii)] Suppose that $L$ has trivial holonomy and let $\Xi$ denote the set of all parallel normal fields on $L$. Then $\F=\{\eta_{\xi}(L)\}_{\xi\in \, \Xi}.$
\end{enumerate}
\end{theorem}

Note that this result guarantees that, given a regular leaf $L$ with trivial holonomy, it is possible to reconstruct $\F$ taking parallel submanifolds to $L$.

In order to state the next theorem, we recall the concepts of slice and local section. Let $q\in M$, and $\tub(P_{q})$ be a tubular neighborhood of a plaque $P_{q}$ that contains $q$. The connected component of $\exp_{q}(\nu P_{q})\cap \tub(P_{q})$ that contains $q$ is called a \emph{slice}\index{Slice}\index{$S_x$} at $q$ and is usually denoted $S_{p}.$ A \emph{local section}\index{Local section} $\sigma$ (centered at $q$) of a section $\Sigma$ is a connected component $\tub(P_{q})\cap\Sigma$ (that contains $q$).

Let us recall some results about the local structure of $\F$, in particular about the structure of the set of singular points in a local section. The next result of Alexandrino~\cite{Alex2} is a generalization of Theorem~\ref{theorem-slice-acaoconjugacao}.

\begin{slthm}\label{slicetheoremrema}
Let $\F$ be a s.r.f.s. on a complete Riemannian manifold $M$, and consider $q$ be a singular point of $M$ and $S_{q}$ a slice at $q$. Then
\begin{enumerate}
\item[(i)] Denote $\Lambda(q)$ the set of local sections $\sigma$ centered at $q$. Then $S_{q}=\bigcup_{\sigma\in\Lambda (q)}\sigma$;
\item[(ii)] $S_{x}\subset S_{q}$ for all $x\in S_{q}$;
\item[(iii)] $\F|S_q$ is a s.r.f.s. on $S_{q}$ with the induced metric of $M$;
\item[(iv)] $\F|S_q$ is diffeomorphic to an isoparametric foliation on an open subset of $\R^{n}$, where $n$ is the dimension of $S_{q}$.
\end{enumerate}
\end{slthm}

\begin{remark}
For the proof of this Slice theorem, see also Pi\~{n}eros~\cite{DiegoMestrado}.
\end{remark}

From (iv) of Theorem~\ref{slicetheoremrema} it is not difficult to derive the following corollary, in Alexandrino~\cite{Alex2}.

\begin{corollary}\label{estratificacao-singular}
Let $\sigma$ be a local section. The set of singular points of $\F$ contained in $\sigma$ is a finite union of totally geodesic hypersurfaces. These hypersurfaces are diffeomorphic to focal hyperplanes contained in a section of an isoparametric foliation on an open set of Euclidean space.
\end{corollary}

We will call the set of singular points of $\F$ contained in $\sigma$ the \emph{singular stratification of the local section}\index{Singular stratification of local section} $\sigma$. Let $M_{r}$ denote the set of regular points in $M$. A \emph{Weyl chamber}\index{Weyl chamber} of a local section $\sigma$ is the closure in $\sigma$ of a connected component of $M_{r}\cap\sigma$ (compare with Definition~\ref{definitionWeylchamber}). It is possible to prove that a Weyl chamber of a local section is a convex set.

Theorem~\ref{frss-eh-equifocal} allows to define the \emph{singular holonomy map}, which will be very useful to study $\F$. This result is also in Alexandrino~\cite{Alex2}.

\begin{proposition}\label{proposition-holonomia-singular}
Let $\F$ be a s.r.f.s. on a complete Riemannian manifold $M$, and $q_0,q_1$ two points contained in a leaf $L_{q}$. Let $\beta:[0,1]\rightarrow L_{p}$ be a smooth curve contained in a regular leaf $L_{p}$, such that $\beta(i)\in S_{q_{i}},$ where $S_{q_{i}}$ is the slice at $q_{i}$ for $i=0,1$. Let $\sigma_{i}$ be a local section contained in $S_{q_{i}}$ that contains $\beta(i)$ and $q_{i}$ for $i=0,1$. Finally, let $[\beta]$ denote the homotopy class of $\beta$. Then there exists an isometry $\varphi_{[\beta]}:U_{0} \rightarrow U_{1}$, where the source $U_{0}$ and target $U_{1}$ are contained in $\sigma_{0}$ and $\sigma_{1}$ respectively,  which has the following properties:
\begin{enumerate}
\item[(i)] $q_{0}\in U_{0}$;
\item[(ii)] $\varphi_{[\beta]}(x)\in L_{x}$ for each $x\in U_{0}$;
\item[(iii)] $\dd\varphi_{[\beta]}\xi(0)=\xi(1),$ where $\xi(s)$ is a parallel normal field along $\beta(s)$.
\end{enumerate}
\end{proposition}

Such an isometry is called the \emph{singular holonomy map along $\beta$}.\index{Holonomy!singular map} We remark that, in the definition of the singular holonomy map, singular points can be contained in the domain $U_0.$ If the domain $U_0$ and the range $U_1$ are sufficiently small and $L_{q}$ is regular, then the singular holonomy map coincides with the usual holonomy map along $\beta$.

Theorem~\ref{slicetheoremrema} establishes a relation between s.r.f.s. and isoparametric foliations. Similarly, as in the usual theory of isoparametric submanifolds, it is
natural to ask if we can define a (generalized) Weyl group action on $\sigma$. The following definitions and results deal with this question.

\begin{definition}\label{definitionWeylPseudogroup}
The pseudosubgroup generated by all singular holonomy maps $\varphi_{[\beta]}$ such that $\beta(0)$ and $\beta(1)$ belong to the same local section $\sigma$ is called the \emph{generalized Weyl pseudogroup}\index{Weyl pseudogroup} of $\sigma$. Let $W_{\sigma}$ denote this pseudogroup. Analogously, define $W_{\Sigma}$ for a section $\Sigma$. Given a slice $S$, define $W_{S}$ as the set of all singular holonomy maps $\varphi_{[\beta]}$ such that $\beta$ is contained in $S$.
\end{definition}

\begin{proposition}\label{propositionWisinvariant}
Let $\sigma$ be a local section. The reflections across the hypersurfaces of the singular stratification of the local section $\sigma$ leave $\F|\sigma$ invariant. Moreover, these reflections are elements of $W_{\sigma}$.
\end{proposition}

\begin{remark}
Compare the above proposition, in Alexandrino~\cite{Alex2}, with (i) of Theorem~\ref{proposition-W-generatedby-reflections} and Remark~\ref{remark-proposition-W-generatedby-reflections}.
\end{remark}

Using the technique of suspensions, one can construct an example of a s.r.f.s. such that $W_{\sigma}$ is larger than the pseudogroup generated by the reflections across the hypersurfaces of the singular stratification of $\sigma$. On the other hand, a sufficient condition to ensure that both pseudogroups coincide is that the leaves of $\F$ have trivial normal holonomy and are compact. So it is natural to ask under which conditions we can guarantee that the normal holonomy of regular leaves is trivial. The next result, in Alexandrino and T\"oben~\cite{AlexToeben} and Alexandrino~\cite{Alex5} deals with this question.

\begin{theorem}\label{theorem-fundamental-domain}
Let $\F$ be a s.r.f.s. on a simply connected Riemannian manifold $M$. Suppose also that the leaves of $\F$ are closed and embedded. Then
\begin{enumerate}
\item[(i)] Each regular leaf has trivial holonomy;
\item[(ii)] $M/\F$ is a simply connected Coxeter orbifold;
\item[(iii)] Let $\Sigma$ be a section of $\F$ and $\Pi:M\rightarrow M/\F$ the canonical projection. Denote by $\Omega$ a connected component of the set of regular points in $\Sigma$. Then $\Pi:\Omega\rightarrow M_{r}/\F$ and $\Pi:{\overline{\Omega}}\to M/\F$ are homeomorphisms, where $M_r$ denotes the set of regular points in $M$. In addition, $\Omega$ is convex, i.e., for any two points $p$ and $q$ in $\Omega$, every minimal geodesic segment between $p$ and $q$ lies entirely in $\Omega$.
\end{enumerate}
\end{theorem}

\begin{remark}
Compare the above result with Theorems~\ref{theorem-OrbitaPrincipal=OrbitaRegular} and~\ref{proposition-W-generatedby-reflections}.
\end{remark} 

To conclude this section, we discuss a result that describes the behavior of a s.r.f.s. whose leaves are not embedded. Molino~\cite{Molino} proved that, if $M$ is compact, the closure of the leaves of a (regular) Riemannian foliation forms a partition of $M$ that is a singular Riemannian foliation. He also proved that the leaf closure are orbits of a locally constant sheaf of germs of (transversal) Killing fields. Furthermore, a conjecture is left in the same book.

\begin{mconj}
Let $\F$ be a singular Riemannian foliation on a compact Riemannian manifold $M$. Then the closure of the leaves of $\F$ forms a partition of $M$ that is also a singular Riemannian foliation.
\end{mconj}

Molino~\cite[Theorem 6.2]{Molino} was able to prove that the closure of the leaves is a \emph{transnormal system}, however it still remains an open problem to prove that $\F$ is a singular foliation. The next result, in Alexandrino~\cite{Alex4}, gives a positive answer to Molino's conjecture if $\F$ is a s.r.f.s.

\begin{theorem}\label{theorem-Molino-conjecture}
Let $\F$ be a s.r.f.s. on a complete Riemannian manifold $M$.
\begin{enumerate}
\item[(i)] The closure $\{\overline{L}\}_{L\in\F}$ of the leaves of $\F$ is a partition of $M$ which is a singular Riemannian foliation;
\item[(ii)] Each point $q$ is contained in an homogeneous submanifold $\TO_{q}$ (possibly $0$--dimensional). If we fix a local section $\sigma$ that contains $q$, then $\TO_{q}$ is a connected component of an orbit of the closure of the Weyl pseudogroup of $\sigma$;
\item[(iii)] If $q$ is a point of the submanifold $\overline{L},$ then a neighborhood of $q$ in $\overline{L}$ is the product of the homogeneous submanifold $\TO_{q}$ with plaques with the same dimension of the plaque $P_{q}$;
\item[(iv)] Let $q$ be a singular point and $T$ the intersection of the slice $S_{q}$ with the singular stratum that contains $q$. Then the normal connection of $T$ in $S_{q}$ is flat;
\item[(v)] Let $q$ be a singular point and $T$ as in \emph{(iv)}. Consider $v$ a parallel normal vector field along $T$, $x\in T$ and $y=\exp_{x}(v)$. Then $\TO_{y}=\eta_{v}(\TO_{x})$.
\end{enumerate}
\end{theorem}

The above theorem is illustrated in Example~\ref{example-MolinoConjectureSRFS}.

\section{Transnormal maps}

In the last section, s.r.f.s. was presented as the natural candidate to generalize isoparametric foliations on Euclidean spaces and, in particular, orbits of adjoint actions. Nevertheless, there exists another possible way to try to generalize them. We can consider preimages of special maps, so--called \emph{isoparametric} maps. In fact, Cartan~\cite{Cartan1,Cartan2,Cartan3,Cartan4}, Harle~\cite{Harle}, Terng~\cite{Terng} and Wang~\cite{Wang1} have chosen this approach.

We start by recalling the definition of isoparametric and transnormal map.

\begin{definition}
Let $M^{n+q}$ be a complete Riemannian manifold. A smooth map $H=(h_{1},\dots,h_{q}):M^{n+q}\rightarrow \R^{q}$ is called \index{Transnormal map}\emph{transnormal} if
\begin{enumerate}
\item[(i)] $H$ has a regular value;
\item[(ii)] For each regular value $c$, there exist a neighborhood $V$ of $H^{-1}(c)$ in $M$ and smooth functions $b_{ij}$ on $H(V)$ such that, for every $x\in V,$ $$\langle \grad h_{i}(x),\grad h_{j}(x) \rangle = b_{ij}\circ H(x);$$
\item[(ii)] There exists a sufficiently small neighborhood of each regular level set such that $[\grad h_{i},\grad h_{j}]$ is a linear combination of $\grad h_{1},\dots,\grad h_{q}$, with coefficients being functions of $H$, for all $i$ and $j$.
\end{enumerate}
In particular, a transnormal map $H$ is said to be \emph{isoparametric}\index{Isoparametric!map} if $V$ can be chosen to be $M$ and $\Delta h_{i}= a_{i}\circ H,$ where $a_{i}$ is a smooth function.
\end{definition}

\begin{remark}
This definition is equivalent to saying that $H$ has a regular value, and for each regular value $c$ there exists a neighborhood $V$ of $ H^{-1}(c)$ in $M$ such that $H|_V\rightarrow H(V)$ is an integrable Riemannian submersion, where the Riemannian metric $g_{ij}$ of $H(V)$ is the inverse matrix of $b_{ij}$.
\end{remark}

It is known (see Palais and Terng~\cite{PalaisTerng}) that given an isoparametric map $H$ on a space form and a regular value $c$, the regular level set $H^{-1}(c)$ is an isoparametric submanifold. The next theorem of Alexandrino~\cite{Alex1} is a generalization of this result.

\begin{theorem}\label{TransormalMaptheorem}
Let $H:M\rightarrow \R^{q}$ be an analytic transnormal map on a real analytic complete Riemannian manifold $M$. Let $c$ be a regular value and $L\subset H^{-1}(c)$ be a connected component of $H^{-1}(c)$. Denote $\Xi$ the set of all parallel normal fields along $L$.
\begin{enumerate}
\item[(i)] $\F_{c,L}=\{\eta_{\xi}(L)\}_{\xi\in\Xi}$ is a s.r.f.s. whose leaves are embedded;
\item[(ii)] For each regular value $\widehat{c}$, the connected components of $H^{-1}(\widehat{c})$ are equifocal manifolds and leaves of $\F_{c,L}$;
\item[(iii)] $\F_{c,L}$ is independent of the choice of $c$ and $L$, i.e., for another regular value $\widetilde{c}$ and connected component $\widetilde{L}\subset H^{-1}(\widetilde{c})$, $\F_{c,L}=\F_{\widetilde{c},\widetilde{L}}$.
\end{enumerate}
\end{theorem}

\begin{remark}
Wang~\cite{Wang1} studied smooth transnormal functions, i.e., a smooth function $f:M\rightarrow \R$ satisfying the equation $\|\grad f\|^{2} =b\circ f$, where $b\in C^{2}(f(M))$. He proved that level sets of $f$ are leaves of a singular Riemannian foliation. It also follows from his proof that the regular level sets are equifocal.
\end{remark}

It is also known (see Carter and West~\cite{CarterWest2} and Terng~\cite{Terng}) that given an isoparametric submanifold $L$ of $\R^{n+k}$, one can construct a polynomial isoparametric map from $\R^{n+k}$ into $\R^{k}$ that has $L$ as a level submanifold. Therefore, it is natural to look for conditions under which a s.r.f.s. can be described as preimages of a transnormal map. It follows from Theorem~\ref{slicetheoremrema} and from the classic theory of isoparametric foliations that this is always locally true, as stated in the following result of Alexandrino~\cite{Alex2}.

\begin{proposition}
Let $\F$ be a s.r.f.s. on a complete Riemannian manifold $M$. Then, for each $p\in M$ there exists a neighborhood $U$ of $p$ such that the plaques of $\F\cap U$ are preimages of a transnormal map.
\end{proposition}

Since there are examples of s.r.f.s. with non trivial holonomy (see Alexandrino~\cite{Alex2}), there are examples of s.r.f.s. that are not given as preimages of a transnormal map. The next result, in Alexandrino and Gorodski~\cite{AlexGorodski} and Alexandrino~\cite{Alex5}, gives sufficient conditions under which a s.r.f.s. can be described by a transnormal map, and hence can be considered a converse of Theorem~\ref{TransormalMaptheorem}.

\begin{theorem}\label{theorem-s.r.f.s-transnormal}
Let $\mathcal F$ be a s.r.f.s. on a complete simply connected Riemannian manifold $M$. Assume that leaves of $\mathcal F$ are closed and embedded and $\mathcal F$ admits a flat section of dimension $n$. Then the leaves of $\mathcal F$ are given by the level sets of a transnormal map $H:M\to\R^n$.
\end{theorem}

\begin{remark}
Heintze, Liu and Olmos~\cite{HOL} proved the above result with the additional assumption that $M$ is a simply connected symmetric space of compact type.
\end{remark}

\section{Recent contributions}

In this last section, we would like to cite some recent works on the area of singular Riemannian foliations.

First, the use of a blow up technique by T\"{o}ben~\cite{Toeben} to study equifocal submanifolds (which he calls submanifolds \emph{with parallel focal structure}). A necessary and sufficient condition is given for a closed embedded local equifocal submanifold to induce a s.r.f.s. (see Alexandrino~\cite{Alex3} for an alternative proof). T\"{o}ben~\cite{Toeben2} also proved the non existence of proper s.r.f.s. in compact manifolds of non positive curvature, and gave a global description of proper s.r.f.s. on Hadamard manifolds.

Another interesting contribution was given by Lytchak and Thorbergsson~\cite{LytchakThorbergsson}, that introduced the concept of \emph{singular Riemannian foliation without horizontal conjugate points} in order to generalize variationally complete isometric actions. They proved that a singular Riemannian foliation without horizontal conjugate points in a complete Riemannian manifold with non negative sectional curvature admits flat sections.

T\"{o}ben and Alexandrino~\cite{AlexToeben2} proved that the regular leaves of a singular Riemannian foliation are equifocal, i.e., the endpoint map of a normal foliated vector field has constant rank. This implies that it is possible to reconstruct the singular foliation by taking all parallel submanifolds of a regular leaf with trivial holonomy. In addition, the endpoint map of a normal foliated vector field on a leaf with trivial holonomy is a covering map.

Recently, Lytchak and Thorbergsson~\cite{LytchakThorbergsson2} proved that the quotient space of a variationally complete group action is a good Riemannian orbifold. This result is also proved to hold for singular Riemannian foliations without horizontal conjugate points. In the same paper, the concept of \emph{infinitesimally polar foliation} is introduced. A s.r.f. $\F$ is called infinitesimally polar if the intersection of $\F$ with each slice is diffeomorphic to a polar foliation on an Euclidean space, i.e., an isoparametric foliation. In fact, this diffeomorphism is given as composition of the exponential map with a linear map.

They proved that $\F$ is an infinitesimally polar foliation if and only if 
for each point $x\in M$ there exists a neighborhood $U$ of $x$ such that the leave space of the restricted foliation $\F|_{U}$ is 
an orbifold. Since the leave space of an isoparametric foliation is a Coxeter orbifold, 
they concluded that the leave space $U/\F|_{U}$ is an orbifold if and only if it is a Coxeter orbifold.

Typical examples of infinitesimally polar foliations are polar foliations (see~\cite{Alex2}), 
singular Riemannian foliations without horizontal conjugate points (see~\cite{LytchakThorbergsson,LytchakThorbergsson2}) 
and s.r.f. with codimension one or two (see~\cite{LytchakThorbergsson2}). 
In particular, a partition by orbits of an isometric action is infinitesimally polar if either the action is polar, 
or variationally complete, or has cohomogeneity lower then three.

Lytchak~\cite{Lytchak1} generalized a previous result due to T\"{o}ben~\cite{Toeben2}, proving that singular Riemannian foliations do not exist on compact negatively curved spaces. Lytchak~\cite{Lytchak2} also generalized the blow up introduced by T\"{o}ben~\cite{Toeben}, and proved that a singular Riemannian foliation admits a resolution preserving the transverse geometry if and only if it is infinitesimally polar. He also deduced that singular Riemannian foliations on simply connected manifolds that either have sections or no horizontal conjugate points are closed.

We would also like to cite the PhD thesis of Boltner~\cite{Boltner}, where singular Riemannian foliations in Euclidean space are studied. Furthermore, we cite the PhD thesis of Nowak~\cite{Eva}. Among other things, singular Riemannian foliations without horizontal conjugate points are studied, and it is proved that the regular leaves are taut if properly embedded, and if the ambient space is simply connected.

In conclusion, we mention the recent work of Alexandrino~\cite{Alex6}, where it is proved that each singular Riemannian foliation on a compact manifold can be desingularized after successive blow ups. This result generalizes a previous result of Molino. It is also proved in~\cite{Alex6} that, if the leaves of the singular Riemannian foliation are compact, then, for each small $\varepsilon$, the regular foliation can be chosen so that the desingularization map induces an $\varepsilon$--isometry between the leaf space of the regular Riemannian foliation and the leaf space of the singular Riemannian foliation. This implies, in particular, that the leaf space of the singular Riemannian foliation is is a Gromov--Hausdorff limit of a sequence of Riemannian orbifolds.


\appendix
\chapter{Rudiments of smooth manifolds}\label{appendix}

In this appendix, we summarize basic definitions and results of smooth manifolds that will be used throughout this text. A more detailed treatment of this subject can be found in Warner \cite{warner}, Spivak \cite{spivak1} and Lang \cite{langRiemannianManifolds}.

\section{Smooth manifolds}

In general, by smooth or differentiable we mean of class $C^\infty$, unless otherwise specified. In addition, a map will be called a {\it function} if its counter domain is $\R$.

\begin{definition}
A {\it smooth $m$--dimensional manifold}\index{Manifold} $M$ is a second countable Hausdorff topological space endowed with a {\it smooth structure}. More precisely, there is a collection of pairs $(U_\alpha,\varphi_\alpha)$ called {\it charts},\index{Manifold!chart} such that $U_\alpha$ are open subsets of $M$, and, for each $\alpha$, the map $\varphi_\alpha:U_\alpha\subset M\to \varphi_\alpha(U_\alpha)\subset\R^m$ is a homeomorphism between open sets satisfying
\begin{itemize}
\item[(i)] $M=\bigcup_{\alpha} U_\alpha$;
\item[(ii)] If $(U_\alpha,\varphi_\alpha)$ and $(U_\beta,\varphi_\beta)$ are charts, then the {\it transition map} $$\varphi_\beta \circ\varphi_\alpha^{-1}:\varphi_\alpha(U_\alpha \cap U_\beta)\longrightarrow\varphi_\beta(U_\alpha \cap U_\beta)$$ is smooth. In this case, the charts $\varphi_\alpha$ and $\varphi_\beta$ are said to be $C^\infty$--compatible;
\item[(iii)] The collection $\{(U_\alpha,\varphi_\alpha)\}_\alpha$ of charts is maximal with respect to (i) and (ii).
\end{itemize}

Such collection of charts is called an {\it atlas}\index{Manifold!atlas} of $M$. If $(U_\alpha,\varphi_\alpha)$ is a chart and $p\in U_\alpha$, the open set $U_\alpha$ is called a {\it coordinate neighborhood} of $p$ and $\varphi_\alpha$ is called a {\it local coordinate} of $p$.
\end{definition}

Similarly, it is possible to define {\it $C^k$--manifolds}\index{Manifold!$C^k$} and {\it analytic manifolds}\index{Manifold!analytic} ($C^\omega$--manifolds), by requiring the transition maps to be respectively $C^k$ and analytic. By {\it analytic structure}\index{Manifold!analytic structure} we mean a collection of coordinate systems which overlap analytically, that is, are locally represented by convergent power series.

In this text we are only interested in $C^\infty$--manifolds, therefore a {\it manifold} will always be considered to be a smooth manifold in the sense above.

\begin{definition}
Let $M$ and $N$ be manifolds with atlases $\{(U_\alpha,\varphi_\alpha)\}_\alpha$ and $\{(V_\beta,\psi_\beta)\}_\beta$, respectively. A continuous map $f:M\rightarrow N$ is said to be {\it smooth on $p\in M$} if there exist coordinate neighborhoods $U_\alpha$ of $p$ and $V_\beta$ of $f(p)$, such that its {\it representation} $\psi_\beta\circ f\circ \varphi_\alpha^{-1}$ is smooth on $\varphi_\alpha(p)$.
\end{definition}

It is easy to see that if this condition is satisfied by a pair of charts $(U_\alpha,\varphi_\alpha)$ and $(V_\beta,\psi_\beta)$, then it holds for any charts. Moreover, the above definition clearly implies that a function $g:M\rightarrow\R$ is {\it smooth on $p\in M$} if there exists a chart $(U_\alpha,\varphi_\alpha)$ such that $g\circ\varphi_\alpha^{-1}$ is smooth on $g(p)$. Finally, maps (and functions) are called {\it smooth} if they are smooth on every point $p\in M$.

Let $M$ be a smooth manifold and $p\in M$. The set $C^\infty(M)$ of smooth functions $f:M\rightarrow\R$ is an algebra with the usual operations. Consider the sub algebra $C^\infty(p)$ of smooth functions whose domain of definition includes some open neighborhood of $p$. We then define the {\it tangent space}\index{Manifold!tangent space} to $M$ at $p$ as the vector space $T_pM$ of linear derivations at $p$, i.e., the set of maps $v:C^\infty(p)\rightarrow\R$ satisfying for all $f,g\in C^\infty(p)$,
\begin{itemize}
\item[(i)] $v(f+g)=v(f)+v(g)$ (linearity);
\item[(ii)] $v(fg)=v(f)g(p)+f(p)v(g)$ (Leibniz rule).
\end{itemize}
The vector space operations are defined for all $v,w \in T_pM, \alpha\in\R$ and $f\in C^\infty(p)$ by
\begin{eqnarray*}
(v+w)f &= & v(f)+w(f);  \\
(\alpha v)(f) &= & \alpha v(f).
\end{eqnarray*}

Let $(U,\varphi=(x_1,\dots,x_n))$ be a chart around $p$ and $\tilde{f}= f\circ\varphi^{-1}$ the representation of $f\in C^\infty(p)$ in coordinates given by $\varphi$. Then the \emph{coordinate vectors} given by $$\left(\frac{\partial}{\partial x_i}\Big|_p\right)f=\frac{\partial \tilde{f}}{\partial x_i}\Big|_{\varphi(p)}$$ form a basis $\{\frac{\partial}{\partial x_i}\big|_p\}$ of $T_pM$. In particular, it follows that $\dim T_pM=\dim M$. In this context, tangent vectors can be explicitly seen as directional derivatives. Indeed, consider $v=\sum_{i=1}^n v_i\frac{\partial}{\partial x_i}\in T_pM$. Then $$v(f)=\sum_{i=1}^n v_i\frac{\partial \tilde{f}}{\partial x_i}\Big|_{\varphi(p)}$$ is called the {\it directional derivative of $f$ in the direction $v$}. It can also be proved that this definition is independent of the chart $\varphi$.

The \emph{tangent vector} to a curve $\alpha:(-\varepsilon,\varepsilon)\rightarrow M$ at $\alpha(0)$ is defined as $\alpha'(0)(f) =\frac{\mathrm d}{\mathrm dt}(f\circ\alpha)\big|_{t=0}$. Indeed, $\alpha'(0)$ is a vector and if $\varphi$ is a local coordinate and $(u_1(t),\dots,u_n(t))=\varphi\circ\alpha$, then $$\alpha'(0)= \sum_{i=1}^n u_i'(0) \frac{\partial}{\partial x_i}\in T_pM.$$

\begin{definition}
Let $f:M\rightarrow N$ be a smooth map and let $p\in M$. The {\it differential} of $f$ at $p$ is the linear map $\dd f_p:T_pM\rightarrow T_{f(p)}N$, such that if $v\in T_pM$, the image of $v$ is a tangent vector at $f(p)$ satisfying $\dd f_p(v)g= v(g\circ f)$, for all $g\in C^\infty(f(p))$.
\end{definition}

Hence the chain rule is automatically valid on manifolds, since the differential of a map is defined to satisfy it. Equipped with the notion of differential, we can define important classes of maps between manifolds. A smooth map $f:M\rightarrow N$ is said to be an {\it immersion} if $\dd f_p$ is injective for all $p\in M$, and it is said to be a {\it submersion} if $\dd f_p$ is surjective for all $p\in M$. An immersion $f:M\rightarrow N$ is called an {\it embedding} if $f:M\rightarrow f(M)\subset N$ is a homeomorphism considering $f(M)$ endowed with the relative topology.

In addition, a smooth bijection $f$ with smooth inverse is called a {\it diffeomorphism}. From the chain rule, its differential on every point is an isomorphism, and the {\it Inverse Function Theorem}\index{Theorem!Inverse Function} gives a local converse. More precisely, this theorem states that if the differential $\dd f_p$ of a smooth map $f$ is an isomorphism, then $f$ restricted to some open neighborhood of $p$ is a diffeomorphism.

\begin{definition}
Let $P$ and $N$ be manifolds, with $P\subset N$. Then $P$ is called an \emph{immersed submanifold}\index{Submanifold}\index{Submanifold!embedded}\index{Submanifold!immersed} of $N$ if $i:P\hookrightarrow N$ is an immersion. In addition, if $i:P\hookrightarrow N$ is an embedding, then $P$ is called an \emph{(embedded) submanifold}.
\end{definition}

We also need the concept of quasi--embedded manifolds, which are more than immersed but not necessarily embedded submanifolds.

\begin{definition}
An immersed submanifold $P\subset N$ is {\it quasi--embedded}\index{Submanifold!quasi--embedded} if it satisfies the following property. If $f:M\rightarrow N$ is a smooth map with image $f(M)$ lying in $P$, then the induced map $f_0:M\rightarrow P$ defined by $i\circ f_0=f$ is smooth, where $i:P\hookrightarrow N$ is the inclusion.
\end{definition}

\begin{proposition}\label{quasiembedded}
Let $P\subset N$ be an embedded submanifold of $N$. Then $P$ is quasi--embedded. 
\end{proposition}

\section{Vector fields}
\label{sec:vecfields}

Let $M$ be a fixed manifold and $TM$ the \emph{(total space of its) tangent bundle}, i.e., $TM=\bigcup_{p\in M} \{p\}\times T_{p} M$. Denote $\pi:TM\rightarrow M$ the \emph{footpoint projection} (or canonical projection), i.e., if $V_{p}\in T_{p}M$ then $\pi(V_{p})=p$. It is possible to prove that $TM$ carries a canonical smooth structure inherited from $M$, such that $\pi$ is smooth. In fact, as explained in Section \ref{Section-properaction-fiberbundle}, $(TM,\pi,M,\R^n)$ is a fiber bundle.

\begin{definition}
A \emph{(smooth) vector field}\index{Vector field} $X$ on $M$ is a (smooth) section of $TM$, i.e., a (smooth) map $X:M \rightarrow TM$ such that $ \pi \circ X = \id$.
\end{definition}

Let $p\in M$ and consider a coordinate neighborhood $U$ of $p$ with local coordinates $\varphi=(x_1,\dots,x_n)$. Given $f\in C^\infty(U)$, the {\it directional derivative} $X(f):U\rightarrow\R$ is defined to be the function $$U\owns p\longmapsto X_p(f)\in\R.$$ Furthermore, being $\left\{\frac{\partial}{\partial x_i}\big|_p\right\}$ a basis of $T_pM$, we may write \begin{equation}\label{localrep}
X|_U=\sum_{i=1}^n a_i\frac{\partial}{\partial x_i},
\end{equation}
where $a_i:U\rightarrow\R$.

\begin{proposition}\label{smoothvecfields}
Consider $X$ a vector field on $M$, not necessarily smooth. Then the following are equivalent.
\begin{itemize}
\item[(i)] $X$ is smooth;
\item[(ii)] For every chart $(U,(x_1,\dots,x_n))$, the functions $a_i$ in (\ref{localrep}) are smooth;
\item[(iii)] For every open set $V$ of $M$ and $f\in C^\infty(V)$, $X(f)\in C^\infty(V)$.
\end{itemize}
\end{proposition}

The set of smooth vector fields on $M$ will be denoted $\mathfrak{X}(M)$,\index{$\mathfrak{X}(M)$} and is clearly a $C^\infty(M)$--module with operations defined pointwisely. If $f:M\rightarrow N$ is a smooth map between two manifolds and $X\in\mathfrak{X}(M)$, $Y\in\mathfrak{X}(N)$, then $X$ and $Y$ are said to be {\it $f$--related}\index{Vector field!$f$--related} if $\dd f\circ X=Y\circ f$, i.e., if the following diagram is commutative. \begin{displaymath}
\xymatrix@+20pt{
TM \ar[r]^{\dd f} & TN \\
M \ar[u]^{X}\ar[r]_f & N \ar[u]_Y
}
\end{displaymath}
Note that $X$ and $Y$ are $f$--related if, and only if, for all $g\in C^\infty(N)$, $(Yg)\circ f=X(g\circ f)$.

\begin{definition}
If $X,Y\in\mathfrak{X}(M)$, the {\it Lie bracket of vector fields}\index{Vector field!Lie bracket} $X$ and $Y$ is the vector field $[X,Y]\in\mathfrak{X}(M)$ given by $$[X,Y]f=X(Y(f))-Y(X(f)), \quad f\in C^\infty(M).$$
\end{definition}

Note that the Lie bracket of vector fields is clearly skew--symmetric and satisfies the Jacobi identity. When $X,Y\in\mathfrak{X}(M)$ are such that $[X,Y]=0$, then $X$ and $Y$ are said to {\it commute}.

In addition, if $f:M\rightarrow N$ is a smooth map such that $X^1,X^2\in\mathfrak{X}(M)$ and $Y^1,Y^2\in\mathfrak{X}(N)$ are $f$--related, then $[X^1,X^2]$ and $[Y^1,Y^2]$ are also $f$--related. Indeed, using the observation above, since $X^i$ is $f$--related to $Y^i$, it follows that $(Y^i g)\circ f=X^i(g\circ f)$. The conclusion is immediate from the following equations, using again the same observation.
\begin{equation}\label{bracketrelated}
\begin{aligned}
([Y^1,Y^2]g)\circ f &=(Y^1(Y^2 g))\circ f - (Y^2(Y^1g))\circ f\\
&=X^1((Y^2 g)\circ f) - X^2((Y^1 g)\circ f)\\
&=X^1(X^2(g\circ f)) - X^2(X^1(g\circ f))\\
&=[X^1,X^2](g\circ f)
\end{aligned}
\end{equation}

In order to obtain a local expression for $[X,Y]$, fix $p\in M$ and consider a chart $(U,\varphi=(x_1,\dots,x_n))$. According to Proposition \ref{smoothvecfields}, there exist functions $a_i,b_j\in C^\infty(M)$ such that $X=\sum_{i=1}^n a_i \frac{\partial}{\partial x_i}$ and $Y=\sum_{j=1}^n b_j \frac{\partial}{\partial x_j}$. Therefore
\begin{align*}
XY &= X \left(\sum_{j=1}^{n}b_j \frac{\partial}{\partial x_j} \right) \\
&= \sum_{i=1}^{n}a_i \frac{\partial}{\partial x_i} \left(\sum_{j=1}^{n}b_j \frac{\partial}{\partial x_j} \right)\\
&= \sum_{i=1}^{n}\sum_{j=1}^{n}a_i \left(\frac{\partial b_j}{\partial x_i}\frac{\partial}{\partial x_j}+b_j\frac{\partial^2}{\partial x_i \partial x_j} \right)\\
&= \sum_{i,j=1}^{n}a_i\frac{\partial b_j}{\partial x_i}\frac{\partial}{\partial x_j} + \sum_{i,j=1}^{n}a_ib_j\frac{\partial^2}{\partial x_i \partial x_j}, \\
\intertext{and similarly,}
YX &= \sum_{i,j=1}^{n}b_j\frac{\partial a_i}{\partial x_j}\frac{\partial}{\partial x_i} + \sum_{i,j=1}^{n}a_ib_j\frac{\partial^2}{\partial x_i \partial x_j}.
\end{align*}
Applying the definition and expressions above, we recover the well-known local expression for the Lie bracket,
\begin{eqnarray}\label{liebracket}
[X,Y]|_U &=& \sum_{i,j=1}^{n} \left(a_i\frac{\partial b_j}{\partial x_i}-b_i\frac{\partial a_j}{\partial x_i} \right)\frac{\partial}{\partial x_j}.
\end{eqnarray}

\begin{remark}
If $M\subset\R^n$ is an open set, then writing $X=(x_1,\dots,x_n)$, $Y=(y_1,\dots,y_n)$, we get $[X,Y]=D_X Y-D_Y X$, where $D$ is the usual derivative of maps from $\R^n$ to itself. In addition, if $L$ is the Lie derivative of tensor fields, then $L_X Y=[X,Y]$.
\end{remark}

\section{Foliations and the Frobenius Theorem}

In this section we recall some facts about flows, distributions and foliations, culminating in Frobenius Theorem.

\begin{definition}
Let $X\in\mathfrak{X}(M)$. An {\it integral curve} of $X$ is a smooth curve $\alpha$ in $M$, with $\alpha'(t)=X(\alpha(t))$ for each $t$ in $\alpha$'s domain.
\end{definition}

A classic result of ODEs is existence and uniqueness of integral curves with prescribed initial data.

\begin{theorem}
Let $X\in\mathfrak{X}(M)$ and $p\in M$. There exists a unique maximal integral curve of $X$, $\alpha_p:(-\tilde{\delta},\tilde{\delta})\rightarrow M$, with $\alpha_p(0)=p$. Maximality is in the sense that every other integral curve satisfying the same initial condition is a restriction of $\alpha_p$ to an open sub interval of $I_{p}=(-\tilde{\delta},\tilde{\delta}).$
\end{theorem}

Let $\mathcal{D}(X)$ be the set of all points $(t,x)\in \R\times M$ such that $t$ lies in $I_{x}$. Then we may define a map
\begin{eqnarray*}
\varphi^{X}:\mathcal{D}(X) &\longrightarrow& M\\
(t,x) &\longmapsto& \alpha_{x}(t),
\end{eqnarray*}
where $\alpha_{x}$ is the integral curve of $X$ with $\alpha_x(0)=x$. This map $\varphi^{X}$ is called the {\it flow}\index{Vector field!flow} determined by $X$ and $\mathcal{D}(X)$ is its {\it domain of definition}.

\begin{theorem}\label{flow}
Let $X\in\mathfrak{X}(M)$. Then $\mathcal{D}(X)$ is an open subset of $\R\times M$ that contains $\{0\}\times M$ and $\varphi^{X}$ is smooth.
\end{theorem}

It is also usual to consider the {\it local flow}, i.e., a restriction of $\varphi^{X}$ to an open subset $(-\delta,\delta)\times U$ of $\mathcal{D}(X)$. It may be useful to fix the time parameter $t$ and deal with local diffeomorphisms $\varphi^{X}_{t}$, defined by $\varphi_{t}^{X}(p)=\varphi^X(t,p)$. It is important to note that, for $x\in U$ such that $\varphi_{t}^{X}(x)\in U$, if $|s|<\delta$, $|t|<\delta$ and $|s+t|<\delta$, then $\varphi_{s}^{X}\circ \varphi_{t}^{X}=\varphi_{s+t}^{X}$.

A smooth vector field $X$ is {\it complete} if $\mathcal{D}(X)=\R\times M$ and in this case $\varphi_{t}^{X}$ form a group of diffeomorphisms parameterized by $t$, called the {\it $1$--parameter group of $X$}.\index{Vector field!$1$--parameter group} In this context, there is an action of $\R$ in $M$, $$\R\times M\owns (t,q)\longmapsto t\cdot q= \varphi_{t}^{X}(q)\in M.$$ Typical examples of complete vector fields are fields with bounded length. For instance, if $M$ is compact, every smooth vector field on $M$ is complete.

\begin{definition}
Let $M$ be a $(n+k)$--dimensional smooth manifold. A $n$--dimensional {\it distribution}\index{Distribution} $D$ on $M$ is an assignment of a $n$--dimensional subspace $D_p\subset T_pM$ for each $p\in M$. The distribution is said to be {\it smooth} if for all $p\in M$, there is an open neighborhood of $p$ such that there exist $n$ smooth vector fields spanning $D$ at each point of $U$.

A distribution is {\it involutive}\index{Distribution!involutive} if for every $p\in M$, given vector fields $X,Y$ in an open neighborhood $U$ of $p$, with $X_q,Y_q\in D_q$, for all $q \in U$, then $[X,Y]_q\in D_q$, for all $q\in U$ as well. That is, a distribution is involutive if it is closed with respect to the bracket of vector fields.
\end{definition}

\begin{definition}\label{definition-foliation}
Let $M$ be a $(n+k)$--dimensional manifold. A {\it $n$--dimensional foliation}\index{Foliation} of $M$ is a partition $\F$ of $M$ by connected immersed $n$--dimensional submanifolds called {\it leaves},\index{Leaf}\index{Foliation!leaf} such that the following property holds. For each $p\in M$ and $v\in T_p L_p$, where $L_p\in\mathcal{F}$ is the leaf that contains $p$, there exists a smooth vector field $X$ on $M$ such that $X(p)=v$ and $X(q)\in T_qL_q$, for all $q\in M$.
\end{definition}

A trivial example of foliation is $\R^m=\R^{n}\times\R^{m-n}$, where the leaves are subspaces of the form $\R^n\times\{p\}$, with $p\in \R^{m-n}$. More generally, the partition by preimages of a smooth submersion is a foliation. Another example of foliation is the partition by orbits of an action, so that each orbit has dimension $n$ (see Chapter \ref{chap3}). The interested reader can find more details in Camacho and Neto \cite{camacho} and Lawson \cite{lawson}.

\begin{definition}
Let $\mathcal{F}=\{L_p:p\in M\}$ be a $n$--dimensional foliation of a smooth $(n+k)$--dimensional manifold $M$. It is possible to prove that for all $p\in M$, there exist an open neighborhood $U$ of $p$, an open neighborhood $V$ of $0\in\R^{n+k}=\R^n\times\R^k$ and a diffeomorphism $\psi:V\rightarrow U$,  such that $\psi(V\cap (\R^n\times\{y_0\}))$ is the connected component of $L_{\psi(0,y_0)}\cap U$ that contains $\psi(0,y_0)$. The open set $U$ is called \emph{simple open neighborhood}\index{Simple open neighborhood}, the diffeomorphisms $\psi$ \emph{trivialization}\index{Trivialization} and $\psi^{-1}$ \emph{foliation chart} and the submanifold $\psi(V\cap (\R^n\times\{y_0\}))$ is called a \emph{plaque}\index{Plaque}.
\end{definition}

\begin{remark} 
In the general case of singular foliations, we usually do not have trivializations. Nevertheless, plaques can be defined at least for singular Riemannian foliations as follows. Let $P_{q}$ be a relatively compact open subset of $L_{q}$ and $z$ a point in a tubular neighborhood $\tub(P_q)$. Since $\F$ is a singular Riemannian foliation, each leaf $L_{z}$ is contained in a stratum (the collection of leaves with the same dimension of $L_{z}$). In this stratum, $L_{z}$ is a regular leaf and hence admits plaques as defined above. A \emph{plaque} $P_{z}$ of the singular leaf $L_{z}$ is then defined as a plaque of $L_{z}$ in its stratum.
\end{remark}

\begin{frobthm}\label{frobbis}\index{Theorem!Frobenius}
Let $M$ be a smooth manifold and $D$ a smooth $n$--dimensional involutive distribution on $M$. Then there exists a unique $n$--dimensional foliation $\mathcal{F}=\{L_p:p\in M\}$ satisfying $$D_q=T_qL_q, \quad q\in M.$$ Furthermore, each leaf $L_p$ is quasi--embedded.
\end{frobthm}

\section{Differential forms and integration}
\label{sec:appdiffint}

To be updated. References for this topics are Bredon \cite{bredon}, Spivak \cite{spivakzinho} and Warner \cite{warner}.

\backmatter

\clearpage
\bibliographystyle{amsplain}

\begin{thebibliography}{500}
\bibitem{ah}{{\sc D. Ahiezer} \textit{Lie group actions in complex analysis}, Vieweg, Braunschweig, Wiesbaden, 1995.}
\bibitem{Alex1}{{\sc M. M. Alexandrino}, \textit{Integrable Riemannian submersion with singularities}, Geom.\ Dedicata, \textbf{108} (2004), 141--152.}
\bibitem{Alex2}{{\sc M. M. Alexandrino}, \textit{Singular Riemannian foliations with sections}, Illinois J.\ Math.\ \textbf{48} (2004) 4, 1163--1182.}
\bibitem{Alex3}{{\sc M. M. Alexandrino}, \textit{Generalizations of isoparametric foliations}, Mat.\ Contemp.\ \textbf{28} (2005), 29--50.}
\bibitem{Alex4}{{\sc M. M. Alexandrino}, \textit{Proofs of conjectures about singular Riemannian foliations} Geom.\ Dedicata \textbf{119} (2006) 1, 219--234.}
\bibitem{escola}{{\sc M. M. Alexandrino, L. Biliotti \and R. Pedrosa}, {\it Lectures on isometric actions}, XV Escola de Geometria Diferencial, Publica\c c\ \hspace{-0.13cm}\~oes matem\'aticas, IMPA (2008)}
\bibitem{AlexToeben}{{\sc M. M. Alexandrino \and T. T\"{o}ben}, \textit{Singular Riemannian foliations on simply connected spaces}, Differential Geom.\ and Appl.\ \textbf{24} (2006) 383--397.}
\bibitem{AlexGorodski}{{\sc M. M. Alexandrino \and C. Gorodski}, \emph{Singular Riemannian foliations with sections, transnormal maps and basic forms}, Annals of Global Analysis and Geometry \ \textbf{32} (3) (2007), 209--223.}
\bibitem{Alex5}{{\sc M. M. Alexandrino}, \emph{Singular holonomy of singular Riemannian foliations with sections}, Mat.\ Contemp.\ \textbf{33} (2007), 23--55.}
\bibitem{AlexToeben2}{{\sc M. M. Alexandrino \and T. T\"{o}ben}, \emph{Equifocality of singular Riemannian foliation},
Proc. Amer. Math. Soc. \textbf{136} (2008), 3271--3280.}
\bibitem{Alex6}{{\sc M. M. Alexandrino}, \emph{Desingularization of singular Riemannian foliation}. To appear in Geom. Dedicata. DOI 10.1007/S 10711-010-9489-4.} 
\bibitem{at}{{\sc M. Atiyah}, \textit{Convexity and commuting Hamiltonians},Bull. London Math. Soc. \textbf{14} (\textbf{1}) (1982), 1--15.}
\bibitem{ar}{{\sc V. I. Arnold}, \textit{Mathematical methods of classical Mechanics}, 2nd edition, Graduate Texts in Mathematics \textbf{60}, Springer--Verlag, New-York, 1989.}
\bibitem{arvanitoyeorgos}{{\sc A. Arvanitoyeorgos}, \emph{An introduction to Lie groups and the geometry of homogeneous spaces}, AMS (1999).}
\bibitem{cant}{{\sc M. Aubin, A. Cannas da Silva \and E. Lerman}, \emph{Symplectic geometry of integrable Hamiltonian systems}, Advanced Courses in Mathematics, CRM Barcelona, Birkh\"{a}user Verlag, Barcelona, 2001.}
\bibitem{lb}{{\sc L. Bates \and E. Lerman}, \emph{Proper group actions and symplectic stratified spaces}, {Pacific J. Math.} \textbf{181} (1997), 201--229.}
\bibitem{br}{{\sc C. Benson \and G. Ratcliff}, \emph{A classification of multiplicity free actions}, J. Algebra \textbf{181} (1996), 152--186.}
\bibitem{besse}{{\sc A. L. Besse}, {\it Einstein Manifolds}, Springer--Verlag, Series of Modern Surveys in Mathematics (2002).}
\bibitem{BerndtConsoleOlmos}{{\sc J. Berndt, S. Console \and C. Olmos}, \emph{Submanifolds and holonomy}, Research Notes in Mathematics, vol 434, Chapman and Hall/CRC 2003.}
\bibitem{Biliotti}{{\sc L. Biliotti}, \emph{Coisotropic and polar actions on compact irreducible Hermitian symmetric spaces}, Trans. Amer. Math. Soc. \textbf{358} (2006) 3003--3022.}
\bibitem{Bishop}{{\sc R. J. Bishop \and R. J. Crittenden}, {\it Geometry of manifolds}, AMS Chelsea Publishing (2001).}
\bibitem{Boltner}{{\sc C. Boltner}, \emph{On the structure of equidistant foliations of Euclidean space}, PhD thesis at University of Augsburg, Germany (2007) advised by E. Heintze, arXiv:0712.0245.}
\bibitem{boothby}{{\sc W. M. Boothby}, {\it An introduction to Differentiable Manifolds and Riemannian Geometry}, 2nd edition, Academic Press (2003).}
\bibitem{Boualem}{{\sc H. Boualem}, \textit{Feuilletages riemanniens singuliers transversalement integrables.} Compositio Mathematica \textbf{95} (1995), 101--125.}
\bibitem{bredon}{{\sc G. Bredon}, \emph{Topology and Geometry}, Graduate Texts in Mathematics \textbf{139}, Springer--Verlag, New York 1993.}
\bibitem{bri}{{\sc M. Brion}, \emph{Sur l'image de l'application moment}, Lecture Notes Math. vol 1296, Springer, Berlin Heidelberg New York, 1987.}
\bibitem{bri2}{{\sc M. Brion}, \emph{Classification des espaces homogenes sph\'erique}, Compositio Math. \textbf{63} (1987), 189--208.}
\bibitem{blv}{{\sc M. Brion, D. Luna \and Th. Vust}, \emph{Espaces Homog\`enes sph\'eriques}, Invent. Math. \textbf{84} (1986), 617--632.}
\bibitem{bt}{{\sc T. Br\"{o}ken \and T. tom Dieck}, \emph{Representations of compact Lie group}, Graduate Texts in Mathematics \textbf{98}, Springer--Verlag, New York 1995.}
\bibitem{Bryant}{{\sc R. L. Bryant}, \emph{An introduction to Lie groups and symplectic geometry}, Geometry and quantum field theory, IAS/Park City, Mathematics Series, \textbf{1} American Mathematical Society (1995).}
\bibitem{gbl}{{\sc D. Burns, V. Guillemin \and E. Lerman}, \emph{K\"ahler cuts}, Preprint, arXiv:0212062 [math.DG].}
\bibitem{camacho}{{\sc C. Camacho \and A. L. Neto}, {\it Teoria geom\'etrica das folhea\c{c}\~oes}, Projeto Euclides, IMPA (1979).}
\bibitem{cannas}{{\sc A. Cannas da Silva}, \emph{Lectures on symplectic geometry}, Lecture Notes in Math. \textbf{1764}, Springer--Verlag, Berlin, 2001.}
\bibitem{manfredo}{{\sc M. P. do Carmo}, {\it Geometria Riemanniana}, Projeto Euclides, 3rd edition, IMPA (2005).}
\bibitem{Cartan1}{{\sc \'{E}. Cartan}, \textit{Familles de surfaces isoparam\'{e}triques dans les espaces \'{a} courbure constante,}
Ann. di Mat. \textbf{17} (1938) 177--191.}
\bibitem{Cartan2}{{\sc \'{E}. Cartan}, \emph{Sur des familles remarquables d'hypersurfaces isopara\-m\'{e}triques dans les espaces sh\'{e}riques,} Math. Z. \textbf{45} (1939), 335--367.}
\bibitem{Cartan3}{{\sc \'{E}. Cartan}, \emph{Sur quelques familles remarquables d'hypersurfaces,} C.R. Congr\`{e}s Math. Li\`{e}ge (1939), 30--41.}
\bibitem{Cartan4}{{\sc \'{E}. Cartan}, \emph{Sur des familles d'hypersurfaces isoparam\'{e}triques des espaces sph\'{e}riques \`{a} 5 et \`{a} 9 dimensions,} Revista Univ. Tucum\'{a}n 1 (1940), 5--22.}
\bibitem{carterSegalMacdonald}{{\sc R. Carter, G. Segal \and I. MacDonald}, {\it Lectures on Lie groups and Lie algebras}, London Mathematical Society, Student Texts {\bf 32}, Cambridge University Press (1995).}
\bibitem{CarterWest1}{{\sc S. Carter \and A. West}, \emph{Isoparametric systems and transnormality}, Proc. London Math. Soc. \textbf{51} (1985), 520--542.}
\bibitem{CarterWest2}{{\sc S. Carter \and A. West}, \emph{Generalised Cartan polynomials}, J. London Math. Soc. \textbf{32} (1985), 305--316.}
\bibitem{CE}{{\sc J. Cheeger \and D. Ebin}, \emph{Comparison theorems in Riemannian geometry}, North--Holland Publishing Co., Amsterdam, 1975.}
\bibitem{ch}{{\sc R. Chiang}, \emph{New Lagrangian submanifolds of $\CP^n$}, Int. Math. Soc. Res. Not.  \textbf{45} (2004), 2437--2441.}
\bibitem{Dadok}{{\sc J. Dadok}, \emph{Polar coordinates induced by actions of compact Lie groups}, Trans. Amer. Math. Soc. \textbf{288} (1985), 125--137.}
\bibitem{Deitmar}{{\sc A. Deitmar}, \emph{A first course in harmonic analysis}, Springer--Verlag, Universitext, Second Edition, 2005.}
\bibitem{de}{{\sc T. Delzant}, \emph{Classification des actions hamiltonien\-nes com\-pl\'e\-te\-me\-nt int\'egrables de rang deux}, Bull. Soc. Math. France \textbf{116} (\textbf{3}), (1988) 315--339.}
\bibitem{duistermaat}{{\sc J. J. Duistermaat \and J. A. C. Kolk}, {\it Lie Groups}, Springer--Verlag, Universitext, (2000).}
\bibitem{Fegan}{{\sc H. D. Fegan}, \emph{Introduction to compact Lie groups}, Series in Pure Mathematics, \textbf{13}, World Scientific, (1998).}
\bibitem{Feher-Pusztai-1}{{\sc L. Feh\'er \and B. G. Pusztai}, \emph{Twisted spin Sutherland models from quantum Hamiltonian reduction}, Phys. A: Math. Theor. 41 (2008).}
\bibitem{Feher-Pusztai-2}{{\sc L. Feh\'er \and B. G. Pusztai}, \emph{Hamiltonian reductions of free particles under polar actions of compact Lie groups} Theor. Math. Phys. 155 (2008).}
\bibitem{gov2}{{\sc B. L. Feigin, D. B. Fuchs, V. V. Gorbatsevich, O. V. Schvartsman, \and E. B. Vinberg}, \emph{Lie groups and Lie algebras II}, Springer--Verlag, Encyclopaedia of Mathematical Sciences vol 21, Moscow (1988).}
\bibitem{FerusKarcherMunzner}{{\sc D. Ferus, H. Karcher \and H. F. M\"{u}nzner}, \emph{Cliffordalgebren und neue isoparametrische Hyperfl\"{a}chen} Math. Z. \textbf{177} (1981), 479--502.}
\bibitem{Futaki}{{\sc A. Futaki,}, \emph{The Ricci curvature of symplectic quotients of Fano manifolds}, Tohoku Math. J. (\textbf{2}) \textbf{39} (1987), no. 3, 329--339.}
\bibitem{Gangolli-Varadarajan}{{\sc R. Gangolli \and V. S. Varadarajan}, \emph{Harmonic analysis of spherical functions on real reductive groups} Springer--Verlag, Ergebnisse der Mathematik und ihrer Grenzgebiete \textbf{101} (1988).}
\bibitem{liedover}{{\sc R. Gilmore}, \emph{Lie groups, Lie algebras and some of their applications}, Dover (2006).}
\bibitem{gov0}{{\sc V. V. Gorbatsevich, A. L. Onishchik \and E. B. Vinberg}, \emph{Foundations of Lie Theory and Lie Transformation Groups}, Springer, Moscow (1988).}
\bibitem{gov1}{{\sc V. V. Gorbatsevich, A. L. Onishchik \and E. B. Vinberg}, \emph{Lie groups and Lie algebras I}, Springer--Verlag, Encyclopaedia of Mathematical Sciences vol 20, Moscow (1988).}
\bibitem{gov3}{{\sc V. V. Gorbatsevich, A. L. Onishchik, \and E. B. Vinberg}, \emph{Lie groups and Lie algebras III}, Springer--Verlag, Encyclopaedia of Mathematical Sciences vol 41, Moscow (1990).}
\bibitem{gp}{{\sc A. Gori \and F. Podest\`a}, \emph{A note on the moment map on compact K\"ahler manifold}, Ann. Global. Anal. Geom. \textbf{26} (\textbf{3}) 2004, 315--318.}
\bibitem{pog}{{\sc C. Gorodski \and F. Podest\`a}, \emph{Homogeneity rank of real representations of compact Lie groups}, J. Lie Theory \textbf{15} (\textbf{1}) 2005, 63--77.}
\bibitem{Guest}{{\sc M. A. Guest}, \emph{Harmonic maps, loop groups and integrable systems}, London Mathematical Society, Student texts, \textbf{38}, 1997.}
\bibitem{gsc}{{\sc V. Guillemin \and S. Sternberg},\emph{Convexity properties of the moment map mapping}, Invent. Math. {\textbf 67} 1982, 491--513.}
\bibitem{gsc2}{{\sc V. Guillemin \and S. Sternberg,} \emph{Convexity properties of the moment map mapping II}, Invent. Math.  \textbf{77} (\textbf{3}) 1984, 533--546.}
\bibitem{GS}{{\sc V. Guillemin \and S. Sternberg}, \emph{Multiplicity--free spaces}, J. Differential Geom. \textbf{19} (1984), 31--56}.
\bibitem{gs}{{\sc V. Guillemin \and S. Sternberg}, \emph{Symplectic techniques in physics}, 2nd ediction, Cambridge University Press, Cambridge, 1990.}
\bibitem{Hall}{{\sc B. C. Hall}, \emph{Lie groups, Lie algebras and representations, an elementary introduction}, Springer--Verlag, Graduate Texts in Mathematics, (2004).}
\bibitem{Harle}{{\sc C. E. Harle}, \emph{Isoparametric families of submanifolds,} Bol. Soc. Brasil. Mat. \textbf{13} (1982) 491--513.}
\bibitem{hl}{{\sc R. Harvey \and H. B. Lawson}, \emph{Calibrates geometries}, Acta Math. \textbf{148} (1982), 47--157.}
\bibitem{HOL}{{\sc E. Heintze, X. Liu \and C. Olmos}, \emph{Isoparametric submanifolds and a Chevalley--type restriction theorem}, arXiv:0004028 [math.DG].}
\bibitem{helgason}{{\sc S. Helgason}, \emph{Differential geometry, Lie groups and symmetric spaces}, Graduate Studies in Mathematics, vol 34, AMS (2001).}
\bibitem{HM}{{\sc R. Howe \and T. Umeda}, \emph{The Capelli identity, the double commutant theorem, and multiplicity--free actions}, Math. Ann. \textbf{290} (1991), 565--619.}
\bibitem{HW}{{\sc A. Huckleberry \and T. Wurzbacher}, \emph{Multiplicity--free complex manifolds}, Math. Ann. \textbf{286} (1990), 261--280.}
 \bibitem{IseTakeuchi}{{\sc M. Ise \and M. Takeuchi}, {\it Lie groups I}, Translations of Mathematical Monographs {\bf 85}, American Mathematical Society (1991).}
\bibitem{Jost}{{\sc J. Jost}, {\it Riemannian geometry and geometric analysis}, Springer--Verlag, Universitext, 2nd edition (1998).}
\bibitem{ka}{{\sc V. G. Kac}, \emph{Some remarks on Nilpotent orbit}, J. Algebra \textbf{64} (1980), 190--213.}
\bibitem{Katznelson}{{\sc Y. Katznelson}, \emph{An introduction to harmonic analysis}, Dover publications, INC (1968).}
\bibitem{livrokawakubo}{{\sc K. Kawakubo}, {\it The theory of transformation groups}, Oxford University Press (1991).}
\bibitem{kir}{{\sc F. Kirwan}, \emph{Convexity properties of the moment map mapping III}, Invent. Math \textbf{77} (\textbf{3}) 1984, 547--552.}
\bibitem{ki}{{\sc F. Kirwan}, \emph{Cohomology quotients in symplectic and algebraic geometry}, Math. Notes \textbf{31}, Princeton, 1984.}
\bibitem{Knapp-representation}{{\sc A. W. Knapp}, \emph{Representation theory of semisimple groups}, Princeton University Press, (1986).}
\bibitem{kb}{{\sc F. Knop \and B. V. Steirteghem}, \emph{Classification of smooth affine spherical varieties}, Transform. Groups \textbf{11} (2006), 495--516.}
\bibitem{kn}{{\sc S. Kobayashi \and N. Nomizu}, \emph{Foundations of Differential Geometry}, Vol. \textbf{I, II} Interscience Publisher, J. Wiles \& Sons, 1963.}
\bibitem{ko}{{\sc T. Kobayashi,} \emph{Lectures on restrictions of unitary representations of real reductive groups}, RIMS Preprint \textbf{1526}, University of
Kyoto, 2005.}
\bibitem{Kollross}{{\sc A. Kollross}, \textit{A classification of hyperpolar and cohomogeneity one actions,} Transactions of The American Mathematical Society.\textbf{354} (2002) 2, 571--612.}
\bibitem{kol2}{{\sc A. Kollross}, \emph{Polar actions on symmetric spaces}, {J. of Differential Geom. \textbf{77} (\textbf{3}) (2007), 425--482.}}
\bibitem{langRiemannianManifolds}{{\sc S. Lang}, {\it Differential and Riemannian Manifolds}, Springer--Verlag, Graduate Texts in Mathematics (1995).}
\bibitem{lawson}{{\sc H. B. Lawson}, {\it Foliations}, Bulletin of the Amer. Math. Soc. {\bf 80}, 3 (1974).}
\bibitem{lee2}{{\sc D.H. Lee}, \emph{The structure of complex Lie groups}, Research Notes in Mathematics \textbf{429}, Chapman \& Hall/CRC, Boca Raton, FL, 2002.}
\bibitem{lee}{{\sc J. M. Lee}, {\it Riemannian manifolds: an introduction to curvature}, Springer-Verlag, Graduate Texts in Mathematics (1997).}
\bibitem{lerman}{{\sc E. Lerman}, \emph{Symplectic cut}, Math. Res. Lett. \textbf{2} (1995), 247--258.}
\bibitem{LytchakThorbergsson}{{\sc A. Lytchak \and G. Thorbergsson}, \emph{Variationally complete actions on nonnegatively curved manifolds}, to appear in Illinois J.\ Math.}
\bibitem{LytchakThorbergsson2}{{\sc A. Lytchak \and G. Thorbergsson}, \emph{Curvature explosion in quotients and applications}, Preprint (2007) arXiv:0709.2607.}
\bibitem{Lytchak1}{{\sc A. Lytchak}, \emph{Singular Riemannian foliations on space without conjugate points}, preprint (2007).}
\bibitem{Lytchak2}{{\sc A. Lytchak}, \emph{Geometric resolution of singular Riemannian foliations}. To appear in Geom. Dedicata, DOI: 10.1007/s10711-010-9488-5.}
\bibitem{Marle}{{C. M. Marle}, \emph{Mod\`ele d'action hamiltonienne d'un groupe de Lie sur une vari\'et\'e symplectique}, Rendiconti del Seminario Matematico, Universit\`a e Politecnico, Torino \textbf{43} (1985), 227--251}
\bibitem{ds}{{\sc D. McDuff \and D. Salamon}, \emph{Introduction to Symplectic Topology}, 2nd edition, Oxford Mathematical Monographs, Oxford University Press, New York, 1998.}
\bibitem{milnorLeftInvariantMetric}{{\sc J. Milnor}, {\it Curvatures of left-invariant metrics on Lie groups}, Advances in Math. {\bf 21}, 3, p. 293--329 (1976).}
\bibitem{gravitation}{{\sc W. Misner, K. S. Thorne \and J. A. Wheeler}, {\it Gravitation}, Freeman and Company (1973).}
\bibitem{Moerdijk}{{\sc I. Moerdijk \and J. Mr\v{c}un}, {\it Introduction to foliations and Lie groupoids}, Cambridge Studies in Advanced Mathematics (2003).}
\bibitem{Molino}{{\sc P. Molino}, \textit{Riemannian foliations}, Progress in Mathematics, vol. 73. Birkh\"{a}user Boston (1988).}
\bibitem{MolinoPierrot}{{\sc P. Molino \and M. Pierrot}, \textit{Th\'{e}or\`{e}mes de slice et holonomie des feuilletages Riemanniens singuliers}, Ann. Inst. Fourier (Grenoble) \textbf{37} (1987) 4, 207--223.}
\bibitem{montgomery}{{\sc D. Montgomery \and L. Zippin}, {\it Topological Transformation Groups}, Interscience Publishers (1955).}
\bibitem{MorganTian}{{\sc J. W. Morgan \and G. Tian}, {\it Ricci Flow and the Poincare' Conjecture}, Clay Mathematics Monographs, Volume {\bf 3}, American Mathematical Society (2007).}
\bibitem{Munzner1}{{\sc H. F. M\"{u}nzner}, \emph{Isoparametrische Hyperfl\"{a}chen in Sph\"{a}ren I}, Math. Ann. \textbf{251} (1980) 57--71.}
\bibitem{Munzner2}{{\sc H. F. M\"{u}nzner}, \emph{Isoparametrische Hyperfl\"{a}chen in Sph\"{a}ren II}, Math. Ann. \textbf{256} (1981) 215--232.}
\bibitem{myerssteenrod}{{\sc S. Myers \and N. Steenrod}, \emph{The group of isometries of a Riemannian manifold}, Annals of Math. \textbf{40} 2 (1939), 400--416.}
\bibitem{ne}{{\sc K. H. Neeb,} \emph{Holomorphy and Convexity in Lie theory}, Expositions in Mathematics, de Gruyter, Berlin, 1999.}
\bibitem{Noumi}{{\sc M. Noumi}, \emph{Painlev\'{e} Equations through symmetry}, AMS, Translations of Mathematical Monographs, \textbf{223} (2004).}
\bibitem{Eva}{{\sc E. Nowak}, \emph{Singular Riemannian Foliations: Exceptional Leaves; Tautness} PhD thesis at University of Koeln, Germany (2008) advised by G. Thorbergsson, arXiv:0812.3316v1 [math.DG].}
\bibitem{oh2}{{\sc Y. G. Oh}, \emph{Second variation and stabilities of minimal Lagrangian subma\-ni\-folds in K\"ahler ma\-ni\-folds}, Invent. Math. \textbf{101} (1990), 501--519.}
\bibitem{oh}{{\sc Y. G. Oh}, \emph{Mean curvature vector and symplectic topology of La\-gra\-ngi\-an submanifolds in Einstein-K\"ahler manifolds}, Math. Z. \textbf{216} (1994), 471--482.}
\bibitem{Olver}{{\sc P. J. Olver}, \emph{Applications of Lie groups to differential equations}, Springer--Verlag, Graduate Texts in Mathematics, \textbf{107}, Second Edition, (2000).}
\bibitem{or}{{\sc J. P. Ortega \and T. S. Ratiu}, \emph{A symplectic slice theorem}, Lett. Math. Phys. \textbf{59} (2002), 81--93.}
\bibitem{Palais-partition}{{\sc R. S. Palais}, \emph{On the existence of slices for actions of noncompact groups}, Ann. of Math. 73 (1961), 295--323.}
\bibitem{PT1}{{\sc R. S. Palais \and C. L. Terng}, \emph{A general theory of canonical forms}, Trans. Am. Math. Soc. \textbf{300} (1987), 236--238.}
\bibitem{PalaisTerng}{{\sc R. S. Palais \and C. L. Terng}, \textit{Critical point theory and submanifold geometry}, Lecture Notes in Mathematics 1353, Springer-Verlag.}
\bibitem{petersen}{{\sc P. Petersen}, {\it Riemannian Geometry}, 2nd edition, Springer-Verlag, Graduate Texts in Mahematics (2000).}
\bibitem{DiegoMestrado}{{\sc D. A. C. Pi\~{n}eros}, \textit{Sobre as folhea\c{c}\~{o}es e o teorema de slice para folhea\c{c}\~{o}es Riemannianas singulares com se\c{c}\~{o}es}, MSc dissertation at University of S\~{a}o Paulo, Brazil 2008, advised by M. M. Alexandrino.}
\bibitem{PodestaThorbergsson}{{\sc F. Podest\`{a} \and G. Thorbergsson}, \textit{Polar actions on rank one symmetric spaces,} J. Differential Geometry. \textbf{53} (1999) 1, 131--175.}
\bibitem{PoT}{{\sc F. Podest\`a \and G. Thorbergsson}, \emph{Polar and coisotropic actions on K\"ahler manifolds}, Trans. Am. Math. Soc. \textbf{354} (2002), 236--238.}
\bibitem{pt}{{\sc P. Piccione \and D. Tausk}, \emph{On the geometry Grassmannians and the symplectic group: the Maslov index and its applications},
Publica\c{c}\~oes do IMPA, 11th School of Differential Geometry, Universidade Federal Fluminense, Niter\'oi (RJ), 2000.}
\bibitem{pue}{{\sc T. Puttmann}, \emph{Homogeneity rank and atoms actions}, An. Global Anal. Geom \textbf{22} (2002), 375--399.}
\bibitem{mau}{{\sc L. Pukanszky}, \emph{Unitary representations of solvable groups}, Ann. Sci. \'Ecole Norm. Sup. \textbf{4} (1972), 475--608.}
\bibitem{SanMartin}{{\sc L. A. B. San Martin}, \emph{\'Algebras de Lie}, Editora da Unicamp, (1999).}
\bibitem{shur}{{\sc L. Schur}, \emph{\"Uber eine Klasse Von Mittelbildungen mit anwedungen auf die determinantentheorie},
Sitzungsberichte der berliner mathematishcen gesenlschaft \textbf{22} (1923), 9--20.}
\bibitem{sl}{{\sc R. Sjamaar \and E. Lerman}, \emph{Stratified symplectic spaces and reduction}, Ann. of Math. \textbf{134} (\textbf{2}) (1991),
375--422.}
\bibitem{Serre}{{\sc J. P. Serre}, \emph{Complex semisimple Lie algebras}, Springer-Verlag, Monographs in Mathematics (2001).}
\bibitem{lucas}{{\sc F. L. N. Spindola}, {\it Grupos de Lie, a\c{c}\~oes pr\'oprias e a conjectura de Palais--Terng}, MSc. dissertation in Mathematics, advised by M. M. Alexandrino, Universidade de S\~ao Paulo IME USP (2008).}
\bibitem{spivakzinho}{{\sc M. Spivak}, {\it Calculus on manifolds: a modern approach to classical theorems of advanced calculus},
Perseus Book Publishing, 1965.}
\bibitem{spivak1}{{\sc M. Spivak}, {\it A Comprehensive Introduction to Differential Geometry}, Vol 1, Publish or Perish (1975).}
\bibitem{Szenthe}{{\sc J. Szenthe}, \textit{Orthogonally transversal submanifolds and the generalizations of the Weyl group}, Periodica Mathematica Hungarica \textbf{15} (1984) 4, 281--299.}
\bibitem{ta}{{\sc M. Takeuchi}, \emph{Stability of certain minimal submanifolds of compact Hermitian symmetric spaces}, Tohoku Math. J. \textbf{36} (1984), 293--314.}
\bibitem{Terng}{{\sc C. L. Terng}, \emph{Isoparametric submanifolds and their Coxeter groups,} J. Differential Geometry, \textbf{21} (1985) 79--107.}
\bibitem{TTh1}{{\sc C. L. Terng \and G. Thorbergsson}, \textit{Submanifold geometry in symmetric spaces,} J. Differential Geometry, \textbf{42} (1995), 665--718.}
\bibitem{Th2}{{\sc G. Thorbergsson}, \textit{Isoparametric foliations and their buildings}, Ann. Math. \textbf{133} (1991), 429--446.}
\bibitem{ThSurvey1}{{\sc G. Thorbergsson}, \textit{A survey on isoparametric hypersurfaces and their generalizations}, Handbook of Differential Geometry, Vol 1, Elsevier Science B.V. (2000).}
\bibitem{ThSurvey2}{{\sc G. Thorbergsson}, \emph{Transformation groups and submanifold geometry}, Rendiconti di Matematica, Serie VII, \textbf{25}, (2005) 1--16.}
\bibitem{Toeben}{{\sc D. T\"oben}, \textit{Parallel focal structure and singular Riemannian foliations}, Trans.\ Amer.\ Math.\ Soc.\ \textbf{358} (2006), 1677--1704.}
\bibitem{Toeben2}{{\sc D. T\"{o}ben}, \textit{Singular Riemannian foliations on nonpositively curved manifolds}, Math. Z. \textbf{255} (2) (2007), 427--436.}
\bibitem{Varadarajan-harmonic}{{\sc V. S. Varadarajan}, \emph{An introduction to harmonic analysis on semisimple Lie groups}, Cambridge University Press, Cambridge studies in advanced mathematics \textbf{16} (1999).}
\bibitem{varadarajan}{{\sc V. S. Varadarajan}, {\it Lie Groups, Lie Algebras, and Their Representations}, Springer--Verlag, Graduate Texts in Mathematics (1984).}
\bibitem{Wang1}{{\sc Q. M. Wang}, \emph{Isoparametric functions on Riemannian manifolds. I}, Math. Ann. \textbf{277} (1987), 639--646.}
\bibitem{warner}{{\sc F. W. Warner}, {\it Foundations of Differentiable Manifolds and Lie Groups}, Springer--Verlag, Graduate Texts in Mathematics (1983).}
\bibitem{we}{{\sc A. Weinstein}, \emph{Symplectic manifolds and their Lagrangian submanifolds}, Adv. in Math. \textbf{6} (1971), 329--346.}
\bibitem{wu}{{\sc T. Wurzbacher}, \emph{On a conjecture of Guillemin and Sternberg in geometric quantization of multiplicity--free symplectic spaces}, J. Geom. Phys. \textbf{7} (\textbf{4}) (1990), 537--552 (1991).}
\end{thebibliography}
\phantomsection

\clearpage
\phantomsection
\printindex

\end{document}